\documentclass[twoside,a4paper,10.5pt]{article}
\usepackage{amsthm,amsmath,amssymb,amscd,graphicx,amsfonts,stmaryrd}
\usepackage{enumitem}
\usepackage{mathrsfs}
\usepackage[all]{xy}
\usepackage[margin=27mm]{geometry}
\usepackage{bigints}
\usepackage{mathtools}
\usepackage{enumitem}
\usepackage{comment}
\usepackage{fancyhdr}
\usepackage{bm}
\usepackage{tikz-cd}
\usepackage[english]{babel}
\usepackage{latexsym}
\usepackage{tikz}
\usetikzlibrary{intersections,calc,arrows.meta}
\usepackage{hyperref}

\setlist{topsep = 1.5ex, itemsep = 0.6ex} 

\allowdisplaybreaks[4]
\thepage

\mathtoolsset{showonlyrefs=true}

\numberwithin{equation}{section}

\theoremstyle{definition}
\newtheorem{exm}{Example}[section]
\newtheorem{dfn}[exm]{Definition}
\newtheorem{rem}[exm]{Remark}

\theoremstyle{plane}
\newtheorem{lem}[exm]{Lemma}
\newtheorem{prop}[exm]{Proposition}
\newtheorem{thm}[exm]{Theorem}
\newtheorem{cor}[exm]{Corollary}
\newtheorem{assum}[exm]{Assumption}

\DeclareMathOperator\id{id}
\DeclareMathOperator\IdFunct{Id}

\DeclareMathOperator\dom{dom}
\DeclareMathOperator\Hom{Hom}
\DeclareMathOperator{\ob}{Ob}
\DeclareMathOperator{\closure}{cl}
\DeclareMathOperator{\diam}{diam}

\newcommand{\NN}{\mathbb{N}}
\newcommand{\ZN}{\mathbb{Z}}
\newcommand{\ZNp}{\mathbb{Z}_{+}}

\newcommand{\RN}{\mathbb{R}} 
\newcommand{\RNp}{\mathbb{R}_{\geq 0}}
\newcommand{\RNpp}{\mathbb{R}_{>0}}

\newcommand{\bcmAB}{\mathsf{bcm}}

\newcommand{\Compact}[1]{\mathcal{C}_c(#1)}
\newcommand{\Closed}[1]{\mathcal{C}(#1)}
\newcommand{\HausMet}[1]{d_{#1}^H}
\newcommand{\lHausMet}[1]{d_{#1}^{F}}

\newcommand{\Image}[1]{\mathrm{Im}_{#1}}

\newcommand{\Prob}[1]{\mathcal{P}(#1)}
\newcommand{\finMeas}[1]{\mathcal{M}_{\mathrm{fin}}(#1)}
\newcommand{\cptMeas}[1]{\mathcal{M}_{\mathrm{cpt}}(#1)}
\newcommand{\Meas}[1]{\mathcal{M}(#1)}
\newcommand{\ProhMet}[1]{d_{#1}^P}
\newcommand{\Vague}[1]{d_{#1}^{V}}

\newcommand{\graphmap}{\mathfrak{g}}

\newcommand{\cC}{\mathcal{C}}

\newcommand{\frakC}{\mathfrak{C}}
\newcommand{\frakD}{\mathfrak{D}}

\newcommand{\rootedBCM}{\mathfrak{M}_{\bullet}}
\newcommand{\BCM}{\mathfrak{M}}

\newcommand{\frakK}{\mathfrak{K}}
\newcommand{\rfrakK}{\mathfrak{K}_\bullet}
\newcommand{\cX}{\mathcal{X}}
\newcommand{\cY}{\mathcal{Y}}

\newcommand{\MTopcat}{\mathbf{MTop}}  
\newcommand{\Metcat}{\mathbf{Met}}
\newcommand{\BCMcat}{\mathbf{BCM}}
\newcommand{\CMcat}{\mathbf{CM}}
\newcommand{\rBCMcat}{\mathbf{rBCM}}
\newcommand{\rCMcat}{\mathbf{rCM}}

\newcommand{\CatC}{\mathscr{C}}
\newcommand{\CatD}{\mathscr{D}}
\newcommand{\CatE}{\mathscr{E}}

\newcommand{\FMet}[1]{#1^{\mathrm{m}}}
\newcommand{\FS}[1]{#1^{\mathrm{sr}}}
\newcommand{\FSMet}[1]{#1^{\mathrm{srm}}}
\newcommand{\FE}[1]{#1^{\mathrm{er}}}
\newcommand{\FEMet}[1]{#1^{\mathrm{erm}}}

\newcommand{\ForgetRoot}{\Gamma_1}
\newcommand{\BCInclusion}{\Gamma_2}
\newcommand{\ForgetMetric}{\Gamma_3}

\newcommand{\STFunct}{\Psi}

\newcommand{\RootSystem}{\mathfrak{r}}

\newcommand{\MeasFunct}{\tau_{\mathcal{M}}}
\newcommand{\finMeasFunct}{\tau_{\mathcal{M}_{\mathrm{fin}}}}

\newcommand{\Dist}[1]{\Theta_{#1}}

\newcommand{\RFMet}{d_{\rootedBCM}}
\newcommand{\cRFMet}{d_{\rfrakK}}
\newcommand{\RVMet}{\check{d}_{\rootedBCM}}
\newcommand{\cRVMet}{\check{d}_{\rfrakK}}

\title{Metrization of Gromov--Hausdorff-type topologies on boundedly-compact metric spaces} 
\date{}
\author{Ryoichiro Noda\thanks{Research Institute for Mathematical Sciences, Kyoto University, Kyoto, 606-8502,
JAPAN. E-mail: sgrndr@kurims.kyoto-u.ac.jp}}

\begin{document}

\maketitle
\begin{abstract}
  We present a new general framework for metrization of Gromov--Hausdorff-type topologies on non-compact metric spaces equipped with additional structures.
  We also give easy-to-check conditions for Polishness
  and hence the measure theoretic requirements are provided
  to study convergence of random spaces with additional random objects.
  In particular,
  our framework enables us to define a metric inducing a suitable Gromov--Hausdorff-type topology 
  on the space of rooted $\bcmAB$ spaces with laws of stochastic processes 
  and/or random fields,
  which was not clear how to do in previous frameworks.
  In addition to general theory,
  this paper includes several examples of Gromov--Hausdorff-type topologies,
  verifying that classical examples such as the Gromov--Hausdorff topology and the pointed Gromov--Hausdorff--Prohorov topology 
  are contained within our framework.
\end{abstract}

\tableofcontents

\section{Introduction}

The \textit{Gromov--Hausdorff metric} (see \eqref{eq: the GH metric} below) 
defines a distance between compact metric spaces
and was originally introduced by Gromov \cite{Gromov_07_Metric} for group theoretic purposes.
However,
it has found important applications in probability theory as well
since it provides a framework for discussing convergence of random compact metric spaces,
such as the scaling limit of critical Galton-Watson trees \cite{LeGall_06_Random},
the critical random graph \cite{Berry_Broutin_Goldschmidt_12_The_continuum},
random planar maps \cite{LeGall_Miermont_12_Scaling}
and percolation on some (random) graphs \cite{Archer_Croydon_23_Scaling, Renaudie_Broutin_Nachmias_pre_The_scaling}.
In many examples,
one's interest is in not only the geometry of spaces but also additional structures on spaces,
such as measures \cite{Abraham_Delmas_Hoscheit_13_A_note}, compact subsets \cite{Miermont_09_Tessellations}
and heat-kernel-type functions \cite{Croydon_Hambly_Kumagai_12_Convergence}.
Moreover, there are many examples of random non-compact metric spaces
for the study of which Gromov--Hausdorff-type topologies have been useful.
These include the uniform spanning tree on $\mathbb{Z}^{d}$
\cite{Angel_Croydon_Hernandez-Torres_Shiraishi_21_Scaling, Barlow_Croydon_Kumagai_17_Subsequential},
the uniform half-plane quadrangulation \cite{Gwynne_Miller_17_Scaling},
and the incipient infinite cluster of critical percolation on $\mathbb{Z}^{d}$
\cite{Arous_Cabezas_Fribergh_19_Scaling_theory}.

In consideration of such metric spaces equipped with additional structures,
various generalizations of the Gromov--Hausdorff metric have been introduced and studied in the literature
\cite{Abraham_Delmas_Hoscheit_13_A_note,Athreya_Lohr_Winter_16_The_gap}.
Recently,  
Khezeli \cite{Khezeli_23_A_unified} proposed a general method for defining a Gromov--Hausdorff-type metric  
on a collection of metric spaces equipped with additional structures.  
In the case where the underlying metric spaces are compact,  
his method requires only mild conditions and can accommodate a wide range of additional structures.  
However, in the non-compact setting, his approach involves certain technical conditions,  
which limit its applicability.  

In this paper,  
we introduce a new method for the non-compact case,  
which naturally extends Khezeli’s compact case framework while retaining similarly mild assumptions.  
Our method allows a broader class of examples to be treated,  
and thus provides a foundation for the study of metric spaces equipped with various additional structures.  
In particular, it offers a new topological framework  
for analyzing convergence of random spaces equipped with random objects such as stochastic processes.

In Section~\ref{sec: introduction to GH-type metrics}, 
we recall some Gromov--Hausdorff-type metrics that serve as a basis for our discussion. 
In the next subsection, we introduce Khezeli's framework, 
which generalizes these metrics using tools from category theory. 
In Section~\ref{subsec: Contributions}, 
we give an overview of our framework and clarify how it differs from Khezeli's approach. 
Finally, we explain the organization of the paper and describe some notational conventions used throughout.
For the purposes of our discussions below,
we set $a \wedge b \coloneqq \min\{a, b\}$ 
for $a, b \in \mathbb{R} \cup \{ \pm \infty \}$,
and, given a metric space $(X, d_X)$,
we write, for each $\rho \in X$ and $r > 0$,
\begin{equation}
  D_X(\rho, r) \coloneqq \{x \in X \mid d_X(\rho, x) \leq r\}.
\end{equation}


\subsection{Introduction to Gromov--Hausdorff-type metrics}  \label{sec: introduction to GH-type metrics}

\textbf{The Gromov--Hausdorff metric.}
As already introduced,
the Gromov--Hausdorff metric defines the distance between compact metric spaces.
The idea used to define the distance is to embed different compact metric spaces 
isometrically into a common compact metric space 
and measure the distance between them using the Hausdorff metric in the ambient space.
(The definition of the Hausdorff metric is recalled in Section~\ref{sec: the Fell topology}.)
More precisely,  
the distance between two compact metric spaces $(K_1, d_{K_1})$ and $(K_2, d_{K_2})$ is defined by  
\begin{equation}  \label{eq: the GH metric}
  d_{\mathrm{GH}} ( K_1, K_2)
  \coloneqq 
  \inf_{f_1, f_2, K} \HausMet{K}( f_1(K_1), f_2(K_2) ),
\end{equation}
where the infimum is taken over all compact metric spaces $K$  
and all isometric embeddings (i.e., distance-preserving maps) $f_i \colon K_i \to K$, $i=1,2$,  
and where $\HausMet{K}$ denotes the Hausdorff metric between compact subsets of $K$.
The Gromov--Hausdorff metric is a separable and complete metric 
on the collection of isometric equivalence classes of compact metric spaces,
and the induced topology is called the \textit{Gromov--Hausdorff topology} \cite{Gromov_07_Metric}.

\begin{rem} \label{1. rem: the collection of equivalence classes}
  One should note that it is not possible to consider the ``set''
  of compact metric spaces nor isometric equivalence classes of compact metric spaces
  from the rigorous viewpoint of set theory.
  Indeed, any two singletons are isometric as compact metric spaces,
  but the collection of all singletons is not a set.
  However, as discussed in \cite{Burago_Burago_Ivanov_01_A_course},
  it is possible to regard the collection of isometric equivalence classes 
  as a legitimate set.
  This is true even when we consider the collection of non-compact metric spaces equipped with additional structures
  (see Sections~\ref{sec: the local GH topology} and~\ref{sec: main results}).
\end{rem}

\textbf{The pointed Gromov--Hausdorff--Prohorov metric.}
One generalization of the Gromov--Hausdorff metric is the \textit{pointed Gromov--Hausdorff--Prohorov metric} $d_{\mathrm{pGHP}}$
(see \eqref{eq: the GHP metric} below),
which gives the distance between two rooted-and-measured compact metric spaces.
Note that a \textit{rooted-and-measured compact metric space} $(K, \rho, \mu)$
is a compact metric space $K$ equipped with a distinguished element $\rho \in K$ called the \textit{root}
and a finite Borel measure $\mu$ on $K$.
The metric $d_{\mathrm{pGHP}}$ was introduced in \cite{Abraham_Delmas_Hoscheit_13_A_note} (and \cite{Abraham_Delmas_Hoscheit_14_Exit})
to study a measured-tree-valued process,
and it is defined in the same spirit as the Gromov--Hausdorff metric.
(See also \cite{Athreya_Lohr_Winter_16_The_gap} for a similar version that is discussed in Remark~\ref{rem: local GHV and GHV} below.)
In particular,
for two rooted-and-measured compact metric spaces $\mathcal{K}_{i} = (K_i, \rho_i, \mu_i)$, $i=1,2$,
the distance between them is given by setting 
\begin{equation}  \label{eq: the GHP metric}
  d_{\mathrm{pGHP}}( \mathcal{K}_{1}, \mathcal{K}_{2} )
  \coloneqq 
  \inf_{f_1, f_2, K} 
  \left\{
    \HausMet{K}( f_1(K_1), f_2(K_2) ) 
    \vee
    \ProhMet{K}( \mu_{1} \circ f_1^{-1}, \mu_{2} \circ f_2^{-1} )
    \vee d_{K}( f_1(\rho_{1}), f_2(\rho_{2}) )
  \right\},
\end{equation}
where the infimum is taken over all compact metric spaces $K$ 
and isometric embeddings $f_{i}: K_i \to K$, $i=1,2$,
and $\ProhMet{K}$ denotes the Prohorov metric between finite Borel measures on $K$
(see Section~\ref{sec: the vague topology} for the definition).
Similar to the Gromov--Hausdorff metric,
the pointed Gromov--Hausdorff--Prohorov metric is a separable and complete metric 
on the collection of equivalence classes of measured compact metric spaces,
and the induced topology is called the \textit{pointed Gromov--Hausdorff--Prohorov topology}.

\textbf{The local Gromov--Hausdorff-vague metric.}
In various applications,  
it is desirable to relax the assumption of compactness.  
For that purpose,  
it is convenient to consider \textit{boundedly-compact} metric spaces (or $\bcmAB$ spaces for short),  
that is, metric spaces in which every closed ball of finite radius is compact.  
The \textit{local Gromov--Hausdorff-vague} metric $d_{\mathrm{GHV}}$ (defined in \eqref{eq: the GHV metric} below)  
is an extension of the pointed Gromov--Hausdorff--Prohorov metric.  
It is a metric on the collection of equivalence classes of rooted $\bcmAB$ spaces equipped with Radon measures,  
called \textit{rooted-and-measured $\bcmAB$ spaces},  
and was first introduced in \cite{Abraham_Delmas_Hoscheit_13_A_note}.  
Although \cite{Abraham_Delmas_Hoscheit_13_A_note} focused on a subclass of $\bcmAB$ spaces known as length spaces,  
\cite{Khezeli_20_Metrization} later verified that the metric $d_{\mathrm{GHV}}$ is well-defined on the full space.  

The idea behind the definition of $d_{\mathrm{GHV}}$ is that  
two rooted-and-measured $\bcmAB$ spaces are close  
if their restrictions to balls of finite radius centered at the roots are close  
with respect to the pointed Gromov--Hausdorff--Prohorov metric $d_{\mathrm{pGHP}}$  
(for Lebesgue-almost every radius).  
More precisely,  
for two rooted-and-measured $\bcmAB$ spaces  
$\cX_{i} = (X_i, \rho_i, \mu_i)$, $i=1,2$,  
where $\rho_i$ is the root and $\mu_i$ is a Radon measure on $X_i$,  
the distance between $\cX_{1}$ and $\cX_{2}$ is given by  
\begin{equation}  \label{eq: the GHV metric}
  d_{\mathrm{GHV}} (\cX_{1}, \cX_{2})
  \coloneqq 
  \int_{0}^{\infty} e^{-r} \left( 1 \wedge d_{\mathrm{pGHP}}( \cX_{1}^{(r)}, \cX_{2}^{(r)} ) \right) dr,
\end{equation}
where $\cX_i^{(r)} = (X_i^{(r)}, \rho_i^{(r)}, \mu_i^{(r)})$ is defined as follows:  
$X_i^{(r)}$ is the closed ball in $X_i$ centered at $\rho_i$ with radius $r$;  
$\rho_i^{(r)} \coloneqq \rho_i$;  
and $\mu_i^{(r)}$ is the restriction of $\mu_i$ to $X_i^{(r)}$.  
In \cite{Khezeli_20_Metrization},  
it is shown that $d_{\mathrm{GHV}}$ is a complete and separable metric,  
and the induced topology is referred to as the \textit{local Gromov--Hausdorff-vague topology}.

\begin{rem} \label{rem: local GHV and GHV}
  The terms the ``local Gromov--Hausdorff-vague metric'' and the ``local Gromov--Hausdorff-vague topology'' 
  are not in common use and are only used in the present paper as a matter of convenience.
  Moreover,
  one should note that the local Gromov--Hausdorff-vague topology is different 
  from the Gromov--Hausdorff-vague topology introduced in \cite{Athreya_Lohr_Winter_16_The_gap}
  in that the local Gromov--Hausdorff-vague topology takes into account the metric structure of the entire underlying space 
  while the Gromov--Hausdorff-vague topology ignores the metric structure outside the support of the measure.
\end{rem}


\subsection{Khezeli's framework} \label{subsec: Khezeli's framework}

Khezeli \cite{Khezeli_23_A_unified} uses functors from category theory,  
which allow the construction of Gromov--Hausdorff-type metrics for general additional structures in a unified manner.  
In this subsection, we briefly recall his framework.  
We emphasize that no prior knowledge of category theory is required of the reader.

\textbf{The compact case.}
We begin with a functor $\tau^c$ on compact metric spaces,
that is, 
$\tau^c$ assigns to each compact metric space $X$ a metric space $(\tau^c(X), d^{\tau^c}_X)$,  
and to each isometric embedding $f \colon X \to Y$ an isometric embedding $\tau^c_f \colon \tau^c(X) \to \tau^c(Y)$.  
(The precise definition of functors is given in Definition~\ref{dfn: functor} below.) 
For example, to model compact metric spaces equipped with finite Borel measures,  
we consider the functor $\tau^c = \tau^c_{\mathcal{M}_{\mathrm{fin}}}$ given as follows:
we set $\tau^c(X) = \finMeas{X}$, i.e., the space of finite Borel measures on $X$ equipped with the Prohorov metric,  
and define $\tau^c_f$ to be the pushforward map induced by $f$.

Let $\frakK_\bullet(\tau^c)$ consisting of (equivalence classes of) triples $(X, \rho, a)$,  
where $(X, \rho)$ is a rooted compact metric space and $a \in \tau^c(X)$.  
For $\mathcal{K}_i = (K_i, \rho_i, a_i) \in \frakK_\bullet(\tau^c)$, $i=1,2$,  
the distance between them is defined by generalizing \eqref{eq: the GHP metric},  
that is,  
\begin{equation}
  d^{\tau^c}_{\frakK_\bullet}(\mathcal{K}_1, \mathcal{K}_2) 
  \coloneqq 
  \inf_{f_1, f_2, K} 
  \left\{
    \HausMet{K}( f_1(K_1), f_2(K_2) ) 
    \vee
    d_K( f_1(\rho_1), f_2(\rho_2) )
    \vee
    d^{\tau^c}_K( \tau^c_{f_1}(a_1), \tau^c_{f_2}(a_2) )
  \right\},
\end{equation}
where the infimum is taken over all compact metric spaces $K$  
and all isometric embeddings $f_i \colon K_i \to K$, $i=1,2$.  
The conditions on $\tau^c$ required by Khezeli can be summarized roughly as follows:
\begin{enumerate} [label = (K\arabic*), leftmargin = *, series = Khezeli conditions]
  \item \label{item: Khezeli cond 1}
    the assignment $f \mapsto \tau^c_f$ is ``continuous'' (\cite[Definition~2.7]{Khezeli_23_A_unified});
  \item \label{item: Khezeli cond 2}
    the assignment $X \mapsto \tau^c(X)$ is ``continuous'' (\cite[Definition~2.11 and Remark~2.14]{Khezeli_23_A_unified}).
\end{enumerate}
Khezeli showed that, under \ref{item: Khezeli cond 1}, 
$d^{\tau^c}_{\frakK_\bullet}$ defines a metric on $\frakK_\bullet(\tau^c)$,
and he investigated topological properties such as Polishness under the additional condition \ref{item: Khezeli cond 2}.
The induced topology is characterized in terms of embeddings as follows:
for $\mathcal{K}_n = (K_n, \rho_n, a_n)$, $n \in \NN \cup \{\infty\}$,
$\mathcal{K}_n \to \mathcal{K}_\infty$ in $\frakK_\bullet(\tau^c)$ if and only if 
\begin{equation} \label{eq: conv in K(tau)}
  \begin{minipage}[c]{0.9\linewidth}
    there exist a compact metric space $K$ and isometric embeddings $f_n \colon K_n \to K$ 
    such that $f_n(K_n) \to f_\infty(K_\infty)$ in the Hausdorff topology, $f_n(\rho_n) \to f_\infty(a_\infty)$ in $K$, and
    $\tau^c_{f_n}(a_n) \to \tau^c_{f_\infty}(a_\infty)$ in $\tau^c(K)$,
  \end{minipage}
\end{equation}
where the Hausdorff topology refers to the topology induced by the Hausdorff metric.

\textbf{The boundedly-compact case.}
Khezeli's framework for boundedly-compact spaces follows the philosophy of the local Gromov--Hausdorff-vague metric given in \eqref{eq: the GHV metric}.
To accurately describe his framework, 
we need some further notions from category theory, such as contra-variant functors and inverse limits. 
However, for our purpose here, they are not essential, so we will proceed by making appropriate simplifications of his framework. 
See \cite[Section~3]{Khezeli_23_A_unified} for details.

We begin with a functor $\tau$ that assigns to each $\bcmAB$ space $X$ a topological space $\tau(X)$, 
and to each isometric embedding $f \colon X \to Y$ a topological embedding $\tau_f \colon \tau(X) \to \tau(Y)$
(i.e., a homeomorphism onto its image).    
The space of interest is the set $\rootedBCM(\tau)$ consisting of (equivalence classes of) triples $(X, \rho, a)$,  
where $(X, \rho)$ is a rooted boundedly-compact metric space and $a \in \tau(X)$.  
To metrize this space, we assume that $\tau$ can be truncated to a functor $\tau^c$ on compact metric spaces.  
In particular, we assume that, for each $a \in \tau(X)$ and $r > 0$, there exists a natural truncation $a^{(r)} \in \tau^c(X^{(r)})$.

Given $\cX = (X, \rho, a) \in \rootedBCM(\tau)$, we write $\cX^{(r)} \coloneqq (X^{(r)}, \rho^{(r)}, a^{(r)})$, which is an element of $\frakK_\bullet(\tau^c)$.  
Then the distance between $\cX_1$ and $\cX_2$ is defined by  
\begin{equation}  \label{eq: the Khezeli's metric}
  d(\mathcal{X}_{1}, \mathcal{X}_{2}) 
  \coloneqq 
  \int_{0}^{\infty} e^{-r} \left( 1 \wedge d^{\tau^c}_{\frakK_\bullet}( \mathcal{X}_{1}^{(r)}, \mathcal{X}_{2}^{(r)} )\right) dr.
\end{equation}  
For example, when $\tau(X) = \Meas{X}$,  
one can take $\tau^c = \tau^c_{\mathcal{M}_{\mathrm{fin}}}$
and $a^{(r)}$ to be the restriction of the measure $a$ to $X^{(r)}$.  
In this setting, the metric \eqref{eq: the Khezeli's metric} coincides with $d_{\mathrm{GHV}}$.

In order to show that the function above defines a metric inducing a suitable topology,  
Khezeli imposes additional conditions on $\tau^c$ and the truncations $\cX^{(r)}$ \cite[Assumptions~3.10 and 3.11]{Khezeli_23_A_unified},  
which are more technical than the previous conditions \ref{item: Khezeli cond 1} and \ref{item: Khezeli cond 2}.  
Although it is difficult to state the precise conditions here,  
they can be summarized, roughly speaking, by the following property:
\begin{enumerate}[resume* = Khezeli conditions]
  \item \label{item: Khezeli cond 3}
    Fix $\cX_i \in \frakK_\bullet(\tau^c)$, $i = 1,2$, and set $\varepsilon = d^{\tau^c}_{\frakK_\bullet}(\cX_1, \cX_2)$.  
    For each $\cX_1^{(r)}$ with $r > 2\varepsilon$, there exists $\tilde{\cX}_2 \in \frakK_\bullet(\tau^c)$ such that  
    $d^{\tau^c}_{\frakK_\bullet}(\cX_1^{(r)}, \tilde{\cX}_2) \leq \varepsilon$  
    and $\tilde{\cX}_2$ is a truncation of $\cX_2$ while containing $\cX_2^{(r-2\varepsilon)}$ as a subspace in a suitable sense.
\end{enumerate}
Under the above mentioned conditions, 
in \cite[Lemma~3.28]{Khezeli_23_A_unified},
it is proven that 
the induced topology is characterized in terms of embeddings as follows:
for $\cX_n = (X_n, \rho_n, a_n)$, $n \in \NN \cup \{\infty\}$,
$\cX_n \to \cX_\infty$ in $\rootedBCM(\tau)$ if and only if 
\begin{equation} \label{eq: Khezeli conv in non-cpt case}
  \begin{minipage}[c]{0.9\linewidth}
    there exist a $\bcmAB$ space $M$ and isometric embeddings $f_n \colon X_n \to M$ 
    such that $f_n(X_n) \to f_\infty(X_\infty)$ in the Fell topology, 
    $f_n(\rho_n) \to f_\infty(\rho_\infty)$ in $M$, 
    and $\tau_{f_n}(a_n) \to \tau_{f_\infty}(a_\infty)$ in $\tau(M)$,
  \end{minipage}
\end{equation}
where the Fell topology is an extension of the Hausdorff topology 
(see \cite[Appendix~C]{Molchanov_17_Theory} and also Section~\ref{sec: the Fell topology}).

As mentioned just after \cite[Example~3.12]{Khezeli_23_A_unified},  
to fulfill \ref{item: Khezeli cond 3},  
the metric space $\tau^c(X)$ should be carefully chosen,  
and, in some cases, a standard metric space is not suited.  
For example, to consider the case where $\tau(X) = D(\RNp, X)$, i.e., the space of $X$-valued cadlag functions,  
one needs to define $\tau^c(X)$ to be the space of killed cadlag functions (see \cite[Example~3.45]{Khezeli_23_A_unified}).  
Moreover, checking \ref{item: Khezeli cond 3} can be far from a trivial exercise for some important examples.  
In fact, even in the relatively simple case where $\tau(X) = \Meas{X}$,  
verifying the condition requires some technical arguments (see \cite[Lemma~3.12]{Khezeli_20_Metrization}).  
A more serious difficulty arises when $\tau(X) = \Prob{D(\RNp, X)}$,  
i.e., the space of probability measures on $D(\RNp, X)$,  
which is crucial for discussing convergence of stochastic processes on varying spaces
(cf.\ \cite{Athreya_Lohr_Winter_17_Invariance, Croydon_18_Scaling}).
In this case, it is not clear whether the condition is satisfied.
Our framework resolves these issues,  
as will be explained in the following subsection.


\subsection{The framework presented in the present paper} \label{subsec: Contributions}

Our framework follows the philosophy underlying the formulation of the pointed Gromov--Hausdorff--Prohorov metric given in \eqref{eq: the GHP metric}, 
and provides a natural extension of Khezeli's framework for compact metric spaces. 
It recovers his results under assumptions analogous to \ref{item: Khezeli cond 1} and \ref{item: Khezeli cond 2},
without relying on the truncation operation or the technical condition \ref{item: Khezeli cond 3}.
This relaxation allows us to treat a broader class of additional structures. 
In what follows, we give a brief overview of our framework. 
A detailed description is provided in Section~\ref{sec: main results}.

We begin with the same functor $\tau$ as before. 
Namely, it assigns to each $\bcmAB$ space $X$ a topological space $\tau(X)$, 
and to each isometric embedding $f \colon X \to Y$ a topological embedding $\tau_f \colon \tau(X) \to \tau(Y)$. 
We assume that there exists a \emph{metrization} $\FMet{\tau}$ of $\tau$, 
that is, for each $X$, there exists a metric $d^{\FMet{\tau}}_X$ on $\tau(X)$ 
such that $\tau_f$ is distance-preserving. 

\begin{exm}
  If we set $\tau(X) \coloneqq D(\RNp, X)$ for each $\bcmAB$ space $X$,
  then its metrization is given by equipping $\tau(X)$ with the usual Skorohod metric.
  This yields the space $\rootedBCM(\tau)$, consisting of (equivalence classes of) $\bcmAB$ spaces equipped with cadlag curves.
  The details are provided in Section~\ref{sec: structure for cadlag curves}.
\end{exm}

\begin{rem}
  The topological space $\tau(X)$ may not admit a canonical metric unless a root is specified. 
  For instance, when $\tau(X) = \Meas{X}$ is equipped with the vague topology, 
  its metrization requires a root in $X$ (see Section~\ref{sec: the vague topology}). 
  In such cases, we assume that for each rooted $\bcmAB$ space $(X, \rho)$, 
  there exists a suitable metric on $\tau(X)$. 
  Such functors can also be treated within our framework 
  by using the notion of \emph{rooted metrization} of $\tau$, 
  which is introduced in Section~\ref{sec: Metrization of structures}.
\end{rem}

In our framework, we introduce two metrizations of $\rootedBCM(\tau)$ as follows. 
(The advantages of these will be discussed in Remark~\ref{rem: advantages of two metrizations} below.)
For $\cX_i = (X_i, \rho_i, a_i) \in \rootedBCM(\tau)$, $i=1,2$,  
we define the distance between them by
\begin{equation} \label{eq: intro. metric for preserved roots}
  \RFMet^{\FMet{\tau}}(\cX_1, \cX_2) 
  \coloneqq 
  \inf_{f_1, f_2, (M, \rho)} 
  \left\{
    \lHausMet{M, \rho}( f_1(X_1), f_2(X_2) ) 
    \vee
    d^{\FMet{\tau}}_M( \tau_{f_1}(a_1), \tau_{f_2}(a_2) )
  \right\},
\end{equation}
where the infimum is taken over all rooted $\bcmAB$ spaces $(M, \rho)$  
and all root-preserving isometric embeddings $f_i \colon X_i \to M$, $i=1,2$.  
Here, $\lHausMet{M, \rho}$ is an extension of the Hausdorff metric that measures the distance between closed subsets of $M$,  
as defined in \eqref{2. eq: the local Hausdorff metric} below.
We also define another distance by
\begin{equation}  \label{eq: intro. metric for non-preserved roots}
  \RVMet^{\FMet{\tau}}(\cX_1, \cX_2) 
  \coloneqq 
  \inf_{f_1, f_2, M} 
  \left\{
    \lHausMet{M}\bigl( (f_1(X_1), f_1(\rho_1)), (f_2(X_2), f_2(\rho_2)) \bigr) 
    \vee
    d^{\FMet{\tau}}_M( \tau_{f_1}(a_1), \tau_{f_2}(a_2) )
  \right\},
\end{equation}
where the infimum is taken over all $\bcmAB$ spaces $M$  
and all isometric embeddings $f_i \colon X_i \to M$, $i=1,2$.  
Here, $\lHausMet{M}$ is another extension of the Hausdorff metric,  
which measures the distance between rooted closed subsets of $M$,  
and is defined in \eqref{eq: local Hausdorff metric on product} below.

We prove that Khezeli's results in the compact case also hold in our framework,  
under conditions analogous to \ref{item: Khezeli cond 1} and \ref{item: Khezeli cond 2}.  
The metric $\RVMet^{\FMet{\tau}}$ induces the same convergence as in \eqref{eq: Khezeli conv in non-cpt case},  
whereas $\RFMet^{\FMet{\tau}}$ induces, in general, a different mode of convergence, which can be described as follows:  
for $\cX_n = (X_n, \rho_n, a_n)$, $n \in \NN \cup \{\infty\}$,  
we have $\cX_n \to \cX_\infty$ with respect to $\RFMet^{\FMet{\tau}}$ if and only if  
\begin{equation} \label{eq: intro. conv with prerved roots}
  \begin{minipage}[c]{0.9\linewidth}
    there exist a rooted $\bcmAB$ space $(M, \rho)$ and root-preserving isometric embeddings $f_n \colon X_n \to M$  
    such that $f_n(X_n) \to f_\infty(X_\infty)$ in the Fell topology, and  
    $\tau_{f_n}(a_n) \to \tau_{f_\infty}(a_\infty)$ in $\tau(M)$.
  \end{minipage}
\end{equation}
The difference between the two notions lies in the treatment of roots.  
In \eqref{eq: intro. conv with prerved roots},  
the roots $\rho_n$ are mapped to the root $\rho$ of the ambient space $M$,  
whereas in \eqref{eq: Khezeli conv in non-cpt case},  
the roots are not fixed to be a single point in $M$.  
However, we will show that, under an additional mild assumption on $\tau$,  
the two convergence notions are in fact equivalent (see Section~\ref{sec: Coincidence of the two topologies}).  
This assumption requires, roughly speaking, the following property:
\begin{enumerate} [label = \textup{(N)}]
  \item \label{item: condition N}
    Fix a topological space $X$, and suppose there exist two boundedly-compact metrics $d_X$ and $\tilde{d}_X$ on $X$ that induce the same topology.  
    If $|d_X(x,y) - \tilde{d}_X(x,y)| \leq \varepsilon$ for all $x, y \in X$, then  
    \begin{equation}
      |d^{\FMet{\tau}}_X(a,b) - \tilde{d}^{\FMet{\tau}}_X(a,b)| \leq \Dist{\FMet{\tau}}(\varepsilon),
      \quad 
      \forall a,b \in \tau(X),
    \end{equation}  
    for some function $\Dist{\FMet{\tau}}$ satisfying $\lim_{\varepsilon \to 0} \Dist{\FMet{\tau}}(\varepsilon) = 0$, depending only on $\FMet{\tau}$.  
    Here, $d^{\FMet{\tau}}_X$ and $\tilde{d}^{\FMet{\tau}}_X$ denote the metrics on $\tau(X)$ associated to $d_X$ and $\tilde{d}_X$, respectively.
\end{enumerate}
(See Definition~\ref{dfn: stability} for the precise formulation.)  
Heuristically, this condition means that the metric on $\tau(X)$ is stable under small deformations of the underlying space $X$.  
As shown in Section~\ref{sec: Examples of functors} below,  
this property is easily verified for most functors.

\begin{rem} \label{rem: advantages of two metrizations}
  A commonly used method for estimating Gromov--Hausdorff-type distances between metric spaces $X$ and $Y$ 
  (equipped with additional structures) is to use a correspondence between $X$ and $Y$
  (see \cite[Section~7.3.3]{Burago_Burago_Ivanov_01_A_course} for details).
  Indeed, a correspondence provides isometric embeddings of $X$ and $Y$ into their disjoint union 
  $Z \coloneqq X \sqcup Y$, equipped with a suitable metric.
  In this setting, it is natural to use the metric $\RVMet^{\FMet{\tau}}$.
  Under condition~\ref{item: condition N}, 
  these isometric embeddings can be upgraded to root-preserving isometric embeddings 
  (by identifying the roots of $X$ and $Y$ in $Z$), which is convenient for applications.
\end{rem}

Our framework applies to all the examples considered by Khezeli~\cite{Khezeli_23_A_unified}.  
Moreover, since the technical condition~\ref{item: Khezeli cond 3} is removed,  
it allows us to treat various new types of additional structures.  
For example,  
our framework applies to the case where $\tau(X) = \Prob{D(\RNp, X)}$,  
and thus provides a suitable topological setting  
for discussing the convergence of stochastic processes on varying spaces  
(see Section~\ref{sec: Laws of structures}).  
Furthermore,  
in~\cite{Noda_pre_Convergence,Noda_pre_Scaling},  
following the framework developed in this paper,  
a Gromov--Hausdorff-type topology is introduced,  
which enables us to treat convergence of laws of Markov processes and their associated local times  
on varying spaces.  
We expect that our framework will also be applicable to a wide range of problems in random geometry,  
such as the quantum zipper (cf.\ \cite{Sheffield_16_Conformal}).


\subsection{Organization of the paper and notational conventions} \label{subsec: Organization of the paper and notational conventions}

The remainder of the article is organized as follows.  
As we have already seen in \eqref{eq: intro. metric for preserved roots} and \eqref{eq: intro. metric for non-preserved roots} above,  
the extensions $\lHausMet{M, \rho}$ and $\lHausMet{M}$ of the Hausdorff metric play a crucial role in our framework.  
We establish a general method for extending metrics in Section~\ref{sec: Metric for non-compact objects}.  
In the following section, we apply the method to concrete examples, and in particular the above-mentioned metrics are defined.  
In Section~\ref{sec: the local GH topology}, we introduce the local Gromov--Hausdorff topology,  
which is an extension of the (pointed) Gromov--Hausdorff topology  
and provides a suitable topological setting to discuss convergence of rooted $\bcmAB$ spaces.  
To deal with $\bcmAB$ spaces equipped with additional structures,  
we need some notions from category theory,  
and they are introduced in Section~\ref{sec: Preliminaries on category theory}.  
Then the main results of this paper are presented in Section~\ref{sec: main results}.  
Section~\ref{sec: Functorial operations and preservation of properties} verifies that properties of functors required by our framework  
are preserved under some functorial operations,  
which enables us to consider complex additional structures.  
In the last section, Section~\ref{sec: Examples of functors},  
we present various example additional structures to which our main results apply.

Below, we remark about the notation and terminology used in the rest of this paper.
\begin{enumerate} [label = \textup{(\arabic*)}]
  \item \label{note: real number}
    We write $\RNp \coloneqq [0, \infty)$ and $\RNpp \coloneqq (0, \infty)$,
    and equip each of these spaces with the usual Euclidean metric.
  \item \label{note: max min}
    For $a, b \in \RN \cup \{ \pm \infty \}$,
    we write $a \wedge b \coloneqq \min\{a, b\}$ and $a \vee b \coloneqq \max\{a, b\}$.
  \item \label{note: metric space is non-empty}
    When we say that $X$ is a metric space, we always assume that $X$ is non-empty and the associated metric is written as $d_X$.
  \item \label{note: balls in metric space}
    Given a metric space $X$, we write 
    \begin{equation}
      B_X(x, r) 
      \coloneqq
      \{ y \in X \mid d_X(x,y) < r \},
      \quad 
      D_X(x,r)
      \coloneqq 
      \{ y \in X \mid d_X(x,y) \leq r \}.
    \end{equation}
  \item \label{note: closure}
    Given a topological space $X$ and a subset $A$ of $X$,
    we denote by $\closure(A) = \closure_X(A)$ the  closure of $A$ in $X$.
  \item \label{note: identity}
    For a set $A$,
    we denote the identity map from $A$ to itself by $\id_{A}$.
  \item \label{note: product map}
    Given maps $f_i \colon X_i \to Y_i$, $i = 1,2$, 
    we define $f_1 \times f_2 \colon X_1 \times X_2 \to Y_1 \times Y_2$ by setting $(f_1 \times f_2)(x_1, x_2) \coloneqq (f_1(x_1), f_2(x_2))$.
  \item \label{note: product space}
    When $X$ and $Y$ are topological spaces, we always equip $X \times Y$ with the product topology.
    Moreover, if $X$ and $Y$ are metric spaces, then we always equip $X \times Y$ with the \emph{max product metric} defined as follows: 
    \begin{equation}
      d_{X \times Y}((x_1, y_1), (x_2, y_2)) \coloneqq d_X(x_1, x_2) \vee d_Y(y_1, y_2).
    \end{equation}
  \item \label{note: topological embedding}
    We say that a map $f \colon X \to Y$ between topological spaces is a \emph{topological embedding} 
    if and only if it is a homeomorphism onto its image with the relative topology.
  \item \label{note: isometric embedding}
    We say that a map $f \colon X \to Y$ between metric spaces is an \emph{isometric embedding} (resp.\ \emph{isometry})
    if and only if it is distance-preserving (resp.\ and bijective).

  \item \label{note: bcm}
    We abbreviate “boundedly-compact metric space” as \emph{$\bcmAB$ space}.
\end{enumerate}


\section{An extension of metrics to non-compact objects} \label{sec: Metric for non-compact objects}
  In this section,  
  we present a method for extending a metric defined only for ``compact'' objects  
  to a metric for ``non-compact'' objects.  
  The idea is inspired by Khezeli's framework in the boundedly-compact case \cite[Section~3]{Khezeli_23_A_unified}.
  The method introduced here will be used in the next section
  to extend certain metrics, such as Hausdorff metric and Prohorov metric.
  Readers who are only interested in the definitions and properties of these metrics  
  may skip this section.

  The arguments in this section are very abstract,
  so we first give a concrete example and then state the aim of this section.
  Fix a non-empty boundedly-compact metric space $X$.
  Let $\frakC(X)$ (resp.\ $\frakD(X)$) be the set of compact (resp.\ closed) subsets of $X$ (including the empty set).
  A commonly used metric $d^\frakC_X$ on $\frakC(X)$ is the Hausdorff metric (see \eqref{eq: dfn of Hausdorff metric} below).
  To extend $d^\frakC_X$ to a metric on $\frakD(X)$,
  we consider a restriction system $(R_x^{(r)})_{r > 0, x \in X}$ given by
  \begin{equation}
    R_x^{(r)}(A) = A|_x^{(r)} \coloneqq A \cap D_X(x, r).
  \end{equation}
  Our aim of this subsection is to construct a metric on $\frakD(X)$ which induces the following convergence:
  $A_n \to A$ if and only if,
  for any $x_n \in X$ converging to $x \in X$,
  \begin{equation}
    A_n|_{x_n}^{(r)} \to A|_{x}^{(r)}\ \text{with respect to}\ d^\frakC_X\ \text{for all but countably many}\ r > 0.
  \end{equation}
  We will also investigate some properties of the metric such as Polishness and precompactness.

  Now, we state the general setting where we will work.
  Let $X$ be a metric space,
  and $\frakC(X)$ and $\frakD(X)$ be associated sets such that $\frakC(X) \subseteq \frakD(X)$,
  and $d^{\frakC}_X$ be an extended metric on $\frakC(X)$.
  (NB. An extended metric is a metric that is allowed to take the value $\infty$.) 
  The following is a generalization of the restriction system defined above for closed subsets.

  \begin{dfn}[{Restriction system}] \label{dfn: restriction system}
    Let $R_x^{(r)}: \frakD(X) \to \frakC(X)$ be a map for each $r>0$ and $x \in X$.
    We call $R =(R_x^{(r)})_{r > 0, x \in X}$ a \textit{restriction system} 
    from $\frakC(X)$ to $\frakD(X)$
    if it satisfies the following:
    \begin{enumerate} [label = (RS\,\arabic*), series=RS, leftmargin=*]
      \item \label{dfn item: RS. 1}
        For any $s,r>0$ and $x \in X$, $R_x^{(r)} \circ R_x^{(s)} = R_x^{(s \wedge r)}$.
      \item \label{dfn item: RS. 2}
        For given $x \in X$ and $a, b \in \frakD(X)$, if $R_x^{(r)}(a) = R_x^{(r)}(b)$ for all $r>0$, then $a = b$.
      \item \label{dfn item: RS. 3}
        For any $\rho \in X$ and $a \in \frakC(X)$, 
        there exists $r > 0$ such that $a|_\rho^{(r)} = a$.
      \item \label{dfn item: RS. 4}
        For all $x, y \in X$ and $s, r > 0$,
        if $s \geq d_X(x,y)$,
        then $R_y^{(r)} \circ R_x^{(s+r)} = R_y^{(r)}$.
    \end{enumerate}
  \end{dfn}

  We fix a restriction system $R =(R_x^{(r)})_{r > 0, x \in X}$,
  and simply write $R_x^{(r)}(a) = a|_x^{(r)}$.
  We first consider a metrization of $\frakD(X) \times X$ instead of $\frakD(X)$ itself.
  Once a metric on $\frakD(X) \times X$ is defined,
  the metrization of $\frakD(X)$ is given by specifying the root $\rho$ of $X$.
  This is because the space $\frakD(X)$ can be regarded as the subset $\frakD(X) \times \{\rho\}$ of $\frakD(X) \times X$.
  At this point,
  it may seem roundabout to consider metrization of $\frakD(X) \times X$, 
  but this has an advantage that it does not require specifying the root of the underlying space $X$.
  This helps the metrization of a certain class of Gromov--Hausdorff-type topologies,
  discussed in Section~\ref{sec: main results} later.

  In the same spirit as the local Gromov--Hausdorff-vague metric given in \eqref{eq: the GHV metric},
  we define the distance between $(a, x), (b, y) \in \frakD(X) \times X$ by setting 
  \begin{equation}  \label{eq: metric on D times X}
    d_X^\frakD \bigl( (a, x), (b, y)\bigr)
    \coloneqq 
    d_X(x,y) \vee    
    \Bigl\{
      \int_{0}^{\infty} e^{-r} \bigl( 1 \wedge d^\frakC_X(a|_x^{(r)}, b|_y^{(r)}) \bigr)\, dr
    \Bigr\}.
  \end{equation}
  To ensure that $d_X^\frakD$ is well-defined
  (i.e., the integrand in \eqref{eq: metric on D times X} is measurable),
  and to investigate its properties,
  we assume the following conditions.

  \begin{assum} \label{assum: metrization of D} \leavevmode
    \begin{enumerate} [label = \textup{(\roman*)}, leftmargin = *]   
      \item \label{assum item: 1. metrization of D}
        Fix $(a, x) \in \frakD \times X$.
        Then the map $(0, \infty) \ni r \mapsto a|_x^{(r)} \in \frakC(X)$ 
        is continuous for all but countably many $r>0$.
      \item \
        Let  $(a_n, x_n)$, $n \in \NN \cup \{\infty\}$, be elements of $\frakD(X) \times X$
        such that $x_n \to x_\infty$,
        and 
        $(r_n)_{n \geq 1}$ be an increasing sequence of positive numbers with $r_n \uparrow \infty$.
        \begin{enumerate} [label = \mbox{\textup{(ii-\alph*)}}]
          \item \label{assum item: 2. metrization of D}
            If $d^\frakC_X(a_n|_{x_n}^{(r_n)}, a_\infty|_{x_\infty}^{(r_n)}) \to 0$,
            then $d^\frakC_X(a_n|_{x_n}^{(r)}, a_\infty|_{x_\infty}^{(r)}) \to 0$ for all but countably many $r>0$.
          \item \label{assum item: 3. metrization of D}
            For each $r > 0$, if $d^\frakC_X(a_n|_{x_n}^{(r)}, a_\infty|_{x_\infty}^{(r)}) \to 0$,
            then $\{a_n|_{x_n}^{(s)}\}_{n \in \NN}$ is precompact in $\frakC(X)$ for all $s \in (0, r]$.
          \item \label{assum item: 4. metrization of D} 
            If $d^\frakC_X(a_n|_{x_n}^{(r_n)}, a_{n+1}|_{x_{n+1}}^{(r_n)}) < 2^{-n}$ for all sufficiently large $n$,
            then $\{a_n|_{x_n}^{(r)}\}_{n \in \NN}$ is precompact in $\frakC(X)$ for all $r > 0$. 
        \end{enumerate}
    \end{enumerate}
  \end{assum}

  We note that \ref{assum item: 4. metrization of D} implies \ref{assum item: 3. metrization of D}.
  This can be easily checked by using \ref{dfn item: RS. 2}.
  Assumption~\ref{assum: metrization of D}\ref{assum item: 1. metrization of D} is used to ensure that 
  the integrand in \eqref{eq: metric on D times X} is measurable,
  \ref{assum item: 2. metrization of D} is used to show that the induced topology on $\frakD(X) \times X$ is a natural extension of 
  the product topology on $\frakC(X) \times X$,
  \ref{assum item: 3. metrization of D} is used to derive a precompactness criterion,
  and \ref{assum item: 4. metrization of D} is related to the completeness.
  We say that the restriction system $R$ satisfies
  \begin{itemize}
    \item \textbf{Condition 1} if Assumption~\ref{assum: metrization of D}\ref{assum item: 1. metrization of D} is satisfied,
    \item \textbf{Condition 2} if Assumption~\ref{assum: metrization of D}\ref{assum item: 1. metrization of D} 
      and \ref{assum item: 2. metrization of D} are satisfied,
    \item \textbf{Condition 3} if Assumption~\ref{assum: metrization of D}\ref{assum item: 1. metrization of D}, 
      \ref{assum item: 2. metrization of D}, and \ref{assum item: 3. metrization of D} are satisfied, 
    \item \textbf{Condition 4} if Assumption~\ref{assum: metrization of D}\ref{assum item: 1. metrization of D}, 
      \ref{assum item: 2. metrization of D}, and \ref{assum item: 4. metrization of D}
      (and hence \ref{assum item: 3. metrization of D}) are satisfied.
  \end{itemize}

  \begin{rem}
    The above conditions are based on those required by Khezeli for the metrization of Gromov--Hausdorff-type topologies in the boundedly-compact case (see \cite[Assumptions~3.11 and 3.17]{Khezeli_23_A_unified}).
    An important difference from Khezeli’s setting is that the underlying space $X$ is fixed here,
    whereas he assumes similar conditions where $X$ varies.
    This makes our conditions easier to verify.
    We also note that our conditions are weaker than his,
    as they focus on what appears to be the essential part of the argument.
  \end{rem}

  \begin{prop}  \label{2. prop: d_D is a metric}
    If $R$ satisfies Condition~1,
    then the function $d_X^\frakD$ is a well-defined metric on $\frakD(X) \times X$.
  \end{prop}

  \begin{proof}
    Write $d(c, d) \coloneqq 1 \wedge d_X^\frakC(c, d)$ for each $c, d \in \frakC(X)$.
    Then $d$ is a metric on $\frakC$.
    Fix $(a, x), (b, y) \in \frakD \times X$.
    By the triangle inequality,
    we deduce that  
    \begin{align}
      &\left|
        d( a|_x^{(r')}, b|_y^{(r')}) - d( a|_x^{(r)}, b|_y^{(r)})
      \right|\\
      \leq &
      \left|
        d( a|_x^{(r')}, b|_y^{(r')}) - d( b|_y^{(r')}, a|_x^{(r)})
      \right|
      +
      \left|
        d( b|_y^{(r')}, a|_x^{(r)}) - d( a|_x^{(r)}, b|_y^{(r)})
      \right|\\
      \leq &
      d( a|_x^{(r')}, a|_x^{(r)}) + d( b|_y^{(r')}, b|_y^{(r)}).
    \end{align}
    This, combined with Assumption~\ref{assum: metrization of D}\ref{assum item: 1. metrization of D},
    implies that the map $r \mapsto 1 \wedge d_X^\frakC(a|_x^{(r)}, b|_y^{(r)}) \in \RNp$ 
    is continuous for all but countably many $r>0$.
    Hence, $d_X^\frakD$ is well-defined.
    Symmetry and the triangle inequality are obvious. 
    If $d_X^\frakD\bigl((a, x), (b,y)\bigr)=0$,
    then $x = y$ and $a|_x^{(r)} = b|_y^{(r)}$ for Lebesgue-almost every $r>0$.
    By \ref{dfn item: RS. 1} and \ref{dfn item: RS. 2},
    we obtain $a = b$.
    Thus, $d^\frakD_X$ is positive definite.
  \end{proof}

  We will give a characterization of convergence with respect to $d^\frakD_X$ in Theorem~\ref{thm: convergence in D times X} below.
  To this end,
  we use the following result.

  \begin{lem} \label{2. lem: convergence of restricted objects}
    Assume that $R$ satisfies Condition~2.
    Fix elements $(a_n, x_n) \in \frakD(X) \times X$, $n \in \NN \cup \{\infty\}$, and $r > 0$.
    If $a_n|_{x_n}^{(r)} \to a_\infty|_{x_\infty}^{(r)}$ in $\frakC(X)$,
    then, for all but countably many $s \in (0, r]$, 
    $a_n|_{x_n}^{(s)} \to a|_x^{(s)}$ in $\frakC(X)$. 
  \end{lem}

  \begin{proof}
    This is a immediate consequence of \ref{dfn item: RS. 2} and Assumption~\ref{assum: metrization of D}\ref{assum item: 2. metrization of D}.
  \end{proof}

  \begin{thm} [{Convergence with respect to $d_X^\frakD$}]
    \label{thm: convergence in D times X}
    Assume that $R$ satisfies Condition~2.
    Let $(a_n, x_n)$, $n \in \NN \cup \{\infty\}$, be elements of $\frakD(X) \times X$.
    The following are equivalent:
    \begin{enumerate} [label = \textup{(\roman*)}, leftmargin = *]
      \item \label{2. thm item: convergence w.r.t. d_D}
        $(a_n, x_n) \to (a_\infty, x_\infty)$ with respect to $d_X^\frakD$;
      \item \label{2. thm item: convergence w.r.t. d_C for a.e r}
        $x_n \to x_\infty$ in $X$, and $a_n|_{x_n}^{(r)} \to a_\infty|_{x_\infty}^{(r)}$ with respect to $d^\frakC_X$
        for all but countably many $r > 0$;
      \item \label{2. thm item: convergence w.r.t. d_C for unbounded rs}
        $x_n \to x_\infty$ in $X$, and there exists an increasing sequence $(r_{k})_{k \geq 1}$ with $r_{k} \uparrow \infty$
        such that $a_n|_{x_n}^{(r_{k})} \to a_\infty|_{x_\infty}^{(r_{k})}$ as $n \to \infty$ with respect to $d^\frakC_X$ for each $k$;
      \item \label{2. thm item: convergence w.r.t d_C with increasing radii}
        $x_n \to x_\infty$ in $X$, and there exists a sequence $(r_{n})_{n \geq 1}$ of positive numbers with $r_{n} \uparrow \infty$ 
        such that $d^\frakC_X(a_n|_{x_n}^{(r_n)}, a_\infty|_{x_\infty}^{(r_n)}) \to 0$.
    \end{enumerate}
  \end{thm}

  \begin{proof} 
    The implications
    \ref{2. thm item: convergence w.r.t d_C with increasing radii}
    $\Rightarrow$
    \ref{2. thm item: convergence w.r.t. d_C for unbounded rs}
    and 
    \ref{2. thm item: convergence w.r.t. d_C for unbounded rs}
    $\Rightarrow$
    \ref{2. thm item: convergence w.r.t. d_C for a.e r}
    follow from Assumption \ref{assum: metrization of D}\ref{assum item: 2. metrization of D}
    and Lemma~\ref{2. lem: convergence of restricted objects}, respectively.
    The dominated convergence theorem yields 
    the implication \ref{2. thm item: convergence w.r.t. d_C for a.e r} $\Rightarrow$ \ref{2. thm item: convergence w.r.t. d_D}.
    Assume that \ref{2. thm item: convergence w.r.t. d_D}  holds.
    By the definition of $d^\frakD_X$, it holds that $x_n \to x$.
    Write $\varepsilon_n \coloneqq d^\frakD_X\bigl((a_n, x_n), (a_\infty, x_\infty)\bigr)$
    and $s_n \coloneqq |\log \sqrt{\varepsilon_n}|$.
    For all sufficiently large $n$,
    we have that $s_n = - \log \sqrt{\varepsilon_n}$ and $\sqrt{\varepsilon_n} e^{-s_n} = \varepsilon_n$.
    Thus, 
    \begin{equation}
      \int_{0}^{\infty} e^{-r} \left( 1 \wedge d_X^\frakC( a_n|_{x_n}^{(r)}, a_\infty|_{x_\infty}^{(r)} ) \right) dr 
      \leq
      \sqrt{\varepsilon_n} e^{-s_n},
    \end{equation}
    which implies that, for each such $n$,
    there exists $r_n > s_n$ such that $d_X^\frakC( a_n|_{x_n}^{(r_n)}, a_\infty|_{x_\infty}^{(r_n)} ) \leq \sqrt{\varepsilon_n}$.
    Hence, \ref{2. thm item: convergence w.r.t d_C with increasing radii} holds.
  \end{proof}

  We next investigate the separability and the completeness of $d^\frakD_X$.

  \begin{thm} [{Separability}]\label{2. thm: separability of d_D}
    Assume that $R$ satisfies Condition~2.
    If both $X$ and $\frakC(X)$ are separable,
    then the topology on $\frakD(X) \times X$ induced by $d^{\frakD}_X$ is also separable.
  \end{thm}

  \begin{proof}
    Let $D$ be a countable subset of $\frakC(X) \times X$ that is dense in the product topology.
    By \ref{dfn item: RS. 3} and Theorem~\ref{thm: convergence in D times X},
    $D$ is also dense in the relative topology on $\frakC(X) \times X$ induced by $d^{\frakD}_X$. 
    Thus, it suffices to show that any element of $\frakD(X) \times X$ is approximated by a sequence in $\frakC(X) \times X$.
    By Theorem~\ref{thm: convergence in D times X},
    for any $(a,x) \in \frakD(X) \times X$,
    $(a|_x^{(r)}, x) \to (a, x)$ as $r \to \infty$ in $\frakD(X) \times X$.
    Since $(a|_x^{(r)}, x) \in \frakC(X) \times X$, 
    we deduce that $D$ is dense in $\frakD(X) \times X$.
  \end{proof} 

  To prove the completeness,
  we assume a condition to ensure that $\frakD(X)$ contains sufficiently many elements.

  \begin{dfn} [{Complete restriction system}] \label{dfn: complete RS}
    Let $x \in X$ be an element, $(a_{k})_{k \geq 1}$ be a sequence in $\frakC$,
    and $(r_{k})_{k \geq 1}$ be an increasing sequence of non-negative numbers with $r_{k} \uparrow \infty$.
    A sequence $(a_{k}, r_{k})_{k \geq 1}$ is said to be a \textit{compatible sequence rooted at $x$} if and only if 
    $a_{k} = a_{k'}|_x^{(r_{k})}$ for all $k \leq k'$.
    The restriction system $R$ is said to be \textit{complete} 
    if it satisfies the following.
    \begin{enumerate} [resume* = RS]
      \item \label{2. dfn item: existence of the inverse limit}
        For any $x \in X$ and any compatible sequence $(a_{k}, r_{k})_{k \geq 1}$ rooted at $x$ ,
        there exists $a \in \frakD$ such that $a_{k} = a|_x^{(r_{k})}$.
    \end{enumerate}
  \end{dfn}

  \begin{lem} \label{lem: existence of limit}
    Assume that $R$ is complete and satisfies Condition~2.
    Let $(a_n, x_n)_{n \geq 1}$ be a sequence in $\frakD(X) \times X$,
    $(\alpha_k)_{k \geq 1}$ be a sequence in $\frakC(X)$,
    and $(r_k)_{k \geq 1}$ be an increasing sequence of positive numbers with $r_{k} \uparrow \infty$.
    Assume that $x_n$ converges to some element $x \in X$ and 
    \begin{equation}
      d_X^\frakC(a_n|_{x_n}^{(r_{k})}, \alpha_{k}) \xrightarrow{n \to \infty} 0,
      \quad \forall k \geq 1.
    \end{equation}
    Then there exists an element $a \in \frakD(X)$
    such that $(a_n, x_n) \to (a, x)$ in $\frakD(X) \times X$.  
  \end{lem}

  \begin{proof}
    If necessary,
    by choosing a subsequence,
    we may assume that $(r_{k})_{k \geq 1}$ is strictly increasing.
    By Lemma~\ref{2. lem: convergence of restricted objects},
    for all but countably many $r>0$,
    it holds that 
    \begin{equation}  \label{2. eq: convergence of restrictions for lemma of completeness}
      d_X^\frakC(a_n|_{x_n}^{(r_{k} \wedge r)}, \alpha_k|_x^{(r)}) \xrightarrow{n \to \infty} 0, 
      \quad
      \forall  k \geq 1.
    \end{equation}
    Choose $s_{l} \in (r_{l-1}, r_{l})$ so that \eqref{2. eq: convergence of restrictions for lemma of completeness} 
    holds with $r=s_{l}$, i.e.,
    \begin{equation}  \label{2. eq: choice of s_l for lemma of completeness}
      d_X^\frakC(a_n|_{x_n}^{(r_k \wedge s_{l})}, \alpha_k|_x^{(s_{l})}) \xrightarrow{n \to \infty} 0, 
      \quad
      \forall  k,\,l \geq 1.
    \end{equation}
    For $k' \geq k$,
    by substituting $(k, l)=(k, k)$ and $(k,l) = (k', k)$ in \eqref{2. eq: choice of s_l for lemma of completeness},
    we obtain that 
    \begin{equation}  \label{2. eq: convergence of a_n^s_k for lemma of completeness}
      d_X^\frakC(a_n|_{x_n}^{(s_{k})}, \alpha_k|_x^{(s_{k})}) \xrightarrow{n \to \infty} 0, \quad 
      d_X^\frakC(a_n|_{x_n}^{(s_{k})}, \alpha_{k'}|_x^{(s_{k})}) \xrightarrow{n \to \infty} 0,
    \end{equation}
    which implies that $\alpha_k|_x^{(s_{k})} = \alpha_{k'}|_x^{(s_{k})}$ if $k' \geq k$. 
    Therefore, 
    $(\alpha_k|_x^{(s_{k})}, s_{k})_{k \geq 1}$ is a compatible sequence.
    Since the restriction system is complete,
    we can find $a \in \frakD(X)$ such that $a|_x^{(s_{k})} = \alpha_k|_x^{(s_{k})}$. 
    By \eqref{2. eq: convergence of a_n^s_k for lemma of completeness},
    it holds that $d_X^\frakC(a_n|_{x_n}^{(s_{k})}, a|_x^{(s_{k})}) \to 0$
    as $n \to \infty$ for each $k \geq 1$.
    From Theorem~\ref{thm: convergence in D times X},
    it follows that $(a_n, x_n) \to (a, x)$.
  \end{proof}

  \begin{thm} [{Completeness}] \label{2. thm: completeness of d_D}
    Assume that the restriction system $R$ is complete and satisfies Condition~4, and the metric $d_X$ is complete.
    Then the metric $d^\frakD_X$ is complete.
  \end{thm}

  \begin{proof}
    Fix a Cauchy sequence $(a_n, x_n)_{n \geq 1}$ in $\frakD(X) \times X$.
    The completeness of $d_X$ implies that the sequence $(x_n)_{n \geq 1}$ converges to an element $x \in X$.
    If necessary,
    by choosing a subsequence,
    we may assume that $d^\frakD_X\bigl((a_n, x_n), (a_n, x_{n+1})\bigr) < 2^{-n} e^{-2^{n}}$.
    By the definition of $d^\frakD_X$,
    for some $r_n >2^{n}$,
    we have $d_X^\frakC( a_n|_{x_n}^{(r_n)}, a_{n+1}|_{x_{n+1}}^{(r_n)} ) < 2^{-n}$.
    Assumption~\ref{assum: metrization of D}\ref{assum item: 4. metrization of D} then yields that 
    $(a_n|_{x_n}^{(r)})_{n \geq 1}$ is compact in $\frakC(X)$ for all but countably many $r > 0$.
    Thus, by a diagonal argument, 
    we can find a subsequence $(n_l)_{l \geq 1}$, an increasing sequence of positive numbers with $r_{k} \uparrow \infty$,
    and a sequence $(\alpha_k)_{k \geq 1}$ in $\frakC(X)$ such that 
    \begin{equation}
      d^{\frakC}_X(a_{n_l}|_{x_{n_l}}^{(r_k)}, \alpha_k) \xrightarrow{l \to \infty} 0,
      \quad 
      \forall k \geq 1.
    \end{equation}
    From Lemma~\ref{lem: existence of limit},
    we deduce that $(a_{n_l}, x_{n_l})_{l \geq 1}$ converges in $\frakD(X) \times X$.
    This completes the proof.
  \end{proof}

  Now, we define metrics on $\frakD(X)$. 
  For each $\rho \in X$, we define the distance between $a, b \in \frakD(X)$ by setting 
  \begin{equation} \label{eq: metric on D}
    d^\frakD_{X, \rho}(a, b) 
    \coloneqq 
    d^\frakD_X((a, \rho), (b, \rho))
    =
    \int_{0}^{\infty} e^{-r} \bigl( 1 \wedge d^\frakC_X(a|_\rho^{(r)}, b|_\rho^{(r)}) \bigr)\, dr.
  \end{equation}
  By Proposition \ref{2. prop: d_D is a metric},
  $d^\frakD_{X, \rho}$ is a metric on $\frakD(X)$.
  Below, we verify that the metric induces natural convergence.
  Condition \ref{dfn item: RS. 4} ensures that the induced topology is independent of $\rho$.

  \begin{thm} \label{thm: convergence in D}
    Fix $\rho \in X$.
    Let $a, a_1, a_2, \cdots$ be elements of $\frakD(X)$. 
    The following statements are equivalent.
    \begin{enumerate} [label = \textup{(\roman*)}, leftmargin = *]
      \item \label{thm item: 1, convergence in D}
        It holds that $a_n \to a$ with respect to $d^\frakD_{X, \rho}$.
      \item \label{thm item: 2, convergence in D}
        There exists a sequence $(x_n)_{n \geq 1}$ in $X$ converging to an element $x \in X$
        such that $a_n|_{x_n}^{(r)} \to a|_x^{(r)}$ in $\frakC(X)$ for all but countably many $r > 0$.
      \item \label{thm item: 3, convergence in D}
        There exist a sequence $(x_n)_{n \geq 1}$ in $X$ converging to an element $x \in X$
        and a increasing sequence $(r_k)_{k \geq 1}$ of positive numbers with $r_k \uparrow \infty$
        such that $a_n|_{x_n}^{(r_k)} \to a|_x^{(r_k)}$ as $n \to \infty$ in $\frakC(X)$ for all $k$.
      \item \label{thm item: 4, convergence in D}
        For any elements $x_n \in X$ converging to an element $x \in X$,
        it holds that $a_n|_{x_n}^{(r)} \to a|_x^{(r)}$ in $\frakC(X)$ for all but countably many $r > 0$.
    \end{enumerate}
    In particular, the topology on $\frakD(X)$ induced by $d^\frakD_{X, \rho}$ is independent of $\rho$.
  \end{thm}

  \begin{proof}
    The implication \ref{thm item: 1, convergence in D} $\Rightarrow$ \ref{thm item: 2, convergence in D} follows 
    from Theorem~\ref{thm: convergence in D times X}.
    The implication \ref{thm item: 2, convergence in D} $\Rightarrow$ \ref{thm item: 3, convergence in D} is obvious.
    The implication \ref{thm item: 4, convergence in D} $\Rightarrow$ \ref{thm item: 1, convergence in D} is an immediate consequence of the dominated convergence theorem.
    Assume that \ref{thm item: 3, convergence in D} holds.
    Fix a sequence $(y_n)_{n \geq 1}$ in $X$ converging to an element $y \in X$.
    By Lemma~\ref{2. lem: convergence of restricted objects},
    for all but countably many $r > 0$,
    \begin{equation}
      (a_n|_{x_n}^{(r_k)})|_{y_n}^{(r)} \xrightarrow{n \to \infty} (a|_x^{(r_k)})|_{y}^{(r)},
      \quad 
      \forall k \geq 1.
    \end{equation}
    Since $x_n \to x$ and $y_n \to y$,
    the constant $r_0 \coloneqq \sup_{n \geq 1} d_X(x_n, y_n)$ is finite.
    By \ref{dfn item: RS. 4},
    for each $r > 0$,
    it holds that 
    \begin{equation}
      (a_n|_{x_n}^{(r_k)})|_{y_n}^{(r)} = a_n|_{y_n}^{(r)}, \quad \forall n \geq 1,
      \quad \text{and}\quad
      (a|_x^{(r_k)})|_x^{(r)} = a|_y^{(r)}
    \end{equation}
    for all sufficiently large $k$ satisfying $r_k \geq r_0 + r$.
    Therefore, we deduce that
    $a_n|_{y_n}^{(r)} \to a|_y^{(r)}$ for all but countably many $r>0$,
    which shows \ref{thm item: 4, convergence in D}.
    This completes the proof.
  \end{proof}

  Henceforth, we equip $\frakD(X)$ with the topology induced by $d^\frakD_{X, \rho}$,
  which is independent of the choice of $\rho$ by Theorem~\ref{thm: convergence in D}.
  The following is an immediate consequence of the theorem and Theorem~\ref{thm: convergence in D times X}.

  \begin{cor} \label{cor: topology on D} \leavevmode
    \begin{enumerate} [label = \textup{(\roman*)}, leftmargin = *]
      \item \label{cor item: 1. topology on D}
        The topology on $\frakD(X) \times X$ induced by $d^\frakD_X$ coincides with its product topology.
      \item \label{cor item: 2. topology on D}
        The relative topology on $\frakC(X)$ induced by $\frakD(X)$ is 
        coarser than the topology on $\frakC(X)$.
      \item \label{cor item: 3. topology on D}
        The topology on $\frakD(X)$ only depends on the restriction system and the topology on $\frakC(X)$.
        In particular, it is independent of the metric $d^\frakC_X$.
    \end{enumerate}
  \end{cor}

  \begin{rem}
    In general,
    the relative topology on $\frakC(X)$ induced by $\frakD(X)$ 
    does not coincide with the topology on $\frakC(X)$.
    (See Remark \ref{2. rem: the local Hausdorff top is strictly coarser}.)
  \end{rem}

  We provide summaries of properties of the metrics $d^\frakD_{X, \rho}$ and $d^\frakD_X$.

  \begin{cor} \label{cor: a summary of D}
    Assume that the restriction system $R$ satisfies Condition~2.
    Then, for each $\rho \in X$, 
    the function $d^\frakD_{X, \rho}$ given in \eqref{eq: metric on D} is a well-defined metric on $\frakD(X)$,
    and the topology induced by $d^\frakD_{X, \rho}$ is independent of $\rho$.
    If both $X$ and $\frakC(X)$ are separable, 
    then so is $\frakD(X)$.
    Moreover, if $d_X$ is complete and $R$ is complete and satisfies Condition~4,
    then $d^\frakD_{X, \rho}$ is complete for each $\rho \in X$.
  \end{cor}

  \begin{proof}
    The separability of $d^\frakD_{X, \rho}$ follows from Theorem~\ref{2. thm: separability of d_D} and Corollary~\ref{cor: topology on D},
    and the completeness of $d^\frakD_{X, \rho}$ follows from Theorem~\ref{2. thm: completeness of d_D}.
    The other assertions are already proven.
  \end{proof}

  \begin{cor} \label{cor: a summary of D times X}
    Assume that the restriction system $R$ satisfies Condition~2.
    Then the function $d^\frakD_X$ given in \eqref{eq: metric on D times X} is a well-defined metric on $\frakD(X) \times X$,
    and the induced topology coincides with the product topology.
    If both $X$ and $\frakC(X)$ are separable, 
    then so is $\frakD(X) \times X$.
    Moreover, if $d_X$ is complete, and $R$ is complete and satisfies Condition~4,
    then $d^\frakD_X$ is complete.
  \end{cor}

  \begin{proof}
    This is a consequence of Proposition~\ref{2. prop: d_D is a metric}, 
    Theorems~\ref{2. thm: separability of d_D} and \ref{2. thm: completeness of d_D},
    and Corollary~\ref{cor: topology on D}.
  \end{proof}

  A precompactness criterion is given below.

  \begin{thm} [Precompactness]\label{thm: precompactness in d_D}
    Assume that $R$ is complete and satisfies Condition~3.
    Fix a non-empty subset $\mathfrak{A} \subseteq \frakD(X)$ and $\rho \in X$.
    Write $\mathfrak{A}|_\rho^{(r)} \coloneqq \{ a|_\rho^{(r)} \mid a \in \mathfrak{A} \} \subseteq \frakC(X)$ for each $r>0$.
    The following statements are equivalent.
    \begin{enumerate} [label = \textup{(\roman*)}, leftmargin = *]
      \item \label{2. thm item: A^r is precompact for each r}
        The set $\mathfrak{A}|_\rho^{(r)}$ is precompact in $\frakC(X)$ for all $r>0$.
      \item \label{2. thm item: A^r is precompact for unbounded r}
        There exists an increasing sequence $(r_{k})_{k \geq 1}$ with $r_{k} \uparrow \infty$ 
        such that $\mathfrak{A}|_\rho^{(r_{k})}$ is precompact in $\frakC(X)$.
      \item \label{2. thm item: A is precompact}
        The set $\mathfrak{A}$ is precompact in $\frakD(X)$.
    \end{enumerate}
  \end{thm}

  \begin{proof}
    The implication \ref{2. thm item: A^r is precompact for each r}
    $\Rightarrow$
    \ref{2. thm item: A^r is precompact for unbounded r}
    is obvious.
    Assume that \ref{2. thm item: A^r is precompact for unbounded r} holds.
    Fix a sequence $(a_{n})_{n \geq 1}$ in $\mathfrak{A}$.
    By a diagonal argument,
    one can find a subsequence $(a_{n_{m}})_{m \geq 1}$ and 
    a sequence $(\alpha_{k})_{k \geq 1}$ in $\frakC(X)$ such that 
    $a_{n_{m}}|_\rho^{(r_{k})} \to \alpha_{k}$
    as $m \to \infty$ for each $k \geq 1$.
    By Lemma~\ref{lem: existence of limit}
    we deduce that $(a_{n_{m}})_{m \geq 1}$ is a convergent sequence in $\frakD(X)$,
    which implies \ref{2. thm item: A is precompact}.

    Suppose that \ref{2. thm item: A is precompact} holds.
    Fix $r>0$ and a sequence $(a_{n}|_\rho^{(r)})_{n \geq 1}$ in $\mathfrak{A}|_\rho^{(r)}$.
    By \ref{2. thm item: A is precompact},
    if necessarily,
    by choosing a subsequence,
    we may assume that $(a_n)_{n \geq 1}$ converges to some element $a \in \frakD(X)$
    with respect to $d^\frakD_{X, \rho}$.
    Moreover, 
    we can assume that $d^\frakD_{X, \rho}(a_{n_{m}}, a) < 2^{-m} e^{-m}$.
    Then there exists $r_m > m$ such that $d_X^\frakC(a_{n_{m}}|_\rho^{(r_m)}, a|_\rho^{(r_m)}) < 2^{-m}$.
    Assumption~\ref{assum: metrization of D}\ref{assum item: 4. metrization of D} yields that 
    $\{a_{n_m}|_\rho^{(r)}\}_{m \geq 1}$ is precompact,
    which shows \ref{2. thm item: A^r is precompact for each r}.
  \end{proof}


\section{Metrization of several topologies} \label{sec: metrization of several topologies}

In this section,  
using the method introduced in the preceding section,  
we construct metrizations of certain extended topologies.  
In Sections~\ref{sec: the Fell topology} and \ref{sec: the vague topology},  
we consider extensions of the Hausdorff topology and the weak topology, respectively.  
In the final subsection, Section~\ref{sec: variable domains},  
we introduce an extension of the compact-convergence topology and discuss its metrization.  
For each of these topologies, we also investigate its topological properties,  
such as Polishness and precompactness.  
Throughout this section, we fix a metric space $X$.


\subsection{The Fell topology} \label{sec: the Fell topology}

In this subsection,
using the framework established in Section~\ref{sec: Metric for non-compact objects},
we extend the Hausdorff metric to a metric on the collection of closed subsets.
The resulting metric induces the Fell topology (cf.\ \cite[Appendix~C]{Molchanov_17_Theory})

\begin{dfn} [{The space $\Compact{X}$ and $\Closed{X}$}]  \label{dfn: space of compact and closed subsets}
  We define $\Closed{X}$ to be the set of closed subsets of $S$.
  We denote by $\Compact{X}$ the subset of $\Closed{X}$ consisting of compact subsets.
  (NB.\ Both sets include the empty set.)
\end{dfn}

We equip $\Compact{X}$ with the Hausdorff metric $\HausMet{X}$.
To recall it,
we write, for each subset $A$ and $\varepsilon \geq 0$,
\begin{equation}  \label{2. eq: e-neighborhood}
  A^{\varepsilon} 
  \coloneqq
  \{
    x \in X \mid \exists y \in A\ \text{such that}\ d_X(x,y) \leq \varepsilon
  \},
\end{equation}
which is the \textit{(closed) $\varepsilon$-neighborhood} of $A$ in $X$.
The Hausdorff metric $\HausMet{X}$ on $\Compact{X}$ is then defined by
\begin{equation} \label{eq: dfn of Hausdorff metric}
  \HausMet{X}(A,B)
  \coloneqq
  \inf \{
    \varepsilon \geq 0 \mid A \subseteq B^{\varepsilon},\, B \subseteq A^{\varepsilon}
  \},
\end{equation} 
where we set the infimum over the empty set to be $\infty$.
The function $\HausMet{X}$ is indeed an extended metric on $\Compact{X}$
(see \cite[Section~17.6]{Cech_69_Point}).
(Note that the distance between the empty set and a non-empty set is always infinite.)
We call the topology on $\Compact{X}$ induced by $\HausMet{X}$ the \textit{Hausdorff topology}
(also known as the \emph{Vietoris topology}).
Below, we collect basic properties of the Hausdorff topology and the Hausdorff metric.

\begin{lem} [{\cite[Sections~17.6.5 and 17.6.7]{Cech_69_Point}}] \label{lem: Polishness of Hausdorff}
  If $d_X$ is complete, then so is $\HausMet{X}$.
  If $X$ is separable, then so is the Hausdorff topology.
\end{lem}

The following gives a characterization of convergence in the Hausdorff topology  
in terms of Painlev\'{e}--Kuratowski convergence (cf.\ \cite[Definition~C.6]{Molchanov_17_Theory}).

\begin{lem} \label{lem: convergence in Hausdorff}
  Let $A, A_1, A_2, \dots$ be elements of $\Compact{X}$.
  Then $A_n \to A$ in the Hausdorff topology if and only if $\bigcup_{n \geq 1} A_n$ is compact and the following conditions are satisfied.
  \begin{enumerate} [label = \textup{(PK\,\arabic*)}, leftmargin = *]
    \item \label{lem item: 1. convergence in Hausdorff}
      If elements $x_n \in A_n$ converge to $x$, then $x \in A$.
    \item \label{lem item: 2. convergence in Hausdorff}
      For any $x \in A$, there exist a subsequence $(n_k)_{k \geq 1}$ and elements $x_{n_k} \in A_{n_k}$
      such that $x_{n_k} \to x$.
  \end{enumerate}
  Condition~\ref{lem item: 1. convergence in Hausdorff} can be replaced by the following:
  \begin{enumerate} [label = \textup{(PK\,2$^\prime$)}, leftmargin = *]
    \item \label{lem item: 2'. convergence in Hausdorff}
      For any $x \in A$, there exist elements $x_n \in A_n$ converging to $x$.
  \end{enumerate}
\end{lem}

\begin{proof}
  The equivalence is well-known (see \cite[Section~17.6.3]{Cech_69_Point}).
  By definition, one can easily verify that \ref{lem item: 2'. convergence in Hausdorff} follows from the convergence of $A_n$ to $A$.
  This completes the proof.
\end{proof}

\begin{rem} \label{rem: Hausdorff topology is independent of metric}
  It follows from Lemma~\ref{lem: convergence in Hausdorff} that the Hausdorff topology depends only on the topology of $X$.  
  In particular, it is independent of the choice of the metric $d_X$.
\end{rem}

\begin{lem} \label{lem: precompact in Hausdorff}
  A subset $\mathcal{A}$ of $\Compact{X}$ is precompact in the Hausdorff topology 
  if and only if there exists $K \in \Compact{X}$ such that $A \subseteq K$ for all $A \in \mathcal{A}$.
  In particular, if $d_X$ is boundedly compact, then so is the Hausdorff metric.
\end{lem}

\begin{proof}
  The precompactness criterion is well-known (see \cite[Section~17.6.8]{Cech_69_Point}).
  The second assertion is straightforward.
\end{proof}

The Hausdorff metric is stable under deformation of the underlying space.
This property plays a fundamental role in the metrization of the Gromov--Hausdorff topology (recall \eqref{eq: the GH metric}).
Henceforth, given a map $f \colon S \to T$, we denote by $\Image{f}(A)$ the image of $A$ under $f$, i.e.,
\begin{equation}  \label{dfn: image map}
  \Image{f}(A) \coloneqq \{f(x) \mid x \in A\}.
\end{equation}

\begin{lem} \label{lem: Hausdorff is stable}
  Fix metric spaces $X$, $Y$, $M_1$, and $M_2$.
  Let $f_i \colon X \to M_i$ and $g_i \colon Y \to M_i$, $i = 1,2$, be isometric embeddings.
  Assume that there exists $\varepsilon \in \RNp$ such that, for all $x \in X$ and $y \in Y$,
  \begin{equation} \label{lem eq: 1. Hausdorff is stable}
    d_{M_2}(f_2(x), g_2(y)) \leq d_{M_1}(f_1(x), g_1(y)) + \varepsilon.
  \end{equation}
  Then, for all $A \in \Compact{X}$ and $B \in \Compact{Y}$, 
  \begin{equation} \label{lem eq: 2. Hausdorff is stable}
    \HausMet{M_2} \bigl( \Image{f_2}(A), \Image{g_2}(B) \bigr) \leq \HausMet{M_1}\bigl( \Image{f_1}(A), \Image{g_1}(B) \bigr) + \varepsilon.
  \end{equation}
\end{lem}

\begin{proof}
  This is straightforward (cf.\ \cite[The proof of Lemma~4.1]{Abraham_Delmas_Hoscheit_13_A_note}).
\end{proof}

From now on, we assume that $X$ is boundedly compact.
Using the framework developed in Section~\ref{sec: Metric for non-compact objects},
we extend the Hausdorff metric to a metric on $\Closed{X}$.
To this end, we define a restriction system $R = (R_{X, x}^{(r)})_{r > 0,\, x \in X}$ from $\Closed{X}$ to $\Compact{X}$ as follows:
for each $r > 0$ and $x \in X$,
we set 
\begin{equation} \label{eq: restriction system for Fell}
  R_x^{(r)}(A)
  =
  A|_x^{(r)} 
  \coloneqq
  A \cap D_X(x, r),
  \quad 
  A \in \Closed{X}.
\end{equation}

\begin{prop} \label{prop: RS for local Hausdorff}
  The restriction system $R$ from $\Closed{X}$ to $\Compact{X}$ is complete and satisfies Condition~4.
\end{prop}

\begin{proof}
  The proof is given in Appendix~\ref{appendix: local Hausdorff}.
\end{proof}

For each $\rho \in X$ and $A, B \in \Closed{X}$,
define
\begin{equation} \label{2. eq: the local Hausdorff metric}
  \lHausMet{X,\rho}(A, B) 
  \coloneqq 
  \int_{0}^{\infty}
  e^{-r}
  \bigl(1 \wedge \HausMet{X}(A|_\rho^{(r)}, B|_\rho^{(r)})\bigr)\,
  dr.
\end{equation}
The following is an immediate consequence of Corollary~\ref{cor: a summary of D}, Lemma~\ref{lem: Polishness of Hausdorff}, and Proposition~\ref{prop: RS for local Hausdorff}.

\begin{thm} \label{thm: local Hausdorff metric}
  The function $\lHausMet{X, \rho}$ is a well-defined complete, separable metric on $\Closed{X}$,
  and the topology on $\Closed{X}$ induced by $\lHausMet{X, \rho}$ is independent of $\rho$.
\end{thm}

We equip $\Closed{X}$ with the topology induced by the metric $\lHausMet{X, \rho}$, which is independent of $\rho$.
The next result shows that this topology coincides with the Fell topology.
Below, we use the characterization of the Fell topology in terms of the Painlev\'{e}--Kuratowski convergence
(see \cite[Theorem~C.7]{Molchanov_17_Theory} and also \cite[Theorem~12.2.2]{Schneider_Weil_Stochastic}).

\begin{thm} [Convergence]
  \label{thm: convergence in the Fell topology}
  Let $A, A_{1}, A_{2}, \ldots$ be elements of $\Closed{X}$.
  Then the following are equivalent.
  \begin{enumerate} [label = \textup{(\roman*)}, leftmargin = *]
    \item \label{thm item: 1. convergence in the local Hausdorff}
      It holds that $A_n \to A$ in $\Closed{X}$.
    \item \label{thm item: 2. convergence in the local Hausdorff}
      There exists a sequence $(x_n)_{n \geq 1}$ in $X$ converging to an element $x \in X$,
      such that $A_n|_{x_n}^{(r)} \to A|_x^{(r)}$ in the Hausdorff topology for all but countably many $r > 0$.
    \item \label{thm item: 3. convergence in the local Hausdorff}
      There exist a sequence $(x_n)_{n \geq 1}$ in $X$ converging to $x \in X$,
      and an increasing sequence $(r_k)_{k \geq 1}$ of positive numbers with $r_k \uparrow \infty$,
      such that $A_n|_{x_n}^{(r_k)} \to A|_x^{(r_k)}$ in the Hausdorff topology for all $k$.
    \item \label{thm item: 4. convergence in the local Hausdorff}
      For any sequence $(x_n) \subset X$ converging to $x \in X$,
      it holds that $A_n|_{x_n}^{(r)} \to A|_x^{(r)}$ in the Hausdorff topology for all but countably many $r > 0$.
    \item \label{thm item: 5. convergence in the local Hausdorff}
      It holds that $A_n \to A$ in the Fell topology,
      that is, 
      Conditions~\ref{lem item: 1. convergence in Hausdorff} 
      and \ref{lem item: 2. convergence in Hausdorff} (or equivalently, \ref{lem item: 1. convergence in Hausdorff} and \ref{lem item: 2'. convergence in Hausdorff}) 
      are satisfied.
  \end{enumerate}
\end{thm}

\begin{proof}
  The equivalence of \ref{thm item: 1. convergence in the local Hausdorff}--\ref{thm item: 4. convergence in the local Hausdorff} follows from Theorem~\ref{thm: convergence in D times X}.
  The implication \ref{thm item: 4. convergence in the local Hausdorff} $\Rightarrow$ \ref{thm item: 5. convergence in the local Hausdorff} is immediate from Lemma~\ref{lem: convergence in Hausdorff}.
  Assume \ref{thm item: 5. convergence in the local Hausdorff} holds.
  Fix $\rho \in X$.
  Let $r > 0$ be such that $A|_\rho^{(r)} = \closure(A \cap B_X(\rho, r))$.
  Note that all but countably many $r>0$ satisfy this property (see Lemma~\ref{lem: RS for local Hausdorff is conti}).
  Then, by Lemma~\ref{lem: convergence in Hausdorff},
  it follows that $A_n|_\rho^{(r)} \to A|_\rho^{(r)}$ in the Hausdorff topology.
  Thus, \ref{thm item: 2. convergence in the local Hausdorff} holds.
\end{proof}

\begin{rem} \label{2. rem: the local Hausdorff top is strictly coarser}
  The relative topology on $\Compact{X}$ induced by the Fell topology on $\Closed{X}$ 
  is strictly coarser than the Hausdorff topology.
  For example,
  consider $X = \mathbb{R}$ with the standard Euclidean metric,
  and the sequence $A_{n} \coloneqq [n, n+1]$.
  Then $A_{n}$ converges to the empty set in the Fell topology, 
  but does not converge in the Hausdorff topology.
\end{rem}

The following is well-known (see \cite[Theorem C.2]{Molchanov_17_Theory}).

\begin{thm} \label{thm: local Hausdorff top is compact}
  The Fell topology on $\Closed{X}$ is compact.
\end{thm}

By \eqref{eq: metric on D times X},
we can metrize the product space $\Closed{X} \times X$ given as follows:
for each $(A, x), (B, y) \in \Closed{X} \times X$,
\begin{equation} \label{eq: local Hausdorff metric on product}
  \lHausMet{X}\bigl((A, x), (B, y)\bigr)
  \coloneqq 
  d_X(x,y) 
  \vee 
  \int_{0}^{\infty}
  e^{-r}
  \bigl(1 \wedge \HausMet{X}(A|_x^{(r)}, B|_y^{(r)})\bigr)\,
  dr.
\end{equation}
This metric will be used in Sections~\ref{sec: the local GH topology} and \ref{sec: main results} later.
The following is an immediate consequence of Corollary~\ref{cor: a summary of D times X} and Proposition~\ref{prop: RS for local Hausdorff}.

\begin{prop} \label{prop: product local Hausdorff metric}
  The function $\lHausMet{X}$ is a well-defined complete, separable metric on $\Closed{X} \times X$,
  and the induced topology coincides with the product topology.
\end{prop}

Similarly to the Hausdorff metric,
the metric $\lHausMet{X}$ is stable under deformation of the underlying space.
This plays a crucial role in the metrization of Gromov--Hausdorff-type topologies in our main results.  
In what follows, given a function $f \colon S \to T$, we write  
\begin{equation} \label{eq: image product map}
  \FE{\Image{f}}(A, x) \coloneqq (\Image{f}(A), f(x))
\end{equation}
for each subset $A \subseteq S$ and point $x \in S$.  
(The notation $\FE{\cdot}$ stands for ``element-rooted,''  
as introduced in Section~\ref{sec: Metrization of structures} below.)

\begin{prop} \label{prop: local Hausdorff is stable}
  Fix $\bcmAB$ spaces $X$, $Y$, $M_1$, and $M_2$.
  Let $f_i \colon X \to M_i$ and $g_i \colon Y \to M_i$, $i = 1,2$, be isometric embeddings.
  Assume that there exists $\varepsilon \in \RNp$ such that, for all $x \in X$ and $y \in Y$,
  \begin{equation} \label{prop eq: 1. local Hausdorff is stable}
    d_{M_2}(f_2(x), g_2(y)) \leq d_{M_1}(f_1(x), g_1(y)) + \varepsilon.
  \end{equation}
  Then, for all $(A, x) \in \Closed{X} \times X$ and $(B, y) \in \Closed{Y} \times Y$, 
  \begin{equation} \label{prop eq: 2. local Hausdorff is stable}
    \lHausMet{M_2} \bigl( \FE{\Image{f_2}}(A, x), \FE{\Image{g_2}}(B, y) \bigr) 
    \leq 
    \lHausMet{M_1}\bigl( \FE{\Image{f_1}}(A, x), \FE{\Image{g_1}}(B, y) \bigr) + \varepsilon \wedge 1.
  \end{equation}
\end{prop}

\begin{proof}
  For any $(A, x) \in \Closed{X} \times X$,
  we have
  \begin{align}
    f(A)|_{f(x)}^{(r)} 
    = \{ f(y) \in Y \mid d_Y(f(x), f(y)) \leq r \} 
    = \{ f(y) \in Y \mid d_X(x, y) \leq r \} 
    = f(A|_x^{(r)}).
  \end{align}
  From this and Lemma~\ref{lem: Hausdorff is stable},
  the desired result follows.
\end{proof}

The following is an immediate consequence of the above proposition.

\begin{prop} \label{prop: the local Hausdorff metric is preserved}
  Let $X$ and $Y$ be $\bcmAB$ spaces,
  and let $f \colon X \to Y$ be an isometric embedding. 
  Then the map $\FE{\Image{f}} \colon \Closed{X} \times X \to \Closed{Y} \times Y$ is an isometric embedding.
  In particular,
  the map $\Image{f} \colon \Closed{X} \to \Closed{Y}$ is an isometric embedding 
  with respect to $\lHausMet{X, \rho}$ and $\lHausMet{Y, f(\rho)}$ for any $\rho \in X$.
\end{prop}


\subsection{The vague topology}  \label{sec: the vague topology}
There are various versions of metrics  
that induce the vague topology (see, e.g., \cite[Section A2.6]{Daley_Jones_03_Vol_1} and \cite[Section 4.1]{Kallenberg_17_Random}).  
In this subsection,  
we define one such metric in a manner analogous to the previous subsection.

\begin{dfn} 
  We define $\Meas{X}$ to be the set of Radon measures $\mu$ on $X$,
  that is, 
  $\mu$ is a Borel measure on $X$ such that $\mu(K) < \infty$ for every compact subset $K$.
  We denote by $\finMeas{X}$ (resp.\ $\Prob{X}$, $\cptMeas{X}$) 
  the subset of $\Meas{X}$ consisting of finite Borel measures (resp.\ probability measures, compactly supported measures).
\end{dfn}

A commonly used metric on $\finMeas{X}$ is the \emph{Prohorov metric},
given as follows:
for $\mu, \nu \in \finMeas{X}$,
we define 
\begin{equation}
  \ProhMet{X}(\mu, \nu)
  \coloneqq
  \inf
  \left\{
    \varepsilon > 0 \mid
    \mu(A) \leq \nu(A^{\varepsilon}) + \varepsilon,\,
    \nu(A) \leq \mu(A^{\varepsilon}) + \varepsilon
    \ \text{for all Borel subsets } A \subseteq X
  \right\}.
\end{equation} 
Since $\Prob{X}$ and $\cptMeas{X}$ are subsets of $\finMeas{X}$,
we also equip these spaces with the Prohorov metric.

\begin{lem} [{\cite[pp.\ 72–73]{Billingsley_99_Convergence}}] \label{lem: basics of Prohorov} 
  The function $\ProhMet{X}$ is a metric on $\finMeas{X}$.
  If $X$ is separable,
  then the topology on $\Prob{X}$ induced by $\ProhMet{X}$ coincides with the weak topology.
  If $d_X$ is complete, then so is $\ProhMet{X}$.
\end{lem}

Similarly to the Hausdorff metric (Lemma~\ref{lem: Hausdorff is stable}),
the Prohorov metric is stable under deformation of the underlying space.
Henceforth, given a measurable map $f \colon S \to T$ between measurable spaces, 
we write $f_*$ for the associated pushforward map,
that is, for each measure $\mu$ on $S$, we define $f_*\mu$ by 
\begin{equation}  \label{eq: pushforward map}
  f_*(\mu)(\cdot) \coloneqq \mu(f^{-1}(\cdot)).
\end{equation}

\begin{lem} \label{lem: Prohorov is stable}
  Fix metric spaces $X$, $Y$, $M_1$, and $M_2$.
  Let $f_i \colon X \to M_i$ and $g_i \colon Y \to M_i$, $i = 1,2$, be isometric embeddings.
  Assume that there exists $\varepsilon \in \RNp$ such that, for all $x \in X$ and $y \in Y$,
  \begin{equation} \label{lem eq: 1. Prohorov is stable}
    d_{M_2}(f_2(x), g_2(y)) \leq d_{M_1}(f_1(x), g_1(y)) + \varepsilon.
  \end{equation}
  Then, for all $\mu \in \finMeas{X}$ and $\nu \in \finMeas{Y}$, 
  \begin{equation} \label{lem eq: 2. Prohorov is stable}
    \ProhMet{M_2} \bigl( (f_2)_*\mu, (g_2)_*\nu \bigr) \leq \ProhMet{M_1} \bigl( (f_1)_*\mu, (g_1)_*\nu \bigr) + \varepsilon.
  \end{equation}
\end{lem}

\begin{proof}
  This follows from the argument in the proof of \cite[Lemma~4.1]{Abraham_Delmas_Hoscheit_13_A_note}.
\end{proof}

To extend the Prohorov metric to a metric on $\Meas{X}$,
we define a restriction system $R = (R_x^{(r)})_{r > 0, x \in X}$ from $\Meas{X}$ to $\cptMeas{X}$ as follows:
for each $r > 0$ and $x \in X$,
\begin{equation}
  R_x^{(r)}(\mu)(\cdot)
  =
  \mu|_x^{(r)}(\cdot)
  \coloneqq
  \mu(\cdot \cap D_X(x, r)),
  \quad 
  \mu \in \Meas{X}.
\end{equation}

\begin{lem} \label{lem: RS for vague}
  The restriction system $R$ from $\Meas{X}$ to $\cptMeas{X}$ is complete and satisfies Condition~4.
\end{lem}

\begin{proof}
  The proof is given in Appendix~\ref{appendix: vague}.
\end{proof}

For each $\rho \in X$ and $\mu, \nu \in \Meas{X}$,
define
\begin{equation} \label{2. eq: definition of the vague metric}
  \Vague{X, \rho}(\mu, \nu)
  \coloneqq
  \int_{0}^{\infty}
  e^{-r}
  \bigl(1 \wedge \ProhMet{X} (\mu|_\rho^{(r)}, \nu|_\rho^{(r)})\bigr)\, dr.
\end{equation}
The following is an immediate consequence of Corollary~\ref{cor: a summary of D} and Proposition~\ref{prop: RS for local Hausdorff}.

\begin{thm} \label{thm: vague metric}
  The function $\Vague{X, \rho}$ is a well-defined complete, separable metric on $\Meas{X}$,
  and the topology on $\Meas{X}$ induced by $\Vague{X, \rho}$ is independent of $\rho$.
\end{thm}

We equip $\Meas{X}$ with the topology induced by the metric $\Vague{X, \rho}$, which is independent of $\rho$.
The next result shows that this topology coincides with the vague topology.

\begin{thm} [Convergence] \label{thm: convergence in vague}
  Let $\mu, \mu_{1}, \mu_{2}, \ldots$ be elements of $\Meas{X}$.
  Then the following are equivalent.
  \begin{enumerate} [label = \textup{(\roman*)}, leftmargin = *]
    \item \label{thm item: 1, convergence in vague}
      The sequence $\mu_n$ converges to $\mu$ in $\Meas{X}$.
    \item \label{thm item: 2, convergence in vague}
      There exists a sequence $(x_n)_{n \geq 1}$ in $X$ converging to some $x \in X$
      such that $\mu_n|_{x_n}^{(r)} \to \mu|_x^{(r)}$ weakly for all but countably many $r > 0$.
    \item \label{thm item: 3, convergence in vague}
      There exist a sequence $(x_n)_{n \geq 1}$ in $X$ converging to some $x \in X$,
      and an increasing sequence $(r_k)_{k \geq 1}$ with $r_k \uparrow \infty$
      such that $\mu_n|_{x_n}^{(r_k)} \to \mu|_x^{(r_k)}$ weakly for each $k$.
    \item \label{thm item: 4, convergence in vague}
      For any sequence $(x_n)$ converging to some $x \in X$,
      we have $\mu_n|_{x_n}^{(r)} \to \mu|_x^{(r)}$ weakly for all but countably many $r > 0$.
    \item \label{thm item: 5, convergence in vague}
      The sequence $\mu_n$ converges to $\mu$ vaguely, that is,
      \begin{equation}
        \lim_{n \to \infty} \int_X f(x)\, \mu_n(dx) 
        =
        \int_X f(x)\, \mu(dx)
      \end{equation}
      for all continuous functions $f \colon X \to \mathbb{R}$ with compact support.
  \end{enumerate}
\end{thm}

\begin{proof}
  The equivalence of \ref{thm item: 1, convergence in vague}--\ref{thm item: 4, convergence in vague}
  follows from Theorem~\ref{thm: convergence in D times X}.
  It is easy to see that 
  \ref{thm item: 2, convergence in vague}
  implies 
  \ref{thm item: 5, convergence in vague}.
  Assume that \ref{thm item: 5, convergence in vague} holds.
  Fix $\rho \in X$.
  Let $r > 0$ be such that $\mu(\{x \in X \mid d_X(\rho, x) = r\}) = 0$.
  Then, by \cite[Lemma~4.1]{Kallenberg_17_Random},
  we have $\mu_n|_\rho^{(r)}(X) \to \mu|_\rho^{(r)}(X)$.
  For any closed set $C \subseteq X$,
  it follows from the same lemma that 
  \begin{equation}
    \limsup_{n \to \infty} \mu_n|_\rho^{(r)}(C)
    \leq 
    \mu|_\rho^{(r)}(C).
  \end{equation}
  Therefore, by \cite[Theorem~A.2.3.II]{Daley_Jones_03_Vol_1},
  we conclude that $\mu_n|_\rho^{(r)} \to \mu|_\rho^{(r)}$ weakly,
  which proves \ref{thm item: 2, convergence in vague}.
\end{proof}

\begin{cor}
  The topology on $\Meas{X}$ coincides with the vague topology.
\end{cor}

\begin{proof}
  This follows immediately from Theorem~\ref{thm: convergence in vague}.
\end{proof}

We define a natural metrization of the product space $\Meas{X} \times X$ as follows:
for each pair $(\mu, x), (\nu, y)\allowbreak \in \Meas{X} \times X$,
\begin{equation} \label{eq: vague product metric}
  \Vague{X}\bigl((\mu, x), (\nu, y)\bigr)
  \coloneqq 
  d_X(x,y) 
  \vee 
  \int_{0}^{\infty}
  e^{-r}
  \bigl(1 \wedge \ProhMet{X}(\mu|_x^{(r)}, \nu|_y^{(r)})\bigr)\,
  dr.
\end{equation}
The following is an immediate consequence of Proposition~\ref{prop: RS for local Hausdorff} and Corollary~\ref{cor: a summary of D}.

\begin{prop} \label{prop: vauge product metric}
  The function $\Vague{X}$ is a well-defined complete, separable metric on $\Meas{X} \times X$,
  and the induced topology coincides with the product topology.
\end{prop}

The metric $\Vague{X}$ inherits the stability of the Prohorov metric (Lemma~\ref{lem: Prohorov is stable}).
Henceforth, given a measurable map $f \colon S \to T$ between measurable spaces, 
we write $\FE{f_*} \coloneqq f_* \times f$.

\begin{prop} \label{prop: vague is stable}
  Fix $\bcmAB$ spaces $X$, $Y$, $M_1$, and $M_2$.
  Let $f_i \colon X \to M_i$ and $g_i \colon Y \to M_i$, $i = 1,2$, be isometric embeddings.
  Assume that there exists $\varepsilon \in \RNp$ such that, for all $x \in X$ and $y \in Y$,
  \begin{equation} \label{prop eq: 1. vague is stable}
    d_{M_2}(f_2(x), g_2(y)) \leq d_{M_1}(f_1(x), g_1(y)) + \varepsilon.
  \end{equation}
  Then, for all $(\mu, x) \in \Meas{X} \times X$ and $(\nu, y) \in \Meas{Y} \times Y$, 
  \begin{equation} \label{prop eq: 2. vague is stable}
    \Vague{M_2} \bigl( \FE{(f_2)_*}(\mu, x), \FE{(g_2)_*}(\nu, y) \bigr) 
    \leq 
    \Vague{M_1} \bigl( \FE{(f_1)_*}(\mu, x), \FE{(g_1)_*}(\nu, y) \bigr) + \varepsilon.
  \end{equation}
\end{prop}

\begin{proof}
  For any $x \in X$, $r > 0$, and Borel subset $A \subseteq M_2$,
  we have
  \begin{align}
    f^{-1}(A \cap D_{M_2}(f(x), r)) 
    &= \{ y \in X \mid f(y) \in A,\ d_{M_2}(f(y), f(x)) \leq r \} \\
    &= \{ y \in X \mid f(y) \in A,\ d_X(y, x) \leq r \} \\
    &= f^{-1}(A) \cap D_X(x, r).
  \end{align}
  Thus, for any $\mu \in \Meas{X}$ and Borel $A \subseteq M_2$,
  \begin{equation}
    (f_*\mu)|_{f(x)}^{(r)}(A)
    =
    f_*\mu ( A \cap D_{M_2}(f(x), r)) 
    =
    \mu (f^{-1}(A) \cap D_X(x, r))
    =
    (\mu|_x^{(r)}) \circ f^{-1}(A).
  \end{equation}
  Hence, $(f_*\mu)|_{f(x)}^{(r)} = (f_* (\mu|_x^{(r)}))$, i.e., restriction and pushforward commute.
  From this and Lemma~\ref{lem: Prohorov is stable}, the result follows.
\end{proof}

The following is an immediate consequence of the above proposition.

\begin{prop} \label{prop: vague metric is distance-preserved}
  Let $X$ and $Y$ be $\bcmAB$ spaces,
  and let $f \colon  X \to Y$ be an isometric embedding. 
  Then the map $\FE{f_*} \colon \Meas{X} \times X \to \Meas{Y} \times Y$ is an isometric embedding.
  In particular,
  the map $f_* \colon \Meas{X} \to \Meas{Y}$ is an isometric embedding 
  with respect to $\Vague{X, \rho}$ and $\Vague{Y, f(\rho)}$ for any $\rho \in X$.
\end{prop}

\begin{rem}
  Another metric inducing the vague topology is given in \cite[Lemma~4.6]{Kallenberg_17_Random}, 
  but it is not clear whether it satisfies the stability property in Proposition~\ref{prop: vague is stable}.
  In \cite[Section~A2.6]{Daley_Jones_03_Vol_1},
  a metric defined similarly to $\Vague{X, \rho}$ is proposed.
  However,
  as pointed out in \cite{Morariu-Patrichi_18_On},
  there are errors in the proofs given in \cite{Daley_Jones_03_Vol_1},
  which is the reason we do not adopt that metric here.
\end{rem}


\subsection{The compact-convergence topology with variable domains}  \label{sec: variable domains}
\newcommand{\hatC}[2]{\widehat{C}(#1, #2)}
\newcommand{\hatCc}[2]{\widehat{C}_c(#1, #2)}

\newcommand{\MarkeredCompact}[2]{\mathcal{C}_c(#1, #2)}
\newcommand{\GraphSp}[2]{\mathcal{C}(#1, #2)}
\newcommand{\subGraphSp}[4]{\mathcal{C}_{#1}^{(#2)}(#3, #4)}

\newcommand{\GraphMet}[2]{d^{\mathcal{C}}_{#1, #2}}

\newcommand{\hatCcMet}[2]{d^{\widehat{C}_c}_{#1, #2}}
\newcommand{\hatCMet}[2]{d^{\widehat{C}}_{#1, #2}}

\newcommand{\hatCond}[3]{\widehat{C}(#1, #2; #3)}

We next introduce a topology 
on a collection of functions with varying domains.
This topological framework arises, for example, 
when considering the metrization of a Gromov--Hausdorff-type topology 
on the collection of triples $(X, \rho_X, f)$, 
where the additional object $f$ is an element of $C(X, \Xi)$ 
and $\Xi$ is a fixed metric space.
Here, $C(X, \Xi)$ denotes the set of continuous functions $f \colon X \to \Xi$, 
equipped with the compact-convergence topology;
that is, 
$f_n$ converges to $f$ if and only if $f_n$ converges to $f$ uniformly on every compact subset.

A key difficulty in defining a Gromov--Hausdorff-type metric on such a space 
is that $C(X, \Xi)$ cannot be naturally embedded into $C(Y, \Xi)$ 
even when $X$ is a subspace of $Y$.
To address this issue, 
we consider the set $\hatC{X}{\Xi}$ 
consisting of functions from closed subsets of $X$ to $\Xi$.
This leads to a natural embedding of $\hatC{X}{\Xi}$ into $\hatC{Y}{\Xi}$.
With this background,
we define a metric on $\hatC{X}{\Xi}$ that extends the compact-convergence topology on $C(X, \Xi)$.
We also mention a related work \cite{Cao_23_Convergence},
which introduces a notion of convergence in $\hatC{X}{\Xi}$ when $X$ is compact.
The framework developed below provides a metrization of this convergence,
and further investigates topological properties such as Polishness and precompactness.
Moreover, the space $\hatC{X}{\Xi}$ will be useful for the study of scaling limits of lattice models; 
see Remark~\ref{rem: hatC applies to lattice model} below.

In addition to $X$, we fix another metric space $\Xi$.

\begin{dfn}[{The sets $\hatCc{X}{\Xi}$ and $\hatC{X}{\Xi}$}]
  We define
  \begin{equation}
    \hatC{X}{\Xi}
    \coloneqq
    \bigcup_{S \in \Closed{X}} 
    C(S, \Xi).
  \end{equation}
  Note that $\hatC{X}{\Xi}$ contains the empty map $\emptyset_{\Xi} \colon \emptyset \to \Xi$.
  For a function $f$, we denote its domain by $\dom(f)$.
  Then we define $\hatCc{X}{\Xi}$ 
  to be the subset of $\hatC{X}{\Xi}$ consisting of those functions $f$ whose domain $\dom(f)$ is compact.
\end{dfn}

To define metrics on $\hatCc{X}{\Xi}$ and $\hatC{X}{\Xi}$,
we borrow an idea from \cite[Section~4.5]{Khezeli_23_A_unified}:
we identify each function with its graph.
For each function $f$,
we write $\graphmap(f)$ for its graph, i.e.,
\begin{equation} \label{eq: def of graphmap}
  \graphmap(f) \coloneqq \{(x, f(x)) \mid x \in \dom(f)\}.
\end{equation}
We then define
\begin{equation}
  \hatCcMet{X}{\Xi}(f, g)
  \coloneqq
  \HausMet{X \times \Xi}(\graphmap(f), \graphmap(g)),
  \quad
  f, g \in \hatCc{X}{\Xi}.
\end{equation}
Here, $\HausMet{X \times \Xi}$ denotes the Hausdorff metric on $\Compact{X \times \Xi}$,
the set of compact subsets of $X \times \Xi$.

\begin{lem}
  The function $\hatCcMet{X}{\Xi}$ is a metric on $\hatCc{X}{\Xi}$.
\end{lem}

\begin{proof}
  Since the Hausdorff metric $\HausMet{X \times \Xi}$ is a metric on $\Compact{X \times \Xi}$,
  and since the map $\graphmap$ is injective from $\hatCc{X}{\Xi}$ into $\Compact{X \times \Xi}$,
  the function $\hatCcMet{X}{\Xi}$ is indeed a metric.
\end{proof}

Henceforth, we equip $\hatCc{X}{\Xi}$ with the topology induced by $\hatCcMet{X}{\Xi}$.
Topological properties of $\Xi$ are naturally inherited by $\hatCc{X}{\Xi}$, as described below.

\begin{thm} \label{thm: Polishness of uniform variable domains}
  The following statements hold.
  \begin{enumerate}[label = \textup{(\roman*)}, leftmargin = *]
    \item \label{thm item: 1. Polishness of uniform variable domains}
      If $\Xi$ is separable, then so is $\hatCc{X}{\Xi}$.
    \item \label{thm item: 2. Polishness of uniform variable domains}
      If $\Xi$ is Polish, then so is $\hatCc{X}{\Xi}$.
      (NB.\ The metric $\hatCcMet{X}{\Xi}$ itself is not necessarily complete. 
        See Remark~\ref{rem: uniform variable domain is not complete} below.)
  \end{enumerate}
\end{thm}

\begin{proof}
  By definition,
  the graph map $\graphmap \colon \hatCc{X}{\Xi} \to \Compact{X \times \Xi}$ is distance-preserving,
  and in particular a topological embedding.
  If $\Xi$ is separable, then $\Compact{X \times \Xi}$ is separable by Lemma~\ref{lem: Polishness of Hausdorff},
  which implies \ref{thm item: 1. Polishness of uniform variable domains}.
  If $\Xi$ is Polish, then $\Compact{X \times \Xi}$ is Polish by the same lemma.
  To prove \ref{thm item: 2. Polishness of uniform variable domains},
  it thus suffices to show that $\graphmap(\hatCc{X}{\Xi})$ is a $G_\delta$ subset of $\Compact{X \times \Xi}$,
  by Alexandrov's theorem (see \cite[Theorem~2.2.1]{Srivastava_98_A_Course}).
  This can be shown in the same manner as the proof of Theorem~\ref{thm: Polishness of variable domains}\ref{thm item: 2. Polishness of variable domains} below,
  and so we omit the proof here.
\end{proof}

\begin{rem} \label{rem: uniform variable domain is not complete}
  To see that $\hatCcMet{X}{\Xi}$ fails to be complete even when $d_\Xi$ is complete,
  consider $X \coloneqq [0,1]$ and $\Xi \coloneqq \mathbb{R}$, both equipped with the Euclidean metric.
  Define continuous functions $f_n \colon X \to \Xi$, for $n \geq 1$, by
  \begin{equation}
    f_n(x) \coloneqq
    \begin{cases}
      1 - nx, & x \in [0, 1/n], \\
      0,      & x \in (1/n, 1].
    \end{cases}
  \end{equation}
  Then $\graphmap(f_n)$ converges in the Hausdorff topology to the compact set
  \begin{equation}
    E \coloneqq (\{0\} \times [0,1]) \cup ([0,1] \times \{0\}).
  \end{equation}
  However, $E$ is not the graph of any function.
  Hence, $(f_n)_{n \geq 1}$ is a Cauchy sequence in $\hatCc{X}{\Xi}$,
  but it does not converge in $\hatCc{X}{\Xi}$.
\end{rem}

The following result verifies that the convergence in $\hatCc{X}{\Xi}$ is equivalent to the convergence introduced in \cite[Definition~2.1]{Cao_23_Convergence}.

\begin{thm}[Convergence]  \label{thm: convergence in uniform variable domains}
  Let $f, f_1, f_2, \ldots$ be elements of $\hatCc{X}{\Xi}$.
  The following conditions are equivalent.
  \begin{enumerate} [label = \textup{(\roman*)}, leftmargin = *]
    \item \label{thm item: 1. convergence in uniform variable domains} 
      The functions $f_{n}$ converge to $f$ in $\hatCc{X}{\Xi}$.
    \item \label{thm item: 2. convergence in uniform variable domains}
      The sets $\dom(f_{n})$ converge to $\dom(f)$ in the Hausdorff topology as subsets of $X$,
      and it holds that 
      \begin{equation}  \label{thm eq: 2. convergence in uniform variable domains}
        \lim_{\delta \to 0}
        \limsup_{n \to \infty}
        \sup_{\substack{ x_n \in \dom(f_{n}) \\ x \in \dom(f) \\ d_X(x_n,x) < \delta}}
        d_\Xi(f_n(x_n), f(x)) 
        = 0.
      \end{equation}
    \item \label{thm item: 3. convergence in uniform variable domains}
      The sets $\dom(f_n)$ converge to $\dom(f)$ in the Hausdorff topology as subsets of $X$,
      and, for any $x_n \in \dom(f_n)$ and $x \in \dom(f)$ with $x_n \to x$ in $X$,
      it holds that $f_n(x_n) \to f(x)$ in $\Xi$.
    \item \label{thm item: 4. convergence in uniform variable domains}
      The sets $\dom(f_{n})$ converge to $\dom(f)$ in the Hausdorff topology as subsets of $X$,
      and there exist functions $g_{n}, g \in C(X, \Xi)$ such that $g_{n}|_{\dom(f_{n})} = f_{n}$, $g|_{\dom(f)} = f$,
      and $g_{n} \to g$ in the compact-convergence topology.
  \end{enumerate}
\end{thm}

\begin{proof}
  Assume that \ref{thm item: 1. convergence in uniform variable domains} holds.
  It is easy to check that $\dom(f_n)$ converges to $\dom(f)$ in the Hausdorff topology.
  Fix $\varepsilon > 0$.
  By the uniform continuity of $f$ on $\dom(f)$,
  we can find $\delta \in (0, \varepsilon)$ such that, for all $x, y \in \dom(f)$,
  \begin{equation}  \label{2. eq: choice of delta for convergence in C(C(S), Xi)}
    d_X(x, y) < 2\delta 
    \implies 
    d_\Xi(f(x), f(y)) < \varepsilon.
  \end{equation}
  Choose $N \in \mathbb{N}$ such that
  \begin{equation}  \label{2. eq: Hausdorff distance between f_n^r and f^r}
    \hatCcMet{X}{\Xi}( f_{n}, f) = \HausMet{X \times \Xi}(\graphmap(f_n), \graphmap(f)) < \delta,
    \quad
    \forall n > N.
  \end{equation}
  Fix $n > N$ and $x_n \in \dom(f_{n}),\, x \in \dom(f)$ with $d_X(x_n, x) < \delta$.
  By \eqref{2. eq: Hausdorff distance between f_n^r and f^r},
  there exists $y \in \dom(f)$ such that 
  \begin{equation}  \label{2. eq: choice of y for convergence in C(C(S), Xi)}
    d_X(x_n, y) \vee d_\Xi(f_n(x_n), f(y)) < \delta.
  \end{equation}
  Since 
  \begin{equation}
    d_X(x, y) \leq d_X(x, x_n) + d_X(x_n, y) < 2 \delta,
  \end{equation}
  it follows from \eqref{2. eq: choice of delta for convergence in C(C(S), Xi)} that 
  $d_\Xi(f(x), f(y)) < \varepsilon$.
  This, combined with \eqref{2. eq: choice of y for convergence in C(C(S), Xi)}, yields that 
  \begin{equation}
    d_\Xi(f_n(x_n), f(x)) \leq d_\Xi(f_n(x_n), f(y)) + d_\Xi(f(y), f(x)) < 2\varepsilon.
  \end{equation}
  Thus, we obtain \ref{thm item: 2. convergence in uniform variable domains}.

  The implication \ref{thm item: 2. convergence in uniform variable domains} $\Rightarrow$ \ref{thm item: 3. convergence in uniform variable domains} is straightforward.
  The equivalence of \ref{thm item: 3. convergence in uniform variable domains} and \ref{thm item: 4. convergence in uniform variable domains} 
  follows from \cite[Proposition~2.3]{Cao_23_Convergence}
  (see also the proof of Theorem~\ref{thm: convergence in variable domains} below).  
  The remaining implication \ref{thm item: 4. convergence in uniform variable domains} $\Rightarrow$ \ref{thm item: 1. convergence in uniform variable domains}
  can be verified by using Lemma~\ref{lem: convergence in Hausdorff}.
\end{proof}

We provide a precompactness criterion in $\hatCc{X}{\Xi}$,
which is a generalization of the Arzel\`{a}-Ascoli theorem.

\begin{thm}[Precompactness] \label{thm: precompactness uniform in variable domains}
  Fix a non-empty index set $\mathscr{A}$.
  A subset $\{ f_{\alpha} \mid \alpha \in \mathscr{A} \}$ of $\hatCc{X}{\Xi}$ 
  is precompact in $\hatCc{X}{\Xi}$ 
  if and only if the following conditions are satisfied.
  \begin{enumerate} [label = \textup{(\roman*)}, leftmargin = *]
    \item \label{thm item: 1. precompactness uniform in variable domains}
      The set $\{ \dom(f_\alpha) \in \Compact{X} \mid \alpha \in \mathscr{A} \}$ is precompact in the Hausdorff topology.
    \item \label{thm item: 2. precompactness uniform in variable domains} 
      The set $\{ f_{\alpha}(x) \mid x \in \dom(f_{\alpha}),\, \alpha \in \mathscr{A} \}$
      is precompact in $\Xi$.
    \item \label{thm item: 3. precompactness uniform in variable domains} 
      It holds that 
      \begin{equation}
        \lim_{\delta \to 0}
        \sup_{\alpha \in \mathscr{A}}
        \sup_{\substack{ x, y \in \dom(f_{\alpha})\\ d_X(x, y) \leq \delta}}
        d_\Xi(f_{\alpha}(x), f_{\alpha}(y)) 
        = 0.
      \end{equation}
  \end{enumerate}
\end{thm}

\begin{proof}
  Suppose that $\{ f_{\alpha} \mid \alpha \in \mathscr{A}\}$ is precompact.
  Condition~\ref{thm item: 1. precompactness uniform in variable domains} follows from Theorem~\ref{thm: convergence in variable domains}.
  Assume that \ref{thm item: 2. precompactness uniform in variable domains} is not satisfied.
  Then we can find a sequence $(\alpha_{n}, x_n)_{n \geq 1}$ 
  with $\alpha_{n} \in \mathscr{A}$ and $x_n \in \dom(f_{\alpha_{n}})$
  such that $(f_{\alpha_{n}}(x_n))_{n \geq 1}$ contains no convergent subsequence.
  If necessary, by choosing a subsequence,
  we may assume that $f_{\alpha_{n}}$ converges to some function $f$ in $\hatCc{X}{\Xi}$.
  By Theorem~\ref{thm: convergence in variable domains},
  $\dom(f_{\alpha_n}) \to \dom(f)$ in the Hausdorff topology.
  Thus, 
  if necessary, by choosing a further subsequence,
  we may also assume that $x_n$ converges to some $x \in \dom(f)$ in $X$.
  It then follows from Theorem~\ref{thm: convergence in variable domains}\ref{thm item: 3. convergence in variable domains} that 
  $f_{\alpha_{n}}(x_n) \to f(x)$ in $\Xi$,
  which is a contradiction.
  Therefore, we obtain \ref{thm item: 2. precompactness uniform in variable domains}.

  Next, assume that \ref{thm item: 3. precompactness uniform in variable domains} is not satisfied.
  Then we can find $\varepsilon > 0$,
  a decreasing sequence $(\delta_{n})_{n \geq 1}$ with $\delta_{n} \downarrow 0$,
  a sequence $(\alpha_{n})_{n \geq 1}$ in $\mathscr{A}$,
  and 
  $x_n, y_{n} \in \dom(f_{\alpha_{n}})$ with $d_X(x_n, y_n) \leq \delta_n$ such that 
  \begin{equation}
    d_\Xi(f_{\alpha_{n}}(x_n), f_{\alpha_{n}}(y_{n})) > \varepsilon.
  \end{equation}
  If necessary, by choosing a subsequence,
  we may assume that $f_{\alpha_{n}}$ converges to some function $f$ in $\hatCc{X}{\Xi}$.
  Moreover,
  by \ref{thm item: 1. precompactness uniform in variable domains} and Lemma~\ref{lem: precompact in Hausdorff},
  if necessary, by choosing a subsequence,
  we may assume that $x_n \to x$ and $y_n \to y$ for some $x, y \in X$.
  Using the convergence of $\dom(f_{\alpha_n})$ to $\dom(f)$ and $d_X(x_n, y_n) \to 0$, 
  we obtain $x = y \in \dom(f)$.
  It then follows from Theorem~\ref{thm: convergence in uniform variable domains}\ref{thm item: 3. convergence in uniform variable domains} that 
  $f_{\alpha_n}(x_n) \to f(x)$ and $f_{\alpha_n}(y_n) \to f(x)$,
  which contradicts the inequality above.
  Therefore, \ref{thm item: 3. precompactness uniform in variable domains} holds.

  Conversely,
  assume that \ref{thm item: 1. precompactness uniform in variable domains}, \ref{thm item: 2. precompactness uniform in variable domains} 
  and \ref{thm item: 3. precompactness uniform in variable domains} are satisfied.
  By \ref{thm item: 1. precompactness uniform in variable domains} and Lemma~\ref{lem: precompact in Hausdorff},
  there exists a compact subset $X'$ of $X$ such that $\dom(f_\alpha) \subseteq X'$ for all $\alpha \in \mathscr{A}$.
  Let $\Xi'$ be the closure of $\{f_\alpha(x) \mid x \in \dom(f_\alpha),\, \alpha \in \mathscr{A}\}$,
  which is compact by \ref{thm item: 2. precompactness uniform in variable domains}.
  For all $\alpha \in \mathscr{A}$,
  the graph $\graphmap(f_\alpha)$ of $f_\alpha$ is contained in the compact set $X' \times \Xi'$.
  Thus, by Lemma~\ref{lem: precompact in Hausdorff}, 
  $\{\graphmap(f_\alpha) \mid \alpha \in \mathscr{A}\}$ is precompact in the Hausdorff topology.

  Fix a sequence $(\alpha_n)_{n \geq 1}$ in $\mathscr{A}$ arbitrarily.
  It is enough to show that $(f_{\alpha_n})_{n \geq 1}$ has a convergent subsequence.
  If necessary, by choosing a subsequence,
  we may assume that 
  $\graphmap(f_{\alpha_{n}})$ converges to some compact subset $G$ of $X \times \Xi$ in the Hausdorff topology.
  Define 
  \begin{equation}
    K \coloneqq \{x \in X \mid (x, a) \in G\ \text{for some}\ a \in \Xi\}.
  \end{equation}
  If $K$ is the empty set,
  then, by the definition of the Hausdorff metric,
  $\graphmap(f_{\alpha_n})$ is the empty set for all sufficiently large $n$,
  which implies that $f_{\alpha_n}$ converges to the empty map in $\hatCc{X}{\Xi}$.
  It thus remains to consider the case where $K$ is not empty.

  Fix $x \in K$ and suppose $(x,a), (x,b) \in G$.
  Since $\graphmap(f_{\alpha_n}) \to G$,
  we can find $x_n, y_n \in \dom(f_{\alpha_n})$ such that 
  $(x_n, f_{\alpha_n}(x_n)) \to (x,a)$ and $(y_n, f_{\alpha_n}(y_n)) \to (x,b)$.
  By \ref{thm item: 3. precompactness uniform in variable domains}, this implies $a = b$.
  Therefore, we can define a map $f \colon K \to \Xi$ so that $\graphmap(f) = G$.
  It remains to show that $f$ is continuous.
  Fix a sequence $(x^k)_{k \geq 1}$ in $\dom(f)$ converging to some $x^\infty \in \dom(f)$.
  Since $\graphmap(f_{\alpha_n}) \to \graphmap(f)$,
  we can find $x_n^k \in \dom(f_{\alpha_n})$, $k \in \NN \cup \{\infty\}$, such that 
  $(x_n^k, f_{\alpha_n}(x_n^k)) \to (x^k, f(x^k))$ as $n \to \infty$ for each $k \in \NN \cup \{\infty\}$.
  The triangle inequality yields
  \begin{equation}
    d_\Xi(f(x^k), f(x^\infty)) 
    \leq 
    d_\Xi(f(x^k), f_{\alpha_n}(x_n^k)) + d_\Xi(f_{\alpha_n}(x_n^k), f_{\alpha_n}(x_n^\infty)) + d_\Xi(f_{\alpha_n}(x_n^\infty), f(x^\infty)).
  \end{equation}
  Letting $n \to \infty$ and then $k \to \infty$, and using \ref{thm item: 3. precompactness uniform in variable domains},
  we obtain that $f(x^k) \to f(x^\infty)$.
  This shows that $f$ is continuous.
\end{proof}

We are then interested in the metrization of the larger space $\hatC{X}{\Xi}$.
To this end, we define a restriction system from $\hatC{X}{\Xi}$ to $\hatCc{X}{\Xi}$ as follows:
for each $r > 0$ and $x \in X$,
\begin{equation} \label{eq: RS for variable domains}
  R_x^{(r)}(f) = f|_x^{(r)} \coloneqq f|_{D_X(x, r)},
  \quad 
  f \in \hatC{X}{\Xi}.
\end{equation} 

\begin{lem} \label{lem: RS for variable domains}
  The restriction system defined above, from $\hatC{X}{\Xi}$ to $\hatCc{X}{\Xi}$, satisfies Condition~3.
\end{lem}

For each $\rho \in X$ and $f, g \in \hatC{X}{\Xi}$,
we define
\begin{equation} \label{eq: the metric on hatC}
  \hatCMet{(X, \rho)}{\Xi}(f, g) 
  \coloneqq 
  \int_{0}^{\infty}
  e^{-r}
  \bigl(1 \wedge \hatCcMet{X}{\Xi}(f|_\rho^{(r)}, g|_\rho^{(r)})\bigr)\,
  dr.
\end{equation}
The following is an immediate consequence of Corollary~\ref{cor: a summary of D} and Lemma~\ref{lem: RS for variable domains}.

\begin{prop} \label{prop: hatC metric}
  For each $\rho \in X$, the function $\hatCMet{(X, \rho)}{\Xi}$ is a well-defined metric on $\hatC{X}{\Xi}$,
  and the topology it induces is independent of the choice of $\rho$.
\end{prop}

Henceforth, we equip $\hatC{X}{\Xi}$ with the topology induced by $\hatCMet{(X, \rho)}{\Xi}$,
which is independent of $\rho$.
Similarly as before,
we have a metrization of the product space $\hatC{X}{\Xi} \times X$ given as follows:
for each $(f,x), (g,y) \in \hatC{X}{\Xi} \times X$,
we define
\begin{equation} \label{eq: the metric on product hatC}
  \hatCMet{X}{\Xi}\bigl( (f,x), (g, y) \bigr) 
  \coloneqq 
  d_X(x,y) 
  \vee
  \int_{0}^{\infty}
  e^{-r}
  \bigl(1 \wedge \hatCcMet{X}{\Xi}(f|_x^{(r)}, g|_y^{(r)})\bigr)\,
  dr.
\end{equation}
The following is an immediate consequence of Corollary~\ref{cor: a summary of D times X} and Lemma~\ref{lem: RS for variable domains}.

\begin{prop} \label{prop: productt hatC metric}
  The function $\hatCMet{X}{\Xi}$ is a well-defined metric on $\hatC{X}{\Xi} \times X$,
  and the induced topology coincides with the product topology.
\end{prop}

The following is an analogue of Theorem~\ref{thm: Polishness of uniform variable domains}.

\begin{thm} \label{thm: Polishness of variable domains}
  The following statements hold.
  \begin{enumerate} [label = \textup{(\roman*)}, leftmargin = *]
    \item \label{thm item: 1. Polishness of variable domains}
      If $\Xi$ is separable, then so is $\hatC{X}{\Xi}$.
    \item \label{thm item: 2. Polishness of variable domains}
      If $\Xi$ is Polish, then so is $\hatC{X}{\Xi}$.
      (NB.\ The metric $\hatCMet{(X, \rho)}{\Xi}$ itself is not necessarily complete.)
  \end{enumerate}
\end{thm}

\begin{proof}
  Statement~\ref{thm item: 1. Polishness of variable domains} follows from 
  Corollary~\ref{cor: a summary of D}, Theorem~\ref{thm: Polishness of uniform variable domains}, and Lemma~\ref{lem: RS for variable domains}.
  The proof of \ref{thm item: 2. Polishness of variable domains} requires some additional input,  
  and is therefore deferred.  
  It will be provided after the proof of Theorem~\ref{thm: precompactness in variable domains} below.
\end{proof}

We provide a characterization of convergence in $\hatC{X}{\Xi}$.

\begin{thm}  [Convergence]  \label{thm: convergence in variable domains}
  Let $f, f_1, f_2, \ldots$ be elements of $\hatC{X}{\Xi}$.
  The following statements are equivalent.
  \begin{enumerate} [label = \textup{(\roman*)}, leftmargin = *]
    \item \label{thm item: 1. convergence in variable domains} 
      The functions $f_{n}$ converge to $f$ in $\hatC{X}{\Xi}$.
    \item \label{thm item: 2. convergence in variable domains} 
      There exists a sequence $(x_n)_{n \geq 1}$ in $X$ converging to an element $x \in X$
      such that $f_n|_{x_n}^{(r)} \to f|_x^{(r)}$ in $\hatCc{X}{\Xi}$ for all but countably many $r > 0$.
    \item \label{thm item: 3. convergence in variable domains}
      There exist a sequence $(x_n)_{n \geq 1}$ in $X$ converging to an element $x \in X$
      and an increasing sequence $(r_{k})_{k \geq 1}$ with $r_{k} \uparrow \infty$
      such that $f_n|_{x_n}^{(r_{k})} \to f|_x^{(r_{k})}$ in $\hatCc{X}{\Xi}$ for each $k$.
    \item \label{thm item: 4. convergence in variable domains}
      For any elements $x_n \in X$ converging to an element $x \in X$,
      $f_n|_{x_n}^{(r)} \to f|_x^{(r)}$ in $\hatCc{X}{\Xi}$ for all but countably many $r > 0$.
    \item \label{thm item: 5. convergence in variable domains}
      The sets $\dom(f_n) \to \dom(f)$ in the Fell topology as subsets of $X$,
      and, for any $x_n \in \dom(f_n)$ and $x \in \dom(f)$ such that $x_n \to x$ in $X$,
      it holds that $f_n(x_n) \to f(x)$ in $\Xi$.
    \item \label{thm item: 6. convergence in variable domains}
      The sets $\dom(f_{n})$ converge to $\dom(f)$ in the Fell topology as subsets of $X$,
      and there exist functions $g_{n}, g \in C(X, \Xi)$ such that $g_{n}|_{\dom(f_{n})} = f_{n}$, $g|_{\dom(f)} = f$
      and $g_{n} \to g$ in the compact-convergence topology.
  \end{enumerate}
\end{thm}

\begin{proof}
  The equivalence of \ref{thm item: 1. convergence in variable domains},
  \ref{thm item: 2. convergence in variable domains}, \ref{thm item: 3. convergence in variable domains}, and \ref{thm item: 4. convergence in variable domains} 
  follows from Theorem~\ref{thm: convergence in D}.
  The implication \ref{thm item: 6. convergence in variable domains} $\Rightarrow$ \ref{thm item: 5. convergence in variable domains} is straightforward.
  The implication \ref{thm item: 5. convergence in variable domains} $\Rightarrow$ \ref{thm item: 2. convergence in variable domains} follows from 
  Theorems~\ref{thm: convergence in the Fell topology} and \ref{thm: convergence in variable domains}\ref{thm item: 3. convergence in variable domains}.
  Thus, it remains to show \ref{thm item: 4. convergence in variable domains} $\Rightarrow$ \ref{thm item: 6. convergence in variable domains}.

  Assume that \ref{thm item: 4. convergence in variable domains} holds.
  By Theorems~\ref{thm: convergence in the Fell topology} and \ref{thm: convergence in variable domains},
  $\dom(f_n) \to \dom(f)$ in the Fell topology.
  Define
  \begin{equation}
    D \coloneqq \{(0, x) \mid x \in \dom(f)\} \cup \bigcup_{n \geq 1} \{(1/n, x) \mid x \in \dom(f_n)\}.
  \end{equation}
  By the convergence of $\dom(f_n)$ to $\dom(f)$,
  we deduce that $D$ is a closed subset of $[0,1] \times X$.
  Define a function $F \colon D \to \Xi$ by setting 
  $F(1/n, \cdot) \coloneqq f_n(\cdot)$, $n \geq 1$, and $F(0, \cdot) = f(\cdot)$.
  By Theorem~\ref{thm: convergence in variable domains}\ref{thm item: 3. convergence in variable domains},
  we deduce that $F$ is continuous.
  Applying the Tietze extension theorem,
  we obtain a continuous function $\tilde{F} \colon [0,1] \times X \to \Xi$ which coincides with $F$ on $D$.
  By defining $g_n(\cdot) \coloneqq \tilde{F}(1/n, \cdot)$, $n \geq 1$, and $g(\cdot) \coloneqq \tilde{F}(0, \cdot)$,
  we verify that \ref{thm item: 6. convergence in variable domains} holds.
\end{proof}

The following corollary is an immediate consequence of the above theorem.
It confirms that the topology on $\hatC{X}{\Xi}$ is a natural extension of the compact-convergence topology.

\begin{cor} \label{cor: embedding of C(X, Xi) into hatC(S, Xi)}
  For any $S \in \Closed{X}$,
  the inclusion map $C(S, \Xi) \ni f \mapsto f \in \hatC{X}{\Xi}$ is a topological embedding,
  where we recall that $C(S, \Xi)$ is equipped with the compact-convergence topology.
\end{cor}

\begin{rem} \label{rem: hatC applies to lattice model}
  In the study of scaling limits of certain functions arising in lattice models,
  one sometimes faces a technical issue concerning the domains of the functions.
  The use of the space $\hatC{X}{\Xi}$ may resolve this issue.
  To see this in a concrete example,
  consider a discrete-time random walk on $\ZN$.
  For each $t \in \ZNp$ and $x,y \in \ZN$,
  write $p(t, x, y)$ for the probability that the random walk starting from $x$ is at $y$ at time $t$.
  For each $n \in \NN$, we define a function $p_n$ on $n^{-1}\ZNp \times n^{-1/2} \ZN \times n^{-1/2} \ZN$ by
  \begin{equation}
    p_n(t,x,y) \coloneqq \sqrt{n}\, p(nt, \sqrt{n}x, \sqrt{n}y).
  \end{equation}
  Under a suitable assumption on the jump distribution of the random walk,
  the local central limit theorem (LCLT) asserts that 
  the functions $p_n$ converge to the transition density $\tilde{p}$ of a one-dimensional Brownian motion
  (cf.\  \cite[Section~2]{Lawler_Limic_10_Random}).
  However, since the domains of $p_n$ differ,
  to formulate this assertion in a rigorous mathematical context,
  one usually extends the domain of $p_n$ to $\RNp \times \RN \times \RN$ 
  as follows: for each $(t, x, y) \in \RNp \times \RN \times \RN$,
  choose a point $(t', x', y')$ in the domain of $p_n$ closest to $(t,x,y)$ (e.g., with respect to the Euclidean distance),
  and set $p_n^*(t,x,y) \coloneqq p_n(t', x', y')$.
  Then the LCLT states that 
  \begin{equation} \label{eq: conv of LCLT}
    p_n^*(t,x,y) \;\longrightarrow\; \tilde{p}(t,x,y)
  \end{equation}
  uniformly on every compact subset of $\RNp \times \RN \times \RN$.
  Using the framework developed in this subsection,
  one can avoid the technical issue regarding the domains.
  Indeed, one can simply write that $p_n \to \tilde{p}$ in $\hatC{\RNp \times \RN \times \RN}{\RNp}$,
  which recovers the convergence~\eqref{eq: conv of LCLT} by Theorem~\ref{thm: convergence in variable domains}. 
\end{rem}

A precompactness criterion is deduced from Theorem~\ref{thm: precompactness uniform in variable domains}, as follows.

\begin{thm}[Precompactness] \label{thm: precompactness in variable domains}
  Fix a non-empty index set $\mathscr{A}$ and $\rho \in X$.
  A non-empty subset $\{ f_{\alpha} \mid \alpha \in \mathscr{A} \}$ 
  is precompact in $\hatC{X}{\Xi}$ 
  if and only if the following conditions are satisfied.
  \begin{enumerate} [label = \textup{(\roman*)}, leftmargin = *]
    \item \label{thm item: 1. precompactness in variable domains}
      For each $r > 0$, 
      the set $\{ f_{\alpha}(x) \mid x \in \dom(f_{\alpha}|_\rho^{(r)}),\, \alpha \in \mathscr{A} \}$
      is precompact in $\Xi$.
    \item \label{thm item: 2. precompactness in variable domains} 
      For each $r > 0$, it holds that 
      \begin{equation}
        \lim_{\delta \to 0}
        \sup_{\alpha \in \mathscr{A}}
        \sup_{\substack{ x, y \in \dom(f_{\alpha}|_\rho^{(r)})\\ d_X(x, y) \leq \delta}}
        d_\Xi(f_{\alpha}(x), f_{\alpha}(y)) 
        = 0.
      \end{equation}
  \end{enumerate}
\end{thm}

\begin{proof}
  By Theorem~\ref{thm: precompactness in d_D} and Lemma~\ref{lem: RS for variable domains},
  $\{f_\alpha \mid \alpha \in \mathscr{A}\}$ is precompact in $\hatC{X}{\Xi}$ 
  if and only if $\{f_\alpha|_\rho^{(r)} \mid \alpha \in \mathscr{A}\}$ is precompact in $\hatCc{X}{\Xi}$ for each $r > 0$.
  By Theorem~\ref{thm: precompactness uniform in variable domains},
  the latter condition is equivalent to 
  \ref{thm item: 1. precompactness in variable domains}, \ref{thm item: 2. precompactness in variable domains},
  together with the following:
  \begin{itemize}
    \item $\{ \dom(f_\alpha|_\rho^{(r)}) \mid \alpha \in \mathscr{A} \}$ is precompact in the Hausdorff topology.
  \end{itemize}
  However, since all the subsets $\dom(f_\alpha|_\rho^{(r)})$ are contained in the compact set $D_X(\rho, r)$,
  the above condition is always satisfied by Lemma~\ref{lem: precompact in Hausdorff}.
  This completes the proof.
\end{proof}

We now begin the proof of Theorem~\ref{thm: Polishness of variable domains}\ref{thm item: 2. Polishness of variable domains},  
i.e., we show that $\hatC{X}{\Xi}$ is Polish when $\Xi$ is Polish.  
To this end, we define a space $\GraphSp{X}{\Xi}$  
into which $\hatC{X}{\Xi}$ can be topologically embedded.

\begin{dfn} \label{dfn: marked Hausdorff space}
  We define $\GraphSp{X}{\Xi}$ to be the set of closed subsets $E \subseteq X \times \Xi$
  such that $E \cap (K \times \Xi)$ is compact for any compact subset $K \subseteq X$. 
\end{dfn}

To metrize $\GraphSp{X}{\Xi}$,
we define a restriction system from $\GraphSp{X}{\Xi}$ to $\Compact{X \times \Xi}$ as follows:
for each $r > 0$ and $x \in X$,
\begin{equation}
  R_x^{(r)}(E)
  =
  E|_x^{(r, *)} 
  \coloneqq 
  E \cap (D_X(x, r) \times \Xi),
  \quad E \in \GraphSp{X}{\Xi}.
\end{equation} 

\begin{lem} \label{lem: RS for marked local Hausdorff}
  The above-defined restriction system from $\GraphSp{X}{\Xi}$ to $\Compact{X \times \Xi}$
  is complete and satisfies Condition~2.
  Moreover, if $d_\Xi$ is complete,
  then it satisfies Condition~4.
\end{lem}

\begin{proof}
  The proof is given in Appendix~\ref{appendix: variable domains}.
\end{proof}

For each $\rho \in X$ and $D, E \in \GraphSp{X}{\Xi}$,
we define 
\begin{equation}  \label{eq: the metric on graph sp}
  \GraphMet{(X, \rho)}{\Xi}(D, E)
  \coloneqq
  \int_{0}^{\infty}
  e^{-r}
  (1 \wedge \HausMet{X \times \Xi}(D|_\rho^{(r, *)}, E|_\rho^{(r, *)}))\,dr.
\end{equation}

\begin{prop} \label{prop: graph metric}
  The function $\GraphMet{(X, \rho)}{\Xi}$ is a well-defined metric on $\GraphSp{X}{\Xi}$ for each $\rho \in X$,
  and the induced topology is independent of $\rho$.
  Moreover, if $\Xi$ is Polish, then the topology on $\GraphSp{X}{\Xi}$ induced by $\GraphMet{X}{\Xi}$ is Polish.
\end{prop}

\begin{proof}
  The first assertion follows from Corollary~\ref{cor: a summary of D} and Lemma~\ref{lem: RS for marked local Hausdorff}.  
  Assume that $\Xi$ is Polish.  
  As a consequence of  Corollary~\ref{cor: a summary of D} and Lemma~\ref{lem: RS for marked local Hausdorff},  
  we deduce that, by replacing $d_\Xi$ with a complete metric inducing the same topology,  
  $\GraphMet{(X, \rho)}{\Xi}$ becomes a complete and separable metric.  
  Combining this with Corollary~\ref{cor: topology on D}\ref{cor item: 3. topology on D},  
  we obtain the second assertion.
\end{proof}

Henceforth, we equip $\GraphSp{X}{\Xi}$ with the topology induced by $\GraphMet{(X,\rho)}{\Xi}$,
which is independent of $\rho$.
Below, we verify that 
$\hatC{X}{\Xi}$ can be regarded as a subspace of $\GraphSp{X}{\Xi}$
through the graph map $\graphmap$ (recall it from \eqref{eq: def of graphmap}).

\begin{lem} \label{lem: embedding of hatC to cC(S, Xi)}
  The graph map $\graphmap \colon \hatC{X}{\Xi} \to \GraphSp{X}{\Xi}$
  is distance-preserving
  with respect to $\hatCMet{(X, \rho)}{\Xi}$ and $\GraphMet{(X, \rho)}{\Xi}$ for any $\rho \in X$.
  In particular, it is a topological embedding.
\end{lem}

\begin{proof}
  By definition,
  we have $\graphmap(f|_\rho^{(r)}) = \graphmap(f)|_\rho^{(r,*)}$ for any $f \in \hatC{X}{\Xi}$ and $r > 0$.
  Thus, for any $f, g \in \hatC{X}{\Xi}$, we deduce that 
  \begin{align}
    \hatCMet{(X,\rho)}{\Xi}(f,g)
    &= 
    \int_0^\infty e^{-r} (1 \wedge \hatCcMet{X}{\Xi}(f|_\rho^{(r)}, g|_\rho^{(r)}))\, dr \\
    &=
    \int_0^\infty e^{-r} \bigl( 1 \wedge \HausMet{X \times \Xi}( \graphmap(f|_\rho^{(r)}), \graphmap(g|_\rho^{(r)}) ) \bigr)\,dr\\
    &=
    \int_0^\infty e^{-r} \bigl( 1 \wedge \HausMet{X \times \Xi}( \graphmap(f)|_\rho^{(r,*)}, \graphmap(g)|_\rho^{(r,*)} )  \bigr)\,dr\\
    &=
    \GraphMet{(X,\rho)}{\Xi} (\graphmap(f), \graphmap(g)),
  \end{align}
  which establishes the desired result.
\end{proof}

Thus, to prove the Polishness of $\hatC{X}{\Xi}$,
it is enough to show that it is a $G_\delta$ set as a subset of $\GraphSp{X}{\Xi}$.
To this end, we introduce a sequence of subsets of $\GraphSp{X}{\Xi}$.

\begin{dfn}  \label{dfn: graph subspace}
  For each $k \geq 1$ and $\rho \in X$,
  we define $\subGraphSp{\rho}{k}{X}{\Xi}$ to be the collection of $D \in \GraphSp{X}{\Xi}$ 
  such that there exist $r = r(D) > k$ and $\delta_{i} = \delta_{i}(D) \in (0, 1/k)$, $i=1,2$,
  satisfying the following.
  \begin{enumerate} [label = (C)]
    \item  \label{dfn item: graph subspace}
      For any $(x, a), (y, b) \in D|_\rho^{(r, *)}$,
      if $d_X(x, y) < \delta_{1}$,
      then $d_\Xi(a,b) < \delta_{2}$.
  \end{enumerate}
\end{dfn}

\begin{lem} \label{lem: 1. Polishness of variable domains}
  For each $k \geq 1$,
  $\subGraphSp{\rho}{k}{X}{\Xi}$ is an open subset of $\GraphSp{X}{\Xi}$.
\end{lem}

\begin{proof}
  Fix $D \in \subGraphSp{\rho}{k}{X}{\Xi}$.
  Let $r>k$ and $\delta_{i} \in (0, 1/k)$, $i=1,2$,
  be constants satisfying \ref{dfn item: graph subspace}
  for $D$.
  Choose $\varepsilon \in (0,1)$ so that
  \begin{equation}
    r- \varepsilon > k,
    \quad 
    r< \varepsilon^{-1},
    \quad 
    2\varepsilon < \delta_{1},
    \quad
    2\varepsilon + \delta_{2}< 1/k.
  \end{equation}
  Fix $E \in \GraphSp{X}{\Xi}$ 
  such that $\GraphMet{(X, \rho)}{\Xi} (D, E) < \varepsilon e^{-1/\varepsilon}$.
  It is enough to show that $E \in \subGraphSp{\rho}{k}{X}{\Xi}$.
  By the definition of $\GraphMet{(X, \rho)}{\Xi}$,
  we can find $\tilde{r} > 1/\varepsilon$ such that 
  \begin{equation}  \label{pr eq: 1. 1. Polishness of variable domains}
    \HausMet{X \times \Xi} ( D|_\rho^{(\tilde{r}, *)}, E|_\rho^{(\tilde{r}, *)} ) < \varepsilon.
  \end{equation}
  Define $r' > k$ and $0 < \delta_{i}' < 1/k$ for $i = 1,2$
  by setting 
  \begin{equation}  \label{pr eq: 2. 1. Polishness of variable domains}
    r' \coloneqq r - \varepsilon,
    \quad
    \delta_{1}' \coloneqq \delta_{1} - 2 \varepsilon,
    \quad
    \delta_{2}' \coloneqq \delta_{2} + 2 \varepsilon,
  \end{equation}
  We will prove that $r', \delta_{1}'$ and $\delta_{2}'$ satisfy \ref{dfn item: graph subspace}
  for $E$.
  Fix $(x, a), (y, b) \in E|_\rho^{(r', *)}$ satisfying $d_X(x, y) < \delta_{1}'$.
  By \eqref{pr eq: 1. 1. Polishness of variable domains},
  there exists $(x', a'), (y', b') \in D|_x^{(\tilde{r}, *)}$ such that
  \begin{equation}  \label{pr eq: 3. 1. Polishness of variable domains}
    d_X(x,x') \vee d_\Xi(a, a') < \varepsilon, 
    \quad 
    d_X(y,y') \vee d_\Xi(b, b') < \varepsilon.
  \end{equation}
  It then follows from \eqref{pr eq: 1. 1. Polishness of variable domains}, \eqref{pr eq: 2. 1. Polishness of variable domains}, 
  and \eqref{pr eq: 3. 1. Polishness of variable domains}
  that 
  $(x', a'), (y', b') \in D|_\rho^{(r, *)}$ and $d_X(x', y') < \delta_{1}$.
  Using \ref{dfn item: graph subspace},
  we obtain $d_\Xi(a', b') < \delta_{2}$.
  This, combined with \eqref{pr eq: 2. 1. Polishness of variable domains}
  and \eqref{pr eq: 3. 1. Polishness of variable domains}, 
  yields
  $d_\Xi(a, b) < \delta_{2}'$,
  which implies $E \in \subGraphSp{\rho}{k}{X}{\Xi}$.
  This completes the proof.
\end{proof}

\begin{lem} \label{lem: 2. Polishness of variable domains}
  For each $\rho \in X$, it holds that $\graphmap(\hatC{X}{\Xi}) = \bigcap_{k \geq 1} \subGraphSp{\rho}{k}{X}{\Xi}$.
\end{lem}

\begin{proof}
  It is easy to check that $\graphmap(\hatC{X}{\Xi}) \subseteq \bigcap_{k \geq 1} \subGraphSp{\rho}{k}{X}{\Xi}$
  by using the uniform continuity of $f \in \hatC{X}{\Xi}$ on compact subsets.
  Fix $D \in \bigcap_{k \geq 1} \subGraphSp{\rho}{k}{X}{\Xi}$.
  It suffices to construct a function $f \in \hatC{X}{\Xi}$ whose graph coincides with $D$.
  Define $S$ to be the subset of $X$
  consisting of $x$ such that $(x, a) \in D$ for some $a \in \Xi$.
  By the definition of $\subGraphSp{\rho}{k}{X}{\Xi}$,
  one can check that $S$ is a closed subset of $X$.
  Condition \ref{dfn item: graph subspace}
  implies that,
  for each $x \in X$,
  an element $a_x \in \Xi$ satisfying $(x, a_x) \in D$ 
  is uniquely determined.
  We then define $f \colon  X \to \Xi$ by setting $f(x) \coloneqq a_x$. 
  Using \ref{dfn item: graph subspace} again,
  we deduce that $f$ is continuous,
  which competes the proof.
\end{proof}

Finally, we complete the proof of Theorem~\ref{thm: Polishness of variable domains}.

\begin{proof} [{Proof of Theorem~\ref{thm: Polishness of variable domains}\ref{thm item: 2. Polishness of variable domains}}]
  Assume that $\Xi$ is Polish.
  By Lemmas~\ref{lem: 1. Polishness of variable domains} and \ref{lem: 2. Polishness of variable domains},
  $\graphmap(\hatC{X}{\Xi})$ is a $G_\delta$ subset of $\GraphSp{X}{\Xi}$.
  Since $\GraphSp{X}{\Xi}$ is Polish by Proposition~\ref{prop: graph metric},
  by Alexandrov's theorem (see \cite[Theorem~2.2.1]{Srivastava_98_A_Course}),
  we deduce that $\graphmap(\hatC{X}{\Xi})$ is Polish.
  This, combined with Lemma~\ref{lem: embedding of hatC to cC(S, Xi)},
  implies the Polishness of $\hatC{X}{\Xi}$.
\end{proof}


\section{The local Gromov--Hausdorff topology} \label{sec: the local GH topology}

In this section,  
we define the \emph{local Gromov--Hausdorff topology}  
on the collection of (equivalence classes of) rooted $\bcmAB$ spaces.  
This topology naturally extends the pointed Gromov--Hausdorff topology (see Remark~\ref{rem: pointed GH topology})
to non-compact metric spaces.
Extensions of the pointed Gromov--Hausdorff topology to non-compact metric spaces  
have already been considered in the literature (cf.\ \cite{Burago_Burago_Ivanov_01_A_course,Gromov_07_Metric,Khezeli_20_Metrization}),  
and we verify that our extension coincides with these.
Most of the results in this section are proved in a more general setting in Section~\ref{sec: main results},  
and therefore the details of the proofs are omitted here.

\begin{rem} \label{rem: pointed GH topology}
  Recall that the pointed Gromov--Hausdorff topology  
  is defined on the collection of (equivalence classes of) rooted compact metric spaces,  
  and is analogous to the pointed Gromov--Hausdorff--Prohorov topology introduced in Section~\ref{sec: introduction to GH-type metrics}.  
  More precisely, by removing the measure terms from the definition of $d_{\mathrm{pGHP}}$ in \eqref{eq: the GHP metric},  
  one obtains a metric that induces the pointed Gromov--Hausdorff topology.  
  (See also \cite[Definition~1.2]{Jansen_17_Notes}.)
\end{rem}

Since the collection of metric spaces is too large to form a set in the strict sense,  
we introduce an identification between $\bcmAB$ spaces.
We say that $\bcmAB$ spaces $X$ and $Y$ are \textit{isometric}  
if and only if there exists an isometry $f \colon  X \to Y$.
Ideally, we would like to define $\BCM$ as the set of isometric equivalence classes of $\bcmAB$ spaces.
However,  
since the collection of all $\bcmAB$ spaces does not form a set,  
the operation of taking equivalence classes is not valid in a strict set-theoretic sense.
Nevertheless,  
by selecting a representative from each equivalence class in an appropriate manner,  
we can construct a well-defined set, as described below.

\begin{prop}  \label{prop: existence of BCM}
  There exists a set $\BCM$ satisfying the following.
  \begin{enumerate} [label = \textup{(\roman*)}, leftmargin = *]
    \item 
      The set $\BCM$ consists of $\bcmAB$ spaces.
    \item 
      For any $\bcmAB$ space $\cY$,
      there exists a unique element $\cX \in \BCM$ isometric to $\cY$.
  \end{enumerate}
\end{prop}

\begin{proof}
  Let $2^{\mathbb{R}}$ be the set of non-empty subsets of $\mathbb{R}$.
  For every $M \in 2^{\mathbb{R}}$,
  we denote by $D(M)$ the set of functions $d_M : M \times M \to \RNp$
  such that $d_M$ is a metric on $M$
  and $M$ is boundedly-compact.
  We then define a set $\mathscr{M}$ by setting 
  \begin{equation}
    \mathscr{M}
    \coloneqq
    \left\{
      M \mid
      M \in 2^{\mathbb{R}},\ d_M \in D(M)
    \right\}.
  \end{equation}
  Let $\cX = X$ be an arbitrary $\bcmAB$ space.
  Since the cardinality of $X$ is smaller than or equal to the cardinality of $\RN$,
  there exists an injective map $f \colon X \to \RN$.
  Let $Y \coloneqq f(X) \in 2^\RN$,
  and define a function $d_Y \colon Y \times Y \to \RNp$ by setting 
  $d_Y(f(x_1), f(x_2)) \coloneqq d_X(x_1, x_2)$.
  It is then the case that $d_Y \in D(M)$ and $f$ is an isometry from $X$ to $Y$.
  Since $\mathscr{M}$ is a legitimate set,
  we can safely choose a representative from each rooted-isometric equivalence class of elements in $\mathscr{M}$,
  and we obtain the desired set $\BCM$.
\end{proof}

To deal with non-compact spaces, it is convenient to consider rooted spaces.  
As before, we introduce an identification between rooted metric spaces.  
Let $\cX = (X, \rho_X)$ and $\cY = (Y, \rho_Y)$ be rooted $\bcmAB$ spaces.  
We say that $\cX$ and $\cY$ are \emph{rooted-isometric}
if and only if there exists a root-preserving isometry $f \colon X \to Y$.  
Here, \emph{root-preserving} means that $f(\rho_X) = \rho_Y$. 
By Proposition~\ref{prop: existence of BCM},  
we can safely define the set of rooted-isometric equivalence classes of rooted $\bcmAB$ spaces as follows.

\begin{prop}  \label{prop: existence of rootedBCM}
  There exists a set $\rootedBCM$ satisfying the following.
  \begin{enumerate} [label = \textup{(\roman*)}, leftmargin = *]
    \item 
      The set $\rootedBCM$ consists of rooted $\bcmAB$ spaces.
    \item 
      For any rooted boundedly-compact metric space $\cY$,
      there exists a unique element $\cX \in \rootedBCM$ such that $\cX$ is rooted-isometric to $\cY$.
  \end{enumerate}
\end{prop}

\begin{proof}
  Let $\BCM$ be the set appearing in Proposition~\ref{prop: existence of BCM}.
  Define $\mathscr{M}_\bullet$ to be the set of all triples $(X, \rho_X)$ such that $X \in \BCM$ and $\rho_X \in X$.
  Since $\mathscr{M}_\bullet$ is a legitimate set, we can define the desired set $\rootedBCM$ 
  by selecting a representative from each rooted-isometric equivalence class.
\end{proof}

Henceforth, we fix the sets $\BCM$ and $\rootedBCM$ as given in Propositions~\ref{prop: existence of BCM} and~\ref{prop: existence of rootedBCM}.  
In the following discussions,  
given a (rooted) $\bcmAB$ space $X$,  
if necessary,  
we identify it with the unique element of $\BCM$ (resp.\ $\rootedBCM$)  
that is (rooted-)isometric to $X$.

Below, we introduce two metrics on $\rootedBCM$.
Recall the metrics $\lHausMet{X,\rho}$ on $\Closed{X}$ and $\lHausMet{X}$ on $\Closed{X} \times X$ from Section~\ref{sec: the Fell topology}.

\begin{dfn} \label{dfn: RF GH metric}
  For $\cX=(X, \rho_X), \cY=(Y, \rho_Y) \in \rootedBCM$,
  we define 
  \begin{equation}  \label{dfn eq: RF GH metric}
    \RFMet(\cX, \cY)
    \coloneqq
    \inf_{f, g, Z}
    \lHausMet{Z,\rho_Z} (f(X), g(Y)),
  \end{equation}
  where the infimum is taken 
  over all $(Z, \rho_{Z}) \in \rootedBCM$ 
  and root-preserving isometric embeddings $f \colon  X \to Z$ and $g \colon Y \to Z$.
\end{dfn}

\begin{dfn}  \label{dfn: RV GH metric}
  For $\cX=(X, \rho_X), \cY=(Y, \rho_Y) \in \rootedBCM$,
  we define 
  \begin{equation}  \label{dfn eq: RV GH metric}
    \RVMet(\cX, \cY)
    \coloneqq
    \inf_{f, g, Z}
    \lHausMet{Z} \bigl( (f(X), f(\rho_X)), (g(Y), g(\rho_Y)) \bigr),
  \end{equation}
  where the infimum is taken 
  over all $(Z, d_Z) \in \BCM$
  and isometric embeddings $f \colon  X \to Z$ and $g \colon Y \to Z$.
\end{dfn}

\begin{rem}
  One needs to check that the infimums in \eqref{dfn eq: RF GH metric} and \eqref{dfn eq: RV GH metric} are well-defined,
  that is,
  it is taken over a non-empty set.
  We give a sketch of how to check it.
  Fix $(X, \rho_X), (Y, \rho_Y) \in \rootedBCM$.
  Define $Z'$ to be the disjoint union $X \sqcup Y$
  and define a pseudometric on $Z'$ by setting 
  $d_{Z'}|_{X \times X} \coloneqq d_X,\, d_{Z'}|_{Y \times Y} \coloneqq d_Y$
  and $d_{Z'}(x, y) \coloneqq d_X(x, \rho_X) + d_Y(\rho_Y, y)$
  for $x \in X,\, y \in Y$.
  Then we have $d_{Z'}(\rho_X, \rho_Y) = 0$.
  Therefore,
  by setting $Z$ to be the quotient space and $\rho_{Z}$ to be the equivalence class $\{\rho_X, \rho_Y\}$,
  we obtain a rooted $\bcmAB$ space $(Z, \rho_{Z})$,
  where $(X, \rho_X)$ and $(Y, \rho_Y)$ 
  are isometrically embedded. 
\end{rem}

\begin{thm} \label{thm: metrics for local GH topology}
  Both functions \(\RFMet\) and \(\RVMet\) are metrics on \(\rootedBCM\),
  and the induced topologies coincide.
\end{thm}

\begin{proof}
  That \(\RFMet\) and \(\RVMet\) are metrics on \(\rootedBCM\) are proven similarly to Theorems~\ref{thm: space-rooted metric} and \ref{thm: element-rooted metric} below, respectively.
  The last assertion is proven similarly to Theorem~\ref{thm: coincidence of RF and RV}.
\end{proof}

\begin{dfn} [{The local Gromov--Hausdorff topology}]
  We call the topology on \(\rootedBCM\) induced by \(\RFMet\) (or, equivalently, \(\RVMet\)) the \textit{local Gromov--Hausdorff topology}.
\end{dfn}

The local Gromov--Hausdorff topology is characterized by convergence of spaces embedded into a common metric space.
More precisely, we have the following result.
In particular, this ensures that the local Gromov--Hausdorff topology  
coincides with the generalization of the pointed Gromov--Hausdorff topology 
considered in the literature 
such as \cite{Burago_Burago_Ivanov_01_A_course,Gromov_07_Metric,Khezeli_20_Metrization}.
Below, for each \(\cX = (X, \rho_X) \in \rootedBCM\) and \(r > 0\), we define \(\cX^{(r)} = (X^{(r)}, \rho_{X^{(r)}})\) by setting
\(X^{(r)} \coloneqq D_X(\rho_X, r)\) equipped with the metric \(d_{X^{(r)}} \coloneqq d_X|_{X^{(r)} \times X^{(r)}}\) 
and \(\rho_{X^{(r)}} \coloneqq \rho_X\).

\begin{thm} [{Convergence}] \label{thm: local GH convergence}
  Let \(\cX_n=(X_n, \rho_{X_n})\), \(n \in \NN \cup \{\infty\}\), be elements of \(\rootedBCM\).
  The following statements are equivalent with each other.
\begin{enumerate} [label = \textup{(\roman*)}, leftmargin = *]
  \item \label{thm item: 1. local GH convergence}
    The sequence \((\cX_n)_{n \in \NN}\) converges to \(\cX_{\infty}\) in the local Gromov--Hausdorff topology.
  \item \label{thm item: 2. local GH convergence}
    There exist a rooted \(\bcmAB\) space \((M, \rho_M)\) 
    and root-preserving isometric embeddings \(f_n \colon X_n \allowbreak \to \allowbreak M\), \(n \in \NN \cup \{\infty\}\),
    such that \(f_n(X_n) \to f(X)\) in the Fell topology as subsets of \(M\).
  \item \label{thm item: 3. local GH convergence}
    There exist a \(\bcmAB\) space \(M\) 
    and isometric embeddings \(f_n \colon X_n \to M\), \(n \in \NN \cup \{\infty\}\),
    such that \(f_n(X_n) \to f(X)\) in the Fell topology as subsets of \(M\)
    and \(f_n(\rho_{X_n}) \to f_\infty(\rho_{X_\infty})\) in \(M\).
  \item \label{thm item: 4. local GH convergence}
      For all but countably many \(r > 0\), 
      \((\cX_n^{(r)})_{n \in \NN}\) converges to \(\cX_{\infty}^{(r)}\) in the pointed Gromov--Hausdorff topology.
\end{enumerate}
\end{thm}

\begin{proof}
  The equivalences \ref{thm item: 1. local GH convergence} \(\Leftrightarrow\) \ref{thm item: 2. local GH convergence}  
  and \ref{thm item: 1. local GH convergence} \(\Leftrightarrow\) \ref{thm item: 3. local GH convergence}  
  are verified similarly to Theorems~\ref{thm: RF convergence} and~\ref{thm: RV convergence}, respectively.  
  The equivalence between \ref{thm item: 3. local GH convergence} and \ref{thm item: 4. local GH convergence}  
  can be proven similarly to \cite[Proposition~5.9]{Athreya_Lohr_Winter_16_The_gap}.
\end{proof}

\begin{rem}
  From the above result,  
  it might be more appropriate to call the local Gromov--Hausdorff topology the \emph{Gromov--Fell topology}.  
  However, since the term ``Gromov--Hausdorff'' has been widely adopted in the literature,  
  including for non-compact underlying spaces,  
  and is more familiar to a broader audience,  
  we follow this convention and use ``Gromov--Hausdorff'' throughout the paper.
\end{rem}

\begin{thm}[Polishness] \label{thm: local GH is Polish}
  The local Gromov--Hausdorff topology is separable,
  and the both metrics \(\RFMet\) and \(\RVMet\) are complete.
\end{thm}

\begin{proof}
  The completeness of the metrics are proven similarly to Theorems~\ref{thm: completeness of RF} and \ref{thm: completeness of RF} below, respectively,
  and therefore we omit the proof.
  For the separability,
  we note that, by Theorem~\ref{thm: local GH convergence} and \cite[Theorem~3.24(v)]{Khezeli_20_Metrization},
  the local Gromov--Hausdorff topology on \(\rootedBCM\) coincides with 
  the topology considered in \cite[Section~4.1]{Khezeli_20_Metrization}
  (where \(\rootedBCM\) is denoted by \(\mathfrak{N}_*\) in that notation).
  Thus, by \cite[Theorem~4.1]{Khezeli_20_Metrization},
  the local Gromov--Hausdorff topology on \(\rootedBCM\) is separable.
\end{proof}

To describe a precompactness criterion in the local Gromov--Hausdorff topology,
we introduce the notion of $\varepsilon$-covering and metric entropy.
\begin{dfn} [{$\varepsilon$-covering, metric entropy}]  \label{dfn: metric entropy}
  Let $S$ be a metric space and $\varepsilon$ be a positive number.
  A subset $A \subseteq S$ is called an \textit{$\varepsilon$-covering} of $S$ 
  if it holds that $S = \bigcup_{x \in A} D_{S}(x, \varepsilon)$.
  We define 
  \begin{equation}
    N(S, \varepsilon) 
    \coloneqq 
    \min\{
      |A| \mid A\ \text{is an}\ \varepsilon \text{-covering of}\ S
    \},
  \end{equation}
  where $| \cdot |$ denotes the cardinality of a set. 
  We call the family $\{N(S, \varepsilon) \mid \varepsilon > 0\}$ the \textit{metric entropy} of $S$.
\end{dfn}

\begin{rem}
  In Definition \ref{dfn: metric entropy},
  we borrow the definition of metric entropy given in \cite{Marcus_Rosen_06_Markov},
  but one should note that the metric entropy is defined to be the logarithm of $N(S, \varepsilon)$ elsewhere in the literature.
\end{rem}

\begin{lem}
  Let $K$ be a compact metric space.
  Then $N(K,\cdot)$ is right-continuous with left-hand limits.
  In particular, it has at most countable discontinuity points.
\end{lem}

\begin{thm} [{Convergence of metric entropies}] \label{thm: metric entropy convergence}
  If a sequence of compact metric spaces $K_n$ converges to a compact metric space $K$ 
  in the Gromov--Hausdorff topology,
  then
  \begin{equation}  \label{thm eq: 1. metric entropy convergence}
    N(K,\varepsilon)
    \leq
    \liminf_{n \to \infty} N(K_n,\varepsilon),
  \end{equation}
  for all $\varepsilon >0$, and,
  for all continuity points $\varepsilon >0$ of $N(K, \cdot)$,
  \begin{equation}  \label{thm eq: 2. metric entropy convergence}
    \lim_{n \to \infty}N(K_n,\varepsilon) = N(K,\varepsilon)
  \end{equation}
  holds.
  In particular, the above equality holds for all but countably many $\varepsilon$.
\end{thm}

\begin{proof}
  Set $C \coloneqq \liminf_{n \to \infty} N(K_n,\varepsilon)$.
  Using that $K_n \to K$ and following the argument in the proof of \cite[Proposition~7.4.11(2)]{Burago_Burago_Ivanov_01_A_course},
  we deduce that, for any $\delta > 0$, there exists a $(\varepsilon + \delta)$-covering of $K$ with cardinality $C$.
  Thus, 
  \begin{equation}
    N(K,\varepsilon + \delta)
    \leq
    \liminf_{n \to \infty} N(K_n,\varepsilon),
    \quad \forall \delta > 0.
  \end{equation}
  Letting $\delta \to 0$ in the above inequality and using the right-continuity of $N(K, \cdot)$,
  we obtain \eqref{thm eq: 1. metric entropy convergence}.
  Suppose that $\varepsilon >0$ is a continuity point of $N(K,\cdot)$.
  Then we can find an $\varepsilon' < \varepsilon$ satisfying
  $N(K, \varepsilon') = N(K, \varepsilon) \eqqcolon C$.
  By a similar argument as before,
  we deduce that, for any $\delta > 0$, 
  there exists an $(\varepsilon' + \delta)$-covering of $K_n$ with cardinality $C$ for all sufficiently large $n$.
  Letting $\delta > 0$ be such that $\varepsilon' + \delta < \varepsilon$,
  we obtain that  
  \begin{equation}
    \limsup_{n \to \infty} N(K_n, \varepsilon) \leq C = N(K, \varepsilon).
  \end{equation}
  This, combined with \eqref{thm eq: 1. metric entropy convergence}, yields \eqref{thm eq: 2. metric entropy convergence}.
\end{proof}

\begin{thm} [Precompactness] \label{thm: precompact in the local GH top}
  A non-empty subset $\{\cX_\alpha = (X_\alpha, \rho_\alpha) \mid \alpha \in \mathscr{A}\}$ of $\rootedBCM$
  is precompact in the local Gromov--Hausdorff topology 
  if and only if the following condition is satisfied.
  \begin{enumerate} [label= \textup{(\roman*)}]
    \item \label{3. thm item: a uniform bound for metric entropies}
      For every $r>0$ and $\varepsilon > 0$,
      it holds that $\sup_{\alpha} N(X_\alpha|_{\rho_\alpha}^{(r)}, \varepsilon) < \infty$.
  \end{enumerate}
\end{thm}

\begin{proof}
  Suppose that $\{\cX_\alpha \mid \alpha \in \mathscr{A}\}$ is precompact.
  Assume that \ref{3. thm item: a uniform bound for metric entropies} does not hold.
  It is then the case that, for some $r>0$ and $\varepsilon >0$,
  we can find a sequence $(\alpha_{n})_{n \geq 1}$ 
  satisfying $N(X_{\alpha_{n}}^{(r)}, \varepsilon) \to \infty$.
  We choose a subsequence $(n_{k})_{k \geq 1}$ such that $(\cX_{\alpha_{n_{k}}})_{k \geq 1}$ 
  converges to some $\cX=(X, \rho_X) \in \rootedBCM$.
  It then follows from Theorem~\ref{thm: metric entropy convergence} that 
  $\lim_{k \to \infty} N(X_{\alpha_{n_{k}}}^{(r)}, \varepsilon') = N(X^{(r)}, \varepsilon') < \infty$
  for some $\varepsilon' < \varepsilon$.
  Since $N(X_{\alpha_{n_{k}}}^{(r)}, \varepsilon) \leq N(X_{\alpha_{n_{k}}}^{(r)}, \varepsilon')$,
  we obtain that 
  $\limsup_{k \to \infty} N(X_{\alpha_{n_{k}}}^{(r)}, \varepsilon) \leq N(X^{(r)}, \varepsilon') < \infty$,
  which is a contradiction.
  Therefore, \ref{3. thm item: a uniform bound for metric entropies} holds.

  Conversely,
  suppose that \ref{3. thm item: a uniform bound for metric entropies} is satisfied.
  Fix a sequence $(\alpha_{n})_{n \geq 1}$ in $\mathscr{A}$.
  For each $\cX_{\alpha_{n}}$,
  we define a rooted-and-measured $\bcmAB$ space $\cX'_{\alpha_{n}}$
  by equipping $\cX_{\alpha_{n}}$ with the zero measure.
  Then, by \cite[Theorem 2.6]{Abraham_Delmas_Hoscheit_13_A_note} and \cite[Theorem 3.28]{Khezeli_20_Metrization},
  the sequence $(\cX'_{\alpha_{n}})_{n \geq 1}$ has a subsequence $(\cX'_{\alpha_{n_{k}}})_{k \geq 1}$
  convergent in the local Gromov--Hausdorff-vague topology.
  This implies that $(\cX_{\alpha_{n_{k}}})_{k \geq 1}$ converges in the local Gromov--Hausdorff topology
  (cf.\ Corollary \ref{cor: projection continuity in RF} and Proposition~\ref{prop: recovering the local GHV top} below).
\end{proof}

\section{Preliminaries on category theory} \label{sec: Preliminaries on category theory}
  
The idea of using functors to provide a unified framework for the metrization of Gromov--Hausdorff-type topologies  
was introduced by Khezeli \cite{Khezeli_23_A_unified}.  
We begin by briefly explaining why category theory naturally arises in the study of Gromov--Hausdorff-type topologies.  
We then introduce several notions from category theory that are used in our framework.  
We emphasize that no prior knowledge of category theory is assumed of the reader.

Our goal is to extend the local Gromov--Hausdorff topology, 
defined in the preceding section, 
to topologies on collections of $\bcmAB$ spaces equipped with additional elements such as measures.  
To this end, we need a rule $\tau$ that defines additional structures on each $\bcmAB$ space:  
\begin{enumerate}[label = \textup{(\arabic*)}, series = functor, leftmargin = *]
\item \label{item: functor 1}
  each $\bcmAB$ space $X$ is assigned a metrizable topological space $\tau(X)$.
\end{enumerate}
We then define the collection $\rootedBCM(\tau)$ as the set of (equivalence classes of) $(X, \rho_X, a_X)$ 
such that $\cX = (X, \rho_X) \in \rootedBCM$ and $a_X \in \tau(X)$  
(see Proposition~\ref{prop: existence of rootedBCM with additional elements} below).  
Our aim is to endow $\rootedBCM(\tau)$ with a metric  
that characterizes convergence analogously to Theorem~\ref{thm: convergence in the Fell topology}:  
namely, convergence in $\rootedBCM(\tau)$ means the existence of isometric embeddings of the underlying spaces  
into a common metric space, under which the embedded elements converge.  
To give meaning to such embedded elements,  
we require that $\tau$ provide embeddings of additional structures  
whenever isometric embeddings of the underlying spaces are given:  
\begin{enumerate}[resume* = functor]
\item \label{item: functor 2}
  for every isometric embedding $f \colon X \to Y$ between $\bcmAB$ spaces $X$ and $Y$,  
  there exists a corresponding topological embedding $\tau_f \colon \tau(X) \to \tau(Y)$.
\end{enumerate}
Moreover, it is natural to assume the following properties:  
\begin{enumerate}[resume* = functor]
\item \label{item: functor 3}
  for each $\bcmAB$ space $X$, $\tau_{\id_X} = \id_{\tau(X)}$;
\item \label{item: functor 4}
  for any isometric embeddings $f \colon X \to Y$ and $g \colon Y \to Z$,  
  $\tau_{g \circ f} = \tau_g \circ \tau_f$.
\end{enumerate}
These properties abstract the essence of $\tau$ into the notion of a functor.  
By employing category theory,  
we can interpret $\tau$ as a functor from the category of $\bcmAB$ spaces to the category of metrizable topological spaces.  
Thus, the language of categories and functors provides a unified framework in which to formulate our setting.

In what follows, we define category and functor.

\begin{dfn}[Category]
A \emph{category} $\CatC$ consists of the following data:
\begin{itemize}
  \item a class $\ob(\CatC)$ of objects,
  \item for every pair of objects $X, Y \in \ob(\CatC)$, a (possibly empty) set $\Hom_{\CatC}(X, Y)$ of morphisms from $X$ to $Y$,
  \item for every triple $X, Y, Z \in \ob(\CatC)$, a composition 
  \begin{equation}
    \circ \colon \Hom_{\CatC}(Y, Z) \times \Hom_{\CatC}(X, Y) \to \Hom_{\CatC}(X, Z),
  \end{equation}
  \item for each object $X \in \ob(\CatC)$, 
    a distinguished morphism $\id_X \in \Hom_{\CatC}(X, X)$ called the identity morphism,
\end{itemize}
such that the following axioms hold.
\begin{enumerate}
  \item For any morphisms $f \in \Hom_{\CatC}(X, Y)$, $g \in \Hom_{\CatC}(Y, Z)$, and $h \in \Hom_{\CatC}(Z, W)$,  
  $h \circ (g \circ f) = (h \circ g) \circ f$.
  \item For any morphism $f \in \Hom_{\CatC}(X, Y)$, $\id_Y \circ f = f = f \circ \id_X$.
\end{enumerate}
\end{dfn}

The categories we will consider are as follows:  
\begin{enumerate}[label = (Cat\arabic*), leftmargin = *]
  \item \label{item: rBCM, category}
    $\rBCMcat$ denotes the category whose objects are rooted $\bcmAB$ spaces  
    and whose morphisms are root-preserving isometric embeddings;
  \item \label{item: BCM, category}
    $\BCMcat$ denotes the category whose objects are $\bcmAB$ spaces and whose morphisms are isometric embeddings;
  \item \label{item: Met, category}
    $\Metcat$ denotes the category whose objects are non-empty metric spaces and whose morphisms are isometric embeddings;
  \item \label{item: Top, category}
    $\MTopcat$ denotes the category whose objects are metrizable topological spaces and whose morphisms are topological embeddings.
\end{enumerate}
In all the categories above, the composition of morphisms is given by the usual composition of maps,  
and the identity morphism for each object is the usual identity map.

\begin{dfn}[Functor] \label{dfn: functor}
Let $\CatC$ and $\CatD$ be categories.
A \emph{functor} $\tau \colon \CatC \to \CatD$ is a mapping that
\begin{itemize}
  \item associates each object $X \in \ob(\CatC)$ to an object $\tau(X) \in \ob(\CatD)$,
  \item associates each morphism $f \in \Hom_{\CatC}(X, Y)$ to a morphism $\tau_f \in \Hom_{\CatD}(\tau(X), \tau(Y))$,
\end{itemize}
such that the following conditions hold.
\begin{enumerate}
  \item For every object $X \in \CatC$, $\tau_{\id_X} = \id_{\tau(X)}$.
  \item For any morphisms $f \in \Hom_{\CatC}(X, Y)$ and $g \in \Hom_{\CatC}(Y, Z)$, $\tau_{g \circ f} = \tau_g \circ \tau_f$.
\end{enumerate}
\end{dfn}

\begin{exm} \label{exm: identity functor}
  For any category $\CatC$, there exists a functor $\IdFunct_\CatC \colon \CatC \to \CatC$,  
  called the \emph{identity functor},  
  which is defined by $\IdFunct_\CatC(X) \coloneqq X$ for any $X \in \ob(\CatC)$ and $(\IdFunct_\CatC)_f \coloneqq f$ for any $f \in \Hom_{\CatC}(X, Y)$.
\end{exm}

\begin{exm} \label{exm: forgetful and inclusion}
  There exist natural functors between the categories introduced in \ref{item: rBCM, category}--\ref{item: Top, category} above.
  For example, by regarding a metric space as a topological space
  and each isometric embedding as a topological embedding,
  we obtain a natural functor from $\Metcat$ to $\MTopcat$.
  This functor simply forgets the metric structure and extracts only the underlying topological structure.
  Such a functor, which is defined by forgetting part of the information in a category,
  is called a \emph{forgetful functor}.
  As $\bcmAB$ spaces are, in particular, metric spaces,
  there is also a natural functor from $\BCMcat$ to $\Metcat$, referred to as the \emph{inclusion functor}.
  In summary, we obtain the following relation between these categories via functors.
  \begin{equation}
    \begin{tikzcd}
      \rBCMcat \arrow[r, "\ForgetRoot"] & \BCMcat \arrow[r, "\BCInclusion"] & \Metcat \arrow[r, "\ForgetMetric"] & \MTopcat 
    \end{tikzcd}
  \end{equation}
  \begin{enumerate} [label = \textup{($\Gamma_{\arabic*}$)}, leftmargin = *]
    \item The functor $\ForgetRoot$ is the forgetful functor that forgets the root information.
    \item The functor $\BCInclusion$ is the natural inclusion functor.
    \item The functor $\ForgetMetric$ is the forgetful functor that forgets the metrics while retaining the underlying topology.
  \end{enumerate}
  In particular, given any two categories $\CatC$ and $\CatD$ appearing above, there exists a natural functor from $\CatC$ to $\CatD$ (or, from $\CatD$ to $\CatC$),
  and we write $\Gamma_{\CatC \to \CatD}$ (or, $\Gamma_{\CatD \to \CatC}$) for it.
  For example, we write $\Gamma_{\BCMcat \to \MTopcat}$ for $\Gamma_3 \circ \Gamma_2$.
\end{exm}

\begin{exm} \label{exm: measure functor}
  As a more concrete example, we define a functor $\MeasFunct$ as follows.
  \begin{itemize}
    \item 
      For each $X \in \ob(\BCM)$, 
      define $\MeasFunct(X) \coloneqq \Meas{X}$ equipped with the vague topology. 
    \item 
      For each $f \in \Hom_{\BCMcat}(X, Y)$,
      define $(\MeasFunct)_f \coloneqq f_*$, i.e., the pushforward map given by $f$
      (recall it from \eqref{eq: pushforward map}).
  \end{itemize}
  Then the set $\rootedBCM(\MeasFunct)$ is the set of (equivalence classes of) rooted boundedly-compact metric spaces equipped with Radon measures.
  In Section~\ref{sec: the functor for measures} below,
  using our main results,
  we recover the local Gromov--Hausdorff-vague topology on $\rootedBCM(\MeasFunct)$.
\end{exm}

Similarly to maps, we can consider composition of functors.

\begin{dfn}[Composition of functors] \label{dfn: composition of functors}
Let $\CatC, \CatD, \CatE$ be categories.
Suppose $\tau \colon \CatC \to \CatD$ and $\sigma \colon \CatD \to \CatE$ are functors.
Then the \emph{composition} $\sigma \circ \tau \colon \CatC \to \CatE$ is the functor defined as follows:
\begin{itemize}
  \item For each object $X \in \CatC$, $(\sigma \circ \tau)(X) \coloneqq \sigma(\tau(X))$;
  \item For each morphism $f \in \Hom_{\CatC}(X, Y)$, $(\sigma \circ \tau)_f \coloneqq \sigma_{\tau_f}$.
\end{itemize}
\end{dfn}

Below, we introduce natural transformations, a fundamental notion in category theory. 
They provide a way to compare functors in a coherent and structure-preserving manner.

\begin{dfn}[Natural transformation]
Let $\tau$ and $\sigma$ be functors between categories $\CatC$ and $\CatD$.
A \emph{natural transformation} $\eta \colon \tau \Rightarrow \sigma$ is a collection $\{\eta_X\}_{X \in \ob(\CatC)}$ satisfying the following:
\begin{enumerate} [label = \textup{(NT\arabic*)}, leftmargin = *]
  \item \label{dfn item: 1. NT}
    $\eta_X \in \Hom_{\CatD}(\tau(X), \sigma(X))$ for each $X \in \ob(\CatC)$;
  \item \label{dfn item: 2. NT}
    for any $X, Y \in \ob(\CatC)$ and $f \in \Hom_{\CatC}(X, Y)$, $\eta_Y \circ \tau_f = \sigma_f \circ \eta_X$,
    that is, the following diagram commutes.
    \begin{equation}
      \begin{tikzcd}
        \tau(X) \arrow[r, "\tau_f"] \arrow[d, "\eta_X"] & \tau(Y) \arrow[d, "\eta_Y"]\\
        \sigma(X) \arrow[r, "\sigma_f"]                 & \sigma(Y)       
      \end{tikzcd}
    \end{equation}
\end{enumerate}
\end{dfn}

Using category theoretic notions,
we can study Gromov--Hausdorff-type topologies in a unified manner.
Below, we introduce further notions that are used in our main results in Section~\ref{sec: main results}.
Note that the notions given below are not common in category theory, and they are only used in this paper.

The completeness and separability defined below will play an important role to prove those for the Gromov--Hausdorff-type topologies.

\begin{dfn} [Completeness/Separability] \label{dfn: completeness and separability of functor}
  Let $\CatC$ be a category.
  \begin{enumerate} [label = \textup{(\roman*)}, leftmargin = *]
    \item We say that a functor $\tau \colon \CatC \to \Metcat$ is complete
      if and only if, for each $X \in \ob(\CatC)$, the metric on $\tau(X)$ is complete.
    \item We say that a functor $\tau \colon \CatC \to \MTopcat$ is separable
      if and only if, for each $X \in \ob(\CatC)$, the topology on $\tau(X)$ is separable.
  \end{enumerate}
\end{dfn}

Subfunctors, defined below, will be used to describe and study subsets of Gromov--Hausdorff-type spaces.

\begin{dfn}[Subfunctors] \label{dfn: subfunctors}
  Let $\CatC$ be a category 
  and $\tau \colon \CatC \to \MTopcat$ a functor.
  A functor $\sigma \colon \CatC \to \MTopcat$ is called a \emph{subfunctor} of $\tau$
  if the following conditions are satisfied.
  \begin{enumerate}[label = (Sub\arabic*), leftmargin = *]
    \item \label{dfn item: 1, subfunctors}
      For each $X \in \ob(\CatC)$, $\sigma(X)$ is a topological subspace of $\tau(X)$;
      that is, $\sigma(X) \subseteq \tau(X)$ and the topology on $\sigma(X)$ coincides with the relative topology.
    \item \label{dfn item: 2, subfunctors}
      For every $f \in \Hom_{\CatC}(X, Y)$,
      one has $\sigma_f = \tau_f|_{\sigma(X)}$;
      in other words, the morphisms assigned by $\sigma$ are the restrictions of those assigned by $\tau$.
  \end{enumerate}
  We say that a subfunctor $\sigma$ is \emph{precompact} (resp.\ \emph{open, closed})
  if $\sigma(X)$ is precompact (resp.\ open, closed) in $\tau(X)$ for each $X \in \ob(\CatC)$.
  Furthermore, we say that $\sigma$ is \emph{pullback-stable} if,
  for any $f \in \Hom_{\CatC}(X, Y)$,
  it holds that 
  \begin{equation}
    \tau_f^{-1} \bigl(\sigma(Y)\bigr) = \tau(X).
  \end{equation}
  (NB.\ The inclusion $\supseteq$ always holds by \ref{dfn item: 2, subfunctors}.)
\end{dfn}

\begin{rem} \label{rem: define subfunctor}
  To define a subfunctor $\sigma$ of $\tau$,
  it suffices to specify, for each $X \in \ob(\CatC)$,
  a subset $\sigma(X) \subseteq \tau(X)$
  that is preserved under morphisms assigned by $\tau$;
  that is, for any $f \in \Hom_{\CatC}(X, Y)$,
  one has $\tau_f(\sigma(X)) \subseteq \sigma(Y)$.
  Then, by equipping each $\sigma(X)$ with the relative topology 
  and setting $\sigma_f \coloneqq \tau_f|_{\sigma(X)}$,
  one obtains a well-defined subfunctor $\sigma$ of $\tau$.
\end{rem}

Topological subfunctors, defined below, generalize subfunctors
and allow us to topologically embed Gromov--Hausdorff-type spaces into larger ones.

\begin{dfn}[Topological subfunctor] \label{dfn: topological subfunctor}
  Let $\CatC$ be a category 
  and $\tau \colon \CatC \to \MTopcat$ a functor.
  A functor $\sigma \colon \CatC \to \MTopcat$ is called a \emph{topological subfunctor} of $\tau$  
  if there exists a natural transformation $\eta \colon \sigma \Rightarrow \tau$.
  We call $\eta$ a \emph{topological embedding} of $\sigma$ into $\tau$,
  and say that $\sigma$ is \emph{topologically embedded} into $\tau$.
\end{dfn}

Using product functors defined below,
we can easily consider multiple additional structures simultaneously.
Details are discussed in Section~\ref{sec: product structures}.

\begin{dfn} [Product functor] \label{dfn: product functor}
  Let $\CatC$ be a category 
  and $(\tau^{(n)})_{n=1}^N$ a sequence of functors from $\CatC$ to $\MTopcat$.
  We define the product functor $\tau = \prod_{n=1}^N \tau^{(n)}$ as follows:
  \begin{enumerate}
    \item For each $X \in \ob(\CatC)$, define $\tau(X) \coloneqq \prod_{n=1}^N \tau^{(n)}(X)$ equipped with the product topology;
    \item For each $f \in \Hom_{\CatC}(X, Y)$, define $\tau_f \coloneqq \prod_{n=1}^N \tau^{(n)}_f$, i.e.,
      \begin{equation}
        \tau_f((a_n)_{n=1}^N) \coloneqq ( \tau^{(n)}_f(a_n) )_{n=1}^N.
      \end{equation}
  \end{enumerate}
\end{dfn}


\section{Main results}  \label{sec: main results}

In this section, we present the framework for the metrization of Gromov--Hausdorff-type topologies, which is the main result of this work.
Our framework naturally extends the framework introduced by Khezeli \cite{Khezeli_23_A_unified} for compact metric spaces 
to boundedly-compact metric spaces.
The assumptions required by our framework are the same as those required by his framework in the compact case.
In particular, it removes the truncation operation and 
the technical assumptions \cite[Assumptions~3.10 and 3.11]{Khezeli_23_A_unified}, 
which were required in Khezeli’s framework for the boundedly-compact case.
This enables us to consider Gromov--Hausdorff-type topologies for a broader class of additional structures,
as we see in Section~\ref{sec: Examples of functors}.
Before going into the details, we give an overview of our framework.

\begin{dfn}[Structure]
  We refer to a functor $\tau \colon \BCMcat \to \MTopcat$ as a \emph{structure}.
\end{dfn}

Throughout this section, we fix a structure $\tau$.
We begin by introducing the space $\rootedBCM(\tau)$ of interest.
For $\cX = (X, \rho_X, a_X)$ and $\cY = (Y, \rho_Y, a_Y)$ such that $(X, \rho_X)$ and $(Y, \rho_Y)$ are rooted $\bcmAB$ spaces,
$a_X \in \tau(X)$ and $a_Y \in \tau(Y)$,
we say that $\cX$ and $\cY$ are \emph{rooted-$\tau$-isometric}
if and only if there exists a root-preserving isometry $f \colon X \to Y$ with $\tau_f(a_X) = a_Y$.
The following result can be shown in a similar way to Proposition~\ref{prop: existence of rootedBCM}.

\begin{prop}  \label{prop: existence of rootedBCM with additional elements}
There exists a set $\rootedBCM(\tau)$ satisfying the following properties.
\begin{enumerate} [label = \textup{(\roman*)}, leftmargin = *]
  \item 
    The set $\rootedBCM(\tau)$ consists of all triples $(X, \rho_X, a_X)$ such that $(X, \rho_X)$ is a rooted $\bcmAB$ space and $a_X \in \tau(X)$.
  \item 
    For any $\cY = (Y, \rho_Y, a_Y)$ such that $(Y, \rho_Y)$ is a rooted $\bcmAB$ space and $a_Y \in \tau(Y)$,
    there exists a unique element $\cX \in \rootedBCM(\tau)$ that is rooted-$\tau$-isometric to $\cY$.
\end{enumerate}
\end{prop}

\noindent
Henceforth, we fix a representative set $\rootedBCM(\tau)$ as guaranteed by the above proposition.
As before, 
given $\cX = (X, \rho_X, a_X)$ such that $(X, \rho_X)$ is a rooted $\bcmAB$ space and $a_X \in \tau(X)$, 
when necessary,
we identify it with the unique element in $\rootedBCM(\tau)$ that is rooted-$\tau$-isometric to $\cX$.

As a generalization of Theorem~\ref{thm: local GH convergence}\ref{thm item: 2. local GH convergence} and \ref{thm item: 3. local GH convergence}, respectively,
two types of convergence in $\rootedBCM(\tau)$ naturally arise as follows.
In the definitions below, we fix elements $\cX_n = (X_n, \rho_{X_n}, a_{X_n}) \in \rootedBCM(\tau)$, $n \in \NN \cup \{\infty\}$.

\begin{dfn} \label{dfn: RF convergence}
    We say that $\cX_n$ converge to $\cX_\infty$ in the \emph{local GH-type topology with preserved roots}
    if and only if there exist a rooted $\bcmAB$ space $(M, \rho_M)$
    and root-preserving isometric embeddings $f_n \colon X_n \to M$, $n \in \NN \cup \{\infty\}$,
    such that $f_n(X_n) \to f_\infty(X_\infty)$ in the Fell topology as subsets of $M$,
    and $\tau_{f_n}(a_{X_n}) \to \tau_{f_\infty}(a_{X_\infty})$ in $\tau(M)$.
\end{dfn}

\begin{dfn} \label{dfn: RV convergence}
  We say that $\cX_n$ converge to $\cX_\infty$ in the \emph{local GH-type topology with non-preserved roots}
  if and only if there exist a $\bcmAB$ space $M$
  and isometric embeddings $f_n \colon X_n \to M$, $n \in \NN \cup \{\infty\}$,
  such that $f_n(X_n) \to f_\infty(X_\infty)$ in the Fell topology as subsets of $M$,
  $f_n(\rho_{X_n}) \to f_\infty(\rho_{X_\infty})$ in $M$,
  and $\tau_{f_n}(a_{X_n}) \to \tau_{f_\infty}(a_{X_\infty})$ in $\tau(M)$.
\end{dfn}

\noindent
The difference between these two types of convergence lies in the treatment of roots.
In the local GH-type topology with preserved roots,
the roots $\rho_{X_n}$ are mapped to the root $\rho_M$ of the ambient space $M$,
whereas in the local GH-type topology with non-preserved roots, the roots are allowed to move in $M$.

In Section~\ref{sec: Metrization of structures} below,  
we introduce the concept of metrization of structures.  
Based on this, in Section~\ref{sec: Metrization of RF and RV},  
we define two metrics that induce the above-defined two types of convergence, respectively.
In Section~\ref{sec: Coincidence of the two topologies},
we provide a sufficient condition on $\tau$ under which the two topologies coincide.
Topological properties such as Polishness are investigated in Section~\ref{sec: topological properties}.


\subsection{Metrization of structures} \label{sec: Metrization of structures}

For the metrization of $\rootedBCM(\tau)$,
we assume that, for each $\bcmAB$ space $X$, $\tau(X)$ is equipped with a metric 
that is consistent with the associated embeddings given by $\tau$.  
This requirement is precisely described in this subsection as metrization of structures.
For the following discussions,  
recall the forgetful functors and the inclusion functor introduced in Example~\ref{exm: forgetful and inclusion}.

\begin{dfn}[Metrization] \label{dfn: metrization of structures}
  The structure $\tau$ is said to be \emph{metrizable}
  if there exists a functor $\FMet{\tau} \colon \BCMcat \to \Metcat$ such that the following diagram commutes:
  \begin{equation}  \label{dfn eq: metrization of structures}
    \begin{tikzcd}
                                                       & \Metcat \arrow[d]  \\
      \BCMcat \arrow[r, "\tau"] \arrow[ru, "\FMet{\tau}"] & \MTopcat  
    \end{tikzcd}
  \end{equation}
  Here, the vertical arrow denotes the forgetful functor.
  For each $X \in \ob(\BCM)$,  
  we denote by $d^{\FMet{\tau}}_X$ the metric on $\tau(X)$ assigned by $\FMet{\tau}$,  
  or write $d_{\tau(X)}$ when the metrization $\FMet{\tau}$ is clear from the context.
\end{dfn}

\begin{rem} \label{rem: construction of metrization}
  Note that the diagram \eqref{dfn eq: metrization of structures} already characterize the mapping of $\FMet{\tau}$ 
  between objects and morphisms except for metric structure,
  that is, it must hold that 
  \begin{gather}
    \FMet{\tau}(X) = \tau(X),\quad \text{as topological spaces for each $X \in \ob(\BCMcat)$},
    \\
    \FMet{\tau}_f = \tau_f,\quad \text{for each}\ f \in \Hom_{\BCMcat}(X, Y). 
  \end{gather}
  Thus, to construct a metrization of $\tau$,  
  it suffices to define a metric $d^{\FMet{\tau}}_X$ on $\tau(X)$ for each $\bcmAB$ space $X$
  satisfying the following conditions.
  \begin{enumerate} [label = \textup{(\roman*)}, leftmargin = *]
    \item \label{rem item: 1. construction of metrization}
      For each $X \in \ob(\BCMcat)$, the topology on $\tau(X)$ induced by $d^{\FMet{\tau}}_X$ coincides with the given topology.
    \item \label{rem item: 2. construction of metrization}
      For each $f \in \Hom_{\BCMcat}(X,Y)$, 
      the map $\tau_f \colon \tau(X) \to \tau(Y)$ is distance-preserving.
  \end{enumerate}
\end{rem}

The metrization of certain topologies, such as the vague topology, requires spaces to be rooted.  
Accordingly, for functors involving such topologies,  
we equip each underlying space $X$ with a root $\rho_X$  
in order to obtain a metric on $\tau(X)$.  
This corresponds to considering a metrization of the following composition:
\begin{equation}
  \FS{\tau} \coloneqq \tau \circ \Gamma_{\rBCMcat \to \BCMcat},
\end{equation}
where, as recalled from Definition~\ref{dfn: composition of functors},  
\begin{itemize}
  \item for each $(X, \rho_X) \in \ob(\BCMcat)$, $\FS{\tau}(X, \rho_X) = \tau(X)$;
  \item for each $f \in \Hom_{\rBCMcat}((X, \rho_X), (Y, \rho_Y))$, $\FS{\tau}_f = \tau_f$.
\end{itemize}

\begin{dfn} [Space-rooted metrization]\label{dfn: space-rooted metrization of structures}
  The structure $\tau$ is said to be \emph{space-rooted metrizable} 
  if there exists a functor $\FSMet{\tau} \colon \rBCMcat \to \Metcat$ such that 
  the following diagram commutes.
  \begin{equation}
    \begin{tikzcd}
                                               & \Metcat \arrow[d] \\
      \rBCMcat \arrow[ru, "\FSMet{\tau}"] \arrow[r, "\FS{\tau}"]     & \MTopcat 
    \end{tikzcd}
  \end{equation}
  Here, the vertical arrow denotes the forgetful functor.
  We refer to $\FSMet{\tau}$ as a \emph{space-rooted metrization} of $\tau$.
  For each $(X, \rho_X) \in \ob(\rBCMcat)$,
  we denote by $d^{\FSMet{\tau}}_X$ the metric on $\tau(X)$ assigned by $\FSMet{\tau}$,  
  or write $d_{\tau(X), \rho_X}$ when the metrization $\FSMet{\tau}$ is clear from the context.
\end{dfn}

\begin{rem} \label{rem: construction of SR metrization}
  Similarly to Remark \ref{rem: construction of metrization},
  to construct a space-rooted metrization $\FSMet{\tau}$,  
  it suffices to define a metric $d^{\FSMet{\tau}}_{X, \rho_X}$ on $\tau(X)$ for each $(X, \rho_X) \in \ob(\rBCMcat)$
  satisfying the following conditions.
  \begin{enumerate} [label = \textup{(\roman*)}, leftmargin = *]
    \item \label{rem item: 1. construction of SR metrization}
      For each $(X, \rho_X) \in \ob(\rBCMcat)$, 
      the topology on $\tau(X)$ induced by $d^{\FSMet{\tau}}_{X, \rho_X}$ coincides with the given topology.
    \item \label{rem item: 2. construction of SR metrization}
      For each $f \in \Hom_{\rBCMcat}((X, \rho_X),(Y, \rho_Y))$, 
      the map $\tau_f \colon \tau(X) \to \tau(Y)$ is distance-preserving
      with respect to $d^{\FSMet{\tau}}_{X, \rho_X}$ and $d^{\FSMet{\tau}}_{Y, \rho_Y}$.
  \end{enumerate}
\end{rem}

The space-rooted metrization $\FSMet{\tau}$ introduced above will be used  
to metrize the local GH-type topology with preserved roots in Section~\ref{sec: Metrization of RF}.  
For the local GH-type topology with non-preserved roots, we will use a different metrization.  
Instead of equipping the underlying space with a root,  
we equip each element of $\tau(X)$ with a root.  
This corresponds to considering the product functor  
\begin{equation}
  \FE{\tau} \coloneqq \tau \times \Gamma_{\BCMcat \to \MTopcat}
\end{equation}
where, as recalled from Definition~\ref{dfn: product functor},  
\begin{itemize}
  \item for each $X \in \ob(\BCMcat)$, $\FE{\tau}(X) = \tau(X) \times X$, equipped with the product topology;
  \item for each $f \in \Hom_{\BCMcat}(X,Y)$, $\FE{\tau}_f = \tau_f \times f$.
\end{itemize}
Even when $\tau$ is not metrizable, the functor $\FE{\tau}$ may be.  
For example, in the case of the measure structure $\MeasFunct$ introduced in Example~\ref{exm: measure functor},  
$\MeasFunct^\times(X) = \Meas{X} \times X$ admits a metric $\Vague{X}$,  
defined in Section~\ref{sec: the vague topology}.  
This leads to the following.

\begin{dfn}[Element-rooted metrization] \label{dfn: ER metrization}
  The structure $\tau$ is said to be \emph{element-rooted metrizable}  
  if $\FE{\tau}$ is metrizable, i.e.,
  there exists a functor $\FEMet{\tau} \colon \BCMcat \to \Metcat$ such that the following diagram commutes.
  \begin{equation} \label{dfn eq: ER metrization}
    \begin{tikzcd}
                                                                   & \Metcat \arrow[d] \\
      \BCMcat \arrow[r, "\FE{\tau}"'] \arrow[ru, "\FEMet{\tau}"]  & \MTopcat 
    \end{tikzcd}
  \end{equation}
  Here, the vertical arrow denotes the forgetful functor.
  We refer to $\FEMet{\tau}$ as an \emph{element-rooted metrization} of $\tau$.
  For each $X \in \ob(\BCMcat)$,
  we denote by $d^{\FEMet{\tau}}_X$ the metric on $\FE{\tau}(X) = \tau(X) \times X$ assigned by $\FEMet{\tau}$,  
  or write $d_{\FE{\tau}(X)}$ when the metrization $\FEMet{\tau}$ is clear from the context.
\end{dfn}

\begin{rem} \label{rem: construction of ER metrization}
  Similarly to Remark \ref{rem: construction of metrization}, to construct an element-rooted metrization $\FEMet{\tau}$,  
  it suffices to define a metric $d^{\FEMet{\tau}}_X$ on $\FE{\tau}(X) = \tau(X) \times X$ for each $\bcmAB$ space $X$
  satisfying the following conditions.
  \begin{enumerate} [label = \textup{(\roman*)}, leftmargin = *]
    \item \label{rem item: 1. construction of ER metrization}
      For each $X \in \ob(\BCMcat)$, the topology on $\FE{\tau}(X)$ induced by $d^{\FEMet{\tau}}_X$ coincides with the product topology.
    \item \label{rem item: 2. construction of ER metrization}
      For each $f \in \Hom_{\BCMcat}(X,Y)$,
      the product map $\tau_f \times f \colon \FE{\tau}(X) \to \FE{\tau}(Y)$ is distance-preserving. 
    \end{enumerate}
\end{rem}

Metrizable structures belong to the class of element-rooted metrizable structures.

\begin{lem} \label{lem: metrization defines space and ER metrization}
  Let $\FMet{\tau}$ be a metrization of $\tau$.
  Then $\FMet{\tau}$ naturally defines an element-rooted metrization $\FEMet{\tau}$ of $\tau$ as follows:
  for each $X \in \ob(\BCMcat)$,
  equip $\FE{\tau}(X)$ with the metric given by
  \begin{equation}
    d^{\FEMet{\tau}}_X \bigl( (a_X, \rho_X), (a'_X, \rho'_X) \bigr) 
    \coloneqq 
    d^{\FMet{\tau}}_X(a_X, a'_X) \vee d_X(\rho_X, \rho'_X),
    \quad 
    (a_X, \rho_X), (a'_X, \rho'_X) \in \FE{\tau}(X).
  \end{equation}
\end{lem}

Similarly, an element-rooted metrization of $\tau$ naturally defines a rooted-metrization as follows.

\begin{lem} \label{lem: ER metrization defines SR metrization}
  Let $\FEMet{\tau}$ be an element-rooted metrization of $\tau$.
  Then $\FEMet{\tau}$ naturally defines a space-rooted metrization $\FSMet{\tau}$ of $\tau$ as follows:
  for each $(X, \rho_X) \in \ob(\rBCMcat)$, equip $\tau(X)$ with the metric given by
  \begin{equation}
    d^{\FSMet{\tau}}_{X, \rho_X}(a_X, a'_X) 
    \coloneqq 
    d^{\FEMet{\tau}}_X \bigl( (a_X, \rho_X), (a'_X, \rho_X) \bigr) ,
    \quad 
    a_X, a'_X \in \tau(X).
  \end{equation}
\end{lem}

The above two lemmas are easily verified by using Remarks~\ref{rem: construction of SR metrization} and \ref{rem: construction of ER metrization},
and so the proofs are omitted.
From the lemmas,
we find the following relation between the classes of structures defined above:
\begin{equation}
  \text{metrizable}\quad \Rightarrow \quad \text{element-rooted metrizable}\quad \Rightarrow \quad \text{space-rooted metrizable}.
\end{equation}
Thus, the class of space-rooted metrizable structures is the most generalized notion among them.


\subsection{Metrization of the topologies}  \label{sec: Metrization of RF and RV}

In this subsection,  
we provide metrizations of the local GH-type topologies with preserved roots and with non-preserved roots, respectively.
To ensure that metrizations of $\tau$ introduced in the previous subsection define metrics on $\rootedBCM(\tau)$,
we will assume the continuity of $\tau$ with respect to isometric embeddings,
described as follows.
This corresponds to \cite[Definition~2.7]{Khezeli_23_A_unified}.

\begin{assum} \label{assum: embedding continuity}
  Fix $X, Y \in \ob(\BCMcat)$ and $f_n \in \Hom_{\BCMcat}(X, Y)$, $n \in \mathbb{N} \cup \{ \infty \}$.
  If $f_{n} \to f_{\infty}$ in the compact-convergence topology,
  then $\tau_{f_{n}}(a) \to \tau_{f_{\infty}}(a)$ in $\tau(Y)$ for all $a \in \tau(X)$.
\end{assum}

\begin{rem} \label{rem: strong continuity of functors}
  In the setting of Assumption~\ref{assum: embedding continuity}, 
  $\tau_{f_n} \to \tau_{f_\infty}$ in the compact-convergence topology.
  Indeed, since each $\tau_{f_n}$ is distance-preserving and $(\tau_{f_n}(a))_{n \geq 1}$ is convergent for each $a \in \tau(X)$,
  the Arzel\`{a}--Ascoli theorem (cf.\ \cite[Theorem~47.1]{Munkres_00_Topology}) implies that the family $\{\tau_{f_n}\}_{n \geq 1}$
  is precompact in the compact-convergence topology.
  Together with the pointwise convergence of $\tau_{f_n}$ to $\tau_f$,
  this yields the desired conclusion.
\end{rem}

\begin{dfn}[Embedding-continuity] \label{dfn: embedding continuity}
  We say that $\tau$ is \emph{embedding-continuous}
  if it satisfies Assumption~\ref{assum: embedding continuity}.
\end{dfn}

\subsubsection{Metrization of the local GH-type topology with preserved roots}  \label{sec: Metrization of RF}

Here, we define a metric inducing the local GH-type topology with preserved roots.
In this subsubsection, we fix a space-rooted metrization $\FSMet{\tau}$.

Using the space-rooted metrization $\FSMet{\tau}$,
we define distance between elements of $\rootedBCM(\tau)$ as follows.
This is a natural generalization of Definition~\ref{dfn: RF GH metric}.

\begin{dfn} \label{dfn: Generalized RF metric}
  We define, for each $\cX \allowbreak =\allowbreak (X, \rho_X, a_X),\,\allowbreak \cY \allowbreak =\allowbreak (Y, \rho_Y, a_Y) \in \rootedBCM(\tau)$,
  \begin{equation}
    \RFMet^{\FSMet{\tau}} (\cX, \cY)
    \coloneqq
    \inf_{f, g, Z}
    \Bigl\{
      \lHausMet{Z,\rho_Z}(f(X), g(Y)) \vee  d^{\FSMet{\tau}}_{Z, \rho_Z}(\tau_{f}(a_X), \tau_{g}(a_{Y}))  
    \Bigr\},
  \end{equation}
  where the infimum is taken 
  over all $(Z, \rho_{Z}) \in \rootedBCM$ 
  and root-preserving isometric embeddings $f \colon  X \to Z$ and $g \colon Y \to Z$.
\end{dfn}

\begin{rem}
  When $\tau$ admits a metrization $\FMet{\tau}$,  
  we use the naturally associated space-rooted metrization $\FSMet{\tau}$  
  in the sense of Lemmas~\ref{lem: metrization defines space and ER metrization} and \ref{lem: ER metrization defines SR metrization}
  to define the distance 
  $\RFMet^{\FMet{\tau}}(\cX, \cY) \coloneqq \RFMet^{\FSMet{\tau}}(\cX, \cY)$.  
  Namely,
  \begin{equation}
    \RFMet^{\FMet{\tau}} (\cX, \cY)
    \coloneqq
    \inf_{f, g, Z}
    \left\{
      \lHausMet{Z,\rho_Z}(f(X), g(Y)) \vee  d^{\FMet{\tau}}_Z(\tau_{f}(a_X), \tau_{g}(a_{Y}))  
    \right\}.
  \end{equation}
  Similarly, when $\tau$ admits an element-rooted metrization $\FEMet{\tau}$,  
  we define $\RFMet^{\FEMet{\tau}}(\cX, \cY) \coloneqq \RFMet^{\FSMet{\tau}}(\cX, \cY)$  
  using the associated space-rooted metrization $\FSMet{\tau}$ in the sense of Lemma~\ref{lem: ER metrization defines SR metrization}.
\end{rem}

We first show that the convergence induced by the above distance coincides
with the convergence in the local GH-type topology with preserved roots.

\begin{thm} \label{thm: RF convergence}
  Let $\cX_n$, $n \in \NN \cup \{\infty\}$ be elements of $\rootedBCM(\tau)$.
  The following statements are equivalent with each other:
  \begin{enumerate} [label = \textup{(\roman*)}, leftmargin = *]
    \item \label{thm item: 1. convergence in RF}
      $\cX_n \to \cX_\infty$ with respect to $\RFMet^{\FSMet{\tau}}$;
    \item \label{thm item: 2. convergence in RF}
      $\cX_n \to \cX_\infty$ in the local GH-type topology with preserved roots.
  \end{enumerate}
\end{thm}

\begin{proof}
  Write $\cX_n=(X_n, \rho_{X_n}, a_{X_n})$, $n \in \NN \cup \{\infty\}$.
  The implication \ref{thm item: 2. convergence in RF} $\Rightarrow$ \ref{thm item: 1. convergence in RF} is immediate 
  from the definition of $\RFMet^{\FSMet{\tau}}$.
  Assume that $\varepsilon_n \coloneqq \RFMet^{\FSMet{\tau}}(\cX_n, \cX_{\infty}) \to 0$.
  Set $\varepsilon_n' \coloneqq \varepsilon_n + n^{-1}$.
  For each $n \in \NN$, by the definition of $\RFMet^{\FSMet{\tau}}$,
  there exists $(Y_n, \rho_{Y_n}) \in \rootedBCM$,
  together with root-preserving isometric embeddings
  $f_n \colon X_n \to Y_n$ and $g_n \colon X_\infty \to Y_n$,
  such that
  \begin{equation}
    \lHausMet{Y_n, \rho_{Y_n}}(f_n(X_n), g_n(X_\infty)) 
    \vee
    d^{\FSMet{\tau}}_{Y_n, \rho_{Y_n}}(\tau_{f_n}(a_{X_n}), \tau_{g_n}(a_{X_\infty})) 
    < \varepsilon_n'.
  \end{equation}
  We define $M$ to be the $\bcmAB$ space obtained by gluing the spaces $Y_n$, $n \in \NN$,
  along the images of $X_\infty$ in each $Y_n$
  and then taking the completion (cf.\ \cite[Lemma~2.42]{Khezeli_23_A_unified}).
  Let $\iota_n \colon Y_n \to M$, $n \in \NN$, be the canonical isometric embeddings induced by the gluing construction.
  Note that the composition $\iota_n \circ g_n \colon X_\infty \to M$ is independent of $n$,
  and we denote this common map by $h_\infty$.
  The following commutative diagram summarizes the spaces and isometric embeddings defined so far.
  \begin{equation} \label{pr eq: 2. convergence in RF}
    \begin{tikzcd}[column sep=10pt, row sep=10pt]
                                                            &                           & M                                                                        &                                    &                                 \\
                                                            & Y_n \arrow[ru, "\iota_n"] &                                                                          & Y_{n+1} \arrow[lu, "\iota_{n+1}"'] &                                 \\
    X_n \arrow[ru, "f_n"] \arrow[rruu, bend left = 50, "h_n"] &                           & X_\infty \arrow[lu, "g_n"] \arrow[uu, "h_\infty"] \arrow[ru, "g_{n+1}"'] &                                    & X_{n+1} \arrow[lu, very near end, "f_{n+1}"'] \arrow[lluu, bend right = 50, "h_{n+1}"']
    \end{tikzcd}
  \end{equation}
  We define the root of $M$ by setting $\rho_M \coloneqq h_\infty(\rho_{X_\infty})$,
  so that $h_\infty$ is a root-preserving isometric embedding from $X_\infty$ to $M$.
  Then $\iota_n \colon Y_n \to M$ is also a root-preserving isometric embedding.
  This yields that 
  \begin{align}
    d^\tau_{M, \rho_M}( \tau_{h_n}(a_{X_n}), \tau_{h_\infty}(a_{X_\infty})) 
    &= 
    d^{\FSMet{\tau}}_{Y_n, \rho_{Y_n}}(\tau_{f_n}(a_{X_n}), \tau_{g_n}(a_{X_\infty})) 
    < \varepsilon_n'.
  \end{align}
  Similarly, using Proposition~\ref{prop: the local Hausdorff metric is preserved},
  we deduce that $\lHausMet{M, \rho_M}( h_n(X_n), h_\infty(X_\infty)) < \varepsilon_n'$.
  Therefore, we obtain \ref{thm item: 2. convergence in RF}.
\end{proof}

Under the embedding-continuity of $\tau$,
we can prove that $\RFMet^{\FSMet{\tau}}$ is a metric.

\begin{thm} \label{thm: space-rooted metric}
  If $\tau$ is embedding-continuous,
  then $\RFMet^{\FSMet{\tau}}$ is a metric on $\rootedBCM(\tau)$.
\end{thm}

\begin{proof}
  The symmetry of $\RFMet^{\FSMet{\tau}}$ is obvious.
  Assume that $\RFMet^{\FSMet{\tau}}(\cX, \cY) = 0$
  for $\cX= (X, \rho_X, a_X)$ and $\cY = (Y, \rho_Y, a_Y)$.
  By Theorem~\ref{thm: RF convergence},
  there exist a rooted $\bcmAB$ space $(M, \rho_M)$
  and root-preserving isometric embeddings $f_n \colon X \to M$, $n \in \NN$, and $g \colon Y \to M$
  such that
  $f_n(X) \to g(Y)$ in the Fell topology as subsets of $M$,
  and $\tau_{f_n}(a_X) \to \tau_g(a_Y)$ in $\tau(M)$.
  Since each $f_n$ is distance-preserving and satisfies $f_n(\rho_X) = \rho_M$ for all $n \in \NN$,
  we can apply the Arzel\`{a}--Ascoli theorem to the sequence $(f_n)_{n \geq 1}$ in $C(X, M)$.
  Thus, we obtain a subsequence $(f_{n_k})_{k \geq 1}$ that converges to a map $f \in C(X, M)$ in the compact-convergence topology.
  It is easy to verify that $f$ is a root-preserving isometric embedding from $X$ to $M$.
  Moreover, by the continuity of $\tau$, we deduce that $f(X) = g(Y)$ and $\tau_f(a_X) = \tau_g(a_Y)$.
  Therefore, $\cX$ is rooted-$\tau$-isometric to $\cY$.
  This proves the positive definiteness.

  To prove the triangle inequality,
  assume that $\RFMet^{\FSMet{\tau}}(\cX_{1}, \cX_{2}) < r$ and $\RFMet^{\FSMet{\tau}}(\cX_{2}, \cX_{3}) < s$ 
  for $\cX_i = (X_i, d_{X_i}, \rho_{X_i}, a_{X_i})$, $i=1,2,3$.
  Using the gluing technique similarly to the proof of Theorem~\ref{thm: RF convergence},
  one can construct a rooted $\bcmAB$ space $(Y, \rho_Y)$
  and root-preserving isometric embeddings $f_i \colon X_i \to Y$, $i = 1,2,3$,
  such that
  \begin{gather}
    \lHausMet{Y, \rho_Y}( f_1(X_{1}), f_2(X_{2}) ) 
    \vee
    d^\tau_{Y, \rho_Y}( \tau_{f_1}(a_{X_{1}}), \tau_{f_2}(a_{X_{2}}) ) 
    < r,
    \\
    \lHausMet{Y, \rho_Y}( f_2(X_{2}), f_{3}(X_{3}) ) 
    \vee
    d^\tau_{Y, \rho_Y}( \tau_{f_2}(a_{X_{2}}), \tau_{f_{3}}(a_{X_{3}}) ) 
    < s.
  \end{gather}
  This implies that $\RFMet^{\FSMet{\tau}}(\cX_{1}, \cX_{3}) < r + s$,
  and hence $\RFMet^{\FSMet{\tau}}$ satisfies the triangle inequality.
  This completes the proof.
\end{proof}

\begin{dfn} [The local GH-type topology with preserved roots]
  When $\tau$ is embedding-continuous,
  we refer to the topology on $\rootedBCM(\tau)$ induced by $\RFMet^{\FSMet{\tau}}$ as the \emph{local GH-type topology with preserved roots}.
\end{dfn}

\begin{rem}
  By Definition~\ref{dfn: RF convergence} and Theorem~\ref{thm: RF convergence},  
  the local GH-type topology with preserved roots is independent of space-rooted metrization of $\tau$,  
  and depends only on the functor $\tau$.  
  Thus, it is natural to define this topology without relying on such metrization.  
  A possible approach is to define the topology via net convergence (cf.\ \cite[Section~11.D]{Willard_70_General}).  
  For example, by replacing sequential convergence in Definition~\ref{dfn: RF convergence} with net convergence,  
  we can define which nets converge.  
  However, we have not been able to verify whether the diagonal principle holds in this setting,  
  which is one of the conditions required to define a topology via net convergence.  
  In particular, we have not succeeded in constructing a $\bcmAB$ space into which all the spaces in a diagonal net can be commonly embedded
  so that they converge.  
  We leave this as an open problem.
\end{rem}

\begin{cor} \label{cor: projection continuity in RF}
  If $\tau$ is embedding-continuous and space-rooted metrizable, then the map
  \begin{equation}
    \rootedBCM(\tau) \ni (X, \rho_X, a_X) \mapsto (X, \rho_X) \in \rootedBCM
  \end{equation}
  is continuous with respect to the local GH-type topology with preserved roots on $\rootedBCM(\tau)$  
  and the local Gromov--Hausdorff topology on $\rootedBCM$.
\end{cor}

\begin{proof}
  This immediately follows from Theorems~\ref{thm: local GH convergence} and \ref{thm: RF convergence}.
\end{proof}
\subsubsection{Metrization of the local GH-type topology with non-preserved roots}  \label{sec: Metrization of the local GH-type topology with non-preserved roots}

Here, we define a metric that induces the local GH-type topology with non-preserved roots,  
based on an element-rooted metrization.  
The proofs are omitted, as they are analogous to those in the preceding subsubsection.
In this subsubsection, we fix a element-rooted metrization $\FSMet{\tau}$.

Using the element-rooted metrization, we define a distance analogously to \eqref{dfn eq: RV GH metric}.
For the following definition, 
recall the map $\FE{\Image{f}}$  from \eqref{eq: image product map}.

\begin{dfn} \label{dfn: Generalized RV metric}
  We define, for $\cX=(X, \rho_X, a_X), \cY=(Y, \rho_Y, a_Y) \in \rootedBCM(\tau)$,
  \begin{equation}  \label{dfn eq: element-rooted GH-type metric}
    \RVMet^{\FEMet{\tau}}(\cX, \cY)
    \coloneqq
    \inf_{f, g, Z}
    \Bigl\{
      \lHausMet{Z} \bigl( \FE{\Image{f}}(X, \rho_X), \FE{\Image{g}}(Y, \rho_Y) \bigr)
      \vee
      d^{\FEMet{\tau}}_Z \bigl( \FE{\tau}_f(a_X, \rho_X), \FE{\tau}_g(a_Y, \rho_Y) \bigr)
    \Bigr\}
    ,
  \end{equation}
  where the infimum is taken 
  over all $Z \in \BCM$ 
  and isometric embeddings $f \colon  X \to Z$ and $g \colon Y \to Z$.
\end{dfn}

\begin{rem}
  When $\tau$ admits a metrization $\FMet{\tau}$, we use the naturally associated element-rooted metrization $\FEMet{\tau}$ 
  in the sense of Lemma~\ref{lem: metrization defines space and ER metrization} to define 
  a distance $\RVMet^{\FMet{\tau}}(\cX, \cY) \coloneqq \RVMet^{\FEMet{\tau}}(\cX, \cY)$.
  Namely, 
  \begin{equation}
    \RVMet^{\FMet{\tau}} (\cX, \cY)
    \coloneqq
    \inf_{f, g, Z}
    \Bigl\{
      \lHausMet{Z}\bigl( \FE{\Image{f}}(X, \rho_X), \FE{\Image{g}}(Y, \rho_Y) \bigr) \vee d^{\FMet{\tau}}_X(\tau_{f}(a_X), \tau_{g}(a_{Y}))  
    \Bigr\}.
  \end{equation}
\end{rem}

The above-defined distance induces the convergence in the local GH-type topology with non-preserved roots, as described below.

\begin{thm} \label{thm: RV convergence}
  For each $n \in \mathbb{N} \cup \{\infty\}$,
  let $\cX_n=(X_n, \rho_{X_n}, a_{X_n})$ be an element of $\rootedBCM(\tau)$.
  The following statements are equivalent with each other:
  \begin{enumerate} [label = \textup{(\roman*)}, leftmargin = *]
    \item \label{thm item: 1. convergence in RV}
      $\cX_n \to \cX_\infty$ with respect to $\RVMet^{\FEMet{\tau}}$;
    \item \label{thm item: 2. convergence in RV}
      $\cX_n \to \cX_\infty$ in the local GH-type topology with non-preserved roots.
  \end{enumerate}
\end{thm}

When $\tau$ is embedding-continuous, we can show that $\RVMet^{\FEMet{\tau}}$ is indeed a metric.  

\begin{thm} \label{thm: element-rooted metric}
  Let $\FEMet{\tau}$ be an element-rooted metrization of $\tau$.
  If $\tau$ is embedding-continuous, then $\RVMet^{\FEMet{\tau}}$ defines a metric on $\rootedBCM(\tau)$.
\end{thm}

\begin{dfn}[The local GH-type topology with non-preserved roots]
  When $\tau$ is embedding-continuous,
  we refer to the topology on $\rootedBCM(\tau)$ induced by $\RVMet^{\FEMet{\tau}}$ as the \emph{local GH-type topology with non-preserved roots}.
\end{dfn}


\subsection{Coincidence of the two topologies}  \label{sec: Coincidence of the two topologies}

When the structure $\tau$ is (element-rooted) metrizable,  
it is also space-rooted metrizable by Lemmas~\ref{lem: metrization defines space and ER metrization} and~\ref{lem: ER metrization defines SR metrization}.  
Thus, both versions of the local GH-type topology — with preserved roots and with non-preserved roots —  
are well-defined on $\rootedBCM(\tau)$.  
In this subsection, we provide a sufficient condition under which these two topologies coincide.  
This condition is sufficiently mild that all the examples considered in Section~\ref{sec: Examples of functors} satisfy it.  
Throughout this subsection,  
we assume that $\tau$ is embedding-continuous.

By Theorems~\ref{thm: RF convergence} and~\ref{thm: RV convergence},  
it is clear that the local GH-type topology with non-preserved roots is coarser than that with preserved roots.  
To prove the reverse inclusion,  
it suffices to show that any sequence converging in the local GH-type topology with non-preserved roots  
also converges in the local GH-type topology with preserved roots.  
As stated in Theorem~\ref{thm: RV convergence},  
convergence in the local GH-type topology with non-preserved roots entails the existence of embeddings  
under which both the roots and the additional elements converge.  
To upgrade this to convergence in the topology with preserved roots,  
we glue all the embedded roots to a single point.  
Since this gluing operation alters the underlying metric structure,  
it is desirable that the distance between the additional elements remains stable under such a spatial deformation.  
The following notion formalizes this idea.

\begin{dfn}[Stability of metrization] \label{dfn: stability}
  We say that a functor $\sigma \colon \BCMcat \to \Metcat$ is \emph{stable} 
  if and only if there exists a function $\Dist{\sigma} \colon \RNp \to \RNp$ satisfying the following conditions.
  \begin{enumerate} [label = \textup{(Stab\arabic*)}, leftmargin = *]
    \item \label{dfn item: 1, stability}
      It holds that $\lim_{\varepsilon \to 0} \Dist{\sigma}(\varepsilon) = \Dist{\sigma}(0) = 0$.
    \item \label{dfn item: 2, stability}
      Fix $\bcmAB$ spaces $X, Y, M_1,$ and $M_2$.
      Let $f_i \colon X \to M_i$ and $g_i \colon Y \to M_i$, $i = 1,2$, be isometric embeddings.
      If there exists $\varepsilon \in \RNp$ such that, for all $x \in X$ and $y \in Y$,
      \begin{equation} \label{assum item eq: 1, stability}
        d_{M_2}(f_2(x), g_2(y)) \leq d_{M_1}(f_1(x), g_1(y)) + \varepsilon,
      \end{equation}
      then, for all $a \in \sigma(X)$ and $b \in \sigma(Y)$,
      \begin{equation} \label{assum item eq: 2, stability}
        d^{\sigma}_{M_2}( \sigma_{f_2}(a), \sigma_{g_2}(b)) 
        \leq     
        d^{\sigma}_{M_1}( \sigma_{f_1}(a), \sigma_{g_1}(b) )  
        + 
        \Dist{\sigma}(\varepsilon),
      \end{equation}
      where $d^\sigma_{M_1}$ and $d^\sigma_{M_2}$ denote the metrics on $\sigma(M_1)$ and $\sigma(M_2)$, respectively.
  \end{enumerate}
  We call such a function $\Dist{\sigma}$ a \emph{distortion} of $\sigma$.
  We say that a structure is \emph{stably metrizable} 
  if it admits a stable element-rooted metrization.
\end{dfn}

The intuitive interpretation of \ref{dfn item: 2, stability} is as follows.
Inequality~\eqref{assum item eq: 1, stability} expresses 
that the metric space $M_2$ is a deformation of $M_1$ with distortion at most $\varepsilon$ in terms of distance.
Inequality~\eqref{assum item eq: 2, stability} then states 
that the resulting distortion in the additional structure, caused by the deformation of the underlying spaces, 
is controlled by $\Dist{\sigma}(\varepsilon)$.

\begin{thm} \label{thm: coincidence of RF and RV}
  If $\tau$ admits a stable element-rooted metrization $\FEMet{\tau}$, then 
  \begin{equation}  \label{thm eq: 1. coincidence of RF and RV}
        \RVMet^{\FEMet{\tau}}(\cX, \cY) 
        \leq
        d^{\FEMet{\tau}}_{\rootedBCM}(\cX, \cY) 
        \leq     
        2\,\RVMet^{\FEMet{\tau}}(\cX, \cY)    
        +
        \Dist{\FEMet{\tau}} \bigl( 2\, \RVMet^{\FEMet{\tau}}(\cX, \cY) \bigr),
        \quad
        \forall \cX, \cY \in \rootedBCM(\tau).
      \end{equation}
  If $\tau$ admits a stable metrization $\FMet{\tau}$, 
  then the same inequalities hold with $\FEMet{\tau}$ replaced by $\FMet{\tau}$.
\end{thm}

\begin{proof}
  Assume that $\tau$ admits a stable element-rooted metrization $\FEMet{\tau}$.
  Note that the left inequality of \eqref{thm eq: 1. coincidence of RF and RV} follows by the definitions of $\RFMet^{\FEMet{\tau}}$ and $\RVMet^{\FEMet{\tau}}$.
  We first prove that, for any $\cX, \cY \in \rootedBCM(\tau)$ and $\varepsilon > 0$,
  \begin{equation}  \label{pr eq: 1, coincidence of Type 1 and 2}
    \RVMet^{\FEMet{\tau}}(\cX, \cY) < \varepsilon 
    \quad
    \Longrightarrow 
    \quad
    \RFMet^{\FEMet{\tau}}(\cX, \cY) < 2\varepsilon + \Dist{\FEMet{\tau}}(\varepsilon).
  \end{equation}
  Fix $\cX, \cY \in \rootedBCM(\tau)$, 
  and let $\varepsilon > 0$ be such that $\RVMet^{\FEMet{\tau}}(\cX, \cY) < \varepsilon$.
  By definition,
  there exist a $\bcmAB$ space $M_1$ and isometric embeddings $f_1 \colon X \to M_1$ and $g_1 \colon Y \to M_1$
  such that 
  \begin{equation}
    \lHausMet{M_1}\bigr( \FE{\Image{f_1}}(X, \rho_X), \FE{\Image{g_1}}(Y, \rho_Y) \bigl) 
    \vee    
    d^{\FEMet{\tau}}_{M_1} \bigl( \FE{\tau}_{f_1}(a_X, \rho_X), \FE{\tau}_{g_1} (a_Y, \rho_Y) \bigr) 
    < \varepsilon.
  \end{equation}
  We define a metric $d_Z$ on the disjoint union $Z \coloneqq X \sqcup Y$ by setting 
  $d_M|_{X \times X} \coloneqq d_X$, $d_M|_{Y \times Y} \coloneqq d_Y$, and 
  \begin{equation}
    d_Z(x,y) \coloneqq d_{M_1}(f_1(x), g_1(y)) + d_{M_1}(f_1(\rho_X), g_1(\rho_Y)),
  \end{equation}
  for each $x \in X$ and $y \in Y$.
  We then define a $\bcmAB$ space $M_2$ by gluing $\rho_X$ and $\rho_Y$ in $Z$.
  Let $q \colon Z \to M_2$ be the associated quatient map.
  We equip $M_2$ with the root $\rho_{M_2} \coloneqq q(\rho_X) = q(\rho_Y)$.
  The triangle inequality yields that, for any $x, x' \in X$
  \begin{align}
    d_Z(x, \rho_X) + d_Z(\rho_Y, x') 
    &= d_X(x, \rho_X) + d_Z(g(\rho_Y), f(x')) + d_Z(f(\rho_X), g(\rho_Y)) 
    \\
    &\geq d_X(x, \rho_X) + d_Z(f(\rho_X), f(x')) 
    \\
    &\geq  
    d_X(x, x').
  \end{align}
  Thus, $f_2 \coloneqq q \circ \iota_X \colon X \to M_2$ is root-preserving isometric embedding.
  Similarly, one can check that $g_2 \coloneqq q \circ \iota_Y \colon Y \to M_2$ is root-preserving isometric embedding.
  By definition,
  it holds that 
  \begin{equation}
    d_{M_2}(f_2(x), g_2(y)) 
    \leq 
    d_Z(f(x), g(y)) + d_Z(f(\rho_X), g(\rho_Y)) 
    < 
    d_Z(f(x), g(y)) + \varepsilon.
  \end{equation}
  Thus, if we write $\FSMet{\tau}$ for the space-rooted metrization of $\tau$ associated with $\FEMet{\tau}$ in the sense of Lemma~\ref{lem: ER metrization defines SR metrization},
  then the stability of $\FEMet{\tau}$ implies that  
  \begin{align}
    d^{\FSMet{\tau}}_{M_2, \rho_{M_2}}(\tau_{f_2}(a_X), \tau_{g_2}(a_Y)) 
    &=
    d^{\FEMet{\tau}}_{M_2} \bigl( (\tau_{f_2}(a_X), \rho_{M_2}), (\tau_{g_2}(a_Y), \rho_{M_2}) \bigr)
    \\
    &=
    d^{\FEMet{\tau}}_{M_2} \bigl( \FE{\tau}_{f_2}(a_X, \rho_X), \FE{\tau}_{g_2}(a_Y, \rho_Y) \bigr)
    \\
    &\leq     
    d^{\FEMet{\tau}}_{M_1} \bigl( \FE{\tau}_{f_1}(a_X, \rho_X), \FE{\tau}_{g_1} (a_Y, \rho_Y) \bigr) 
    + 
    \Dist{\FEMet{\tau}}(\varepsilon)
    \\
    &< 
    \varepsilon + \Dist{\FEMet{\tau}}(\varepsilon).
  \end{align}
  Similarly, using Proposition~\ref{prop: local Hausdorff is stable}, we deduce that 
  \begin{equation}
    \lHausMet{M_2, \rho_{M_2}}\bigr( \Image{f_2}(X), \Image{g_2}(Y) \bigl)
    \leq
    \lHausMet{M_1}\bigr( \FE{\Image{f_1}}(X, \rho_X), \FE{\Image{g_1}}(Y, \rho_Y) \bigl)
    + 
    \varepsilon 
    <
    2\varepsilon.
  \end{equation}
  Thus, we obtain \eqref{pr eq: 1, coincidence of Type 1 and 2}.

  Now, we prove \eqref{thm eq: 1. coincidence of RF and RV}.
  Fix $\cX, \cY \in \rootedBCM(\tau)$.
  If $\cX = \cY$, the result is immediate since all terms vanish.
  Otherwise, the desired inequality follows by setting $\varepsilon = 2\, \RVMet^{\FEMet{\tau}}(\cX, \cY)$ in \eqref{pr eq: 1, coincidence of Type 1 and 2}.
\end{proof}

\begin{rem}
  When $\Dist{\FEMet{\tau}}$ is right-continuous,  
  instead of setting $\varepsilon = 2\, \RVMet^{\FEMet{\tau}}(\cX, \cY)$  
  in the final step of the above proof,  
  one can let $\varepsilon \downarrow \RVMet^{\FEMet{\tau}}(\cX, \cY)$,  
  and thereby easily verify that the inequality \eqref{thm eq: 1. coincidence of RF and RV} improves as follows:
  \begin{equation}  \label{rem eq: coincidence of Type 1 and 2}
    \RVMet^{\FEMet{\tau}}(\cX, \cY) 
    \leq
    \RFMet^{\FSMet{\tau}}(\cX, \cY) 
    \leq     
    2\RVMet^{\FEMet{\tau}}(\cX, \cY)    
    +
    \Dist{\FEMet{\tau}} \bigl( \RVMet^{\FEMet{\tau}}(\cX, \cY) \bigr),
    \quad
    \forall \cX, \cY \in \rootedBCM(\tau).
  \end{equation}
\end{rem}

The following results are immediate consequences of the above theorem.

\begin{cor} \label{cor: coincidence of topologies}
  If $\tau$ is stably metrizable,
  then the local GH-type topology with preserved roots and with non-preserved roots coincide.
\end{cor}

\begin{cor}
  If $\tau$ admits a stable metrization $\FMet{\tau}$ (resp.\ element-rooted metrization $\FEMet{\tau}$),
  then the completeness of $\RFMet^{\FMet{\tau}}$ and $\RVMet^{\FMet{\tau}}$ 
  (resp.\ $\RFMet^{\FEMet{\tau}}$ and $\RVMet^{\FEMet{\tau}}$) is equivalent.
\end{cor}

\begin{dfn}
  When the local GH-type topologies with preserved roots and with non-preserved roots coincide, 
  we simply refer to the resulting topology as the \emph{local GH-type topology}.
\end{dfn}

\begin{rem}
  When working with rooted compact metric spaces equipped with additional structures,
  one can define two Gromov--Hausdorff-type topologies analogously to Definitions~\ref{dfn: RF convergence} and \ref{dfn: RV convergence},
  replacing the Fell convergence of the underlying spaces with Hausdorff convergence.
  By a minor adaptation of the notion of stability in Definition~\ref{dfn: stability},
  one can show that the two topologies coincide, following the same lines as the proof of Theorem~\ref{thm: coincidence of RF and RV}.
  See Appendix~\ref{appendix: GH metrization in the compact case} for details. 
\end{rem}


\subsection{Topological properties} \label{sec: topological properties}

In this subsection,  
we investigate topological properties of the local GH-type topologies with preserved roots,  
such as Polishness and precompactness.  
Although analogous arguments apply to the case of non-preserved roots,  
the results presented here suffice for our purposes.  
Indeed, as we will see in Section~\ref{sec: Examples of functors},  
most examples of interest are stably metrizable,  
so that the two topologies coincide by Corollary~\ref{cor: coincidence of topologies}.  
One of the main results is Theorem~\ref{thm: main result},  
which shows that, under mild assumptions on $\tau$,  
the two topologies coincide and the resulting topology is Polish.  
This applies to all the examples considered in Section~\ref{sec: Examples of functors}.  
A precompactness criterion is also established in Theorem~\ref{thm: precompact in RF}.

For the following discussions, we impose additional continuity assumptions on $\tau$,  
beyond the embedding-continuity introduced in Definition~\ref{dfn: embedding continuity}.  
These assumptions correspond to those stated in \cite[Remark~2.14]{Khezeli_23_A_unified},  
which were introduced to ensure the completeness and separability of Gromov--Hausdorff-type metrics in the compact case.  

In what follows, for notational convenience,  
we adopt the convention that when a $\bcmAB$ space $X$ is isometrically embedded into another $\bcmAB$ space $Y$ via an isometric embedding $f$,  
we regard $\tau(X)$ as a subspace of $\tau(Y)$ through the topological embedding $\tau_f \colon \tau(X) \to \tau(Y)$.

\begin{assum} \label{assum: semicontinuity}
  Let $X_n$, $n \in \mathbb{N} \cup \{\infty\}$, be $\bcmAB$ spaces
  that are isometrically embedded into a common $\bcmAB$ space $Y$
  such that $X_n \to X_\infty$ in the Fell topology as subsets of $Y$.
  \begin{enumerate}[label=\textup{(\roman*)}, leftmargin=*]
    \item \label{assum item: upper semicontinuity}
      If a sequence $a_n \in \tau(X_n)$ converges to some $a \in \tau(Y)$, then $a$ belongs to $\tau(X_\infty)$.
    \item \label{assum item: lower semicontinuity}
      For every $a \in \tau(X_\infty)$, there exists a subsequence $(n_k)_{k \geq 1}$ and elements $a_k \in \tau(X_{n_k})$ 
      such that $a_k \to a$ in $\tau(Y)$. 
  \end{enumerate}
\end{assum}

\begin{rem} \label{rem: relation to PK convergence}
  As mentioned in \cite[Remark~2.15]{Khezeli_23_A_unified},
  the above conditions~\ref{assum item: upper semicontinuity} and~\ref{assum item: lower semicontinuity} 
  are equivalent to the conditions~\ref{lem item: 1. convergence in Hausdorff} and~\ref{lem item: 2. convergence in Hausdorff},
  which characterize Painlev\'{e}--Kuratowski convergence (cf.\ \cite[Definition~C.6]{Molchanov_17_Theory}).
  Thus, 
  Assumption~\ref{assum: semicontinuity} implies that $\tau(X_n) \to \tau(X_\infty)$ in the sense of Painlev\'{e}--Kuratowski convergence.
  In particular, if all the sets $\tau(X_n)$, $n \in \NN \cup \{\infty\}$, are closed in $\tau(Y)$,
  then this convergence is equivalent to convergence in the Fell topology (see \cite[Theorem~C.7]{Molchanov_17_Theory}).
\end{rem}

\begin{dfn}[Semicontinuity] \label{dfn: Semicontinuity} \leavevmode
  \begin{enumerate}[label=\textup{(\roman*)}, leftmargin=*]
    \item 
      We say that $\tau$ is \emph{upper semicontinuous} if it satisfies 
      Assumption~\ref{assum: semicontinuity}\ref{assum item: upper semicontinuity}.
    \item 
      We say that $\tau$ is \emph{lower semicontinuous} if it satisfies 
      Assumption~\ref{assum: semicontinuity}\ref{assum item: lower semicontinuity}.
  \end{enumerate}
\end{dfn}

\begin{rem}
  We use the term ``semicontinuity'' here
  because Assumptions~\ref{assum: semicontinuity}\ref{assum item: upper semicontinuity} and~\ref{assum item: lower semicontinuity}
  respectively resemble the definitions of upper and lower semicontinuity for set-valued maps
  (cf.\ \cite[Section~1.4]{Aubin_Frankowska_09_Set}).
  Indeed, formally, 
  these conditions correspond to the semicontinuity of the assignment $X \mapsto \tau(X)$.
\end{rem}

\begin{dfn}[Continuity] \label{dfn: continuity of structure}
  We say that $\tau$ is \emph{continuous}
  if it is embedding-continuous and both upper and lower semicontinuous.
\end{dfn}

In the following two theorems, we prove the completeness and separability of the metric $\RFMet^{\FSMet{\tau}}$.
These results generalize \cite[Theorem~2.12]{Khezeli_23_A_unified} to the boundedly-compact case.
For the following discussions, 
recall the notions of completeness and separability of functors from Definition~\ref{dfn: completeness and separability of functor}.

\begin{thm} [{Completeness}] \label{thm: completeness of RF}
  Assume that $\tau$ is embedding-continuous and upper semicontinuous.
  If the space-rooted metrization $\FSMet{\tau}$ is complete,
  then $\RFMet^{\FSMet{\tau}}$ is a complete metric on $\rootedBCM(\tau)$.
\end{thm}

\begin{proof}
  We follow the proof of \cite[Theorem~2.12]{Khezeli_23_A_unified}.
  Let $\cX_n = (X_n, \rho_{X_n}, a_{X_n}) \in \rootedBCM(\tau)$, $n \in \mathbb{N}$, 
  be a Cauchy sequence in $(\rootedBCM(\tau), \RFMet^{\FSMet{\tau}})$.
  It suffices to find a convergent subsequence.
  Thus, by relabelling if necessary,
  we may assume that $\RFMet^{\FSMet{\tau}}(\cX_n, \cX_{n+1}) < 2^{-n}$.
  By the same argument as in the proof of \cite[Theorem~2.12]{Khezeli_23_A_unified}
  (see also \cite[Lemma~2.42]{Khezeli_23_A_unified}),
  we deduce that there exist a rooted $\bcmAB$ space $(M, \rho_M)$ 
  and root-preserving isometric embeddings $f_n \colon X_n \to M$, $n \geq 1$,
  such that
  \begin{equation}
    \lHausMet{M, \rho_M}(f_n(X_n), f_{n+1}(X_{n+1})) 
    \vee
    d^{\FSMet{\tau}}_{M, \rho_M}(\tau_{f_n}(a_{X_n}), \tau_{f_{n+1}}(a_{X_{n+1}})) 
    < 2^{-n}.
  \end{equation}
  Thus, the sequence $(f_n(X_n))_{n \geq 1}$ is Cauchy with respect to $\lHausMet{M,\rho_M}$.
  By the completeness of the metric (see Theorem~\ref{thm: local Hausdorff top is compact}),
  there exists a closed subset $X \subseteq M$ such that $f_n(X_n) \to X$ in the Fell topology as subsets of $M$.
  We equip $X$ with the metric $d_X \coloneqq d_M|_{X \times X}$ and the root $\rho_X \coloneqq \rho_M$.
  Similarly, the sequence $(\tau_{f_n}(a_{X_n}))_{n \geq 1}$ is Cauchy with respect to $d^{\FSMet{\tau}}_{M, \rho_M}$.
  By the completeness of $d^{\FSMet{\tau}}_{M, \rho_M}$,
  this sequence converges to some $a_X \in \tau(M)$.
  By Assumption~\ref{assum: semicontinuity}\ref{assum item: upper semicontinuity},
  we have $a_X \in \tau(X)$.
  Therefore, by Theorem~\ref{thm: RF convergence},
  we conclude that $\cX_n$ converges to $(X, \rho_X, a_X)$,
  which completes the proof.
\end{proof}

\begin{thm}[Separability] \label{thm: separability of RF}
  If $\tau$ is separable, embedding-continuous, lower semicontinuous, and space-rooted metrizable, 
  then the local GH-type topology with preserved roots on $\rootedBCM(\tau)$ is separable.
\end{thm}

\begin{proof}
  By Theorem~\ref{thm: local GH is Polish},
  there exists a countable subset $\mathfrak{A} \subseteq \rootedBCM$ 
  which is dense in $\rootedBCM$ with respect to the local Gromov--Hausdorff topology.
  For each $(X, \rho_X) \in \mathfrak{A}$,
  choose a countable dense subset $D(X) \subseteq \tau(X)$,
  and define a countable subset $\mathfrak{D} \subseteq \rootedBCM(\tau)$ by
  \begin{equation}
    \mathfrak{D} 
    \coloneqq 
    \{
      (X, \rho_X, a_X) \mid (X, \rho_X) \in \mathfrak{A},\, a_X \in D(X)
    \}.
  \end{equation}
  Fix $\cX = (X, \rho_X, a_X) \in \rootedBCM(\tau)$.
  Choose $(X_n, \rho_{X_n}) \in \mathfrak{A}$ such that $(X_n, \rho_{X_n})$ converges to $(X, \rho_X)$ in the local Gromov--Hausdorff topology.
  Then there exist a rooted $\bcmAB$ space $(Y, \rho_Y)$ 
  and root--preserving isometric embeddings $f_n \colon X_n \to Y$ and $f \colon X \to Y$
  such that $f_n(X_n)$ converges to $f(X)$ in the Fell topology as subsets of $Y$.
  By Assumption~\ref{assum: semicontinuity}\ref{assum item: lower semicontinuity},
  we can find a subsequence $(n_k)_{k \geq 1}$ and elements $a_{X_{n_k}} \in \tau(X_{n_k})$ 
  such that $\tau_{f_{n_k}}(a_{X_{n_k}})$ converges to $\tau_f(a_X)$ in $\tau(Y)$.
  Since $D(X_{n_k})$ is dense in $\tau(X_{n_k})$,
  we may assume that $a_{X_{n_k}}$ is an element of $D(X_{n_k})$.
  By Theorem~\ref{thm: RF convergence},
  we obtain that $(X_{n_k}, \rho_{X_{n_k}}, a_{X_{n_k}}) \in \mathfrak{D} $ converges to $\cX$.
  Hence $\mathfrak{D} $ is dense in $\rootedBCM(\tau)$,
  which completes the proof.
\end{proof}

The following is a summary of the results so far.

\begin{cor} \label{cor: metrically Polish Type 1}
  If $\tau$ is continuous and its space-rooted metrization $\FSMet{\tau}$ is complete,
  then $\RFMet^{\FSMet{\tau}}$ is a complete separable metric on $\rootedBCM(\tau)$
  and the induced topology is the local GH-type topology with preserved roots.
\end{cor}

Even when each $\tau(X)$ is Polish as a topological space,
the associated metric may not be complete.
This occurs, for instance, in the case of the compact-convergence topology with variable domains,
together with the metric introduced in Section~\ref{sec: variable domains}.
In such situations,
under natural assumptions,
the Polishness of the local GH-type topology with preserved roots can still be verified.
To this end,
we first present a method to study topological properties of certain subsets of $\rootedBCM(\tau)$.
The subsets of interest are specified via subfunctors (recall these from Definition~\ref{dfn: subfunctors}).
The following result is a natural extension of \cite[Lemma~2.21]{Khezeli_23_A_unified} 
to the boundedly-compact case.

\begin{lem} \label{lem: open and closed in RF}
  Assume that $\tau$ is embedding-continuous and space-rooted metrizable.
  Let $\FS{\tau}_{\mathrm{sub}} \colon \rBCMcat \to \MTopcat$ be a subfunctor of $\FS{\tau}$.
  Define a subset of $\rootedBCM(\tau)$ by   
  \begin{equation}
    \rootedBCM(\FS{\tau}_{\mathrm{sub}})
    \coloneqq 
    \{(X, \rho_X, a_X) \in \rootedBCM(\tau) \mid a_X \in \FS{\tau}_{\mathrm{sub}}(X, \rho_X)\}.
  \end{equation}
  If $\FS{\tau}_{\mathrm{sub}}$ is pullback-stable and open (resp.\ closed),
  then the subset $\rootedBCM(\FS{\tau}_{\mathrm{sub}})$
  is open (resp.\ closed) in $\rootedBCM(\tau)$ with respect to the local GH-type topology with preserved roots.
\end{lem}

\begin{proof}
  Assume that the subfunctor $\FS{\tau}_{\mathrm{sub}}$ is pullback-stable and closed.
  Let $(X_n, \rho_{X_n}, a_{X_n})$, $n \geq 1$, be a sequence in $\rootedBCM(\FS{\tau}_{\mathrm{sub}})$
  converging to some $(X_\infty, \rho_{X_\infty}, a_{X_\infty})$ in $\rootedBCM(\tau)$.
  By Theorem~\ref{thm: RF convergence},
  there exist a rooted $\bcmAB$ space $(M, \rho_M)$
  and root-preserving isometric embeddings $f_n \colon X_n \to M$, $n \in \NN \cup \{\infty\}$,
  such that $f_n(X_n) \to f_\infty(X_\infty)$ in the Fell topology as subsets of $M$,
  and $\tau_{f_n}(a_{X_n}) \to \tau_{f_\infty}(a_{X_\infty})$ in $\tau(M)$.
  By \ref{dfn item: 1, subfunctors} and \ref{dfn item: 2, subfunctors},
  we have that $\tau_{f_n}(a_{X_n}) \in \FS{\tau}_{\mathrm{sub}}(M, \rho_M)$ for all $n$.
  Since $\FS{\tau}_{\mathrm{sub}}(M, \rho_M)$ is closed in $\tau(M)$,
  it follows that $\tau_{f_\infty}(a_{X_\infty}) \in \FS{\tau}_{\mathrm{sub}}(M, \rho_M)$.
  Then the pullback-stability of $\FS{\tau}_{\mathrm{sub}}$ implies that $a_{X_\infty} \in \FS{\tau}_{\mathrm{sub}}(X_\infty)$.
  Hence, $\rootedBCM(\FS{\tau}_{\mathrm{sub}})$ is closed in $\rootedBCM(\tau)$.

  Next, assume that the subfunctor $\FS{\tau}_{\mathrm{sub}}$ is pullback-stable and open.
  By pullback-stability, we can define another subfunctor $\hat{\tau}^\bullet_{\mathrm{sub}}$ of $\tau$ as follows:
  for each $(X, \rho_X) \in \ob(\rBCMcat)$, define
  \begin{equation}
    \hat{\tau}^\bullet_{\mathrm{sub}}(X, \rho_X) \coloneqq \tau(X) \setminus \FS{\tau}_{\mathrm{sub}}(X, \rho_X)
  \end{equation}
  (see Remark \ref{rem: define subfunctor} for the construction of the subfunctor).
  It is straightforward to check that $\hat{\tau}^\bullet_{\mathrm{sub}}$ is also pullback-stable and closed.
  Thus, by the previous argument, $\rootedBCM(\hat{\tau}^\bullet_{\mathrm{sub}})$ is closed in $\rootedBCM(\tau)$.
  Since $\rootedBCM(\FS{\tau}_{\mathrm{sub}}) = \rootedBCM(\tau) \setminus \rootedBCM(\hat{\tau}^\bullet_{\mathrm{sub}})$,
  we conclude that $\rootedBCM(\FS{\tau}_{\mathrm{sub}})$ is open.
  This completes the proof.
\end{proof}

Topological subfunctors, introduced in Definition~\ref{dfn: topological subfunctor}, provide topological embeddings of the corresponding Gromov--Hausdorff-type spaces, as described below.  
This enables us to study the topological properties of a given Gromov--Hausdorff-type space via a larger, more tractable space.

\begin{lem}  \label{lem: subfunctor embedding in RF}
  Let $\tilde{\tau}$ be another structure,
  and assume that $\tau$ is topologically embedded into $\tilde{\tau}$ via a topological embedding $\eta$.
  If $\tilde{\tau}$ is space-rooted metrizable and embedding-continuous,
  then so is $\tau$.
  Moreover, the following map is a topological embedding with respect to the local GH-type topologies with preserved roots:
  \begin{equation}  \label{lem eq: subfunctor embedding in RF}
    \rootedBCM(\tau) \ni (X, \rho_X, a_X) \mapsto (X, \rho_X, \eta_X(a_X)) \in \rootedBCM(\tilde{\tau}).
  \end{equation}
\end{lem}

\begin{proof}
  Let $\FSMet{\tilde{\tau}}$ be a space-rooted metrization of $\tilde{\tau}$.
  We define a space-rooted metrization $\FSMet{\tau}$ of $\tau$ as follows:
  for each $(X, \rho_X) \in \rBCMcat$,
  we equip $\tau(X)$ with a metric given by
  \begin{equation}
    d^{\FSMet{\tau}}_{X, \rho_X}( a_X, a'_X) 
    \coloneqq 
    d^{\FSMet{\tilde{\tau}}}_{X, \rho_X}( \eta_X(a_X), \eta_X(a'_X) ),
    \quad 
    a_X, a'_X \in \tau(X).
  \end{equation}
  Since $\eta_X \colon \tau(X) \to \tilde{\tau}(X)$ is a topological embedding,
  the function $d^{\FSMet{\tau}}_{X, \rho_X}$ is indeed a metric 
  on $\tau(X)$ that induces the given topology.
  Moreover, if $f \colon X \to Y$ is a root-preserving isometric embedding between rooted $\bcmAB$ spaces,
  then $\tau_f \colon \tau(X) \to \tau(Y)$ is distance-preserving with respect to the above-defined metrics,
  which can be verified by using \ref{dfn item: 2. NT}.
  Hence, $\FSMet{\tau}$ is a space-rooted metrization of $\tau$.

  To prove the embedding-continuity of $\tau$,
  fix $\bcmAB$ spaces $X$ and $Y$.
  Let $f_n \colon X \to Y$, $n \in \mathbb{N} \cup \{ \infty \}$, be isometric embeddings,
  and assume that $f_n \to f_\infty$ in the compact-convergence topology.
  Fix $a \in \tau(X)$.
  The embedding-continuity of $\tilde{\tau}$ implies that 
  $\tilde{\tau}_{f_n}( \eta_X(a) ) \to \tilde{\tau}_{f_\infty}( \eta_X(a) )$.
  By \ref{dfn item: 2. NT},
  this is equivalent to $\eta_Y(\tau_{f_n}(a)) \to \eta_Y(\tau_{f_\infty}(a))$.
  Since $\eta_X$ is a topological embedding,
  this is further equivalent to $\tau_{f_n}(a) \to \tau_{f_\infty}(a)$,
  which shows the embedding-continuity of $\tau$.

  One can also prove that the map given in~\eqref{lem eq: subfunctor embedding in RF} is a topological embedding 
  in a similar manner, by using Theorem~\ref{thm: RF convergence}.
  This completes the proof.
\end{proof}

We now provide a method to check the Polishness of $\rootedBCM(\tau)$ 
for a functor $\tau$ whose space-rooted metrization is not necessarily complete.
This method will be used in Section~\ref{sec: structure of continuous functions}.

\begin{dfn} [{Polishness with preserved roots}] \label{dfn: Polish in RF}
  We say that $\tau$ is \textit{Polish with preserved roots}  
  if there exists a topological embedding $\eta \colon \tau \Rightarrow \tilde{\tau}$ into another structure $\tilde{\tau}$,  
  and a sequence $(\FS{\tilde{\tau}}_k)_{k \geq 1}$ of subfunctors of $\FS{\tilde{\tau}}$  
  satisfying the following conditions.
  \begin{enumerate} [label = (P\arabic*), leftmargin = *]
    \item \label{dfn item: 1. Polish functor in RF}
      The functor $\tilde{\tau}$ is continuous and admits a complete space-rooted metrization. 
    \item \label{dfn item: 2. Polish functor in RF}
      Each subfunctor $\FS{\tilde{\tau}}_k$ is pullback-stable and open.
    \item \label{dfn item: 3. Polish functor in RF}
      For each $(X, \rho_X) \in \ob(\rBCMcat)$, $\eta_X(\tau(X)) = \bigcap_{k \geq 1} \FS{\tilde{\tau}}_k(X, \rho_X)$.
  \end{enumerate}
  We call $(\tilde{\tau}, \eta, (\FS{\tilde{\tau}}_{k})_{k \geq 0})$ a \textit{root-preserving Polish system} of $\tau$.
\end{dfn}

\begin{rem} \label{3. rem: decreasing Polish system}
  If $(\tilde{\tau}, \eta, (\FS{\tilde{\tau}}_{k})_{k \geq 0})$ is a root-preserving Polish system of $\tau$,
  then, by setting $\sigma_{k}(X) \coloneqq \bigcap_{l=1}^{k} \FS{\tilde{\tau}}_k(X, \rho_X)$
  for each $(X, \rho_X) \in \rBCMcat$,
  one can easily check that $(\tilde{\tau}, \eta, (\sigma_{k})_{k \geq 1})$ is another root-preserving Polish system of $\tau$.
  Thus,
  we can always assume that $(\FS{\tilde{\tau}}_{k}(X, \rho_X))_{k \geq 1}$ is a decreasing sequence.
\end{rem}

The intuition for the above conditions is as follows:
\ref{dfn item: 1. Polish functor in RF} say that the space $\rootedBCM(\tau)$ is topologically embedded into a Polish space $\rootedBCM(\tilde{\tau})$;
\ref{dfn item: 2. Polish functor in RF} and \ref{dfn item: 3. Polish functor in RF} then imply that 
$\rootedBCM(\tau)$ is a $G_\delta$ subset of $\rootedBCM(\tilde{\tau})$.
Hence, we deduce the following result.

\begin{thm} \label{thm: Polish with preserved roots}
  If $\tau$ is Polish with preserved roots,
  then the associated local GH-type topology with preserved roots on $\rootedBCM(\tau)$ is Polish.
\end{thm}

\begin{proof}
  We first note that by Lemma~\ref{lem: subfunctor embedding in RF} $\tau$ is embedding-continuous and space-rooted metrizable,
  and thus the local GH-type topology with preserved roots on $\rootedBCM(\tau)$ is well-defined.
  Let $(\tilde{\tau}, \eta, (\FS{\tilde{\tau}}_{k})_{k \geq 0})$ be a root-preserving Polish system of $\tau$.
  By Lemma~\ref{lem: subfunctor embedding in RF}, 
  the following map is a topological embedding:
  \begin{equation}
    \rootedBCM(\tau) \ni (X, \rho_X, a_X) \mapsto (X, \rho_X, \eta_X(a_X)) \in \rootedBCM(\tilde{\tau}).
  \end{equation}
  Condition \ref{dfn item: 3. Polish functor in RF} implies that 
  the image of the above map is $\mathfrak{A} \coloneqq \bigcap_{k \geq 1} \rootedBCM(\FS{\tilde{\tau}}_{k})$.
  Thus, it suffices to show that $\mathfrak{A}$ is Polish with the topology induced by $\rootedBCM(\tau)$.
  By Corollary \ref{cor: metrically Polish Type 1} and \ref{dfn item: 1. Polish functor in RF},
  $(\rootedBCM(\tilde{\tau}), d^{\FSMet{\tilde{\tau}}}_{\rootedBCM})$ is a complete separable metric space.
  Moreover, by Lemma~\ref{lem: open and closed in RF},
  each $\rootedBCM(\tilde{\tau}_k)$ is open in $\rootedBCM(\tilde{\tau})$.
  Hence, $\mathfrak{A}$ is a $G_\delta$ subset of $\rootedBCM(\tilde{\tau})$.
  Applying Alexandrov's theorem (see \cite[Theorem~2.2.1]{Srivastava_98_A_Course}),
  we deduce that $\mathfrak{A}$ is Polish, which completes the proof.
\end{proof}

From the above theorem and Corollary~\ref{cor: coincidence of topologies},
we obtain the following result.

\begin{thm} \label{thm: main result}
  If $\tau$ is stably metrizable and Polish with preserved roots,
  then the associated local GH-type topologies with preserved roots and with non-preserved roots coincide 
  and the resulting topology is Polish.
\end{thm}

The notion of subfunctor is also useful to describe a precompactness criterion.
The following is a natural extension of \cite[Theorem~2.16]{Khezeli_23_A_unified} to the boundedly-compact case.

\begin{thm} [Precompactness] \label{thm: precompact in RF}
  Assume that $\tau$ is space-rooted metrizable and embedding-continuous.
  Fix a non-empty subset 
  \begin{equation}
    \{
      \cX_\alpha = (X_\alpha, \rho_\alpha, a_\alpha) \in \rootedBCM(\tau) \mid \alpha \in \mathscr{A}
    \},
  \end{equation}
  where $\mathscr{A}$ denotes an index set.
  Consider the following statements.
  \begin{enumerate}[label = \textup{(\roman*)}, leftmargin=*]
    \item \label{thm item: 1. precompact in RF}
    The set $\{\cX_\alpha \mid \alpha \in \mathscr{A} \}$ is precompact in the local GH-type topology with preserved roots on $\rootedBCM(\tau)$.
    \item \label{thm item: 2. precompact in RF}
    The set $\{(X_\alpha, \rho_\alpha) \mid \alpha \in \mathscr{A} \}$ of $\rootedBCM$ is precompact in the local Gromov--Hausdorff topology,
    and there exists a precompact subfunctor $\tau'$ of $\FS{\tau}$ 
    such that $\{\cX_\alpha \mid \alpha \in \mathscr{A} \} \subseteq \rootedBCM(\tau')$.
  \end{enumerate}
  Then the implication \ref{thm item: 1. precompact in RF} $\Rightarrow$ \ref{thm item: 2. precompact in RF} holds.
  Moreover, if $\tau$ is upper semicontinuous, then the converse implication also holds.
\end{thm}

\begin{proof}
  The second assertion, namely the implication \ref{thm item: 2. precompact in RF} $\Rightarrow$ \ref{thm item: 1. precompact in RF},
  follows from essentially the same argument as in the proof of \cite[Theorem~2.16]{Khezeli_23_A_unified}.
  The first assertion can also be proved by a similar argument.
  Although some minor modifications are needed due to the fact that the underlying spaces are boundedly compact rather than compact,
  these adjustments are straightforward.
  For instance, in places where the proof of \cite[Theorem~2.16]{Khezeli_23_A_unified} relies on a precompactness criterion for the Gromov--Hausdorff topology,
  one may instead apply Theorem~\ref{thm: precompact in the local GH top}.
  Therefore, we omit the proof.
\end{proof}

\begin{rem}
  All the results above, except for the precompactness criterion,  
  readily extend to the case of non-preserved roots with minor modifications.  
  However, we have not yet found a useful characterization of precompactness in that case.  
  As noted at the beginning, though, this rarely causes any issues in practical applications.
\end{rem}

\section{Functorial operations and preservation of properties} \label{sec: Functorial operations and preservation of properties}

Recall composition and product of functors from Definitions~\ref{dfn: composition of functors} and \ref{dfn: product functor}.  
In this section, we explore these functorial operations  
and study properties of the resulting structures.  
We begin with product, since its associated arguments are simpler to handle,  
and then proceed to composition.  
We note that these operations are also discussed in \cite[Examples~2.6.3 and 2.6.6]{Khezeli_23_A_unified}  
in the context of metrization of GH-type topologies for compact metric spaces.


\subsection{Product of structures}  \label{sec: product structures}

In the framework presented in the preceding section,  
product functors (recall Definition~\ref{dfn: product functor}) allow for an easy treatment of multiple additional structures.  
In this subsection, we prove that the product inherits the properties of its constituent structures.  

Fix $N \in \mathbb{N} \cup \{\infty\}$.
Let $(\tau_n)_{n=1}^{N}$ be 
a sequence of structures.
Write $\tau \coloneqq \prod_{n=1}^{N} \tau_n$ for the product functor,
which is again a structure.
The following results are straightforward, so we omit the proofs.

\begin{prop} \label{prop: topological properties of product st}
  If each $\tau_n$ is embedding-continuous
  (resp.\ upper semicontinuous, lower semicontinuous),
  then so is  $\tau$.
\end{prop}

\begin{proof}
  This is straightforward.
\end{proof}

\begin{prop} \label{prop: topological subfunctor of product st}
  Let $(\tilde{\tau}_n)_{n=1}^{N}$
  be another sequence of functors such that 
  $\tau_k$ is a topological subfunctor of $\tilde{\tau}_n$ for each $n$. 
  Then $\tau$ is a topological subfunctor of $\prod_{n=1}^N \tilde{\tau}_n$.
\end{prop}

\begin{proof}
  For each $n$, let $\eta^n \colon \tau_n \Rightarrow \tilde{\tau}_n$ be a topological embedding.
  For each $X \in \ob(\BCMcat)$, we define $\eta_X \colon \tau(X) \to \tilde{\tau}(X)$ to be the product of $(\eta^n_X)_{n =1}^N$.
  One can easily verify that $\eta = \{\eta_X\}_{X \in \ob(\BCMcat)}$ defines a natural transformation from $\tau$ to $\tilde{\tau}$.
  Thus, the desired result follows.
\end{proof}

We then turn to the metrization of $\tau$.  
If each $\tau_n$ admits a metrization $\FMet{\tau}_n$,  
then $\tau$ admits a naturally associated metrization $\FMet{\tau}$ defined as follows:  
for each $X \in \ob(\BCMcat)$, if $N < \infty$, we equip $\tau(X)$ with the max product metric;  
otherwise, we equip $\tau(X)$ with the metric given by  
\begin{equation} \label{eq: infinite max product metirc}
  d^{\FMet{\tau}}_X \bigl( (a_n)_{n=1}^\infty, (b_n)_{n=1}^\infty \bigr)
  \coloneqq 
  \sum_{n=1}^{\infty} 2^{-n} \bigl( 1 \wedge d^{\FMet{\tau}_n}_X(a_n, b_n) \bigr).
\end{equation}
Similarly, if each $\tau_n$ admits a space-rooted (resp.\ element-rooted) metrization,  
then $\tau$ admits a naturally associated space-rooted (resp.\ element-rooted) metrization $\FSMet{\tau}$ (resp.\ $\FEMet{\tau}$).

\begin{prop}  \label{prop: completeness of product st}
  If each metrization $\FMet{\tau}_n$ of $\tau_n$ is complete, then so is $\FMet{\tau}$.  
  The same conclusion holds for space-rooted and element-rooted metrization.
\end{prop}

\begin{proof}
  The result is straightforward,  
  since the max product metric (or the metric of the form \eqref{eq: infinite max product metirc}) constructed from complete metrics is itself complete.
\end{proof}

The following two theorems are the main results of this subsection,  
showing that both stability and Polishness are inherited by $\tau$.

\begin{thm}  \label{thm: stability of product st}
  If each $\tau_n$ is stably metrizable,
  then so is  $\tau$.
\end{thm}

\begin{proof}
  Assume that each $\tau_n$ admits a stable element-rooted metrization $\FEMet{\tau}_n$.
  Let $\FEMet{\tau}$ be the naturally associated element-rooted metrization of $\tau$.
  One can easily verify that $\FEMet{\tau}$ is stable and its distortion is given as follows:
    \begin{equation}
    \Dist{\FEMet{\tau}}(\varepsilon)= 
    \begin{cases}
      \displaystyle \max_{1 \leq n \leq N} \Dist{\FEMet{\tau}_n}(\varepsilon),& \text{if}\ N < \infty,\\
      \displaystyle \sum_{n=1}^{\infty} 2^{-n} (1 \wedge \Dist{\FEMet{\tau}_n}(\varepsilon)),& \text{if}\ N = \infty.
    \end{cases}
  \end{equation}
\end{proof}

\begin{thm} \label{thm: Polishness of product st}
  If each $\tau_n$ is Polish with preserved roots, then so is $\tau$.
\end{thm}

\begin{proof}
  Assume that each $\tau_n$ is Polish with preserved roots.  
  For each $n$, let $(\tilde{\tau}_n, \eta_n, (\FS{\tilde{\tau}}_{n,k})_{k \geq 1})$ be a root-preserving Polish system of $\tau_n$.  
  Define $\tilde{\tau} \coloneqq \prod_{n=1}^N \tilde{\tau}_n$.
  By Propositions~\ref{prop: topological properties of product st} and \ref{prop: completeness of product st}, 
  $\tilde{\tau}$ is continuous and admits a complete space-rooted metrization.
  Moreover, by Proposition~\ref{prop: topological subfunctor of product st},
  $\tau$ is topologically embedded into $\tilde{\tau}$.
  Let $\eta$ be the associated topological embedding constructed in the proof of  Proposition~\ref{prop: topological subfunctor of product st}.
  For each $k \geq 1$,
  we define a subfunctor $\FS{\tilde{\tau}}_k$ of $\FS{\tilde{\tau}}$ by setting $\FS{\tilde{\tau}}_k \coloneqq \prod_{n=1}^N \FS{\tilde{\tau}}_{n,k}$.  
  We then deduce that $(\tilde{\tau}, \eta, (\FS{\tilde{\tau}}_k)_{k \geq 1})$ forms a root-preserving Polish system of $\tau$,  
  and hence $\tau$ is Polish with preserved roots.
\end{proof}


\subsection{Structures via composition with space transformations}  \label{sec: Composition of st}

When $\tau$ is a structure and $\STFunct \colon \BCMcat \to \BCMcat$ is a functor,
one can consider a new structure given by the composition $\tau \circ \STFunct$.
In this subsection, we study properties of $\tau \circ \STFunct$, such as Polishness, in terms of the properties of $\tau$ and $\STFunct$.
This allows us to reduce the verification of properties for a complex structure 
to the verification for its simpler components.

In order to inherit properties of $\tau$ to $\tau \circ \Psi$, 
we assume that $\Psi$ belongs to a class of space transformations, which we define below.
Note that, in the following definition, 
we say that $\STFunct$ is continuous if and only if the composition  
$\Gamma_{\BCMcat \to \MTopcat} \circ \STFunct$  
is continuous in the sense of Definition~\ref{dfn: continuity of structure}.

\begin{dfn} [Space transformation]
  We call a functor $\STFunct \colon \BCMcat \to \BCMcat$ a \emph{space transformation}
  if it is continuous and 
  there exists a natural transformation $\RootSystem \colon \Gamma_{\BCMcat \to \MTopcat} \Rightarrow \Gamma_{\BCMcat \to \MTopcat} \circ \STFunct$,
  that is, $\RootSystem$ is a collection of maps $\RootSystem_X \colon X \to \STFunct(X)$, where $X \in \ob(\BCMcat)$, such that 
  \begin{enumerate} [label = \textup{(RT\,\arabic*)}, leftmargin = *]
    \item \label{dfn item: 1. rooted}
      $\RootSystem_X \colon X \to \STFunct(X)$ is a topological embedding for each $X \in \ob(\BCMcat)$,
    \item \label{dfn item: 2. rooted}
      $\STFunct_f \circ \RootSystem_X = \RootSystem_Y \circ f$, for each $f \in \Hom_{\BCMcat}(X, Y)$.
  \end{enumerate}
  We call $\RootSystem$ a \emph{rooting system} of $\STFunct$.
  For each $\bcmAB$ space $X$, we write $d^\STFunct_X = d_{\STFunct(X)}$ for the metric on $\STFunct(X)$.
\end{dfn}

The continuity of $\Psi$ ensures that the continuity of $\tau$ is inherited by the composition $\tau \circ \Psi$,  
as we will see in Proposition~\ref{prop: continuity of composition} below.  
The rooting system specifies the root of $\Psi(X)$ for each rooted space $X$,  
and will be used to construct a rooted metrization of $\tau \circ \Psi$  
in Propositions~\ref{prop: composition is SR} and \ref{prop: composition is ER metrizable} below.

\begin{exm} \label{exm: 1. space transformation}
  A canonical example of space transformations is a functor 
  that maps each $\bcmAB$ space $X$ to the Cartesian product $X^k$, where $k$ is a fixed natural number.
  Another example is a functor $\STFunct$ that maps $X$ to $X \times \Xi$, 
  where $\Xi$ is a fixed $\bcmAB$ space.
  From this functor and the measure structure $\MeasFunct$ defined in Example~\ref{exm: measure functor},
  one obtains the space $\rootedBCM(\MeasFunct \circ \STFunct)$,
  consisting of (equivalence classes of) rooted $\bcmAB$ spaces equipped with marked measures, i.e., measures on $X \times \Xi$.
  More precise definitions of these functors are given in Example~\ref{exm: 2. space transformation} below.
\end{exm}

Henceforth, we fix a space transformation $\Psi$ with a rooting system $\RootSystem$ and a structure $\tau$.

\begin{prop} \label{prop: continuity of composition}
  If $\tau$ is embedding-continuous (resp.\ upper/lower semicontinuous),
  then so is $\tau \circ \Psi$.
  In particular, if $\tau$ is continuous, then so is $\tau \circ \STFunct$.
\end{prop}

\begin{proof}
  Assume that $\tau$ is embedding-continuous.
  Fix $\bcmAB$ spaces $X$ and $Y$,
  and let $f_n \colon X \to Y$, $n \in \NN \cup \{\infty\}$, be isometric embeddings such that $f_n \to f_\infty$ in the compact-convergence topology.
  By the continuity of $\STFunct$ and Remark~\ref{rem: strong continuity of functors},
  we have $\STFunct_{f_n} \to \STFunct_{f_\infty}$ in the compact-convergence topology.
  Then the continuity of $\tau$ implies that $(\tau \circ \STFunct)_{f_n}(a) \to (\tau \circ \STFunct)_{f_\infty}(a)$ for all $a \in \tau(\STFunct(X))$.
  Hence, $\tau \circ \STFunct$ is embedding-continuous.

  To prove the assertions regarding semicontinuity,
  let $X_n$, $n \in \mathbb{N} \cup \{\infty\}$, be $\bcmAB$ spaces 
  that are isometrically embedded into a common $\bcmAB$ space $Y$
  in such a way that $X_n \to X_\infty$ in the Fell topology as subsets of $Y$.
  Note that, since $\STFunct_f \colon \STFunct(X_n) \to \STFunct(Y)$ is distance-preserving 
  and the metric on $\STFunct(X_n)$ is complete,
  the set $\STFunct(X_n)$ is closed in $\STFunct(Y)$.
  Thus, by Remark~\ref{rem: relation to PK convergence},
  we deduce that $\STFunct(X_n) \to \STFunct(X_\infty)$ in the Fell topology as subsets of $\STFunct(Y)$.
  Now it is easy to see that 
  the upper (resp.\ lower) semicontinuity of $\tau$ implies the upper (resp.\ lower) semicontinuity of $\tau \circ \STFunct$.
  This completes the proof.
\end{proof}

\begin{prop} \label{prop: topological subfunctor of composition}
  Let $\tilde{\tau}$ be another structure.
  If $\tau$ is a topological subfunctor of $\tilde{\tau}$,
  then $\tau \circ \STFunct$ is a topological subfunctor of $\tilde{\tau} \circ \STFunct$.
\end{prop}

\begin{proof}
  Let $\eta \colon \tau \Rightarrow \tilde{\tau}$ be the associated topological embedding.
  For each $X \in \BCMcat$, define $\zeta_X \coloneqq \eta_{\STFunct(X)} \colon \tau(\STFunct(X)) \to \tilde{\tau}(\STFunct(X))$.
  Then one can verify that $\zeta$ defines a natural transformation from $\tau \circ \STFunct$ to $\tilde{\tau} \circ \STFunct$,
  and hence the desired result follows.
\end{proof}

We next study the metrization of $\tau \circ \STFunct$.

\begin{prop}
  If $\tau$ is metrizable, 
  then so is $\tau \circ \STFunct$.
\end{prop}

\begin{proof}
  If we write $\FMet{\tau}$ for a metrization of $\tau$,
  then the following diagram commutes.
  \begin{equation}
    \begin{tikzcd}
                                                                                       &                                                  & \Metcat \arrow[d, "\ForgetMetric"] \\
      \BCMcat \arrow[rru, bend left = 10, "\FMet{\tau} \circ \STFunct"] \arrow[r, "\STFunct"] & \BCMcat \arrow[ru, "\FMet{\tau}"] \arrow[r, "\tau"] & \MTopcat 
    \end{tikzcd}
  \end{equation}
  This shows that $\tau \circ \STFunct$ is metrizable and $\FMet{\tau} \circ \STFunct$ is its metrization.
\end{proof}

\begin{prop} \label{prop: composition is SR}
  If $\tau$ admits a space-rooted metrization $\FSMet{\tau}$,
  then $\tau \circ \STFunct$ admits a space-rooted metrization $\FSMet{(\tau \circ \STFunct)}$ given as follows:
  for each $(X, \rho_X) \in \ob(\rBCMcat)$,
  equip $\tau(\STFunct(X))$ with a metric given by
  \begin{equation} \label{prop eq: composition is SR}
    d^{\FSMet{(\tau \circ \STFunct)}}_{X, \rho_X} ( a, b )
    \coloneqq 
    d^{\FSMet{\tau}}_{\STFunct(X), \RootSystem_X(\rho_X)} (a , b).
  \end{equation}
  Moreover, if $\FSMet{\tau}$ is complete,
  then so is $\FSMet{(\tau \circ \STFunct)}$.
\end{prop}

\begin{proof}
  The second assertion regarding completeness is obvious.
  For the first assertion,
  by Remark~\ref{rem: construction of SR metrization},
  it is enough to show that, if $f \colon X \to Y$ is a root-preserving isometric embedding between rooted $\bcmAB$ space 
  $(X, \rho_X)$ and $(Y, \rho_Y)$,
  then $(\tau \circ \STFunct)_f$ is distance-preserving.
  From \ref{dfn item: 2. rooted}, we have $\STFunct_f(\RootSystem_X(\rho_X)) = \RootSystem_Y(\rho_Y)$,
  which means that the isometric embedding $\STFunct_f \colon \STFunct(X) \to \STFunct(Y)$ is also root-preserving
  with respect to the roots $\RootSystem_X(\rho_X)$ and $\RootSystem_Y(\rho_Y)$.
  Since $\FSMet{\tau}$ is a space-rooted metrization,
  it follows that $(\tau \circ \STFunct)_f = \tau_{\STFunct_f}$ is distance-preserving.
\end{proof}

\begin{prop}  \label{prop: composition is ER metrizable}
  If $\tau$ admits an element-rooted metrization $\FEMet{\tau}$,
  then $\tau \circ \STFunct$ admits an element-rooted metrization $\FEMet{(\tau \circ \STFunct)}$ given as follows:
  for each $X \in \ob(\BCMcat)$,
  equip $\FE{(\tau \circ \STFunct)}(X) = \tau(\STFunct(X)) \times X$ with a metric given by 
  \begin{equation} \label{prop eq: composition is ER metrizable}
    d^{\FEMet{(\tau \circ \STFunct)}}_X \bigl( (a, x), (b, y) \bigr)
    \coloneqq 
    d^{\FEMet{\tau}}_{\STFunct(X)} \bigl( (a, \RootSystem_X(x)), (b, \RootSystem_X(y)) \bigr)
    \vee
    d_X(x,y).
  \end{equation}
  Moreover, if $\FEMet{\tau}$ is complete,
  then so is $\FEMet{(\tau \circ \STFunct)}$.
\end{prop}

\begin{proof}
  We verify the conditions given in Remark \ref{rem: construction of ER metrization}.
  By using \ref{dfn item: 1. rooted},
  we deduce that the metric given in \eqref{prop eq: composition is ER metrizable} yields the product topology.
  Hence, Remark~\ref{rem: construction of ER metrization}\ref{rem item: 1. construction of ER metrization} is satisfied.
  Let $X$ and $Y$ be $\bcmAB$ spaces and $f \colon X \to Y$ be an isometric embedding.
  Using \ref{dfn item: 2. rooted} and that $\FE{\tau}_{\STFunct_f}$ is distance-preserving,
  we deduce that, for any $(a, x), (b, y) \in \tau(\STFunct(X)) \times X$,
  \begin{align} 
    d^{\FEMet{(\tau \circ \STFunct)}}_Y \bigl( \FE{(\tau \circ \STFunct)}_f(a, x), \FE{(\tau \circ \STFunct)}_f(b, y) \bigr)
    &=
    d^{\FEMet{\tau}}_{\STFunct(Y)} \bigl( (\tau_{\STFunct_f}(a), \RootSystem_Y \circ f(x)), (\tau_{\STFunct_f}(b), \RootSystem_Y \circ f(y)) \bigr)
    \\
    &=
    d^{\FEMet{\tau}}_{\STFunct(Y)} \bigl( (\tau_{\STFunct_f}(a), \STFunct_f \circ \RootSystem_X(x)), (\tau_{\STFunct_f}(b), \STFunct_f \circ \RootSystem_Y(y)) \bigr)
    \\
    &=
    d^{\FEMet{\tau}}_{\STFunct(X)} \bigl( (a, \RootSystem_X(x)), (b, \RootSystem_Y(y)) \bigr)
    \\
    &=
    d^{\FEMet{(\tau \circ \STFunct)}}_X \bigl( (a, x), (b, y) \bigr).
    \label{pr eq: 1. composition is ER}
  \end{align}
  This shows that $\FE{(\tau \circ \STFunct)}_f$ is distance-preserving, which verifies Remark~\ref{rem: construction of ER metrization}\ref{rem item: 2. construction of ER metrization}.

  We finally prove the last assertion of the result.
  Assume that $\FEMet{\tau}$ is complete.
  Fix a $\bcmAB$ space $X$ and a Cauchy sequence $(a_n, x_n)_{n \geq 1}$ in $\FE{(\tau \circ \STFunct)}(X)$.
  Then $(a_n, \RootSystem_X(x_n))_{n \geq 1}$ and $(x_n)_{n \geq 1}$ are Cauchy sequences 
  in $\tau(\STFunct(X)) \times \STFunct(X)$ and $X$, respectively.
  The completeness of $d_X$ implies that $x_n$ converges to some element $x \in X$.
  Similarly, the completeness of $\FEMet{\tau}$ implies that $(a_n, \RootSystem_X(x_n))$ converges to some element $(a, \psi) \in \tau(\STFunct(X)) \times \STFunct(X)$.
  In particular, $a_n \to a$ in $\tau(\STFunct(X))$.
  It follows that $(a_n, x_n) \to (a, x)$,
  which shows that $\FEMet{(\tau \circ \STFunct)}$ is complete.
\end{proof}

We say that $\Psi$ is stable
if and only if the composition $\Gamma_{\BCMcat \to \MTopcat} \circ \Psi$ is stable in the sense of Definition~\ref{dfn: stability}.  
In this case, we write $\Dist{\Psi}$ for the corresponding distortion.
Below, we verify that the stability of $\tau$ and $\Psi$ is inherited by $\tau \circ \Psi$.

\begin{thm} \label{thm: composition is stable}
  If $\Psi$ is stable and $\tau$ is stably metrizable,
  then $\tau \circ \Psi$ is also stably metrizable.
\end{thm}

\begin{proof}
  Let $\FEMet{\tau}$ be a stable element-rooted metrization of $\tau$.
  Fix $\bcmAB$ spaces $X, Y, M_1$, and $M_2$.
  Let $f_i \colon X \to M_i$ and $g_i \colon Y \to M_i$, $i = 1,2$, be isometric embeddings.
  Suppose that there exists $\varepsilon \in \RNp$ such that, for all $x \in X$ and $y \in Y$,
  \begin{equation} \label{pr eq: 1. composition is stable}
    d_{M_2}(f_2(x), g_2(y)) \leq d_{M_1}(f_1(x), g_1(y)) + \varepsilon.
  \end{equation}
  The stability of $\STFunct$ implies that, for all $\alpha \in \STFunct(X)$ and $\beta \in \STFunct(Y)$, 
  \begin{equation} \label{pr eq: 2. composition is stable}
    d_{\STFunct(M_2)}(\STFunct_{f_2}(\alpha), \STFunct_{g_2}(\beta)) 
    \leq 
    d_{\STFunct(M_1)}(\STFunct_{f_1}(\alpha), \STFunct_{g_1}(\beta)) + \Dist{\STFunct}(\varepsilon).
  \end{equation}
  It then follows from the stability of $\FEMet{\tau}$ that, 
  for all $(a, x) \in \tau(\STFunct(X)) \times X$ and $(b, y) \in \tau(\STFunct(Y)) \times Y$,
  \begin{align}
    &d^{\FEMet{\tau}}_{\STFunct(M_2)} \bigl( \FE{\tau}_{\STFunct_{f_2}}(a, \RootSystem_X(x)), \FE{\tau}_{\STFunct_{g_2}}(b, \RootSystem_Y(y)) \bigr)
    \\
    &\leq
    d^{\FEMet{\tau}}_{\STFunct(M_1)} \bigl( \FE{\tau}_{\STFunct_{f_1}}(a, \RootSystem_X(x)), \FE{\tau}_{\STFunct_{g_1}}(b, \RootSystem_Y(y)) \bigr)
    + 
    \Dist{\FEMet{\tau}}(\Dist{\STFunct}(\varepsilon)).
  \end{align}
  Following the argument in \eqref{pr eq: 1. composition is ER},
  we deduce that 
  \begin{align}
    d^{\FEMet{(\tau \circ \STFunct)}}_{M_2} 
    \bigl( \FE{(\tau \circ \STFunct)}_{f_2}(a, x), \FE{(\tau \circ \STFunct)}_{g_2}(b, y) \bigr)
    &=
    d_{M_2}^{\FEMet{(\tau \circ \STFunct)}}
    \bigl( (\tau_{\STFunct_{f_2}}(a), f_2(x)), (\tau_{\STFunct_{g_2}}(b), g_2(y)) \bigr)
    \\
    &=
    d^{\FEMet{\tau}}_{\STFunct(M_2)} 
    \bigl( (\tau_{\STFunct_{f_2}}(a), \RootSystem_{M_2}(f_2(x))), (\tau_{\STFunct_{g_2}}(b), \RootSystem_{M_2}(g_2(y))) \bigr)
    \\
    &=
    d^{\FEMet{\tau}}_{\STFunct(M_2)} 
    \bigl( (\tau_{\STFunct_{f_2}}(a), \STFunct_{f_2}(\RootSystem_X(x))), (\tau_{\STFunct_{g_2}}(b), \STFunct_{g_2}(\RootSystem_Y(y))) \bigr)
    \\
    &=
    d^{\FEMet{\tau}}_{\STFunct(M_2)} 
    \bigl( \FE{\tau}_{\STFunct_{f_2}}(a, \RootSystem_X(x)), \FE{\tau}_{\STFunct_{g_2}}(b, \RootSystem_Y(y)) \bigr),
  \end{align}
  and similarly
  \begin{equation}
    d^{\FEMet{(\tau \circ \STFunct)}}_{M_1} 
    \bigl( \FE{(\tau \circ \STFunct)}_{f_1}(a, x), \FE{(\tau \circ \STFunct)}_{g_1}(b, y) \bigr)
    =
    d^{\FEMet{\tau}}_{\STFunct(M_1)} \bigl( \FE{\tau}_{\STFunct_{f_1}}(a, \RootSystem_X(x)), \FE{\tau}_{\STFunct_{g_1}}(b, \RootSystem_Y(y)) \bigr).
  \end{equation}
  Thus, we deduce that 
  \begin{align}
    &d^{\FEMet{(\tau \circ \STFunct)}}_{M_2} 
    \bigl( \FE{(\tau \circ \STFunct)}_{f_2}(a, x), \FE{(\tau \circ \STFunct)}_{g_2}(b, y) \bigr)
    \\
    &\leq
    d^{\FEMet{(\tau \circ \STFunct)}}_{M_1} 
    \bigl( \FE{(\tau \circ \STFunct)}_{f_1}(a, x), \FE{(\tau \circ \STFunct)}_{g_1}(b, y) \bigr)
    + 
    \Dist{\FEMet{\tau}}(\Dist{\STFunct}(\varepsilon)).
  \end{align}
  This shows that $\FEMet{(\tau \circ \STFunct)}$ is stable with a distortion $\Dist{\FEMet{\tau}} \circ \Dist{\STFunct}$.
\end{proof}

Finally, we show that Polishness of $\tau$ is inherited by $\tau \circ \Psi$.

\begin{thm} \label{thm: composition is Polish}
  If $\tau$ is Polish with preserved roots, then so is $\tau \circ \Psi$.
\end{thm}

\begin{proof}
  To simplify notation, we write $\sigma \coloneqq \tau \circ \STFunct$.
  Let $(\tilde{\tau}, \eta, (\FS{\tilde{\tau}}_k)_{k \geq 1})$ be a root-preserving Polish system of $\tau$.
  Define $\tilde{\sigma} \coloneqq \tilde{\tau} \circ \STFunct$.
  By Proposition~\ref{prop: topological subfunctor of composition},
  $\sigma$ is a topological subfunctor of $\tilde{\sigma}$,
  and we write $\zeta$ for the associated topological embedding.
  For each $k \geq 1$, define a subfunctor $\FS{\tilde{\sigma}}_k$ of $\FS{\tilde{\sigma}}$ as follows:
  for each $(X, \rho_X) \in \rBCMcat$, define 
  \begin{equation}
    \FS{\tilde{\sigma}}_k(X, \rho_X) 
    \coloneqq 
    \FS{\tilde{\tau}}_k(\STFunct(X), \RootSystem_X(\rho_X)).
  \end{equation}
  We will verify that $(\tilde{\sigma}, \zeta, (\FS{\tilde{\sigma}}_k)_{k \geq 1})$ is a root-preserving Polish system of $\sigma$.

  Proposition~\ref{prop: composition is ER metrizable} implies that $\tilde{\sigma}$ admits a complete space-rooted metrization,
  and so \ref{dfn item: 1. Polish functor in RF} is satisfied.
  For each $(X, \rho_X) \in \ob(\BCMcat)$,
  $\FS{\tilde{\sigma}}_k(X)$ is open in $\FS{\tilde{\sigma}}(X)$.
  Let $f \colon X \to Y$ be a root-preserving isometric embedding between rooted $\bcmAB$ spaces $(X, \rho_X)$ and $(Y, \rho_Y)$.
  As discussed in the proof of Proposition~\ref{prop: composition is SR},
  the isometric embedding $\Psi_f \colon \Psi(X) \to \Psi(Y)$ is root-preserving with respect to the roots $\RootSystem_X(\rho_X)$ and $\RootSystem_Y(\rho_Y)$.
  Thus, the pullback-stability of $\FS{\tilde{\tau}}_k$ yields that 
  we have that 
  \begin{align}
    (\FS{\tilde{\sigma}}_f)^{-1}(\FS{\tilde{\sigma}}_k(Y, \rho_Y))
    &=
    \tau_{\Psi_f}^{-1}\bigl( \FS{\tilde{\tau}}_k(\STFunct(Y), \RootSystem_Y(\rho_Y)) \bigr) 
    =
    \FS{\tilde{\tau}}_k(\STFunct(X), \RootSystem_X(\rho_X))
    = 
    \FS{\tilde{\sigma}}_k(X, \rho_X),
  \end{align}
  which shows that $\FS{\tilde{\sigma}}_k$ is pullback-stable.
  Thus, \ref{dfn item: 2. Polish functor in RF} is satisfied.
  For any $(X, \rho_X) \in \ob(\rBCMcat)$, it holds that 
  \begin{align}
    \bigcap_{k \geq 1} \FS{\tilde{\sigma}}_k(X, \rho_X) 
    &= 
    \bigcap_{k \geq 1} \FS{\tilde{\tau}}_k(\STFunct(X), \RootSystem_X(\rho_X)) \\
    &=
    \eta_{\Psi(X)}(\tau(\Psi(X)))
    \quad \text{(by \ref{dfn item: 3. Polish functor in RF} of $(\tilde{\tau}, \eta, (\FS{\tilde{\tau}}_k)_{k \geq 1})$)}\\
    &=
    \zeta_X(\sigma(X)),
  \end{align}
  which proves \ref{dfn item: 3. Polish functor in RF}.
  Therefore, $\sigma$ is Polish with preserved roots.
\end{proof}

\begin{exm} \label{exm: 2. space transformation}
  Here, we provide useful space transformations, which are briefly introduced in Example~\ref{exm: 1. space transformation}.
  Fix $k \in \NN$ and a rooted $\bcmAB$ space $(\Xi, \rho_\Xi)$.
  We define a space transformation $\Psi$ as follows.
  \begin{itemize}
  \item 
    For each $X \in \ob(\BCMcat)$, 
    define $\Psi(X) \coloneqq X^k \times \Xi$ equipped with the max product metric. 
  \item 
    For each $f \in \Hom_{\BCMcat}(X, Y)$, 
    define $\Psi_f \coloneqq \underbrace{f \times \cdots \times f}_{k\ \text{times}} \times \id_\Xi$.
\end{itemize}
  It admits a rooting system $\RootSystem$ given by $\RootSystem_X(x) = (x, \dots, x, \rho_\Xi) \in X^k \times \Xi$.
  One can easily verify that $\Psi$ is continuous and stable with distortion $\Dist{\Psi}(\varepsilon) = \varepsilon$.
  Similarly, one can define a space transformation $\Psi$ given by setting $\Psi(X) \coloneqq X^k$ for each $X \in \ob(\BCMcat)$.
  Again, this space transformation is stable.
\end{exm}


\section{Examples of structures}  \label{sec: Examples of functors}

In this section, we introduce examples of structures.
All of them are turned out to be stably metrizable and Polish with preserved roots.
Thus, by Theorem~\ref{thm: main result}, 
if $\tau$ is any one of structures given below,
then the local GH-type topologies with preserved roots and with non-preserved roots on $\rootedBCM(\tau)$ coincide 
and the resulting topology is Polish.
We also provide precompactness criterion.
The author is not aware of any natural example that fails to satisfy our conditions, such as stability.

\subsection{Fixed structures}  \label{sec: Fixed structures}
\newcommand{\Fixedst}[1]{\tau_{#1}}

Here, we consider equipping metric spaces with objects from a fixed metric space.
While this may not be of much interest on its own,
it allows for a rich variety of additional structures to be considered,
through the operation of product and composition discussed in Section~\ref{sec: Functorial operations and preservation of properties} 
or in combination with other structures introduced below.

Fix a Polish space $\Xi$.
Define a structure $\tau = \Fixedst{\Xi}$ as follows.
\begin{itemize}
  \item 
    For each $X \in \ob(\BCMcat)$, 
    define $\tau(X) \coloneqq \Xi$. 
  \item 
    For each $f \in \Hom_{\BCMcat}(X, Y)$, 
    define $\tau_{f} \coloneqq \id_{\Xi}$.
\end{itemize}

\begin{lem}
  The structure $\Fixedst{\Xi}$ is continuous and separable.
\end{lem}

\begin{proof}
  This is straightforward.
\end{proof}

\begin{thm}
  The structure $\Fixedst{\Xi}$ is stably metrizable and Polish with preserved roots.
\end{thm}

\begin{proof}
  Fix a complete metric $d_\Xi$ on $\Xi$ inducing the given topology.
  Then $\tau$ admits a complete metrization $\FMet{\Fixedst{\Xi}}$ given by equipping, for each $X \in \ob(\BCMcat)$,
  $\tau(X) = \Xi$ with the metric $d_\Xi$.
  It is obvious that $\FMet{\Fixedst{\Xi}}$ is stable and complete.
  Thus, the result is deduced.
\end{proof}

The following result provides a precompactness criterion for $\rootedBCM(\Fixedst{\Xi})$.
Since it is proven easily by using Theorem~\ref{thm: precompact in RF},
we omit the proof (cf.\ the proof of Theorem~\ref{thm: precompact for the point functor} below).

\begin{thm} [Precompactness] \label{thm: Precompactness for fixed st}
  Let $\mathscr{A}$ be a non-empty index set.
  A subset $\{\cX_\alpha = (X_\alpha, \rho_\alpha, \xi_\alpha) \mid \alpha \in \mathscr{A} \}$ 
  of $\rootedBCM(\Fixedst{\Xi})$ is precompact if and only if the following conditions are satisfied.
  \begin{enumerate} [label = \textup{(\roman*)}, leftmargin = *]
    \item \label{thm item: the functor for a fixed space, spaces are precompact in local GH top}
      The subset $\{(X_\alpha, \rho_\alpha) \mid \alpha \in \mathscr{A} \}$ of $\rootedBCM$
      is precompact in the local Gromov--Hausdorff topology.
    \item \label{4. thm item2. : the functor for a fixed space,} 
      The set $\{\xi_\alpha \mid \alpha \in \mathscr{A} \}$ is precompact in $\Xi$.
  \end{enumerate}
\end{thm}


\subsection{Points} \label{sec: the functor for points}
\newcommand{\PointFunc}{\tau_{\mathrm{id}}}
\newcommand{\PointsFunct}[1]{\tau_{\mathrm{id}^{#1}}}

In this subsection, we introduce a structure,
which equip each spaces with additional points.
The resulting topology is useful for discussing convergence of glued/fused metric spaces 
(e.g.\ \cite[Section 4]{Berry_Broutin_Goldschmidt_12_The_continuum} and \cite[Section 8.3]{Croydon_18_Scaling}).

Define $\PointFunc \coloneqq \Gamma_{\BCMcat \to \MTopcat}$, that is,
\begin{itemize}
  \item 
    For each $X \in \BCMcat$, 
    define $\PointFunc(X) \coloneqq X$.
  \item 
    For each $f \in \Hom_{\BCMcat}(X, Y)$, 
    define $(\PointFunc)_{f} \coloneqq f$.
\end{itemize}
It admits a natural metrization $\FMet{\PointFunc}$ given by equipping each $\PointFunc(X) = X$ with the metric $d_X$.
The following result is straightforward and thus we omit the proof.

\begin{thm}  \label{thm: Polishness of point functor}
  The structure $\PointFunc$ is continuous and separable,
  and its metrization $\FMet{\PointFunc}$ is complete and stable with distortion $\Dist{\PointFunc}(\varepsilon) = \varepsilon$.
\end{thm}

\begin{thm} [Precompactness] \label{thm: precompact for the point functor}
  Let $\mathscr{A}$ be a non-empty index set.
  A subset $\{\cX_\alpha = (X_\alpha, \rho_\alpha, v_\alpha) \mid \alpha \in \mathscr{A} \}$ 
  of $\rootedBCM(\PointFunc)$ is precompact if and only if the following conditions are satisfied.
  \begin{enumerate} [label = \textup{(\roman*)}, leftmargin = *]
    \item \label{thm item: 1. precompact for the point functor}
      The subset $\{(X_\alpha, \rho_\alpha) \mid \alpha \in \mathscr{A} \}$ of $\rootedBCM$
      is precompact in the local Gromov--Hausdorff topology.
    \item \label{thm item: 2. precompact for the point functor} 
      For some $r>0$, it holds that 
      $v_\alpha \in X_\alpha|_{\rho_\alpha}^{(r)} = D_{X_\alpha}(\rho_\alpha, r)$ for all $\alpha \in \mathscr{A}$.
  \end{enumerate}
\end{thm}

\begin{proof}
  Assume that $\{ \cX_\alpha \mid \alpha \in \mathscr{A} \}$ is precompact.
  From Corollary \ref{cor: projection continuity in RF},
  we obtain \ref{thm item: 1. precompact for the point functor}.
  If \ref{thm item: 2. precompact for the point functor} is not satisfied,
  we can find an increasing $(r_{n})_{n \geq 1}$ with $r_{n} \uparrow \infty$ 
  and a sequence $(v_{\alpha_{n}})_{n \geq 1}$ with $v_{\alpha_{n}} \in X_{\alpha_{n}}$ 
  such that $v_{\alpha_{n}} \notin X_{\alpha_{n}}|_{\rho_{\alpha_n}}^{(r_{n})}$ for all $n$.
  If necessary,
  by choosing a subsequence,
  we may assume that $(X_{\alpha_{n}}, \rho_{\alpha_{n}}, v_{\alpha_{n}})$ 
  converges to some $(X, \rho, v) \in \rootedBCM(\PointFunc)$.
  By Theorem~\ref{thm: RF convergence},
  there exist a rooted $\bcmAB$ space $(Z, \rho_{Z})$
  and root-preserving isometric embeddings $f_{n} \colon X_{\alpha_{n}} \to Z$ and $f \colon  X \to Z$
  such that 
  $f_{n}(X_{\alpha_{n}}) \to f(X)$ in the Fell topology as subsets of $Z$ 
  and $f_{n} (v_{\alpha_{n}}) \to f(v)$ in $Z$.
  It is then the case that, for some $r>0$, $v_{\alpha_{n}} \in X_{\alpha_{n}}|_{\rho_{\alpha_n}}^{(r)}$ for all $n$,
  which is a contradiction.
  Therefore, we obtain \ref{thm item: 2. precompact for the point functor}.
  The converse assertion follows from Theorem~\ref{thm: precompact in RF} 
\end{proof}

By using product functors,
one can consider equipping metric spaces multiple points.
For $n \in \mathbb{N} \cup \{\infty\}$, we define $\PointsFunct{n}$ to be the $n$-product functor of $\PointFunc$.
Then $\rootedBCM(\PointsFunct{n})$ is the collection of 
equivalence classes of rooted $\bcmAB$ spaces equipped with additional $n$ points 
and the metric $\RFMet^{\PointsFunct{n}}$ induces a suitable Polish topology on $\rootedBCM(\PointsFunct{n})$.


\subsection{Subsets}  \label{sec: the functor for subsets}
\newcommand{\SubsetSt}{\tau_{\mathcal{C}}}
\newcommand{\cptSubsetSt}{\tau_{\mathcal{C}_c}}
\newcommand{\SetsFunct}[1]{\tau^{#1 \text{-}\mathrm{sts}}}

In \cite[Section 6.4]{Miermont_09_Tessellations},  
a Gromov--Hausdorff-type topology was introduced on  
a set of equivalence classes of measured compact metric spaces equipped with subsets.  
In this subsection, we provide a structure that gives a natural extension of that topology to non-compact underlying spaces.
For the discussions below, we recall several notations from Section~\ref{sec: the Fell topology}.

Define a functor $\SubsetSt$ as follows.  
\begin{itemize}
  \item 
    For each $X \in \ob(\BCMcat)$,  
    define $\SubsetSt(X) \coloneqq \Closed{X}$ equipped with the Fell topology.
  \item 
    For each $f \in \Hom_{\BCMcat}(X, Y)$,  
    define $(\SubsetSt)_{f} \coloneqq \Image{f}$.
\end{itemize}  
It admits a natural element-rooted metrization $\FEMet{\SubsetSt}$ given by equipping $\SubsetSt^\times(X) = \Closed{X} \times X$  
with the metric $\lHausMet{X}$.  

\begin{thm}  \label{thm: Polishness of subset st}
  The functor $\SubsetSt$ is continuous and separable,  
  and its metrization $\FEMet{\SubsetSt}$ is complete and stable with distortion $\Dist{\FEMet{\SubsetSt}}(\varepsilon) = \varepsilon \wedge 1$.  
\end{thm}

\begin{proof}
  The embedding-continuity is straightforward.  
  To prove the semicontinuity,  
  let $X_n$, $n \in \mathbb{N}\cup\{\infty\}$, be $\bcmAB$ spaces  
  that are embedded isometrically into a common $\bcmAB$ space $Y$  
  in such a way that $X_n \to X_\infty$ in the Fell topology as subsets of $Y$.  
  Assume that closed subsets $A_n \in \Closed{X_n}$ converge to some $A \in \Closed{Y}$  
  in the Fell topology.  
  By Theorem~\ref{thm: convergence in the Fell topology},  
  for each $x \in A$,  
  there exist elements $x_n \in A_n$ converging to $x$.  
  Since $X_n \to X_\infty$,  
  it follows that $x \in X_\infty$.  
  Hence, $A \in \Closed{X}$,  
  which verifies the upper semicontinuity.  
  Let $A$ be a finite subset of $X$.  
  Since $X_n \to X_\infty$,  
  we can construct $A_n \in \Closed{X_n}$ converging to $A$ in the Fell topology.  
  Since such subsets $A$ are dense in $\Closed{X_\infty}$,  
  we obtain the lower semicontinuity of $\SubsetSt$.  
  We deduce the completeness and stability of $\FEMet{\SubsetSt}$  
  from Propositions~\ref{prop: product local Hausdorff metric} and~\ref{prop: local Hausdorff is stable}, respectively.  
  This completes the proof.
\end{proof}

\begin{thm} [Precompactness]
  Let $\mathscr{A}$ be a non-empty index set.
  A subset $\{\cX_\alpha = (X_\alpha, \rho_\alpha, A_{\alpha}) \mid \alpha \in \mathscr{A} \}$  
  of $\rootedBCM(\SubsetSt)$ is precompact  
  if and only if  
  the subset $\{(X_\alpha, \rho_\alpha) \mid \alpha \in \mathscr{A} \}$ of $\rootedBCM$  
  is precompact in the local Gromov--Hausdorff topology.
\end{thm}

\begin{proof}
  Using Theorems~\ref{thm: local Hausdorff top is compact} and \ref{thm: RF convergence},  
  one can prove the desired result in the same manner as in the proof of Theorem~\ref{thm: precompact for the point functor}.
\end{proof}

Since the Fell topology is compact (see Theorem~\ref{thm: local Hausdorff top is compact}),  
the element-rooted metrization $\FEMet{\SubsetSt}$ naturally defines a space transformation $\Psi_{\mathcal{C}}$ as follows.  
\begin{itemize}
  \item 
    For each $X \in \ob(\BCMcat)$,  
    define $\SubsetSt(X) \coloneqq \Closed{X} \times X$ equipped with the metric $\lHausMet{X}$.
  \item 
    For each $f \in \Hom_{\BCMcat}(X, Y)$,  
    define $(\Psi_{\mathcal{C}})_{f} \coloneqq \FE{\Image{f}}$.
\end{itemize}  
Its rooting system $\RootSystem$ is given by $\RootSystem_X(x) \coloneqq (\{x\}, x)$.

\begin{rem}
  To treat spaces equipped with compact subsets,  
  one may use the structure $\cptSubsetSt$ defined as follows.  
  \begin{itemize}
    \item 
      For each $X \in \ob(\BCMcat)$,  
      set $\cptSubsetSt(X) \coloneqq \Compact{X}$ endowed with the Hausdorff topology.
    \item 
      For each $f \in \Hom_{\BCMcat}(X, Y)$,  
      set $(\cptSubsetSt)_{f} \coloneqq \Image{f}$.
  \end{itemize}  
  An analogue of Theorem~\ref{thm: Polishness of subset st} holds for $\cptSubsetSt$.
\end{rem}


\subsection{Measures} \label{sec: the functor for measures}

The local Gromov--Hausdorff-vague topology (recall it from Section~\ref{sec: introduction to GH-type metrics}) 
is commonly used for studying random measured spaces.
In this subsection, 
we recover this topology by the measure structure $\MeasFunct$ introduced in Example~\ref{exm: measure functor}.
For the discussions below, we recall several notations from Section~\ref{sec: the Fell topology}.

Recall that the structure $\tau = \MeasFunct$ is defined as follows.
\begin{itemize}
    \item 
      For each $X \in \ob(\BCMcat)$, 
      define $\tau(X) \coloneqq \Meas{X}$ equipped with the vague topology. 
    \item 
      For each $f \in \Hom_{\BCMcat}(X, Y)$,
      define $\tau_f \coloneqq f_*$, i.e., the pushforward map given by $f$.
  \end{itemize}
It admits a natural element-rooted metrization $\FEMet{\MeasFunct}$ given by equipping  $\Meas{X} \times X$ with the metric $\Vague{X}$,
which is defined in \eqref{eq: vague product metric}.

\begin{thm}  \label{thm: Polishness of measure functor}
  The functor $\MeasFunct$ is continuous and separable,
  and its metrization $\FEMet{\MeasFunct}$ is complete and stable with distortion $\Dist{\FEMet{\MeasFunct}}(\varepsilon) = 1 \wedge \varepsilon$.
\end{thm}

\begin{proof}
  The separability of $\MeasFunct$ follows from the separability of the vague topology; see Theorem~\ref{thm: vague metric}.
  The completeness and stability of $\FEMet{\MeasFunct}$ follows from Proposition~\ref{prop: vauge product metric}.
  The embedding-continuity follows immediately from the definition of the vague topology
  and the dominated convergence theorem.
  To prove the semicontinuity,
  let $X_n$, $n \in \mathbb{N}\cup\{\infty\}$, be $\bcmAB$ spaces
  that are embedded isometrically into a common $\bcmAB$ space $Y$
  in such a way that $X_n \to X_\infty$ in the Fell topology as subsets of $Y$.
  Assume that measures $\mu_{n} \in \Meas{X_n}$ converging vaguely to some $\mu \in \Meas{Y}$ vaguely.
  Fix $x \in Y \setminus X_\infty$.
  Since $Y \setminus X_\infty$ is open,
  there exists $\varepsilon >0$ such that $B_{Y}(x, \varepsilon) \cap X = \emptyset$. 
  The Fell convergence $X_n \to X_\infty$ ensures that 
  $X_n \cap B_{Y}(x, \varepsilon) = \emptyset$ for all sufficiently large $n$.
  Therefore, by \cite[Lemma~4.1]{Kallenberg_17_Random},
  we obtain 
  \begin{equation}
    \mu(B_Y(x, \varepsilon))
    \leq 
    \liminf_{n \to \infty} \mu_n(B_Y(x, \varepsilon)) 
    = 0,
  \end{equation}
  which implies $\mu$ is supported on $X_\infty$, i.e., $\mu \in \Meas{X_\infty}$.
  This proves the upper semicontinuity of $\MeasFunct$.
  It remains to verify the lower semicontinuity.
  Let $\mu$ be a discrete measure with finitely many atoms in $X_\infty$.
  By approximating the atoms by elements in $X_n$,
  it is not difficult to construct a finite measure $\mu_{n}$ with atoms in $X_n$
  such that $\mu_{n}$ converges to $\mu$ vaguely.
  Since such measures $\mu$ are dense in $\Meas{X_\infty}$,
  we obtain the lower semicontinuity of $\MeasFunct$.
\end{proof}

By Theorems~\ref{thm: main result} and \ref{thm: Polishness of measure functor},
the local GH-type topologies with preserved roots and with non-preserved roots on $\rootedBCM(\MeasFunct)$ coincide,
and the resulting topology is Polish.
Below, we verify that this topology coincides with the local Gromov--Hausdorff-vague topology 
(recall it from Section~\ref{sec: introduction to GH-type metrics}).

\begin{prop}  \label{prop: recovering the local GHV top}
  For each $n \in \mathbb{N} \cup \{ \infty \}$,
  let $\cX_n = (X_n, \rho_n, \mu_n)$ be an element of $\rootedBCM(\MeasFunct)$.
  The following statements are equivalent.
  \begin{enumerate} [label = \textup{(\roman*)}, leftmargin = *]
    \item \label{prop item: 1. recovering the local GHV top}
      The elements $\cX_n$ converge $\cX_\infty$ in the local GH-type topology,
    \item \label{prop item: 2. recovering the local GHV top}
      There exist a rooted $\bcmAB$ space $(M, \rho_M)$
      and root-preserving isometric embeddings $f_n \colon X_n \to M$, $n \in \NN \cup \{\infty\}$,
      such that $f_n(X_n) \to f_\infty(X_\infty)$ in the Fell topology as subsets of $M$,
      and $(f_n)_*(\mu_n) \to (f_\infty)_*(\mu_\infty)$ vaguely as measures on $M$.
    \item \label{prop item: 3. recovering the local GHV top}
      There exist a $\bcmAB$ space $M$
      and isometric embeddings $f_n \colon X_n \to M$, $n \in \NN \cup \{\infty\}$,
      such that $f_n(X_n) \to f_\infty(X_\infty)$ in the Fell topology as subsets of $M$,
      $f_n(\rho_n) \to f_\infty(\rho_\infty)$ in $M$,
      and $(f_n)_*(\mu_n) \to (f_\infty)_*(\mu_\infty)$ vaguely as measures on $M$.
    \item \label{prop item: 4. recovering the local GHV top}
      The elements $\cX_n$ converge $\cX_\infty$ in the local Gromov--Hausdorff-vague topology.
  \end{enumerate}
\end{prop}

\begin{proof} 
  The equivalence of \ref{prop item: 1. recovering the local GHV top}, \ref{prop item: 2. recovering the local GHV top}, and \ref{prop item: 3. recovering the local GHV top}
  is a consequence of Theorems~\ref{thm: RF convergence} and \ref{thm: RV convergence}.
  Following the proof of \cite[Proposition~5.9]{Athreya_Lohr_Winter_16_The_gap}, one can verify that \ref{prop item: 3. recovering the local GHV top} and \ref{prop item: 4. recovering the local GHV top} are equivalent.  
  (Indeed, the only difference is that the authors consider the Fell convergence of the supports of $\mu_n$, 
  rather than the underlying spaces $X_n$.)
  This completes the proof.
\end{proof}

\begin{thm} [Precompactness] \label{thm: precompactness in the local GHV top}
  Let $\mathscr{A}$ denote a non-empty index set.
  A subset $\{ \cX_\alpha = (X_\alpha, \rho_\alpha, \mu_\alpha) \mid \alpha \in \mathscr{A}\}$ 
  of $\rootedBCM(\MeasFunct)$ is precompact in the local Gromov--Hausdorff-vague topology if and only if 
  the following conditions are satisfied.
  \begin{enumerate} [label = \textup{(\roman*)}, leftmargin = *]
    \item \label{thm item: 1. precompactness in the local GHV}
      The set $\{(X_\alpha, \rho_\alpha) \mid \alpha \in \mathscr{A}\}$ 
      is precompact in the local Gromov--Hausdorff topology.
    \item \label{thm item: 2. precompactness in the local GHV top}
      For every $r>0$, it holds that $\sup_{\alpha \in \mathscr{A}} \mu_\alpha(X_\alpha|_{\rho_\alpha}^{(r)}) < \infty$.
  \end{enumerate}
\end{thm}

\begin{proof}
  If $\{\cX_\alpha \mid \alpha \in \mathscr{A}\}$ satisfies \ref{thm item: 1. precompactness in the local GHV} 
  and \ref{thm item: 2. precompactness in the local GHV top},
  then by Theorem~\ref{thm: precompact in the local GH top},
  \cite[Theorem 2.6]{Abraham_Delmas_Hoscheit_13_A_note} and \cite[Theorem 3.28]{Khezeli_20_Metrization},
  $\{\cX_\alpha \mid \alpha \in \mathscr{A} \}$ is precompact in the local Gromov--Hausdorff-vague topology.
  The converse direction is proved by contradiction, similarly to the proof of Theorem~\ref{thm: precompact for the point functor}.
\end{proof}

\begin{rem}
  To consider spaces equipped with finite Borel measures,  
  one may use the structure $\finMeasFunct$ defined as follows.  
  \begin{itemize}
    \item 
      For each $X \in \ob(\BCMcat)$,  
      set $\finMeasFunct(X) \coloneqq \finMeas{X}$ endowed with the weak topology.
    \item 
      For each $f \in \Hom_{\BCMcat}(X, Y)$,  
      set $(\finMeasFunct)_{f} \coloneqq f_*$.
  \end{itemize}  
  An analogue of Theorem~\ref{thm: Polishness of measure functor} holds for $\finMeasFunct$.
\end{rem}


\subsection{Cadlag functions}  \label{sec: structure for cadlag curves}
\newcommand{\SkorohodSt}{\tau_D}

Stochastic processes with cadlag paths are random cadlag functions,
and are fundamental objects of interest in probability theory.
In this subsection, we define a structure for cadlag functions.
Combined with the structure introduced in Section~\ref{sec: Laws of structures} below,
this structure provides a suitable topological framework for studying stochastic processes defined on different spaces.

We first recall the usual $J_1$-Skorohod topology and its complete metrization.  
For details, see \cite[Chapter~3]{Billingsley_99_Convergence} and \cite{Whitt_80_Some}.  
Given a metric space $S$ and an interval $I \subseteq \RNp$,  
we denote by $D(I, S)$ the set of cadlag functions $F \colon I \to S$.  
For each $t \in I$, we write $F(t-)$ for the left-hand limit of $F$ at~$t$.
For $t \in (0,\infty)$, the complete Skorohod metric on $D([0,t], S)$ is defined as follows:  
for each $F, G \in D([0,t], S)$,
\begin{equation}
  d^{J_1,t}_S(F, G) 
  \coloneqq 
  \inf_{\lambda \in \Lambda_t} 
  \left\{ 
    \sup_{0 \leq s_1 < s_2 \leq t} \left|\log \frac{\lambda(s_2) - \lambda(s_1)}{s_2 - s_1} \right| 
    \vee \sup_{s \in [0,t]} d_S\bigl( F(s), G(\lambda(s)) \bigr) 
  \right\}.
\end{equation}
where $\Lambda_t$ denotes the set of all increasing continuous bijections $\lambda \colon [0,t] \to [0,t]$.
For the unbounded interval $I = \RNp$, the complete Skorohod metric on $D(\RNp, S)$ is defined as follows:  
for each $F, G \in D(\RNp, S)$,
\begin{equation} \label{eq: Skorohod metric for unbounded domain}
  d^{J_1}_S(F, G) 
  \coloneqq 
  \int_0^\infty e^{-t} \left( 1 \wedge d^{J_1, t}_S\bigl(F|_{[0,t]}, G|_{[0,t]}\bigr) \right)\, dt.
\end{equation}
The topology on $D(I, S)$, where $I = [0,t]$ or $I = \RNp$, induced by the above metric is called the $J_1$-Skorohod topology.
Moreover, if $d_S$ is complete, the complete Skorohod metric is indeed complete.

Fix a Polish structure $\sigma$ that is stably metrizable and Polish with preserved roots,
and fix an interval $I \subseteq \RNp$ with the form $I = [0,t_\infty]$ for some fixed $t_\infty \in [0,\infty]$,
where we interpret $I = [0, \infty]$ as a shorthand for $\RNp = [0, \infty)$ when $t_\infty = \infty$.
We define a structure $\tau = \SkorohodSt(I, \sigma)$ as follows.
\begin{itemize}
  \item 
    For $X \in \ob(\BCMcat)$, 
    define $\tau(X) \coloneqq D(I, \sigma(X))$ equipped with the $J_1$-Skorohod topology. 
  \item 
    For each $f \in \Hom_{\BCMcat}(X, Y)$, 
    define $\tau_f(F) \coloneqq \sigma_f \circ F$ for $F \in D(I, \sigma(X))$.
\end{itemize}

\begin{prop} \label{prop: continuity of Skorohod st}
  The structure $\SkorohodSt(I, \sigma)$ is embedding-continuous.
  Moreover, if $\sigma$ is upper (resp.\ lower) semicontinuous, then so is $\SkorohodSt(I, \sigma)$.
\end{prop}

\begin{proof}
  Since $\sigma$ is embedding-continuous by Lemma~\ref{lem: subfunctor embedding in RF},
  one can readily verify that $\SkorohodSt(I, \sigma)$ is also embedding-continuous.
  To prove the semicontinuity,
  let $X_n$, $n \in \mathbb{N}\cup\{\infty\}$, be $\bcmAB$ spaces
  that are embedded isometrically into a common $\bcmAB$ space $Y$
  in such a way that $X_n \to X_\infty$ in the Fell topology as subsets of $Y$.
  Assume that $\sigma$ is upper semicontinuous.
  Let $F_{n} \in D(I, \sigma(X_n))$ be 
  such that $F_{n} \to F$ in the usual $J_{1}$-Skorohod topology
  for some $F \in D(I, \sigma(Y))$.
  Then,
  for each $t \in I$,
  there exists $t_{n} \geq 0$
  such that $F_{n}(t_{n}) \to F(t)$ in $\sigma(Y)$.
  It follows from the upper semicontinuity of $\sigma$ that $F(t) \in \sigma(X_\infty)$.
  Therefore, we obtain that $F \in D(I, \sigma(X_\infty))$,
  which implies the upper semicontinuity of $\SkorohodSt(I, \sigma)$.

  Next, assume that $\sigma$ is lower semicontinuous.
  We note that every function in $D(I, X_\infty)$
  is approximated by a sequence of step functions, where a step function is 
  a function that can be written in the following form:
  \begin{equation}  \label{pr eq: 1. continuity of the Skorohod functor}
    F(t) 
    =
    \begin{cases}
      a_{k} & t \in [t_{k-1}, t_{k})\\
      a_{m+1} & t \in [t_{m}, t_\infty]
    \end{cases}
  \end{equation}
  for some $a_{k} \in \sigma(X_\infty)$, $k=1,2,\ldots,m+1$ 
  and $0=t_{0} < t_{1} <t_{2} < \cdots < t_{m}<t_\infty$.
  Using the lower semicontinuity of $\sigma$,
  for every step function $F$ in $D(I, \sigma(X_\infty))$,
  one can construct, for some subsequence $(n_k)_{k \geq 1}$, step functions in $D(I, \sigma(X_{n_k}))$ 
  converging to $F$.
  Therefore, $\SkorohodSt(I, \sigma)$ is lower semicontinuous.
\end{proof}

\begin{thm} \label{thm: Polishness of Skorohod structure}
  The structure $\SkorohodSt(I, \sigma)$ is stably metrizable and Polish with preserved roots.
\end{thm}

\begin{proof} 
  We first prove that $\SkorohodSt(I, \sigma)$ is stably metrizable.
  Let $\FEMet{\sigma}$ be an element-rooted metrization of $\sigma$ that is complete and stable.
  To construct an element-rooted metrization of $\tau$,
  we introduce, for each $X \in \ob(\BCMcat)$, 
  a topological embedding $\iota^D_X \colon D(I, \sigma(X)) \times X \to D(I, \sigma(X) \times X)$ defined by
  \begin{equation}
    \iota^D_X(F, x)(t) \coloneqq (F(t), x).
  \end{equation}
  We then define an element-rooted metrization $\FEMet{\tau}$ of $\tau$ by equipping,  
  for each $X \in \ob(\BCMcat)$, the space $\FE{\tau}(X) = D(I, \sigma(X)) \times X$ with the metric
  \begin{equation} \label{eq: ER metrization of Skorohod functor}
    d^{\FEMet{\tau}}_X \bigl((F, x), (G, y)\bigr)
    \coloneqq 
    d^{J_1}_{\sigma(X) \times X} \bigl( \iota^D_X(F, x), \iota^D_X(G, y) \bigr),
  \end{equation}
  where $d^{J_1}_{\sigma(X) \times X}$ is the complete Skorohod metric associated with the metric $d^{\FEMet{\sigma}}_{X}$ on $\sigma(X) \times X$.
  Then, by the definition of the complete Skorohod metric,
  we obtain the stability of $\FEMet{\tau}$.

  We next prove that $\SkorohodSt(I, \sigma)$ is Polish with preserved roots.
  Let $\mathfrak{P} = (\tilde{\sigma}, \eta, (\FS{\tilde{\sigma}}_k)_{k \geq 1})$ be a root-preserving Polish system of $\sigma$.
  Define $\tilde{\tau} \coloneqq \SkorohodSt(\tilde{\sigma})$.
  Then $\tau$ is a topological subfunctor of $\tilde{\tau}$.
  The associated topological embedding $\zeta \colon \tau \Rightarrow \tilde{\tau}$ is given as follows:
  for each $X \in \ob(\BCMcat)$, 
  we define $\zeta_X \colon \tau(X) \to \tilde{\tau}(X)$ by $\zeta_X(F) \coloneqq \eta_X \circ F$.  
  For each $k \geq 1$, we define a subfunctor $\FS{\tilde{\tau}}_k$ of $\FS{\tilde{\tau}}$ as follows.
  For each $(X, \rho_X) \in \ob(\BCMcat)$, define $\FS{\tilde{\tau}}_k(X, \rho_X)$ to be the set of $F \in \tilde{\tau}(X)$
  satisfying the following condition: there exists $\varepsilon_F \in (0, 1/k)$ such that
  \begin{equation}
    \{(F(t), x), (F(t-), x)\} \subseteq \FS{\tilde{\sigma}}_k(X, \rho_X)\ \text{for all}\ t \in [0, (k-\varepsilon_F) \wedge t_\infty].
  \end{equation}
  We will prove that $\mathfrak{Q} \coloneqq (\tilde{\tau}, \zeta, (\FS{\tilde{\tau}}_k)_{k \geq 1})$ is a Polish system of $\tau$.

  By Proposition~\ref{prop: continuity of Skorohod st} and the separability of the $J_1$-Skorohod topology,
  $\tilde{\tau}$ is continuous and separable.
  Let $\FSMet{\tilde{\sigma}}$ be a complete space-rooted metrization of $\tilde{\sigma}$.
  We define a space-rooted metrization of $\tilde{\tau}$ as follows:
  for each $(X, \rho_X) \in \ob(\rBCMcat)$, we equip $\tilde{\tau}(X)$ with the metric given by 
  \begin{equation}
    d^{\FSMet{\tilde{\tau}}}_{X, \rho_X} (F, G) 
    \coloneqq  
    d^{J_1}_{\tilde{\sigma}(X), \rho_X} (F, G),
  \end{equation}
  where $d^{J_1}_{\tilde{\sigma}(X), \rho_X}$ denotes the complete Skorohod metric 
  associated with the metric $d^{\FSMet{\tilde{\sigma}}}_{X, \rho_X}$ on $\tilde{\sigma}(X)$.
  The completeness of the complete Skorohod metric implies that $\FSMet{\tilde{\tau}}$ is complete.
  Hence, we obtain \ref{dfn item: 1. Polish functor in RF}.

  For each $(X, \rho_X) \in \ob(\rBCMcat)$ and $k \geq 1$,
  since $\FS{\tilde{\sigma}}_k(X, \rho_X)$ is open in $\tilde{\sigma}(X)$,  
  we deduce that $\FS{\tilde{\tau}}_k(X, \rho_X)$ is open in $\tilde{\tau}(X) = D(I, \tilde{\sigma}(X))$.
  Moreover, using the pullback-stability of $\FS{\tilde{\sigma}}_k$,  
  one can show that $\FS{\tilde{\tau}}_k$ is also pullback-stable.  
  Thus, \ref{dfn item: 2. Polish functor in RF} is satisfied.  
  Finally, that $\mathfrak{Q}$ satisfies \ref{dfn item: 3. Polish functor in RF} follows from the corresponding property of $\mathfrak{P}$.  
  Therefore, $\mathfrak{Q}$ is a root-preserving Polish system of $\tau$.
\end{proof}

\begin{exm} \label{exm: simple Skorohod st}
  As a simple example, consider the case where $\sigma = \PointFunc$, as defined in Section~\ref{sec: the functor for points}. 
  By Theorem~\ref{thm: Polishness of point functor} and~\ref{thm: Polishness of Skorohod structure},
  $\SkorohodSt(\RNp, \PointFunc)$ is stably metrizable and Polish with preserved roots.
  This structure, combined with the structure introduced in Section~\ref{sec: Laws of structures} below,  
  provides a suitable topological framework for studying metric spaces equipped with laws of stochastic processes.
\end{exm}

\begin{exm}
  As a more involved example, consider the case where $\sigma = \MeasFunct$, as defined in Section~\ref{sec: the functor for measures}.  
  In this case, for each $X \in \ob(\BCMcat)$, we have $\tau(X) = D(\RNp, \Meas{X})$,  
  i.e., the space of measure-valued cadlag functions.  
  By Theorems~\ref{thm: Polishness of measure functor} and~\ref{thm: Polishness of Skorohod structure},  
  $\tau$ is stably metrizable and Polish with preserved roots.  
  This structure, combined with the structure introduced in Section~\ref{sec: Laws of structures} below,   
  provides a suitable topological framework for studying metric spaces equipped with laws of measure-valued stochastic processes,  
  such as superprocesses, which are widely studied in probability theory (cf.~\cite{Li_22_Measure}).
\end{exm}

We then turn to a precompactness criterion for the space $\rootedBCM(\SkorohodSt(I, \sigma))$.
We only consider the case $I = \RNp$
because the same argument is also valid for a compact interval $I$
(see Remark~\ref{rem: precompactness for compact interval in Skorohod} below).
To this end,
we first recall a precompactness criterion for the $J_1$-Skorohod topology.
For $F \in D(\RNp, S)$, where $S$ is a metric space,
we define  
\begin{equation}
  w_{d_S}(F, h, t) 
  \coloneqq 
  \inf_{(I_{k})_{k=1}^m \in \Pi_t^h} \max_{1 \leq k \leq m} \sup_{r, s \in I_{k}} d_S(F(r), F(s)),
  \quad 
  t, h > 0,
\end{equation}
where $\Pi_t^h$ denotes the set of finite partitions of the interval $[0, t)$
into subintervals $I_{k} = [u_k, u_{k+1})$ with $u_{k+1} - u_k \geq h$ when $u_{k+1} < t$.

\begin{lem} [{\cite[Theorem A5.4]{Kallenberg_21_Foundations}}] \label{lem: precompact in Skorohod}
  Let $S$ be a Polish space.
  Fix a dense set $T \subseteq \RNp$ and an index set $\mathscr{A}$.
  A non-empty subset $\{ F_\alpha \mid \alpha \in \mathscr{A} \}$ of $D(\RNp, S)$ 
  is precompact in the $J_{1}$-Skorohod topology 
  if and only if the following conditions are satisfied.
  \begin{enumerate} [label = \textup{(\roman*)}, leftmargin = *]
    \item \label{lem item: 1. precompact in Skorohod}
      For each $t \in T$, 
      the set $\{F_\alpha(t) \mid \alpha \in \mathscr{A}\}$ is precompact in $S$.
    \item  \label{lem item: 2. precompact in Skorohod}
      It holds that 
      $\displaystyle \lim_{h \to 0} \sup_{\alpha \in \mathscr{A}} w_{d_S}(F_\alpha, h, t) = 0$
      for all $t >0$.
  \end{enumerate}
\end{lem}

From the above lemma,
we deduce a precompactness criterion for $\rootedBCM(\SkorohodSt(\RNp, \sigma))$ as follows.

\begin{thm} [{Precompactness}] \label{thm: precompact in Skorohod functor}
  Assume that $\sigma$ is upper semicontinuous.
  Fix a space-rooted metrization $\FSMet{\sigma}$,
  which exists by Lemma~\ref{lem: subfunctor embedding in RF}.
  Fix a dense set $T \subseteq \RNp$ and an index set $\mathscr{A}$.
  A non-empty subset $\{ \cX_\alpha = (X_\alpha, \rho_\alpha, F_\alpha) \mid \alpha \in \mathscr{A}\}$ 
  of $\rootedBCM(\SkorohodSt(\RNp, \sigma))$ is precompact in the local GH-type topology
  if and only if 
  the following conditions are satisfied.
  \begin{enumerate} [label = \textup{(\roman*)}, series = precompactness for the cadlag functor]
    \item \label{thm item: 1. precompact in Skorohod functor}
      The set $\{(X_\alpha, \rho_\alpha) \mid \alpha \in \mathscr{A}\}$ 
      is precompact in the local Gromov--Hausdorff topology. 
    \item \label{thm item: 2. precompact in Skorohod functor}
      For each $t \in T$,
      there exists a precompact subfunctor $\FS{\sigma}_t$ of $\FS{\sigma}$ 
      such that $F_\alpha(t) \in \FS{\sigma}_t(X_\alpha, \rho_\alpha)$ for all $\alpha \in \mathscr{A}$.
    \item \label{thm item: 3. precompact in Skorohod functor}
      It holds that 
      $\displaystyle \lim_{h \to 0} \sup_{\alpha \in \mathscr{A}} w_{d^{\FSMet{\sigma}}_{X_\alpha, \rho_\alpha}}(F_\alpha, h, t) = 0$
      for all $t >0$. 
  \end{enumerate}
\end{thm}

\begin{proof}
  Assume that \ref{thm item: 1. precompact in Skorohod functor},  
  \ref{thm item: 2. precompact in Skorohod functor}, and \ref{thm item: 3. precompact in Skorohod functor} are satisfied.  
  For each $h > 0$ and $t > 0$, we set
  \begin{equation}
    \delta(h, t) 
    \coloneqq 
    \sup_{\alpha \in \mathscr{A}} w_{d^{\FSMet{\sigma}}_{X_\alpha, \rho_\alpha}}(F_\alpha, h, t).
  \end{equation}
  We then define a subfunctor $\FS{\tau}_{\mathrm{sub}}$ of $\FS{\tau}$ as follows:  
  for each $(X, \rho_X) \in \ob(\rBCMcat)$, we let $\FS{\tau}_{\mathrm{sub}}(X, \rho_X)$ be the set of all $F \in \tau(X) = D(\RNp, \sigma(X))$ such that
  \begin{itemize}
    \item $F(t) \in \sigma_t(X, \rho_X)$ for each $t \in T$,
    \item $w_{d^{\FSMet{\sigma}}_{X, \rho_X}}(F, h, t) \leq \delta(h, t)$ for all $t > 0$.
  \end{itemize}
  By Lemma~\ref{lem: precompact in Skorohod}, the functor $\FS{\tau}_{\mathrm{sub}}$ is a precompact subfunctor of $\FS{\tau}$,  
  and $\rootedBCM(\FS{\tau}_{\mathrm{sub}})$ contains $\{\cX_\alpha \mid \alpha \in \mathscr{A}\}$.  
  By Proposition~\ref{prop: continuity of Skorohod st},  
  we may apply Theorem~\ref{thm: precompact in RF} to conclude that the family $\{\cX_\alpha \mid \alpha \in \mathscr{A}\}$ is precompact.

  Conversely, assume that $\{\cX_\alpha \mid \alpha \in \mathscr{A}\}$ is precompact.
  Condition~\ref{thm item: 1. precompact in Skorohod functor} follows from Corollary~\ref{cor: projection continuity in RF}.
  Condition~\ref{thm item: 3. precompact in Skorohod functor} can be proven by contradiction, 
  similarly to the proof of Theorem~\ref{thm: precompact for the point functor}.
  To obtain \ref{thm item: 2. precompact in Skorohod functor},
  we let $\FS{\tau}_{\mathrm{sub}}$ be a precompact subfunctor of $\FS{\tau}$,
  which exists by Theorem~\ref{thm: precompact in RF}.
  For each $t \in T$,
  we define a subfunctor $\FS{\sigma}_t$ of $\FS{\sigma}$ as follows:
  for each $(X, \rho_X) \in \ob(\rBCMcat)$, we define 
  \begin{equation}
    \FS{\sigma}_t(X, \rho_X) \coloneqq \{F(t) \mid F \in \FS{\tau}_{\mathrm{sub}}(X, \rho_X)\}.
  \end{equation}
  By Lemma~\ref{lem: precompact in Skorohod}, we deduce that $\sigma_t$ is precompact.
  Hence, \ref{thm item: 2. precompact in Skorohod functor} is satisfied.
\end{proof}

\begin{rem} \label{rem: precompactness for compact interval in Skorohod}
  For a compact interval $I = [0, t_\infty]$ with a fixed $t_\infty \in (0, \infty)$,
  the only change occurs in \ref{thm item: 3. precompact in Skorohod functor}.
  The details go as follows.
  For $F \in D([0, t_\infty], S)$, where $S$ is a metric space,
  we define  
  \begin{equation}
    w_{d_S}(F, h) 
    \coloneqq 
    \inf_{(I_{k})_{k=1}^m \in \Pi^h} \max_{1 \leq k \leq m} \sup_{r, s \in I_{k}} d_S(F(r), F(s)),
    \quad 
    t, h > 0,
  \end{equation}
  where $\Pi^h$ denotes the set of partitions of the interval $[0, t_\infty]$
  into subintervals $I_{k} = [u_k, u_{k+1})$ with $u_{k+1} - u_k \geq h$ for all $k$.
  The Condition~\ref{thm item: 3. precompact in Skorohod functor} is then replaced by the following:
  \begin{enumerate} [resume* = precompactness for the cadlag functor]
    \item \label{thm item: 4. precompact in Skorohod functor}
      It holds that 
      $\displaystyle \lim_{h \to 0} \sup_{\alpha \in \mathscr{A}} w_{d^{\FSMet{\sigma}}_{X_\alpha, \rho_\alpha}}(F_\alpha, h) = 0$. 
  \end{enumerate}
  One can easily verify this by using a precompactness criterion for the Skorohod topology on $D([0,t_\infty], S)$
  (cf.\ \cite[Theorem~12.3]{Billingsley_99_Convergence}).
\end{rem}


\subsection{Continuous functions}  \label{sec: structure of continuous functions}
\newcommand{\taucT}{\tau^{C(T)}}
\newcommand{\taucCT}{\tau^{\cC(T)}}
\newcommand{\ContiFunct}{\tau_C}

\subsubsection{Continuous functions with a fixed domain} \label{sec: st for conti funct fixed}

In \cite{Gwynne_Miller_17_Scaling}, 
a Gromov--Hausdorff-type topology on a set of equivalence classes of metric spaces equipped with continuous curves 
was introduced,
where a continuous curve is used to capture the boundary of a space.
However, the focus was on length spaces for technical reasons.
In this subsubsection,
we define a class of structures which gives a natural generalization of that topology.

Fix a locally compact separable metric space $T$ and a structure $\sigma$ that is stably metrizable and Polish with preserved roots.
Define a structure $\tau = \ContiFunct(T, \sigma)$ as follows.
\begin{itemize}
  \item 
    For $X \in \ob(\BCMcat)$, 
    define $\tau(X) \coloneqq C(T, \sigma(X))$ equipped with the compact-convergence topology. 
  \item 
    For each $f \in \Hom_{\BCMcat}(X, Y)$, 
    define $\tau_f(F) \coloneqq \sigma_f \circ F$ for $F \in C(T, \sigma(X))$.
\end{itemize}

\begin{prop} \label{prop: continuity of conti fixed st}
  The structure $\ContiFunct(T, \sigma)$ is embedding-continuous.
  Moreover, if $\sigma$ is upper semicontinuous, then so is $\ContiFunct(T, \sigma)$.
\end{prop}

\begin{proof}
  This is proven similarly to Proposition~\ref{prop: continuity of Skorohod st}.
\end{proof}

\begin{rem}
  Even when $\sigma$ is lower semicontinuous, $\ContiFunct(T, \sigma)$ may fail to be lower semicontinuous.
  This is because, in general, there is no natural discretization of continuous functions,  
  unlike the case of cadlag functions where step-function approximations are available as in \eqref{pr eq: 1. continuity of the Skorohod functor}.  
\end{rem}

\begin{thm} \label{thm: Polishness of conti fixed st}
  The structure $\ContiFunct(T, \sigma)$ is stably metrizable adn Polish with preserved roots.
\end{thm}

\begin{proof}
  Note that, if necessary, by replacing the metric $d_T$, we may assume that $(T, d_T)$ is boundedly compact
  (cf.\ \cite[Theorem~1]{Williamson_Janos_87_Construction}).
  We fix an element $o \in T$, which serves as the root of $T$.
  Given a metric space $S$, the compact-convergence topology on $C(T, S)$ is metrized by the following metric:
  \begin{equation} \label{pr eq: 1. Polishness of conti fixed st}
    d^C_{T, S}(F, G) 
    \coloneqq 
    \sum_{k=1}^\infty 2^{-k} \bigl( 1 \wedge \sup_{t \in D_T(o, k)} d_S(F(t), G(t)) \bigr).
  \end{equation}

  We first prove that $\ContiFunct(T, \sigma)$ is stably metrizable.
  Let $\FEMet{\sigma}$ be a stable element-rooted metrization of $\sigma$.
  To construct an element-rooted metrization of $\tau$,
  we introduce, for each $X \in \ob(\BCMcat)$, 
  a topological embedding $\iota^C_X \colon C(T, \sigma(X)) \times X \to C(T, \sigma(X) \times X)$ defined by
  \begin{equation}
    \iota^C_X(F, x)(t) \coloneqq (F(t), x).
  \end{equation}
  We then define an element-rooted metrization $\FEMet{\tau}$ of $\tau$ by equipping,  
  for each $X \in \ob(\BCMcat)$, the space $\FE{\tau}(X) = C(T, \sigma(X)) \times X$ with the metric
  \begin{equation} 
    d^{\FEMet{\tau}}_X \bigl((F, x), (G, y)\bigr)
    \coloneqq 
    d^C_{T, \sigma(X) \times X} \bigl( \iota^C_X(F, x), \iota^C_X(G, y) \bigr),
  \end{equation}
  where $d^C_{T, \sigma(X) \times X}$ is a metric defined in \eqref{pr eq: 1. Polishness of conti fixed st},
  with $d^{\FEMet{\sigma}}_{X}$ used as the  metric on $\sigma(X) \times X$.
  Then, by definition,
  one can readily verify that $\FEMet{\tau}$ is stable.

  We next prove that $\ContiFunct(T, \sigma)$ is Polish with preserved roots.
  Let $\mathfrak{P} = (\tilde{\sigma}, \eta, (\FS{\tilde{\sigma}}_k)_{k \geq 1})$ be a root-preserving Polish system of $\sigma$.
  Define a structure $\tilde{\tau}$ as follows.
  \begin{itemize}
    \item 
      For $X \in \ob(\BCMcat)$, 
      define $\tilde{\tau}(X) \coloneqq \GraphSp{T}{\tilde{\sigma}(X)}$
      (recall this space from Definition~\ref{dfn: marked Hausdorff space}). 
    \item 
      For each $f \in \Hom_{\BCMcat}(X, Y)$, 
      define $\tilde{\tau}_f \coloneqq \Image{\id_T \times f}$.
  \end{itemize}
  Then $\tau$ is a topological subfunctor of $\tilde{\tau}$.
  The associated topological embedding $\zeta \colon \tau \Rightarrow \tilde{\tau}$ is given as follows:
  for each $X \in \ob(\BCMcat)$, we define $\zeta_X(F) \coloneqq \graphmap (\eta_X \circ F)$ for $F \in \tau(X)$,
  where we recall the graph map $\graphmap$ from \eqref{eq: def of graphmap}.
  For each $k \geq 1$, we define a subfunctor $\FS{\tilde{\tau}}_k$ of $\FS{\tilde{\tau}}$ as follows:
  for each $(X, \rho_X) \in \ob(\rBCMcat)$, 
  we write $p^1_X \colon T \times \tilde{\sigma}(X) \to T$ and $p^2_X \colon T \times \tilde{\sigma}(X) \to \tilde{\sigma}(X)$
  for the projections;
  we then define $\FS{\tilde{\tau}}_k(X, \rho_X)$ to be the set of $D \in  \tilde{\tau}(X)$
  satisfying the following conditions;
  \begin{enumerate} [label = \textup{(\roman*)}]
    \item \label{pr item: 1. Polishness of conti fixed st}
      there exists $\varepsilon_1 = \varepsilon_1(D) \in (0, 1/k)$ 
      such that 
      \begin{equation}
        \HausMet{T}\Bigl( p^1_X \bigl( D|_o^{(k-\varepsilon, *)} \bigr), T|_o^{(k)} \Bigr) < 1/k,
      \end{equation}
    \item \label{pr item: 2. Polishness of conti fixed st}
      there exists $\varepsilon_1 = \varepsilon_1(D) \in (0, 1/k)$ 
      such that 
      \begin{equation}
        p^2_X \bigl( D|_o^{(k-\varepsilon, *)} \bigr)
        \subseteq \FS{\tilde{\sigma}}_k(X, \rho_X),
      \end{equation}
    \item \label{pr item: 3. Polishness of conti fixed st}
       $\displaystyle D \in \subGraphSp{\rho_X}{k}{T}{\tilde{\sigma}(X)}$
      (recall this space from Definition~\ref{dfn: graph subspace}).
  \end{enumerate} 
  We will prove that $\mathfrak{Q} \coloneqq (\tilde{\tau}, \zeta, (\FS{\tilde{\tau}}_k)_{k \geq 1})$ is a root-preserving Polish system of $\tau$.

  The separability of $\tilde{\tau}$ follows from Proposition~\ref{prop: graph metric}.
  Using the continuity of $\tilde{\sigma}$,
  one can prove the continuity of $\tilde{\tau}$ similarly to $\SubsetSt$ (see Theorem~\ref{thm: Polishness of subset st}),
  using the continuity of $\tilde{\sigma}$.
  Let $\FSMet{\tilde{\sigma}}$ be a complete space-rooted metrization of $\sigma$.
  We define a space-rooted metrization of $\tilde{\tau}$ as follows:  
  for each $(X, \rho_X) \in \ob(\rBCMcat)$, we equip $\tilde{\tau}(X)$ with the metric  
  \begin{equation}
    d^{\FSMet{\tilde{\tau}}}_{X, \rho_X}(F, G) 
    \coloneqq  
    \GraphMet{(T, o)}{\tilde{\sigma}(X)}(F, G),
  \end{equation}
  where $\GraphMet{(T, o)}{\tilde{\sigma}(X)}$ is the metric defined in \eqref{eq: the metric on graph sp},  
  with $d^{\FSMet{\tilde{\sigma}}}_{X, \rho_X}$ used as the metric on $\tilde{\sigma}(X)$.
  By Proposition~\ref{prop: graph metric}, the space-rooted metrization $\FSMet{\tilde{\tau}}$ is complete.
  Hence, we obtain \ref{dfn item: 1. Polish functor in RF}.

  Fix $(X, \rho_X) \in \ob(\rBCMcat)$.
  Assume that a sequence $(D_n)_{n \geq 1}$ in $\tilde{\tau}(X)$ converging to some $D$.
  By Theorem~\ref{thm: convergence in D},
  it holds that $D_n|_o^{(r, *)} \to D|_o^{(r, *)}$ in the Hausdorff topology for all but countably many $r > 0$.
  Thus, the set of $D \in \tilde{\tau}(X)$ not satisfying \ref{pr item: 1. Polishness of conti fixed st} is closed in $\tilde{\tau}(X)$.
  Moreover, using that $\FS{\tilde{\sigma}}_k(X, \rho_X)$ is open in $\tilde{\sigma}(X)$,  
  one can verify that the set of $D \in \tilde{\tau}(X)$ not satisfying \ref{pr item: 2. Polishness of conti fixed st} is also closed in $\tilde{\tau}(X)$.
  Combining these with Lemma~\ref{lem: 1. Polishness of variable domains},
  we deduce that $\FS{\tilde{\tau}}_k$ is a open subfunctor of $\FS{\tilde{\tau}}$.
  Moreover, using the pullback-stability of $\FS{\tilde{\sigma}}_k$,  
  one can show that $\FS{\tilde{\tau}}_k$ is also pullback-stable.  
  Thus, \ref{dfn item: 2. Polish functor in RF} is satisfied.
  
  To verify \ref{dfn item: 3. Polish functor in RF},
  fix $(X, \rho_X) \in \ob(\rBCMcat)$.
  Let $D \in \bigcap_{k \geq 1} \FS{\tilde{\tau}}_k(X, \rho_X)$.
  By Lemma~\ref{lem: 2. Polishness of variable domains},
  there exists $F \in \hatC{T}{\tilde{\sigma}(X)}$ such that $\graphmap(F) = D$.
  Condition~\ref{pr item: 1. Polishness of conti fixed st} implies that $\dom(F) = T$.
  Moreover, from \ref{pr item: 2. Polishness of conti fixed st} and \ref{dfn item: 3. Polish functor in RF} of the root-preserving Polish system $\mathfrak{P}$,
  we deduce that $F$ takes values in $\eta_X(\sigma(X))$.
  Thus, if we define $G \in \tau(X) = C(T, \sigma(X))$ by $G \coloneqq \eta_X^{-1} \circ F$,
  then it holds that $\zeta_X(G) = D$.
  This shows that $\zeta_X(\tau(X)) \supseteq \bigcap_{k \geq 1} \FS{\tilde{\tau}}_k(X, \rho_X)$.
  The converse inclusion is straightforward.
  Hence $\mathfrak{Q}$ satisfies \ref{dfn item: 3. Polish functor in RF},
  which completes the proof.
\end{proof}

\begin{exm}
  In the setting of \cite{Gwynne_Miller_17_Scaling},
  $T$ is the one-dimensional Euclidean metric space $\RN$.
  If one sets $T \coloneqq [0,1]/\{0,1\}$,
  then $C(T,X)$ is a set of loops in $X$.
  By taking the countably many products of copies of $\ContiFunct(T, \PointFunc)$,
  one obtains a structure for spaces equipped with countably many loops,
  which might be useful for studying random loop soups (e.g.\ \cite{Lawler_Werner_04_The_Brownian}).
\end{exm} 

Using a precompactness criterion for the compact-convergence topology (see \cite[Theorem A5.2]{Kallenberg_21_Foundations} for example),
we derive a precompactness criterion for the space $\rootedBCM(\ContiFunct(T, \sigma))$ 
in a similar way to the proof of Theorem~\ref{thm: precompact in Skorohod functor}.

\begin{thm} [Precompactness]\label{thm: precompactness for conti fixed st}
  Assume that $\sigma$ is upper semicontinuous.
  Fix a space-rooted metrization $\FSMet{\sigma}$.
  Fix a dense set $T' \subseteq T$
  and a non-empty index set $\mathscr{A}$.
  A subset $\{ \cX_\alpha = (X_\alpha, \rho_\alpha, F_\alpha) \mid \alpha \in \mathscr{A}\}$ 
  of $\rootedBCM(\ContiFunct(T, \sigma))$ is precompact if and only if 
  the following conditions are satisfied.
  \begin{enumerate} [label = \textup{(\roman*)}, leftmargin = *]
    \item \label{thm item: (continuous curves) spaces are precompact in the local GH top}
      The subset $\{(X_\alpha, \rho_\alpha) \mid \alpha \in \mathscr{A}\}$ of $\rootedBCM$
      is precompact in the local Gromov--Hausdorff topology. 
    \item \label{thm item: values of continuous functions at each time is bounded}
      For each $t \in T'$,
      there exists a precompact subfunctor $\FS{\sigma}_t$ of $\FS{\sigma}$ 
      such that $F_\alpha(t) \in \FS{\sigma}_t(X_\alpha, \rho_\alpha)$ for all $\alpha \in \mathscr{A}$.
    \item \label{thm item: uniform convergence of moduli continuity}
      It holds that 
      $\displaystyle \lim_{h \to 0} \sup_{\alpha \in \mathscr{A}} 
      \sup_{\substack{s,t \in K\\ d_T(s, t) \leq h}} 
      d^{\FSMet{\sigma}}_{X_\alpha, \rho_\alpha}\bigl( F_\alpha(s), F_\alpha(t) \bigr)
      = 0$
      for all non-empty compact subset $K \subseteq T$.
  \end{enumerate}
\end{thm}

It is known that
the restriction of the usual $J_{1}$-Skorohod topology to the set of continuous functions 
is the compact-convergence topology
(cf.\ \cite[Chapter VI. Proposition 1.17]{Jacod_Shiryaev_03_Limit}). 
The following result is a generalization of this fact.

\begin{prop}
  The structure $\ContiFunct(\RNp, \PointFunc)$ is topologically embedded into $\SkorohodSt(\RNp, \PointFunc)$,
  which is defined in Section~\ref{sec: structure for cadlag curves}.
  As a consequence, the following map is a topological embedding:
  \begin{equation}
    \rootedBCM(\ContiFunct(\RNp, \PointFunc)) \ni 
    (X, \rho_X, \xi_{X}) \mapsto (X, \rho_X, \xi_{X}) 
    \in \rootedBCM(\SkorohodSt(\RNp, \PointFunc)).
  \end{equation}
\end{prop}

\begin{proof}
  For each $X \in \ob(\BCMcat)$,
  the inclusion map from $C(\RNp, X)$
  to $D(\RNp, X)$ is a topological embedding
  (see \cite[Chapter VI. Proposition 1.17]{Jacod_Shiryaev_03_Limit}).
  Therefore, we deduce that $\ContiFunct(\RNp, \PointFunc)$ is a topological subfunctor of $\SkorohodSt(\RNp, \PointFunc)$.
  The last assertion immediately follows from Proposition \ref{lem: subfunctor embedding in RF}.
\end{proof}


\subsubsection{Continuous functions with space-dependent domains}  \label{sec: functor for space-domain continuous maps}
\newcommand{\taukXi}{\tau^{k, \Xi}} 
\newcommand{\overXk}{\overline{X^{k}}}
\newcommand{\overYk}{\overline{Y^{k}}}
\newcommand{\ContiVFunc}{\tau_{\widehat{C}}}
\newcommand{\ConticVFunc}{\tau_{\widehat{C}_c}}

In \cite{Croydon_Hambly_Kumagai_12_Convergence},
a Gromov--Hausdorff-type topology on a set of equivalence classes of compact metric spaces $X$ equipped with heat-kernel-type functions  
was introduced,
where a heat-kernel-type function $f$ means a continuous function $f \colon  \RNpp \times X \times X \to \RN$.
In \cite{Angel_Croydon_Hernandez-Torres_Shiraishi_21_Scaling,Barlow_Croydon_Kumagai_17_Subsequential},
a Gromov--Hausdorff-type topology on a set of equivalence classes 
of real trees $X$ equipped with embedding maps,
where an embedding map means a continuous map from $X$ to some fixed metric space.
In these Gromov-Hausdorff-type topologies,
unlike the previous subsubsection,
the domain of continuous functions depends on the underlying space $X$.
In this subsubsection, 
we define a class of structures for such continuous functions,
which includes the above-mentioned examples. 
Moreover, 
combined with structures introduced in Section~\ref{sec: Laws of structures} below, 
it provides a suitable topological setting for studying local times of stochastic processes living on different spaces,
which is in used in \cite{Noda_pre_Convergence}.

Fix a stable space transformation $\Psi$ and a structure $\sigma$ that is stably metrizable and Polish with preserved roots.
Define a structure $\tau = \ContiVFunc(\Psi, \sigma)$ as follows.
\begin{itemize}
  \item 
    For $X \in \ob(\BCMcat)$, 
    define $\tau(X) \coloneqq \hatC{\Psi(X)}{\sigma(X)}$. 
  \item 
    For each $f \in \Hom_{\BCMcat}(X, Y)$, 
    define $\tau_f(F) \coloneqq \sigma_f \circ F \circ \Psi_f^{-1}$ with $\dom(\tau_f(F)) \coloneqq \Psi_f(\dom(F))$ 
    for $F \in \tau(X)$.
\end{itemize}

\begin{prop} \label{prop: continuity of conti variable st}
  The structure $\ContiVFunc(\Psi, \sigma)$ is embedding-continuous.
  Moreover, if $\sigma$ is upper semicontinuous, then so is $\ContiVFunc(\Psi, \sigma)$.
\end{prop}

\begin{proof}
  This is proven similarly to Proposition~\ref{prop: continuity of conti fixed st}.
\end{proof}

\begin{thm} \label{thm: Polishness of conti variable st}
  The structure $\ContiVFunc(\Psi, \sigma)$ is stably metrizable and Polish with preserved roots.
\end{thm}

\begin{proof}
  The proof is similar to that of Theorem~\ref{thm: Polishness of conti fixed st}.
  We first prove that $\ContiFunct(T, \sigma)$ is stably metrizable.
  Let $\FEMet{\sigma}$ be a stable element-rooted metrization of $\sigma$.
  To construct an element-rooted metrization of $\tau$,
  we introduce, for each $X \in \ob(\BCMcat)$, 
  a topological embedding $\iota^{\widehat{C}}_X \colon \hatC{\Psi(X)}{\sigma(X)} \times X \to \hatC{\Psi(X)}{\sigma(X) \times X}$ defined by
  \begin{equation}
    \iota^{\widehat{C}}_X(F, x)(\alpha) \coloneqq (F(\alpha), x).
  \end{equation}
  We then define an element-rooted metrization $\FEMet{\tau}$ of $\tau$ by equipping,  
  for each $X \in \ob(\BCMcat)$, the space $\FE{\tau}(X) = \hatC{\Psi(X)}{\sigma(X)} \times X$ with the metric
  \begin{equation} 
    d^{\FEMet{\tau}}_X \bigl((F, x), (G, y)\bigr)
    \coloneqq 
    \hatCMet{\Psi(X)}{\FE{\sigma}(X)} \bigl( \iota^{\widehat{C}}_X(F, x), \iota^{\widehat{C}}_X(G, y) \bigr),
  \end{equation}
  where $\hatCMet{\Psi(X)}{\sigma(X) \times X}$ is the metric defined in \eqref{eq: the metric on product hatC},
  with $d^{\FEMet{\sigma}}_{X}$ used as the  metric on $\sigma(X) \times X$.
  Then, similarly to the proof of Theorem~\ref{thm: Polishness of conti fixed st},
  one can verify that $\FEMet{\tau}$ is stable, using the stability of $\Psi$ and $\FEMet{\sigma}$.

  We next prove that $\ContiVFunc(\Psi, \sigma)$ is Polish with preserved roots.
  Let $\mathfrak{P} = (\tilde{\sigma}, \eta, (\FS{\tilde{\sigma}}_k)_{k \geq 1})$ be a root-preserving Polish system of $\sigma$.
  Define a structure $\tilde{\tau}$ as follows.
  \begin{itemize}
    \item 
      For $X \in \ob(\BCMcat)$, 
      define $\tilde{\tau}(X) \coloneqq \GraphSp{\Psi(X)}{\tilde{\sigma}(X)}$. 
    \item 
      For each $f \in \Hom_{\BCMcat}(X, Y)$, 
      define $\tilde{\tau}_f \coloneqq \Image{\Psi_f \times \tilde{\sigma}_f}$.
  \end{itemize}
  Then $\tau$ is a topological subfunctor of $\tilde{\tau}$.
  The associated topological embedding $\zeta \colon \tau \Rightarrow \tilde{\tau}$ is given as follows:
  for each $X \in \ob(\BCMcat)$, we define $\zeta_X(F) \coloneqq \graphmap (\eta_X \circ F)$ for $F \in \tau(X)$.
  For each $k \geq 1$, we define a subfunctor $\FS{\tilde{\tau}}_k$ of $\FS{\tilde{\tau}}$ as follows:
  for each $(X, \rho_X) \in \ob(\rBCMcat)$, 
  we define $\FS{\tilde{\tau}}_k(X, \rho_X)$ to be the set of $D \in  \tilde{\tau}(X)$
  satisfying the following conditions;
  \begin{enumerate} [label = \textup{(\roman*)}]
    \item $\displaystyle D \in \subGraphSp{\RootSystem_X(\rho_X)}{k}{\Psi(X)}{\tilde{\sigma}(X)}$,
    \item there exists $\varepsilon = \varepsilon(D) \in (0, 1/k)$ such that 
      \begin{equation}
        p^2_X\bigl( D|_{\RootSystem_X(\rho_X)}^{(k-\varepsilon, *)} \bigr) 
      \subseteq \FE{\tilde{\sigma}}_k(X, \rho_X),
      \end{equation}
      where $p^2_X \colon \Psi(X) \times \tilde{\sigma}(X) \to \tilde{\sigma}(X)$ denotes the projection.
  \end{enumerate} 
  Similarly to the proof of Theorem~\ref{thm: Polishness of conti fixed st},
  one can check that $\mathfrak{Q} \coloneqq (\tilde{\tau}, \zeta, (\tilde{\tau}^\times_k)_{k \geq 1})$ is a root-preserving Polish system of $\tau$,
  which completes the proof.
\end{proof}

\begin{rem}
  To discuss convergence of spaces equipped with heat-kernel-type functions,  
  consider the following setting:  
  we define $\Psi$ to be the space transformation given by $\Psi(X) \coloneqq \RNpp \times X \times X$ for each $X \in \ob(\BCMcat)$  
  (see Example~\ref{exm: 2. space transformation}),  
  and set $\sigma \coloneqq \Fixedst{\RNp}$ (recall this from Section~\ref{sec: Fixed structures}).  
  Then the space $\rootedBCM(\ContiVFunc(\Psi,\sigma))$ consists of rooted $\bcmAB$ spaces $(X, \rho_X)$ equipped with continuous functions  
  $p \in \hatC{\RNpp \times X \times X}{\RNp}$, where $p$ serves as the heat kernel of a stochastic process on $X$.
  By Theorem~\ref{thm: Polishness of conti variable st},
  the local GH-type topology on $\rootedBCM(\ContiVFunc(\Psi, \Fixedst{\RNp}))$ is Polish.

  Another way to view a heat-kernel-type function $p \in C(\RNp \times X \times X, \RNp)$ is as a family of density functions $p(t, x, \cdot)$.  
  That is, one may regard $p$ as a continuous map  
  \begin{equation}
    \RNpp \times X \ni (t,x) \mapsto p(t,x, \cdot) \in C(X, \RNp).
  \end{equation} 
  In this case, we take the space transformation $\Psi$ given by $\Psi(X) = \RNpp \times X$ for each $X \in \ob(\BCMcat)$,  
  and set $\sigma \coloneqq \ContiVFunc(\PointFunc, \tau_{\RNp})$.  
  Then the space $\rootedBCM(\ContiVFunc(\Psi,\sigma))$ consists of rooted $\bcmAB$ spaces $(X, \rho_X)$  
  equipped with continuous functions $p \in \hatC{\RNp \times X}{\hatC{X}{\RNp}}$,  
  and is Polish.
\end{rem}

By using Theorem~\ref{thm: precompactness in variable domains},
we obtain a precompactness criterion for the space $\rootedBCM(\ContiVFunc(\Psi, \sigma))$
in a similar way to the proof of Theorem~\ref{thm: precompact in Skorohod functor}.

\begin{thm} [Precompactness]\label{thm: precompactness for conti variable st}
  Assume that $\sigma$ is upper semicontinuous.
  Fix a space-rooted metrization $\FSMet{\sigma}$.
  Write $\RootSystem$ for a rooting system of $\Psi$.
  Fix a non-empty index set $\mathscr{A}$.
  A subset $\{ \cX_\alpha = (X_\alpha, \rho_\alpha, F_\alpha) \mid \alpha \in \mathscr{A} \}$ 
  of $\rootedBCM(\ContiVFunc(\Psi, \sigma))$ is precompact if and only if the following conditions are satisfied.
  \begin{enumerate} [label= (\roman*)]
    \item \label{thm item: space-indexed functions, the space is pre-cpt}
      The subset $\{ (X_\alpha, \rho_\alpha) \mid \alpha \in \mathscr{A}\}$ of $\rootedBCM$
      is precompact in the local Gromov--Hausdorff topology.
    \item \label{thm item: space-indexed functions, functions are relatively compact pointwise}
      For every $r>0$, 
      there exists a precompact subfunctor $\FS{\sigma}_r$ of $\FS{\sigma}$ 
      such that 
      \begin{equation}
        \Bigl\{ F_\alpha(x) \mid x \in \dom(F_\alpha)|_{\RootSystem_{X_\alpha}(\rho_\alpha)}^{(r)} \Bigr\}
        \subseteq 
        \FS{\sigma}_r(X_\alpha, \rho_\alpha),
        \quad 
        \forall \alpha \in \mathscr{A}.
      \end{equation}
    \item \label{thm item: space-indexed functions, equicontinuity of functions}
      For every $r>0$, 
      \begin{equation}
        \lim_{\delta \to 0}
        \sup_{\alpha \in \mathscr{A}}
        \sup_{\substack{x,y \in \dom(F_\alpha)|_{\RootSystem_{X_\alpha}(\rho_\alpha)}^{(r)} \\ d_{\Psi(X_\alpha)}(x,y) \leq \delta}}
        d_\Xi (F_\alpha(x), F_\alpha(y)) 
        = 0.
      \end{equation}
  \end{enumerate}
\end{thm}

\begin{rem}
  When one wants to consider only continuous functions with compact domains,
  one may use the structure $\tau = \ConticVFunc(\Psi, \sigma)$ defined as follows.  
  \begin{itemize}
    \item 
      For $X \in \ob(\BCMcat)$, 
      define $\tau(X) \coloneqq \hatCc{\Psi(X)}{\sigma(X)}$. 
    \item 
      For each $f \in \Hom_{\BCMcat}(X, Y)$, 
      define $\tau_f(F) \coloneqq \sigma_f \circ F \circ \Psi_f^{-1}$ with $\dom(\tau_f(F)) \coloneqq \Psi_f(\dom(F))$ 
      for $F \in \tau(X)$.
  \end{itemize}
  An analogue of Theorem~\ref{thm: Polishness of conti variable st} holds for $\ConticVFunc(\Psi, \sigma)$.
\end{rem}

\subsection{Measurable functions}  \label{sec: Measurable functions}
\newcommand{\LzeroMet}[2]{d^{L^0_\mu}_{#1, #2}}
\newcommand{\LzeroFunct}[2]{\tau_{L_{#2}^0(#1)}}

In certain scaling limit problems such as \cite{Croydon_Muirhead_15_Functional,Fontes_Mathieu_14_On}, 
the usual $J_1$-Skorohod topology is too strong to capture convergence of the processes of interest. 
For example, in \cite{Croydon_Muirhead_15_Functional}, which studies the Bouchaud trap model with slowly varying traps,
during the time that the process is based in a certain deep trap,
it makes many short excursions away from that trap. 
These excursions vanish on the time scale of the limit but have macroscopic spatial size, 
preventing convergence in $J_1$ (and even in the coarser $M_1$ topology). 
The $L_{1,\mathrm{loc}}$ topology (see Example~\ref{exm: L_1,loc topology} below) accommodates such behaviour
and is a natural choice in settings 
where macroscopic, but short-time oscillations occur around the main limiting trajectory.
In what follows, we discuss a generalization of the $L_{1,\mathrm{loc}}$ topology
and apply our main results to this setting.

Fix a measurable space $(T, \mathcal{T})$ equipped with a Borel probability measure $\mu$.
We assume that $\mathcal{T}$ is \emph{countably generated}, that is,
there exists a countable subfamily $\mathcal{T}' \subseteq \mathcal{T}$ that generates $\mathcal{T}$.
For instance, any separable metric space $T$ with Borel $\sigma$-algebra $\mathcal{T}$ satisfies this assumption.
Given a Polish metric space $S$, we write $L^0_\mu(T, S)$ for the set of measurable functions $f \colon T \to S$.
We adopt the convention that any functions $\xi$ and $\eta$ in $L^0_\mu(T, S)$ are identified whenever $f = g$, $\mu$-a.e.
We then define a metric $\LzeroMet{T}{S}$ on $L^0_\mu(T, S)$ by setting, for each $f, g \in L^0_\mu(T, S)$,
\begin{equation}
  \LzeroMet{T}{S}(\xi, \eta) \coloneqq \int_T \bigl( 1 \wedge d_S(\xi(t), \eta(t)) \bigr)\, \mu(dt).
\end{equation}
Convergence with respect to $\LzeroMet{T}{S}$ coincides with convergence in probability
(see \cite[Lemma~5.2]{Kallenberg_21_Foundations}).
We endow $L^0_\mu(T, S)$ with the topology induced by $\LzeroMet{T}{S}$.
We note that this topology is independent of the choice of $d_S$ and depends only on the topology of $S$,
which can be readily verified from \cite[Lemma~5.2(iii)]{Kallenberg_21_Foundations}.

\begin{exm} \label{exm: L_1,loc topology}
  Let $T \coloneqq \RNp$, and $\mu(dt) \coloneqq e^{-t}\, dt$, where $dt$ on the right-hand side denotes the Lebesgue measure. 
  Then the space $D(\RNp, S)$ of cadlag functions can be viewed as a subset of $L^0_\mu(T, S)$, 
  and the topology on $L^0_\mu(T, S)$ induces the relative topology on $D(\RNp, S)$. 
  This is weaker than the $J_1$-Skorohod topology, 
  and is referred to as the $L_{1,\mathrm{loc}}$ topology in \cite{Croydon_Muirhead_15_Functional}.
\end{exm}

\begin{prop} \label{prop: metric for conv in prob}
  The space $L^0_\mu(T, S)$ is separable, and the metric $\LzeroMet{T}{S}$ is complete.
\end{prop}

\begin{proof}
  The assertion regarding completeness follows from \cite[Lemma~5.6]{Kallenberg_21_Foundations}.
  Separability can be established similarly to that of $L^p$ spaces for $p \geq 1$ (cf.\ \cite[Proposition~3.4.5]{Cohn_13_Measure}).
  However, since the range is now a general space $S$, some arguments must be handled with care, so we provide a complete proof here.
  Because $\mathcal{T}$ is countably generated,
  there exists a countable subfamily $\mathcal{T}_0 \subseteq \mathcal{T}$
  that generates $\mathcal{T}$ and forms an algebra (not a $\sigma$-algebra); see \cite[Proof of Proposition~3.4.5]{Cohn_13_Measure}.
  Fix a countable dense subset $S' \subseteq S$.
  Let $\mathcal{D}$ be the subset of $L^0_\mu(T, S)$ consisting of
  simple functions $\xi$ with values in $S'$ and measurable with respect to $\mathcal{T}_0$,
  that is,
  $\xi$ has the form:
  there exist $n \in \mathbb{N}$, a disjoint family $\{A_i\}_{i=1}^n \subseteq \mathcal{T}_0$ with $\bigcup_{i=1}^n A_i = T$, 
  and a collection $\{a_i\}_{i=1}^n$ of points in $S'$
  such that 
  \begin{equation}
    \xi(t) = a_i \quad \text{if} \quad t \in A_i .
  \end{equation}
  Note that $\mathcal{D}$ is countable.
  We claim that $\mathcal{D}$ is dense in $L^0_\mu(T, S)$.
  For any simple function $\xi$,
  one can construct a sequence in $\mathcal{D}$ converging to $\xi$ in probability
  by applying \cite[Lemma~3.4.6]{Cohn_13_Measure}.
  Thus, it remains to show that any measurable function can be approximated by simple functions.
  This is well known for $\mathbb{R}$-valued measurable functions,
  and the same approximation works in our setting as follows.
  Since $S$ is locally compact and separable, 
  we can find an increasing sequence of non-empty compact sets $(K_n)_{n \geq 1}$ with $\bigcup_{n \geq 1} K_n = S$.
  For each $n \geq 1$, let $(K_n^i)_{i = 1}^{k_n}$ be a finite partition of $K_n$ into non-empty sets with $\diam(K_n^i) \leq 1/n$,
  where $\diam$ denotes the diameter.
  Fix an element $a_n^i \in K_n^i$ for each $n \geq 1$ and $i \in \{1, \dots, k_n\}$.
  Given a measurable function $\xi \colon T \to S$,
  for each $n \geq 1$ define a simple function $\xi_n$ by
  \[
    \xi_n(t) \coloneqq 
    \begin{cases}
      a_n^i, & \text{if } \xi(t) \in K_n^i \text{ for some } i \in \{1, \dots, k_n\}, \\
      a_1^1, & \text{otherwise}.
    \end{cases}
  \]
  It is straightforward to check that $\xi_n \to \xi$ pointwise,
  which completes the proof.
\end{proof}

Define a structure $\tau = \LzeroFunct{T}{\mu}$ as follows.
\begin{itemize}
  \item 
    For $X \in \ob(\BCMcat)$, 
    define $\tau(X) \coloneqq L^0_\mu(T, X)$.
  \item 
    For each $f \in \Hom_{\BCMcat}(X, Y)$,
    define $\tau_f(\xi) \coloneqq f \circ \xi$ for each $\xi \in L^0_\mu(T, X)$.
\end{itemize}
This admits a natural metrization $\FMet{\LzeroFunct{T}{\mu}}$ given by equipping each $\LzeroFunct{T}{\mu}(X) = L^0_\mu(T, X)$ 
with the metric $\LzeroMet{T}{X}$.

\begin{thm} \label{thm: continuity of Lzero st}
  The structure $\LzeroFunct{T}{\mu}$ is continuous and separable,
  and its metrization $\FMet{\LzeroFunct{T}{\mu}}$ is complete and stable with distortion $\Dist{\PointFunc}(\varepsilon) = 1 \wedge \varepsilon$.
\end{thm}

\begin{proof}
  The separability of $\LzeroFunct{T}{\mu}$ and the completeness of $\FMet{\LzeroFunct{T}{\mu}}$ follow from Proposition~\ref{prop: metric for conv in prob}.  
  The stability of $\FMet{\LzeroFunct{T}{\mu}}$ is straightforward.  
  The embedding-continuity readily follows from the dominated convergence theorem.
  To prove the semicontinuity,  
  let $X_n$, $n \in \mathbb{N}\cup\{\infty\}$, be $\bcmAB$ spaces
  that are embedded isometrically into a common $\bcmAB$ space $Y$
  in such a way that $X_n \to X_\infty$ in the Fell topology as subsets of $Y$.

  Let $\xi_n$, $n \geq 1$, be elements of $L^0_\mu(T, X_n)$ 
  converging to an element $\xi \in L^0_\mu(T, Y)$ in probability.
  Then there exists a subsequence $(n_k)_{k \geq 1}$ such that 
  $\xi_{n_k} \to \xi$ almost surely.
  This, combined with the Fell convergence of $X_n$ to $X_\infty$,
  implies that $\xi$ takes values in $X_\infty$ almost surely. 
  Therefore, $\LzeroFunct{T}{\mu}$ is upper semicontinuous.
  The lower semicontinuity follows from an argument similar to that in the proof of Theorem~\ref{thm: Polishness of measure functor}.
  Let $\xi$ be a simple function taking values in $X_\infty$.
  Using the Fell convergence of $X_n$ to $X_\infty$,
  it is not difficult to construct simple functions $\xi_n$ taking values in $X_n$
  such that $\xi_{n}$ converges to $\xi$ in probability.
  Since such simple functions $\xi$ are dense in $L^0_\mu(T, X_\infty)$,
  we obtain the lower semicontinuity of $\LzeroFunct{T}{\mu}$.
\end{proof}


\subsection{Laws of structures}  \label{sec: Laws of structures}
\newcommand{\ProbFunct}[1]{\tau_{\mathcal{P}}(#1)}

In this subsection,
we define a class of structures which provides a topological setting 
suitable for studying random objects in different spaces.

Let $\sigma$ be a structure that is stable and Polish with preserved roots.
Define a structure $\tau = \ProbFunct{\sigma}$ as follows.
\begin{itemize}
  \item 
    For $X \in \ob(\BCMcat)$, 
    define $\tau(X) \coloneqq \Prob{\sigma(X)}$,
    i.e., the set of probability measures on $\sigma(X)$ equipped with the weak topology. 
  \item 
    For each $f \in \Hom_{\BCMcat}(X, Y)$,
    define $\tau_f \coloneqq (\sigma_f)_*$, i.e., the pushforward map given by $\sigma_f$.
\end{itemize}

\begin{prop} \label{prop: continuity of prob st}
  The structure $\ProbFunct{\sigma}$ is embedding-continuous.
  Moreover, if $\sigma$ is upper (resp.\ lower) semicontinuous, then so is $\ProbFunct{\sigma}$.
\end{prop}

\begin{proof}
  The embedding-continuity is straightforward by the dominated convergence theorem.
  To prove the semicontinuity,  
  let $X_n$, $n \in \mathbb{N}\cup\{\infty\}$, be $\bcmAB$ spaces
  that are embedded isometrically into a common $\bcmAB$ space $Y$
  in such a way that $X_n \to X_\infty$ in the Fell topology as subsets of $Y$.
 
  Assume that $\sigma$ is upper semicontinuous.
  Let $P_n$, $n \geq 1$, be probability measures on $\sigma(X_n)$  
  converging weakly to some probability measure $P$ on $\sigma(Y)$.  
  By the Skorohod representation theorem,  
  there exists a probability space $(\Omega, \mathcal{F}, \mathbb{P})$  
  and random elements $\xi_n$ taking values in $\sigma(X_n)$ and $\xi$ in $\sigma(Y)$  
  such that  
  $\mathbb{P}(\xi_n \in \cdot) = P_n$,  
  $\mathbb{P}(\xi \in \cdot) = P$,  
  and  
  $\xi_n \to \xi$ in $\sigma(Y)$ almost surely.  
  The upper semicontinuity of $\sigma$ yields that $\xi \in \sigma(X_\infty)$ almost surely,
  which implies that $P$ is supported on $\sigma(X_\infty)$.  
  Therefore, $\ProbFunct{\sigma}$ is upper semicontinuous.
  The lower semicontinuity is proven similarly to the proof of Theorem~\ref{thm: Polishness of measure functor}.
\end{proof}

\begin{thm} \label{thm: Polishness of prob st}
  The structure $\ProbFunct{\sigma}$ is stably metrizable and Polish with preserved roots.
\end{thm}

\begin{proof}
  We first prove that $\tau = \ProbFunct{\sigma}$ is stably metrizable.
  Let $\FEMet{\sigma}$ be a stable element-rooted metrization of $\sigma$ with distortion $\Dist{\FEMet{\sigma}}$.
  We then define an element-rooted metrization $\FEMet{\tau}$ of $\tau$ as follows:
  for each $X \in \ob(\BCMcat)$, 
  we equip $\FE{\tau}(X)$ with a metric given by   
  \begin{equation}
    d^{\FEMet{\tau}}_X\bigl( (P, x), (Q, y) \bigr)
    \coloneqq 
    d_X(x,y)
    \vee
    \ProhMet{\FE{\sigma}(X)}( P \otimes \delta_x,  Q \otimes \delta_y),
  \end{equation}
  where $\otimes$ denotes the product of two measures, 
  $\delta_x$ denotes the Dirac probability measure at $x$,
  and $\ProhMet{\FE{\sigma}(X)}$ denotes the Prohorov metric constructed from the metric $d^{\FEMet{\sigma}}_X$ on $\FE{\sigma}(X)$.
  Fix $\bcmAB$ spaces $X, Y, M_1$, and $M_2$.
  Let $f_i \colon X \to M_i$ and $g_i \colon Y \to M_i$, $i = 1,2$, be isometric embeddings.
  Assume that there exists $\varepsilon \in \RNp$ such that, for all $x \in X$ and $y \in Y$,
  \begin{equation} \label{pr eq: 1. stability of prob functor}
    d_{M_2}(f_2(x), g_2(y)) \leq d_{M_1}(f_1(x), g_1(y)) + \varepsilon,
  \end{equation}
  The stability of $\FEMet{\sigma}$ implies that, for all $(a, x) \in \FE{\sigma}(X)$ and $(b, y) \in \FE{\sigma}(Y)$,
  \begin{equation}
    d_{\FE{\sigma}(M_2)} \bigl( \FE{\sigma}_{f_2}(a, x), \FE{\sigma}_{g_2}(b, y) \bigr)
    \leq   
    d_{\FE{\sigma}(M_1)} \bigl( \FE{\sigma}_{f_1}(a, x), \FE{\sigma}_{g_1}(b, y) \bigr) + \Dist{\FEMet{\sigma}}(\varepsilon).
  \end{equation}
  It then follows from Lemma~\ref{lem: Prohorov is stable} that, for any $\mu \in \Prob{\FE{\sigma}(X)}$ and $\nu \in \Prob{\FE{\sigma}(Y)}$,
  \begin{equation}
    \ProhMet{\FE{\sigma}(M_2)} \bigl( (\FE{\sigma}_{f_1})_* \mu, (\FE{\sigma}_{g_2})_* \nu \bigr)
    \leq   
    \ProhMet{\FE{\sigma}(M_1)} \bigl( (\FE{\sigma}_{f_1})_* \mu, (\FE{\sigma}_{g_1})_* \nu \bigr) 
    + 
    \Dist{\FEMet{\sigma}}(\varepsilon).
  \end{equation}
  Now, fix $(P, x) \in \Prob{\sigma(X)} \times X$ and $(Q, y) \in \Prob{\sigma(Y)} \times Y$.
  Setting $\mu = P \otimes \delta_x$ and $\nu = Q \otimes \delta_y$ in the above inequality,
  we deduce the stability of $\FEMet{\tau}$.

  We next prove that $\tau$ is Polish with preserved roots.
  Let $\mathfrak{P} = (\tilde{\sigma}, \eta, (\FS{\tilde{\sigma}}_k)_{k \geq 1})$ be a root-preserving Polish system of $\sigma$.
  Define $\tilde{\tau} \coloneqq \ProbFunct{\tilde{\sigma}}$.
  Then $\tau$ is a topological subfunctor of $\tilde{\tau}$.
  The associated topological embedding $\zeta \colon \tau \Rightarrow \tilde{\tau}$ is given as follows:
  for each $X \in \ob(\BCMcat)$, we define $\zeta_X \coloneqq (\eta_X)_*$,
  i.e, the pushforward map given by $\eta_X$.
  For each $k \geq 1$, we define a subfunctor $\FS{\tilde{\tau}}_k$ of $\FS{\tilde{\tau}}$ as follows:
  for each $(X, \rho_X) \in \ob(\rBCMcat)$, 
  we define 
  \begin{equation} \label{pr eq: 1. Polishness of prob functor in RV}
    \FS{\tilde{\tau}}_k(X)
    \coloneqq 
    \bigl\{
      \mu \in \tilde{\tau}(X) 
      \mid 
      \mu(\FS{\tilde{\sigma}}_k(X, \rho_X)) > 1- 1/k
    \bigr\},
  \end{equation}
  We will prove that $\mathfrak{Q} \coloneqq (\tilde{\tau}, \zeta, (\FS{\tilde{\tau}}_k)_{k \geq 1})$ is a root-preserving Polish system of $\tau$.

  The separability of $\tilde{\tau}$ follows from the separability of the weak topology,
  and the continuity follows from Proposition~\ref{prop: continuity of prob st}.
  We define a space-rooted metrization of $\tilde{\tau}$ as follows:  
  for each $(X, \rho_X) \in \ob(\rBCMcat)$, we equip $\tilde{\tau}(X)$ with the metric  
  \begin{equation}
    d^{\FSMet{\tilde{\tau}}}_{X, \rho_X}(\mu, \nu) 
    \coloneqq  
    \ProhMet{\tilde{\sigma}(X), \rho_X}(\mu, \nu),
  \end{equation}
  where $\ProhMet{\tilde{\sigma}(X), \rho_X}$ denotes the Prohorov metric,  
  with $d^{\FSMet{\tilde{\sigma}}}_{X, \rho_X}$ used as the metric on $\tilde{\sigma}(X)$.
  By the completeness of the Prohorov metric,
  we deduce that the space-rooted metrization $\FSMet{\tilde{\tau}}$ is complete.
  Hence, we obtain \ref{dfn item: 1. Polish functor in RF}.

  Fix $(X, \rho_X) \in \ob(\rBCMcat)$.
  Suppose that a sequence $(\mu_n)_{n \geq 1}$ in $\tilde{\tau}(X) \setminus \FS{\tilde{\tau}}_{k}(X,\rho_X)$ 
  converges weakly to $\mu \in \tilde{\tau}(X)$.
  Since $\FS{\tilde{\sigma}}_{k}(X)$ is open in $\tilde{\sigma}(X)$,
  we deduce from the Portmanteau theorem that 
  \begin{equation} \label{pr eq: 3. Polishness of prob functor in RV}
    \mu( \FS{\tilde{\tau}}_{k}(X) )
    \leq 
    \liminf_{n \to \infty} 
    \mu_n( \FS{\tilde{\tau}}_{k}(X) )
    \leq 
    1 - k^{-1},
  \end{equation}
  which implies that $\mu \in \tilde{\tau}(X) \setminus \FS{\tilde{\tau}}_{k}(X)$.
  Hence, $\FS{\tilde{\tau}}_k$ is an open subfunctor of $\FS{\tilde{\tau}}$.
  Moreover, using the pullback-stability of $\FS{\tilde{\sigma}}_k$,  
  one can show that $\FS{\tilde{\tau}}_k$ is also pullback-stable.  
  Thus, \ref{dfn item: 2. Polish functor in RF} is satisfied. 

  Finally, we verify \ref{dfn item: 3. Polish functor in RF}.
  Fix $(X, \rho_X) \in \ob(\rBCMcat)$ and $\mu \in \bigcap_{k \geq 1} \FS{\tilde{\tau}}_{k}(X, \rho_X)$.
  Since we may assume that $\FS{\tilde{\sigma}}_{k}(X, \rho_X)$ is decreasing to $\eta_X (\sigma(X))$
  (see Remark \ref{3. rem: decreasing Polish system}),
  it follows that 
  \begin{equation}
    \mu \bigl( \eta_X(\sigma(X)) \bigr) 
    = 
    \lim_{k \to \infty} \mu(\FS{\tilde{\sigma}}_k(X))
    = 1. 
  \end{equation}
  Thus, we can define a probability measure $P$ on $\sigma(X)$ by setting 
  $P(\cdot) \coloneqq \mu(\eta_X(\cdot))$.
  It then holds that $\zeta_X(P) = P \circ \eta_X^{-1} = \mu$,
  which implies $\zeta_X(\tau(X)) \supseteq \bigcap_{k \geq 1} \FS{\tilde{\tau}}_k(X, \rho_X)$.
  The converse inclusion is straightforward and hence \ref{dfn item: 3. Polish functor in RF} is satisfied.
  This completes the proof.
\end{proof}

Using Prohorov's theorem (cf.\ \cite[Theorem~5.1]{Billingsley_99_Convergence}),
we obtain the following sufficient condition for precompactness in the space $\rootedBCM(\ProbFunct{\sigma})$

\begin{thm} [Precompactness]\label{thm: precompactness for prob st}
  Assume that $\sigma$ is upper semicontinuous.
  Fix a space-rooted metrization $\FSMet{\sigma}$.
  Fix a non-empty index set $\mathscr{A}$.
  A subset $\{ \cX_\alpha = (X_\alpha, \rho_\alpha, P_\alpha) \mid \alpha \in \mathscr{A} \}$ 
  of $\rootedBCM(\ProbFunct{\sigma})$ is precompact if the following conditions are satisfied.
  \begin{enumerate} [label= \textup{(\roman*)}]
    \item \label{thm item: 1. precompactness for prob st}
      The subset $\{ (X_\alpha, \rho_\alpha) \mid \alpha \in \mathscr{A}\}$ of $\rootedBCM$
      is precompact in the local Gromov--Hausdorff topology.
    \item \label{thm item: 2. precompactness for prob st}
      For every $\varepsilon \in (0,1)$, 
      there exists a precompact subfunctor $\sigma_\varepsilon$ of $\FS{\sigma}$ 
      such that 
      \begin{equation}
        P_\alpha(\sigma_\varepsilon(X_\alpha, \rho_\alpha)) > 1 - \varepsilon,
        \quad 
        \forall \alpha \in \mathscr{A}.
      \end{equation}
  \end{enumerate}
\end{thm}

\begin{proof}
  Assume that \ref{thm item: 1. precompactness for prob st} and \ref{thm item: 2. precompactness for prob st} are satisfied.  
  For each $h > 0$ and $t > 0$, we set
  \begin{equation}
    \delta(h, t) 
    \coloneqq 
    \sup_{\alpha \in \mathscr{A}} w_{d^{\FSMet{\sigma}}_{X_\alpha, \rho_\alpha}}(F_\alpha, h, t).
  \end{equation}
  We then define a subfunctor $\tau'$ of $\FS{\tau}$ as follows:  
  for each $(X, \rho_X) \in \ob(\rBCMcat)$, we let $\tau'(X, \rho_X)$ be the set of all $P \in \tau(X) = \Prob{\sigma(X)}$ such that 
  \begin{equation}
    P(\sigma_\varepsilon(X, \rho_X)) > 1- \varepsilon,
    \quad 
    \forall \varepsilon \in (0,1).
  \end{equation}
  By Prohorov's theorem (cf.\ \cite[Theorem~5.1]{Billingsley_99_Convergence}), 
  the functor $\tau'$ is a precompact subfunctor of $\FS{\tau}$,  
  and $\rootedBCM(\tau')$ contains $\{\cX_\alpha \mid \alpha \in \mathscr{A}\}$.  
  By Proposition~\ref{prop: continuity of prob st},  
  we may apply Theorem~\ref{thm: precompact in RF} to conclude that the family $\{\cX_\alpha \mid \alpha \in \mathscr{A}\}$ is precompact.
\end{proof}

In practice, however, 
a more useful precompactness criterion can be obtained 
by using a tightness criterion for each concrete $\sigma$ and applying Theorem~\ref{thm: precompact in RF} directly. 
Here we give a precompactness criterion in the case where $\sigma = \SkorohodSt(\RNp, \PointFunc)$, 
which is useful for studying convergence of stochastic processes living on different spaces.
We first recall a tightness criterion for probability measures on cadlag functions.
Note that given a random element $\xi$ we denote by $P_\xi$ its underlying probability measure.

\begin{lem} [{Tightness in the Skorohod topology, \cite[Theorem 23.4]{Kallenberg_21_Foundations}}]
  \label{lem: tightness in the Skorohod}
  Fix a dense set $T \subseteq \RNp$,
  a rooted $\bcmAB$ space $(S, \rho_{S})$,
  and a non-empty index set $\mathscr{A}$.
  A family $\{\xi_\alpha\}_{\alpha \in \mathscr{A}}$ of random elements of $D(\RNp, S)$ 
  is tight if and only if the following conditions are satisfied.
  \begin{enumerate} [label = \textup{(\roman*)}, series = tightness for stochastic processes with cadlag paths]
    \item \label{4. lem item: tightness of values of each point}
      For each $t \in T$, 
      it holds that 
      $\displaystyle
        \lim_{r \to \infty} 
        \sup_{\alpha \in \mathscr{A}}
        P_{\xi_\alpha}\left( \xi_\alpha(t) \notin S|_{\rho_S}^{(r)} \right) 
        = 0.
      $
    \item \label{4. lem item: uniform convergence in probability}
      For each $t > 0$, it holds that, for all $\varepsilon > 0$,
      $\displaystyle
        \lim_{h \to 0}
        \sup_{\alpha \in \mathscr{A}} 
        P_{\xi_\alpha}\bigl( w_{d_S}(\xi_\alpha, h, t) > \varepsilon \bigr) 
        = 0
      $.
  \end{enumerate}
  In that case, the following result stronger than \ref{4. lem item: tightness of values of each point} holds.
  \begin{enumerate} [resume* = tightness for stochastic processes with cadlag paths]
    \item \label{4. lem item: tightness of values on compact intervals}
      For each $t \geq 0$, 
      it holds that 
      $\displaystyle
        \lim_{r \to \infty} 
        \sup_{\alpha \in \mathscr{A}} 
        P_{\xi_\alpha}\left( \xi_\alpha(s) \notin S|_{\rho_S}^{(r)},\, \forall s \leq t \right) 
        = 0.
      $
  \end{enumerate}
\end{lem}

Using the above tightness criterion,
we deduce the following useful precompactness criterion for the space $\rootedBCM(\ProbFunct{\SkorohodSt(\RNp, \PointFunc)})$.

\begin{thm} [Precompactness in $\rootedBCM(\ProbFunct{\SkorohodSt(\RNp, \PointFunc)})$]
  Fix a dense set $T \subseteq \RNp$,
  and a non-empty index set $\mathscr{A}$.
  Let $\{\cX_\alpha = (X_\alpha, \rho_\alpha, P_\alpha)\}_{\alpha \in \mathscr{A}}$ 
  be a family of elements of $\rootedBCM(\ProbFunct{\SkorohodSt(\RNp, \PointFunc)})$. 
  Write $\cX_\alpha = (X_\alpha, \rho_\alpha, P_\alpha)$.
  For each $\alpha$,
  we set $\xi_\alpha$ to be a random element whose law coincides with $P_\alpha$.
  Then the family $\{\cX_\alpha\}_{\alpha \in \mathscr{A}}$ is precompact 
  if and only if the following conditions are satisfied.
  \begin{enumerate} [label= \textup{(\roman*)}, series = precompactness for prob st with cadlag st]
    \item \label{thm item: law of stoc. proc., the spaces are pre-cpt}
      The family $\{(X_\alpha, \rho_\alpha)\}_{\alpha \in \mathscr{A}}$ in $\rootedBCM$
      is precompact in the local Gromov--Hausdorff topology.
    \item \label{thm item: law of stoc. proc., values at each point is tight}
      For each $t \in T$, 
      it holds that  
      $\displaystyle
        \lim_{r \to \infty} 
        \sup_{\alpha \in \mathscr{A}} 
        P_{\xi_\alpha}\left( \xi_\alpha(t) \notin X_\alpha|_{\rho_\alpha}^{(r)} \right) 
        = 0.
      $
      \item \label{thm item: law of stoc. proc., uniform convergence in probability}
      For each $t > 0$, it holds that, for all $\varepsilon > 0$,
      $\displaystyle
        \lim_{h \to 0}
        \sup_{\alpha \in \mathscr{A}} 
        P_{\xi_\alpha}( w_{d_{X_\alpha}}(\xi_\alpha, h, t) > \varepsilon) 
        = 0
      $
  \end{enumerate}
  In that case, the following result stronger than \ref{thm item: law of stoc. proc., values at each point is tight} holds.
  \begin{enumerate} [resume* = precompactness for prob st with cadlag st]
    \item \label{4. lem item: law of stoc. proc., tightness of values on compact intervals}
      For each $t \geq 0$, 
      it holds that 
      $\displaystyle
        \lim_{r \to \infty} 
        \sup_{\alpha \in \mathscr{A}} 
        P_{\xi_\alpha}\left( \xi_\alpha(s) \in X_\alpha|_{\rho_\alpha}^{(r)},\, \forall s \leq t \right) 
        = 0.
      $
  \end{enumerate}
\end{thm}

\begin{proof}
  Using Lemma~\ref{lem: tightness in the Skorohod},
  one can prove the desired result similarly to Theorem~\ref{thm: precompact in Skorohod functor}.
\end{proof}

\appendix

\section{Omitted proofs}  \label{appendix: proofs}

In this appendix,
we provide the omitted proofs in Section~\ref{sec: metrization of several topologies}, regarding restriction systems.
Throughout this appendix, we fix a $\bcmAB$ space $X$.

\subsection{Section~\ref{sec: the Fell topology}} \label{appendix: local Hausdorff}

Recall that the restriction system $R$ from $\Closed{X}$ to $\Compact{X}$ is defined as follows:
for each $r > 0$ and $x \in X$,
\begin{equation}
  R^{(r)}_{X, x}(A) = A|_\rho^{(r)} \coloneqq A \cap D_X(x, r),
  \quad 
  A \in \Closed{X}.
\end{equation}
The aim of this subsection is to prove Proposition~\ref{prop: RS for local Hausdorff},
that is, the restriction system is complete and satisfies Condition~4
(recall it from Section~\ref{sec: Metric for non-compact objects}).

The following result asserting the continuity of the restriction system at almost every radius is straightforward.

\begin{lem} \label{lem: RS for local Hausdorff is conti}
  For each $x \in X$ and $A \in \Closed{X}$, the map $(0,\infty) \ni r \mapsto A|_x^{(r)} \in \Compact{X}$ 
  is cadlag and the left limit at $r$ is the closure of $A \cap B_X(x, r)$.
\end{lem}

\begin{proof}
  This is proven similarly to \cite[Lemma 3.2(i)]{Khezeli_20_Metrization}
\end{proof}

\begin{lem} \label{lem: RS for local Hausdorff satisfies 2}
  The restriction system from $\Closed{X}$ to $\Compact{X}$ satisfies Assumption~\ref{assum: metrization of D}\ref{assum item: 2. metrization of D}.
\end{lem}

\begin{proof}
  Let  $(A_n, x_n)$, $n \in \NN \cup \{\infty\}$, be elements of $\Closed{X} \times X$
  such that $x_n \to x_\infty$,
  and 
  $(r_n)_{n \geq 1}$ be an increasing sequence of positive numbers with $r_n \uparrow \infty$.
  Suppose that 
  \begin{equation} \label{pr eq: 1. RS for local Hausdorff satisfies 2}
    \varepsilon_n \coloneqq \HausMet{X}(A_n|_{x_n}^{(r_n)}, A_\infty|_{x_\infty}^{(r_n)}) \to 0.
  \end{equation}
  Fix $r > 0$ which is a continuity point of the map  $s \mapsto A_\infty|_{x_\infty}^{(s)} \in \Compact{X}$.
  Assume that elements $y_n \in A_n|_{x_n}^{(r)}$ converge to some $y \in X$.
  By \eqref{pr eq: 1. RS for local Hausdorff satisfies 2},
  we can find $z_n \in A_\infty$ such that $d_X(y_n, z_n) \leq \varepsilon_n$
  for all sufficiently large $n$.
  It then follows that $y_n \to y$ and $y \in D_X(x_\infty, r)$.
  Since $A$ is closed, we have $y \in A$.
  Hence, $y \in A|_{x_\infty}^{(r)}$,
  which establishes \ref{lem item: 1. convergence in Hausdorff}.
  Next, fix $y \in A_\infty|_{x_\infty}^{(r)}$.
  Since $r$ is a continuity point,
  there exists a sequence $y^{(k)} \in A \cap B_X(x_\infty, r)$ converging to $y$ as $k \to \infty$.
  For each $y^{(k)}$, by \eqref{pr eq: 1. RS for local Hausdorff satisfies 2},
  we can find elements $z^{(k)}_n \in A_n$ converging to $y^{(k)}$ as $n \to \infty$.
  Since $y^{(k)} \in B_X(x_\infty, r)$ and $x_n \to x_\infty$,
  we have $z^{(k)}_n \in A_n|_{x_n}^{(r)}$ for all sufficiently large $n$.
  Hence, one can find a subsequence $(n_k)_{k \geq 1}$ such that 
  $z^{(k)}_{n_k} \to y$ as $k \to \infty$ and $z^{(k)}_{n_k} \in A_{n_k}|_{x_{n_k}}^{(r)}$,
  which establishes \ref{lem item: 2. convergence in Hausdorff}.
  Thus, by Lemma~\ref{lem: convergence in Hausdorff}, we obtain that $A_n|_{x_n}^{(r)} \to A_\infty|_{x_\infty}^{(r)}$ in the Hausdorff topology.
  This completes the proof.
\end{proof}

\begin{lem} \label{lem: RS for local Hausdorff satisfies 4}
  The restriction system from $\Closed{X}$ to $\Compact{X}$ satisfies Assumption~\ref{assum: metrization of D}\ref{assum item: 4. metrization of D}.
\end{lem}

\begin{proof}
  Let $x_n \in X$, $n \in \NN$ be such that $d(x_n, x_{n+1}) < 2^{-n}$
  and $A_n \in \Closed{X}$, $n \in \NN$.
  Fix $r > 0$.
  Write $c \coloneqq \sup_{n \geq 1} d_X(x_n, \rho_1)$, which is finite.
  It then holds that 
  \begin{equation}
    A_n|_{x_n}^{(r)} = A_n \cap D_X(x_n, r) \subseteq D_X(\rho_1, r + c),
    \quad 
    \forall n \in \NN.
  \end{equation}
  Since $D_X(x_\infty, r + c)$ is compact,
  the set $\{A_n|_{x_n}^{(r)}\}_{n \in \NN}$ is precompact in the Hausdorff topology,
  which shows Assumption~\ref{assum: metrization of D}\ref{assum item: 4. metrization of D}.
\end{proof}

\begin{lem} \label{lem: RS for local Hausdorff is complete}
  The restriction system from $\Closed{X}$ to $\Compact{X}$ is complete.
\end{lem}

\begin{proof}
  Let $(r_{k}, A_{k})_{k \geq 1}$ be a compatible sequence rooted at $x \in X$ (see Definition~\ref{dfn: complete RS}).
  Define $A \coloneqq \bigcup_{k \geq 1}A_{k}$.
  If a sequence in $A$ is convergent, then it is contained in $D_X(x, r)$ for some $r > 0$.
  From this and $(r_{k}, A_{k})_{k \geq 1}$ being compatible, one can verify that $A$ is closed.
  It is also easy to check that $A|_x^{(r_{k})} = A_{k}$,
  which shows that $R$ is complete.
\end{proof}

By the above lemmas, we obtain Proposition~\ref{prop: RS for local Hausdorff}.

\begin{proof} [{Proof of Proposition~\ref{prop: RS for local Hausdorff}}]
  By Lemmas~\ref{lem: RS for local Hausdorff is conti}, \ref{lem: RS for local Hausdorff satisfies 2}, 
  \ref{lem: RS for local Hausdorff satisfies 4}, and \ref{lem: RS for local Hausdorff is complete},
  the desired result follows.
\end{proof}


\subsection{Section~\ref{sec: the vague topology}} \label{appendix: vague}

Recall that the restriction system $R$ from $\Meas{X}$ to $\cptMeas{X}$ is defined as follows:
for each $r > 0$ and $x \in X$,
\begin{equation}
  R_x^{(r)}(\mu)(\cdot) = \mu|_x^{(r)}(\cdot) \coloneqq \mu(\cdot \cap D_X(x, r)),
  \quad 
  \mu \in \Meas{X}.
\end{equation}
The aim of this subsection is to prove Proposition~\ref{prop: RS for local Hausdorff},
that is, the restriction system is complete and satisfies Condition~4.

The following result is an analogue of Lemma~\ref{lem: RS for local Hausdorff is conti}.

\begin{lem} \label{lem: RS for vague is conti}
  For each $\rho \in X$ and $\mu \in \cptMeas{X}$, the map $(0,\infty) \ni r \mapsto \mu|_\rho^{(r)} \in \cptMeas{X}$ 
  is cadlag and the left limit at $r$ is $\mu(\cdot \cap B_X(\rho, r))$.
\end{lem}

\begin{proof}
  This is proven similarly to \cite[Lemma 3.2(i)]{Khezeli_20_Metrization}
\end{proof}

\begin{lem} \label{lem: RS for vague satisfies 2}
  The restriction system from $\Meas{X}$ to $\cptMeas{X}$ satisfies Assumption~\ref{assum: metrization of D}\ref{assum item: 2. metrization of D}.
\end{lem}

\begin{proof}
  Let  $(\mu_n, x_n)$, $n \in \NN \cup \{\infty\}$, be elements of $\Meas{X} \times X$
  such that $x_n \to x_\infty$,
  and 
  $(r_n)_{n \geq 1}$ be an increasing sequence of positive numbers with $r_n \uparrow \infty$.
  Suppose that $\ProhMet{X}(\mu_n|_{x_n}^{(r_n)}, \mu_\infty|_{x_\infty}^{(r_n)}) \to 0$.
  Fix $r > 0$ which is a continuity point of the map  $s \mapsto \mu_\infty|_{x_\infty}^{(s)} \in \cptMeas{X}$.
  Fix $\varepsilon > 0$ arbitrarily.
  Since $r$ is a continuity point,
  there exists $\delta \in (0, \varepsilon \wedge (r/2))$ such that 
  \begin{equation} \label{pr eq: 2. RS for vague}
    \ProhMet{X}(\mu_\infty|_{x_\infty}^{(s)}, \mu_\infty|_{x_\infty}^{(r)}) < \varepsilon,
    \quad 
    \forall s \in [r - 2\delta, r + 2\delta].
  \end{equation}
  For all sufficiently large $n$,
  we have that 
  \begin{equation} \label{pr eq: 3. RS for vague}
    r_n  > r,
    \quad
    d_X(x_n, x_\infty) \vee \ProhMet{X}(\mu_n|_{x_n}^{(r_n)}, \mu_\infty|_{x_\infty}^{(r_n)}) < \delta.
  \end{equation}
  Fix such an $n$.
  It is enough to show that 
  \begin{equation} \label{pr eq: 4. RS for vague}
    \ProhMet{X}(\mu_n|_{x_n}^{(r)}, \mu_\infty|_{x_\infty}^{(r)}) \leq 2\varepsilon.
  \end{equation}
  Fix a Borel subset $A \subseteq X$.
  Using \eqref{pr eq: 3. RS for vague}, we deduce that 
  \begin{align}
    \mu_n|_{x_n}^{(r)}(A) &= \mu_n|_{x_n}^{(r_n)}(A \cap D_X(x_n, r))\\
                             &\leq \mu_\infty|_{x_\infty}^{(r_n)}\bigl( (A \cap D_X(x_n, r))^\delta \bigr) + \delta \\
                             &\leq  \mu_\infty|_{x_\infty}^{(r_n)}( A^\delta \cap D_X(x_n, r+\delta) ) + \delta
  \end{align}
  Since $d_X(x_n, x_\infty) < \delta$, we have $D_X(x_n, r + \delta) \subseteq D_X(\rho, r+ 2\delta)$.
  This, combined with \eqref{pr eq: 2. RS for vague}, yields that 
  \begin{equation}
    \mu_\infty|_{x_\infty}^{(r_n)}( A^\delta \cap D_X(x_n, r+\delta) ) 
    \leq     
    \mu_\infty|_{x_\infty}^{(r+2\delta)}( A^\delta )
    \leq    
    \mu_\infty|_{x_\infty}^{(r)}(A^{\delta + \varepsilon}) + \varepsilon,
  \end{equation}
  and hence 
  \begin{equation}
    \mu_n|_{x_n}^{(r)}(A) 
    \leq 
    \mu_\infty|_{x_\infty}^{(r)}(A^{\delta + \varepsilon}) + \varepsilon + \delta
    \leq    
    \mu_\infty|_{x_\infty}^{(r)}(A^{2\varepsilon}) + 2\varepsilon.
  \end{equation}
  Similarly, one can check that 
  \begin{equation}
    \mu_\infty|_{x_\infty}^{(r)}(A) \leq \mu_n|_{x_n}^{(r)}(A^{2\varepsilon}) + 2\varepsilon.
  \end{equation}
  Therefore, we obtain \eqref{pr eq: 4. RS for vague}.
\end{proof}

\begin{lem} \label{lem: RS for vague satisfies 4}
  The restriction system from $\Meas{X}$ to $\cptMeas{X}$ satisfies Assumption~\ref{assum: metrization of D}\ref{assum item: 4. metrization of D}.
\end{lem}

\begin{proof}
  Let $(\mu_n, x_n) \in \Meas{X} \times X$, $n \geq 1$, be such that 
  \begin{equation} \label{pr eq: 5. RS for vague}
    d_X(x_n, x_{n+1}) \vee \ProhMet{X}(\mu_n|_{x_n}^{(r_n)}, \mu_{n+1}|_{x_{n+1}}^{(r_n)}) < 2^{-n},
    \quad 
    \forall n \geq 1.
  \end{equation}
  For $r > 0$ and $n \in \NN$ such that $r_n \geq r + 2^{-n+1}$,
  we have that 
  \begin{align}
    \mu_{n+1}|_{x_{n+1}}^{(r)}(X) = \mu_{n+1}|_{x_{n+1}}^{(r_n)}(D_X(x_{n+1}, r))
                                     \leq \mu_n|_{x_n}^{(r_n)}(D_X(x_{n+1}, r+2^{-n})) + 2^{-n}.
  \end{align}
  Since $d_X(x_n, x_{n+1}) < 2^{-n}$ and $r_n \geq r + 2^{-n+1}$, it holds that 
  \begin{equation}
    D_X(x_{n+1}, r+2^{-n}) \subseteq D_X(x_n, r+2^{-n+1}) \subseteq D_X(x_n, r_n).
  \end{equation}
  Hence,
  \begin{equation} \label{pr eq: 6. RS for vague}
    \mu_{n+1}|_{x_{n+1}}^{(r)}(X) \leq \mu_n|_{x_n}^{(r + 2^{-n+1})}(X) + 2^{-n}
    \quad 
    \text{if}\ r_n \geq r + 2^{-n+1}.
  \end{equation}
  Now, fix $r > 0$ arbitrarily.
  Let $N \in \NN$ be such that $r_N > r+3$.
  For any $k \geq 1$,
  we have from \eqref{pr eq: 6. RS for vague} that 
  \begin{align}
    \mu_{N+k}|_{\rho_{N+k}}^{(r)}(X) 
    &\leq 
    \mu_{N+k-1}|_{\rho_{N+k-1}}^{(r + 2^{-N-k+2})}(X) + 2^{-N-k+1}\\
    &\leq    
    \mu_{N+k-2}|_{\rho_{N+k-2}}^{(r + 2^{-N-k+2} + 2^{-N-k+3})}(X) + 2^{-N-k+1} + 2^{-N-k+2}\\
    &\leq     
    \mu_N|_{x_N}^{(r+ \sum_{l=1}^k 2^{-N-l+2})}(X) + \sum_{l=1}^k 2^{-N-l+1}\\
    &\leq
    \mu_N|_{x_N}^{(r+2)}(X) + 1.
  \end{align}
  Hence, $\{\mu_n|_{x_n}^{(r)}(X)\}_{n \geq 1}$ is bounded.
  By \eqref{pr eq: 5. RS for vague}, we have $d_X(x_1, x_n) \leq 1$ for all $n \in \NN$,
  which implies that the supports of $\mu_n|_{x_n}^{(r)}$ are contained in the compact set $D_X(x_1, r+1)$.
  Therefore, we deduce that $\{\mu_n|_{x_n}^{(r)}\}_{n \geq 1}$ is precompact in the weak topology.
  Obviously, the limit of any convergent subsequence is supported on the compact set $D_X(x_1, r+1)$.
  Thus, $\{\mu_n|_{x_n}^{(r)}\}_{n \geq 1}$ is precompact in $\cptMeas{X}$,
  which shows Assumption~\ref{assum: metrization of D}\ref{assum item: 4. metrization of D}.
\end{proof}

We now complete the proof of Lemma~\ref{lem: RS for vague}.

\begin{proof} [{Proof of Lemma~\ref{lem: RS for vague}}]
  By Lemmas~\ref{lem: RS for vague is conti}, \ref{lem: RS for vague satisfies 2}, and \ref{lem: RS for vague satisfies 4},
  it remains to show the completeness of the restriction system.
  Let $(\mu_k, r_k)_{k \geq 1}$ be a compatible sequence rooted at $x \in X$.
  For any Borel subset $A \subseteq X$ and any $k \leq k'$,
  we have that 
  \begin{equation}  \label{pr eq: 1. RS for vague}
    \mu_k(A) = \mu_{k'}|_x^{(r_k)}(A) = \mu_{k'}(A \cap D_X(x, r_k)) \leq \mu_{k'}(A).
  \end{equation}
  Thus, we can define a Borel measure $\mu$ on $X$ by setting $\mu(\cdot) \coloneqq \lim_{k \to \infty} \mu_k(\cdot)$.
  By \eqref{pr eq: 1. RS for vague},
  we have that $\mu(A) = \mu_k(A)$ for any $A \subseteq D_X(x, r_k)$.
  This yields that $\mu$ is a Radon measure and $\mu|_x^{(r_k)} = \mu_k$ for any $k$.
  Hence, the restriction system is complete.
\end{proof}


\subsection{Section~\ref{sec: variable domains}} \label{appendix: variable domains}

Recall the setting of Section~\ref{sec: variable domains}.
The aim of this appendix is to prove Lemmas~\ref{lem: RS for variable domains} and \ref{lem: RS for marked local Hausdorff}.
We first prove Lemma~\ref{lem: RS for variable domains}.
The following result is an analogue of Lemma~\ref{lem: RS for local Hausdorff is conti}.

\begin{lem} \label{lem: RS for variable domains is conti}
  For each $\rho \in X$ and $f \in \hatC{X}{\Xi}$, the map $(0,\infty) \ni r \mapsto f|_\rho^{(r)} \in \hatCc{X}{\Xi}$ 
  is cadlag and the left limit at $r$ is $f|_{\closure(B_X(\rho, r))}$.
\end{lem}

\begin{proof}
  Fix $r > 0$.
  Let $r_n > r$, $n \geq 1$, be such that $r_n \downarrow r$.
  By Lemma~\ref{lem: RS for local Hausdorff is conti}, 
  the domains $\dom(f|_\rho^{(r_n)}) = \dom(f)|_\rho^{(r_n)}$ converge to $\dom(f)|_\rho^{(r)} = \dom(f|_\rho^{(r)})$ in the Hausdorff topology.
  Let $x_n \in \dom(f|_\rho^{(r_n)})$, $n \geq 1$, be elements converging to an element $x \in \dom(f|_\rho^{(r)})$.
  By the continuity of $f$, it holds that $f(x_n) \to f(x)$.
  Hence, we deduce from Theorem~\ref{thm: convergence in variable domains} that $f|_\rho^{(r_n)} \to f|_\rho^{(r)}$ in $\hatCc{X}{\Xi}$.
  Similarly, one can check that the left limit at $r$ is $f|_{\closure(B_X(\rho, r))}$.
\end{proof}

\begin{lem} \label{lem: RS for variable domains satisfies 2}
  The restriction system from $\hatC{X}{\Xi}$ to $\hatCc{X}{\Xi}$ satisfies Assumption~\ref{assum: metrization of D}\ref{assum item: 2. metrization of D}.
\end{lem}

\begin{proof}
  Using Lemma~\ref{lem: RS for variable domains is conti},  
  the result can be proved in the same way as in the proof of Lemma~\ref{lem: RS for local Hausdorff satisfies 2}.
\end{proof}

\begin{lem} \label{lem: RS for variable domains satisfies 3}
  The restriction system from $\hatC{X}{\Xi}$ to $\hatCc{X}{\Xi}$ satisfies Assumption~\ref{assum: metrization of D}\ref{assum item: 3. metrization of D}.
\end{lem}

\begin{proof}
  Let $(f_n, x_n) \in \hatC{X}{\Xi} \times X$, $n \in \NN \cup \{\infty\}$ and $r > 0$
  be such that $x_n \to x_\infty$ and $f_n|_{x_n}^{(r)} \to f_\infty|_{x_\infty}^{(r)}$ in $\hatCc{X}{\Xi}$.
  For each $s \leq r$, one can easily verify that $\{f_n|_{x_n}^{(s)}\}_{n \in \NN}$ is precompact in $\hatCc{X}{\Xi}$
  by using Theorem~\ref{thm: precompactness uniform in variable domains}.
  This completes the proof.
\end{proof}

\begin{proof} [{Proof of Lemma~\ref{lem: RS for variable domains} }]
  The desired result follows from Lemmas~\ref{lem: RS for variable domains is conti}, \ref{lem: RS for variable domains satisfies 2}, and \ref{lem: RS for variable domains satisfies 3}.
\end{proof}

Next, we prove Lemma~\ref{lem: RS for marked local Hausdorff}.

\begin{proof} [{Proof of Lemma~\ref{lem: RS for marked local Hausdorff}}]
  In a similar manner to Lemmas~\ref{lem: RS for local Hausdorff is conti} and~\ref{lem: RS for local Hausdorff is complete}, 
  one can verify that the restriction system is complete and Condition~2.
  Now, assume that $d_\Xi$ is complete.
  It remains to prove Assumption~\ref{assum: metrization of D}\ref{assum item: 4. metrization of D}.
  Let $(E_n, x_n) \in \GraphSp{X}{\Xi}$, $n \geq 1$, be such that 
  \begin{equation} \label{pr eq: 1. RS for marked local Hausdorff}
    d_X(x_n, x_{n+1}) \vee \HausMet{X \times \Xi}(E_n|_{x_n}^{(r_n, *)}, E_{n+1}|_{x_{n+1}}^{(r_{n+1}, *)}) < 2^{-n},
    \quad 
    \forall n \geq 1.
  \end{equation}
  Fix $r > 0$ and $n \in \NN$ such that $r_n \geq r + 2^{-n+1}$.
  For each $(x, a) \in E_{n+1}|_{x_{n+1}}^{(r, *)}$,
  by \eqref{pr eq: 1. RS for marked local Hausdorff},
  we can find $(y, b) \in E_n|_{x_n}^{(r_n, *)}$ satisfying
  \begin{equation}
    d_X(x,y) \vee d_\Xi(a, b) < 2^{-n}.
  \end{equation}
  The triangle inequality yields that 
  \begin{equation}
    d_X(x_{n+1}, y) \leq d_X(x_{n+1}, x_n) + d_X(x_n, x) + d_X(x, y) \leq r + 2 \cdot 2^{-n},
  \end{equation}
  which implies that $(y, b) \in E_n|_{x_n}^{(r+2\cdot 2^{-n}, *)}$.
  It thus follows that 
  \begin{equation}
    E_{n+1}|_{x_{n+1}}^{(r)} \subseteq \bigl( E_n|_{x_n}^{(r+2\cdot 2^{-n}, *)} \bigr)^{2^{-n}}
    \quad 
    \text{for any}\ r_n  \geq r,
  \end{equation}
  where we recall the $\varepsilon$-neighborhood from \eqref{2. eq: e-neighborhood}.
  Now, fix $r > 0$ arbitrarily.
  Let $N \in \NN$ be such that $r_N > r+3$.
  For any $k \geq 1$,
  we have from \eqref{pr eq: 6. RS for vague} that 
  \begin{align}
    E_{N+k}|_{\rho_{N+k}}^{(r, *)}
    &\subseteq 
    \bigl( E_{N+k-1}|_{\rho_{N+k-1}}^{(r + 2 \cdot 2^{-N-k+1}, *)} \bigr)^{2^{-N-k+1}}\\
    &\subseteq    
    \Bigl( \bigl( E_{N+k-2}|_{\rho_{N+k-2}}^{(r + 2 \cdot 2^{-N-k+1} + 2 \cdot 2^{-N-k+2}, *)} \bigr)^{2^{-N-k+2}} \Bigr)^{2^{-N-k+1}}\\
    &\subseteq    
    \bigl( E_{N+k-2}|_{\rho_{N+k-2}}^{(r + 2 \cdot 2^{-N-k+1} + 2 \cdot 2^{-N-k+2}, *)} \bigr)^{2^{-N+k+2} + 2^{-N-k+1}}\\
    &\subseteq
    \bigl( E_N|_{x_N}^{(r+ 2 \cdot \sum_{l=1}^k 2^{-N-l + 1}, *)} \bigr) ^{\sum_{l=1}^k 2^{-N-l+1}}\\
    &\subseteq
    \bigl( E_N|_{x_N}^{(r+2^{-N+2}, *)} \bigr)^{2^{-N+1}}.
  \end{align}
  From this,
  we deduce that $U \coloneqq \bigcup_{n \geq 1} E_{n}|_{\rho_{n}}^{(r, *)}$ is totally bounded.
  Since the max product metric $d_{X \times \Xi}$ is complete,
  it follows that $K \coloneqq \closure(U)$ is compact in $X \times \Xi$.
  All the sets $E_n|_{x_n}^{(r,*)}$, $n \geq 1$, are contained in $K$,
  and thus $\{E_n|_{x_n}^{(r,*)}\}_{n \geq 1}$ is precompact in $\Compact{X \times \Xi}$,
  which shows that Assumption~\ref{assum: metrization of D}\ref{assum item: 4. metrization of D} is satisfied.
\end{proof}


\section{Metrization of Gromov--Hausdorff-type topologies on rooted compact metric spaces}  \label{appendix: GH metrization in the compact case}

In \cite[Section~2]{Khezeli_23_A_unified},
Khezeli established a method for metrization of Gromov--Hausdorff-type topologies for compact metric spaces 
that are not necessarily rooted. 
Applying this method to rooted compact metric spaces, 
one obtains a metric that induces the GH-type topology with non-preserved roots. 
On the other hand, one can also consider the GH-type topology with preserved roots, as discussed in Section~\ref{sec: main results}. 
In this appendix,
we present the modifications of our main results in the compact setting.
Since the proofs are identical to those in the boundedly-compact case, we omit them here.

We first define two categories as follows:
\begin{itemize}
  \item \label{item: rCM, category}
    $\rCMcat$ denotes the category whose objects are rooted compact metric spaces  
    and whose morphisms are root-preserving isometric embeddings;
  \item \label{item: CM, category}
    $\CMcat$ denotes the category whose objects are compact metric spaces and whose morphisms are isometric embeddings.
\end{itemize}
Fix a functor $\tau \colon \CMcat \to \MTopcat$,
which we refer to as a \emph{structure on compact metric spaces}.
The notion for structures introduced in Section~\ref{sec: main results},
such as metrization, stability, and continuity, applies to $\tau$ with the obvious modifications, 
e.g., replacing $\BCMcat$ by $\CMcat$ wherever they occur in the relevant definitions.
Accordingly, in this appendix, the notation introduced in that section also applies to $\tau$.
For example, we set $\FE{\tau} \coloneqq \tau \times \Gamma_{\CMcat \to \MTopcat}$,
where $\Gamma_{\CMcat \to \MTopcat} \colon \CMcat \to \MTopcat$ denotes the forgetful functor.
Assume that $\tau$ admits an element-rooted metrization $\FEMet{\tau}$, i.e.,
a functor from $\CMcat$ to $\Metcat$ making the following diagram commute.
\begin{equation} \label{dfn eq: cpt ER metrization}
  \begin{tikzcd}
                                                                & \Metcat \arrow[d] \\
    \CMcat \arrow[r, "\FE{\tau}"'] \arrow[ru, "\FEMet{\tau}"]  & \MTopcat 
  \end{tikzcd}
\end{equation}
Here, the vertical arrow denotes the forgetful functor.
We write $\FSMet{\tau} \colon \rCMcat \to \Metcat$ for the associated space-rooted metrization of $\tau$.

Define $\rfrakK(\tau)$ to be the set of rooted-$\tau$-isometric equivalence classes of $(X, \rho_X, a_X)$,
where $(X, \rho_x)$ is a rooted compact metric spaces and $a_X \in \tau(X)$
(cf.\ Proposition~\ref{prop: existence of rootedBCM with additional elements}).
Following Definitions~\ref{dfn: Generalized RF metric} and \ref{dfn: Generalized RV metric},
we introduce two functions that measure the distances between elements of $\rfrakK(\tau)$.

\begin{dfn} \label{dfn: cpt Generalized RF metric}
  We define, for each $\cX \allowbreak =\allowbreak (X, \rho_X, a_X),\,\allowbreak \cY \allowbreak =\allowbreak (Y, \rho_Y, a_Y) \in \rfrakK(\tau)$,
  \begin{equation}
    \cRFMet^{\FSMet{\tau}} (\cX, \cY)
    \coloneqq
    \inf_{f, g, Z}
    \Bigl\{
      \HausMet{Z}(f(X), g(Y)) \vee  d^{\FSMet{\tau}}_{Z, \rho_Z}(\tau_{f}(a_X), \tau_{g}(a_{Y}))  
    \Bigr\},
  \end{equation}
  where the infimum is taken 
  over all rooted compact metric spaces $(Z, \rho_{Z})$ 
  and root-preserving isometric embeddings $f \colon  X \to Z$ and $g \colon Y \to Z$.
\end{dfn}

\begin{dfn} \label{dfn: cpt Generalized RV metric}
  We define, for $\cX=(X, \rho_X, a_X), \cY=(Y, \rho_Y, a_Y) \in \rfrakK(\tau)$,
  \begin{equation}  \label{dfn eq: cpt element-rooted GH-type metric}
    \cRVMet^{\FEMet{\tau}}(\cX, \cY)
    \coloneqq
    \inf_{f, g, Z}
    \Bigl\{
      \HausMet{Z} \bigl( f(X), g(Y) \bigr)
      \vee
      d_Z(f(\rho_X), g(\rho_Y)) 
      \vee
      d^{\FEMet{\tau}}_Z \bigl( \FE{\tau}_f(a_X, \rho_X), \FE{\tau}_g(a_Y, \rho_Y) \bigr)
    \Bigr\}
    ,
  \end{equation}
  where the infimum is taken 
  over all compact metric spaces $Z$ 
  and isometric embeddings $f \colon  X \to Z$ and $g \colon Y \to Z$.
\end{dfn}

\begin{rem}
  The function $\cRVMet^{\FEMet{\tau}}$ coincides with the metric 
  obtained by applying Khezeli's framework \cite[Section~2]{Khezeli_23_A_unified}.
\end{rem}

One can easily verify analogues of Theorems~\ref{thm: RF convergence} and \ref{thm: RV convergence}, as shown below.

\begin{thm}
  Let $\cX_n = (X_n, \rho_{X_n}, a_{X_n}), n \in \NN \cup \{\infty\}$ be elements of $\rfrakK(\tau)$.
  The following statements are equivalent with each other:
  \begin{enumerate} [label = \textup{(\roman*)}]
    \item $\cX_n \to \cX_\infty$ with respect to $\cRFMet^{\FSMet{\tau}}$;
    \item there exist a rooted compact metric space $(K, \rho_K)$
    and root-preserving isometric embeddings $f_n \colon X_n \to K$, $n \in \NN \cup \{\infty\}$,
    such that $f_n(X_n) \to f_\infty(X_\infty)$ in the Hausdorff topology as subsets of $K$,
    and $\tau_{f_n}(a_{X_n}) \to \tau_{f_\infty}(a_{X_\infty})$ in $\tau(K)$.
  \end{enumerate}
\end{thm}

\begin{thm}
  Let $\cX_n = (X_n, \rho_{X_n}, a_{X_n}), n \in \NN \cup \{\infty\}$ be elements of $\rfrakK(\tau)$.
  The following statements are equivalent with each other:
  \begin{enumerate} [label = \textup{(\roman*)}]
    \item $\cX_n \to \cX_\infty$ with respect to $\cRVMet^{\FEMet{\tau}}$;
    \item there exist a compact metric space $K$
    and isometric embeddings $f_n \colon X_n \to K$, $n \in \NN \cup \{\infty\}$,
    such that $f_n(X_n) \to f_\infty(X_\infty)$ in the Hausdorff topology as subsets of $K$,
    $f_n(\rho_{X_n}) \to f_\infty(\rho_{X_{\infty}})$ in $K$,
    and $\tau_{f_n}(a_{X_n}) \to \tau_{f_\infty}(a_{X_\infty})$ in $\tau(K)$.
  \end{enumerate}
\end{thm}

Similarly to Theorems~\ref{thm: space-rooted metric} and \ref{thm: element-rooted metric},
the embedding-continuity ensures that the above-defined functions are metrics.

\begin{thm}
  If $\tau$ is embedding-continuous, 
  then the functions $\cRFMet^{\FSMet{\tau}}$ and $\cRVMet^{\FEMet{\tau}}$ are metrics on $\rfrakK(\tau)$.
\end{thm}

\begin{dfn} 
  When $\tau$ is embedding-continuous,
  we refer to the topology on $\rootedBCM(\tau)$ induced by $\cRFMet^{\FSMet{\tau}}$ (resp.\ $\cRVMet^{\FEMet{\tau}}$) 
  as the \emph{GH-type topology with preserved (resp.\ non-preserved) roots}.
\end{dfn}

As seen in Section~\ref{sec: Coincidence of the two topologies},
the stability of $\FEMet{\tau}$ suffices to ensure that the topologies defined above on $\rfrakK(\tau)$ coincide.
For clarity, we record the definition of stability for $\FEMet{\tau}$ in the compact case below.

\begin{dfn} \label{dfn: stability in cpt}
  We say that the functor $\FEMet{\tau}$ is \emph{stable} 
  if and only if there exists a function $\Dist{\FEMet{\tau}} \colon \RNp \to \RNp$ satisfying the following conditions.
  \begin{enumerate} [label = \textup{(\roman*)}, leftmargin = *]
    \item \label{dfn item: 1, cpt stability}
      It holds that $\lim_{\varepsilon \to 0} \Dist{\FEMet{\tau}}(\varepsilon) = \Dist{\FEMet{\tau}}(0) = 0$.
    \item \label{dfn item: 2, cpt stability}
      Fix compact metric spaces $X, Y, K_1,$ and $K_2$.
      Let $f_i \colon X \to K_i$ and $g_i \colon Y \to K_i$, $i = 1,2$, be isometric embeddings.
      If there exists $\varepsilon \in \RNp$ such that, for all $x \in X$ and $y \in Y$,
      \begin{equation} \label{assum item eq: 1, cpt stability}
        d_{K_2}(f_2(x), g_2(y)) \leq d_{K_1}(f_1(x), g_1(y)) + \varepsilon,
      \end{equation}
      then, for all $(a,x) \in \FE{\tau}(X)$ and $(b, y) \in \FE{\tau}(Y)$,
      \begin{equation} \label{assum item eq: 2, cpt stability}
        d^{\FEMet{\tau}}_{K_2}\bigl( \FE{\tau}_{f_2}(a, x), \FE{\tau}_{g_2}(b, y) \bigr) 
        \leq     
        d^{\FEMet{\tau}}_{K_1}\bigl( \FE{\tau}_{f_1}(a, x), \FE{\tau}_{g_1}(b, y) \bigr)  
        + 
        \Dist{\FEMet{\tau}}(\varepsilon),
      \end{equation}
      where $d^{\FEMet{\tau}}_{K_1}$ and $d^{\FEMet{\tau}}_{K_2}$ denote the metrics on $\FE{\tau}(K_1)$ and $\FE{\tau}(K_2)$, 
      respectively.
  \end{enumerate}
  We call such a function $\Dist{\FEMet{\tau}}$ a \emph{distortion} of $\FEMet{\tau}$.
\end{dfn}

The following can be deduced by the same argument as in Theorem~\ref{thm: coincidence of RF and RV}.

\begin{thm} \label{thm: cpt coincidence}
  If $\FEMet{\tau}$ is stable with distortion $\Dist{\FEMet{\tau}}$, then 
  \begin{equation}  \label{thm eq: 1, cpt coincidence}
    \cRVMet^{\FEMet{\tau}}(\cX, \cY) 
    \leq
    \cRFMet^{\FEMet{\tau}}(\cX, \cY) 
    \leq     
    2\,\cRVMet^{\FEMet{\tau}}(\cX, \cY)    
    +
    \Dist{\FEMet{\tau}} \bigl( 2\, \cRVMet^{\FEMet{\tau}}(\cX, \cY) \bigr),
    \quad
    \forall \cX, \cY \in \rfrakK(\tau).
  \end{equation}
  In particular, under this assumption, the GH-type topologies with preserved roots and with non-preserved roots coincide.
\end{thm}

The notion of semicontinuity introduced in Definition~\ref{dfn: Semicontinuity} can be applied to $\tau$ 
in the compact setting considered here, after minor modifications.
The only change is to replace, in Assumption~\ref{assum: semicontinuity}, 
``$X_n \to X_\infty$ in the Fell topology'' with ``$X_n \to X_\infty$ in the Hausdorff topology''.
Imposing semicontinuity on $\tau$ then allows one to repeat the arguments of Section~\ref{sec: topological properties} and, 
in particular, to verify the Polishness of the GH-type topology with preserved roots.

Since $\CMcat$ is a subcategory of $\BCMcat$,
any functor $\tau \colon \BCMcat \to \MTopcat$ defines, by restriction, a functor $\tau^c \colon \CMcat \to \MTopcat$.
It is straightforward to check that if $\tau$ is embedding-continuous (resp.\ stably metrizable, upper/lower semicontinuous),
then so is $\tau^c$ in the sense discussed above.
In particular, 
for any functor $\tau$ (and products of such functors) introduced in Section~\ref{sec: Examples of functors},
the GH-type topologies with preserved roots and non-preserved roots on $\rfrakK(\tau^c)$ coincide, 
and the common topology is Polish.

\section*{Acknowledgement}
I would like to thank my supervisor, Dr David Croydon, for his support and fruitful discussions,  
and Dr Ali Khezeli for his valuable advice, which has significantly improved the results.
This work was supported by 
JSPS KAKENHI Grant Number JP 24KJ1447
and 
the Research Institute for Mathematical Sciences, 
an International Joint Usage/Research Center located in Kyoto University.

\bibliographystyle{amsplain}
\bibliography{references_Gromov}

\providecommand{\bysame}{\leavevmode\hbox to3em{\hrulefill}\thinspace}
\providecommand{\MR}{\relax\ifhmode\unskip\space\fi MR }
\providecommand{\MRhref}[2]{%
  \href{http://www.ams.org/mathscinet-getitem?mr=#1}{#2}
}
\providecommand{\href}[2]{#2}
\begin{thebibliography}{10}

\bibitem{Abraham_Delmas_Hoscheit_13_A_note}
R.~Abraham, J.-F. Delmas, and P.~Hoscheit, \emph{A note on the
  {G}romov-{H}ausdorff-{P}rokhorov distance between (locally) compact metric
  measure spaces}, Electron. J. Probab. \textbf{18} (2013), no. 14, 21.
  \MR{3035742}

\bibitem{Abraham_Delmas_Hoscheit_14_Exit}
\bysame, \emph{Exit times for an increasing {L}\'{e}vy tree-valued process},
  Probab. Theory Related Fields \textbf{159} (2014), no.~1-2, 357--403.
  \MR{3201925}

\bibitem{Berry_Broutin_Goldschmidt_12_The_continuum}
L.~Addario-Berry, N.~Broutin, and C.~Goldschmidt, \emph{The continuum limit of
  critical random graphs}, Probab. Theory Related Fields \textbf{152} (2012),
  no.~3-4, 367--406. \MR{2892951}

\bibitem{Angel_Croydon_Hernandez-Torres_Shiraishi_21_Scaling}
O.~Angel, D.~A. Croydon, S.~Hernandez-Torres, and D.~Shiraishi, \emph{Scaling
  limits of the three-dimensional uniform spanning tree and associated random
  walk}, Ann. Probab. \textbf{49} (2021), no.~6, 3032--3105. \MR{4348685}

\bibitem{Archer_Croydon_23_Scaling}
E.~Archer and D.~A. Croydon, \emph{Scaling limit of critical percolation
  clusters on hyperbolic random half-planar triangulations and the associated
  random walks}, 2023, Preprint. {A}vailable at arXiv:2311.11993.

\bibitem{Athreya_Lohr_Winter_16_The_gap}
S.~Athreya, W.~L\"{o}hr, and A.~Winter, \emph{The gap between {G}romov-vague
  and {G}romov-{H}ausdorff-vague topology}, Stochastic Process. Appl.
  \textbf{126} (2016), no.~9, 2527--2553. \MR{3522292}

\bibitem{Athreya_Lohr_Winter_17_Invariance}
\bysame, \emph{Invariance principle for variable speed random walks on trees},
  Ann. Probab. \textbf{45} (2017), no.~2, 625--667. \MR{3630284}

\bibitem{Aubin_Frankowska_09_Set}
J.-P. Aubin and H.~Frankowska, \emph{Set-valued analysis}, Modern Birkh\"auser
  Classics, Birkh\"auser Boston, Inc., Boston, MA, 2009, Reprint of the 1990
  edition [MR1048347]. \MR{2458436}

\bibitem{Barlow_Croydon_Kumagai_17_Subsequential}
M.~T. Barlow, D.~A. Croydon, and T.~Kumagai, \emph{Subsequential scaling limits
  of simple random walk on the two-dimensional uniform spanning tree}, Ann.
  Probab. \textbf{45} (2017), no.~1, 4--55. \MR{3601644}

\bibitem{Arous_Cabezas_Fribergh_19_Scaling_theory}
G.~Ben~Arous, M.~Cabezas, and A.~Fribergh, \emph{Scaling limit for the ant in
  high-dimensional labyrinths}, Comm. Pure Appl. Math. \textbf{72} (2019),
  no.~4, 669--763. \MR{3914881}

\bibitem{Billingsley_99_Convergence}
P.~Billingsley, \emph{Convergence of probability measures}, second ed., Wiley
  Series in Probability and Statistics: Probability and Statistics, John Wiley
  \& Sons, Inc., New York, 1999, A Wiley-Interscience Publication. \MR{1700749}

\bibitem{Renaudie_Broutin_Nachmias_pre_The_scaling}
A.~Blanc-Renaudie, N.~Broutin, and A.~Nachmias, \emph{The scaling limit of
  critical hypercube percolation}, 2024, Preprint. {A}vailable at
  arXiv:2401.16365.

\bibitem{Burago_Burago_Ivanov_01_A_course}
D.~Burago, Y.~Burago, and S.~Ivanov, \emph{A course in metric geometry},
  Graduate Studies in Mathematics, vol.~33, American Mathematical Society,
  Providence, RI, 2001. \MR{1835418}

\bibitem{Cao_23_Convergence}
S.~Cao, \emph{Convergence of energy forms on {S}ierpinski gaskets with added
  rotated triangle}, Potential Anal. \textbf{59} (2023), no.~4, 1793--1825.
  \MR{4684376}

\bibitem{Cohn_13_Measure}
D.~L. Cohn, \emph{Measure theory}, second ed., Birkh\"auser Advanced Texts:
  Basler Lehrb\"ucher. [Birkh\"auser Advanced Texts: Basel Textbooks],
  Birkh\"auser/Springer, New York, 2013. \MR{3098996}

\bibitem{Croydon_18_Scaling}
D.~A. Croydon, \emph{Scaling limits of stochastic processes associated with
  resistance forms}, Ann. Inst. Henri Poincar\'{e} Probab. Stat. \textbf{54}
  (2018), no.~4, 1939--1968. \MR{3865663}

\bibitem{Croydon_Hambly_Kumagai_12_Convergence}
D.~A. Croydon, B.~M. Hambly, and T.~Kumagai, \emph{Convergence of mixing times
  for sequences of random walks on finite graphs}, Electron. J. Probab.
  \textbf{17} (2012), no. 3, 32. \MR{2869250}

\bibitem{Croydon_Muirhead_15_Functional}
D.A. Croydon and S.~Muirhead, \emph{Functional limit theorems for the
  {B}ouchaud trap model with slowly varying traps}, Stochastic Process. Appl.
  \textbf{125} (2015), no.~5, 1980--2009. \MR{3315620}

\bibitem{Daley_Jones_03_Vol_1}
D.~J. Daley and D.~Vere-Jones, \emph{An introduction to the theory of point
  processes. {V}ol. {I}}, second ed., Probability and its Applications (New
  York), Springer-Verlag, New York, 2003, Elementary theory and methods.
  \MR{1950431}

\bibitem{Fontes_Mathieu_14_On}
L.~R.~G. Fontes and P.~Mathieu, \emph{On the dynamics of trap models in {$\Bbb
  Z^d$}}, Proc. Lond. Math. Soc. (3) \textbf{108} (2014), no.~6, 1562--1592.
  \MR{3218319}

\bibitem{Gromov_07_Metric}
M.~Gromov, \emph{Metric structures for {R}iemannian and non-{R}iemannian
  spaces}, english ed., Modern Birkh\"{a}user Classics, Birkh\"{a}user Boston,
  Inc., Boston, MA, 2007, Based on the 1981 French original, With appendices by
  M. Katz, P. Pansu and S. Semmes, Translated from the French by Sean Michael
  Bates. \MR{2307192}

\bibitem{Gwynne_Miller_17_Scaling}
E.~Gwynne and J.~Miller, \emph{Scaling limit of the uniform infinite half-plane
  quadrangulation in the {G}romov-{H}ausdorff-{P}rokhorov-uniform topology},
  Electron. J. Probab. \textbf{22} (2017), Paper No. 84, 47. \MR{3718712}

\bibitem{Jacod_Shiryaev_03_Limit}
J.~Jacod and A.~N. Shiryaev, \emph{Limit theorems for stochastic processes},
  second ed., Grundlehren der mathematischen Wissenschaften [Fundamental
  Principles of Mathematical Sciences], vol. 288, Springer-Verlag, Berlin,
  2003. \MR{1943877}

\bibitem{Jansen_17_Notes}
D.~Jansen, \emph{Notes on pointed {G}romov--{H}ausdorff convergence}, 2017,
  Preprint. {A}vailable at arXiv:1703.09595.

\bibitem{Kallenberg_17_Random}
O.~Kallenberg, \emph{Random measures, theory and applications}, Probability
  Theory and Stochastic Modelling, vol.~77, Springer, Cham, 2017. \MR{3642325}

\bibitem{Kallenberg_21_Foundations}
\bysame, \emph{Foundations of modern probability}, third ed., Probability
  Theory and Stochastic Modelling, vol.~99, Springer, Cham, [2021] \copyright
  2021. \MR{4226142}

\bibitem{Khezeli_20_Metrization}
A.~Khezeli, \emph{Metrization of the {G}romov-{H}ausdorff (-{P}rokhorov)
  topology for boundedly-compact metric spaces}, Stochastic Process. Appl.
  \textbf{130} (2020), no.~6, 3842--3864. \MR{4092421}

\bibitem{Khezeli_23_A_unified}
\bysame, \emph{A unified framework for generalizing the {G}romov-{H}ausdorff
  metric}, Probab. Surv. \textbf{20} (2023), 837--896. \MR{4671147}

\bibitem{Lawler_Limic_10_Random}
G.~F. Lawler and V.~Limic, \emph{Random walk: a modern introduction}, Cambridge
  Studies in Advanced Mathematics, vol. 123, Cambridge University Press,
  Cambridge, 2010. \MR{2677157}

\bibitem{Lawler_Werner_04_The_Brownian}
G.~F. Lawler and W.~Werner, \emph{The {B}rownian loop soup}, Probab. Theory
  Related Fields \textbf{128} (2004), no.~4, 565--588. \MR{2045953}

\bibitem{LeGall_06_Random}
J.-F. Le~Gall, \emph{Random real trees}, Ann. Fac. Sci. Toulouse Math. (6)
  \textbf{15} (2006), no.~1, 35--62. \MR{2225746}

\bibitem{LeGall_Miermont_12_Scaling}
J.-F. Le~Gall and G.~Miermont, \emph{Scaling limits of random trees and planar
  maps}, Probability and statistical physics in two and more dimensions, Clay
  Math. Proc., vol.~15, Amer. Math. Soc., Providence, RI, 2012, pp.~155--211.
  \MR{3025391}

\bibitem{Li_22_Measure}
Z.~Li, \emph{Measure-valued branching {M}arkov processes}, second ed.,
  Probability Theory and Stochastic Modelling, vol. 103, Springer, Berlin,
  [2022] \copyright 2022. \MR{4704078}

\bibitem{Marcus_Rosen_06_Markov}
M.~B. Marcus and J.~Rosen, \emph{Markov processes, {G}aussian processes, and
  local times}, Cambridge Studies in Advanced Mathematics, vol. 100, Cambridge
  University Press, Cambridge, 2006. \MR{2250510}

\bibitem{Miermont_09_Tessellations}
G.~Miermont, \emph{Tessellations of random maps of arbitrary genus}, Ann. Sci.
  \'{E}c. Norm. Sup\'{e}r. (4) \textbf{42} (2009), no.~5, 725--781.
  \MR{2571957}

\bibitem{Molchanov_17_Theory}
I.~Molchanov, \emph{Theory of random sets}, second ed., Probability Theory and
  Stochastic Modelling, vol.~87, Springer-Verlag, London, 2017. \MR{3751326}

\bibitem{Morariu-Patrichi_18_On}
M.~Morariu-Patrichi, \emph{On the weak-hash metric for boundedly finite
  integer-valued measures}, Bull. Aust. Math. Soc. \textbf{98} (2018), no.~2,
  265--276. \MR{3849587}

\bibitem{Munkres_00_Topology}
J.~R. Munkres, \emph{Topology}, second ed., Prentice Hall, Inc., Upper Saddle
  River, NJ, 2000. \MR{3728284}

\bibitem{Noda_pre_Convergence}
R.~Noda, \emph{Convergence of local times of stochastic processes associated
  with resistance forms}, 2023, Preprint. {A}vailable at arXiv:2305.13224.

\bibitem{Noda_pre_Scaling}
\bysame, \emph{Scaling limits of discrete-time {M}arkov chains and their local
  times on electrical networks}, 2024, Preprint. {A}vailable at
  arXiv:2405.01871.

\bibitem{Schneider_Weil_Stochastic}
R.~Schneider and W.~Weil, \emph{Stochastic and integral geometry}, Probability
  and its Applications (New York), Springer-Verlag, Berlin, 2008. \MR{2455326}

\bibitem{Sheffield_16_Conformal}
S.~Sheffield, \emph{Conformal weldings of random surfaces: {SLE} and the
  quantum gravity zipper}, Ann. Probab. \textbf{44} (2016), no.~5, 3474--3545.
  \MR{3551203}

\bibitem{Srivastava_98_A_Course}
S.~M. Srivastava, \emph{A course on {B}orel sets}, Graduate Texts in
  Mathematics, vol. 180, Springer-Verlag, New York, 1998. \MR{1619545}

\bibitem{Cech_69_Point}
E.~\v{C}ech, \emph{Point sets}, Academic Press, New York-London; Academia
  [Publishing House of the Czechoslovak Academy of Sciences], Prague, 1969,
  Translated from the Czech by Ale\v s{} Pultr. \MR{256344}

\bibitem{Whitt_80_Some}
W.~Whitt, \emph{Some useful functions for functional limit theorems}, Math.
  Oper. Res. \textbf{5} (1980), no.~1, 67--85. \MR{561155}

\bibitem{Willard_70_General}
S.~Willard, \emph{General topology}, Addison-Wesley Publishing Co., Reading,
  Mass.-London-Don Mills, Ont., 1970. \MR{264581}

\bibitem{Williamson_Janos_87_Construction}
R.~Williamson and L.~Janos, \emph{Constructing metrics with the {H}eine-{B}orel
  property}, Proc. Amer. Math. Soc. \textbf{100} (1987), no.~3, 567--573.
  \MR{891165}

\end{thebibliography}

\end{document}